\documentclass{article}
\usepackage{float}
\usepackage{cite}
\usepackage{shapepar}
\usepackage{listings}
\usepackage{amsmath}
\usepackage{amsthm}
\usepackage{mathtools}
\usepackage{graphics}
\usepackage{amssymb}
\usepackage{makecell}
\usepackage{mathrsfs}
\usepackage{cite}
\usepackage{framed}
\usepackage{diagbox}
\usepackage{booktabs}
\usepackage{fancybox}
\usepackage{geometry}
\usepackage{multirow}
\usepackage{enumerate}
\usepackage{caption}
\usepackage{subcaption}
\usepackage{hyperref}
\usepackage{boldline}
\usepackage[table]{xcolor}
\usepackage{slashbox}
\usepackage{tabularray}
\usepackage{vcell}
\usepackage{titlesec}
\numberwithin{equation}{section}

  \def\ss{\smallskip}
\newcommand{\q}{\quad}    

\titleformat{\paragraph}
{\normalfont\normalsize\bfseries}{\theparagraph}{1em}{}
\titlespacing*{\paragraph}
{0pt}{3.25ex plus 1ex minus .2ex}{1.5ex plus .2ex}

\newtheorem{theorem}{Theorem}
\usepackage[ruled,linesnumbered]{algorithm2e}
\newtheorem{lemma}{Lemma}[section]
\newtheorem{remark}{Remark}

\newcommand{\subsubsubsection}[1]{\paragraph{#1}}
\SetCommentSty{mycommfont}

\title{
	A Direct Sampling Method and Its Integration with Deep Learning for Inverse Scattering Problems with Phaseless Data}

\author{Jianfeng Ning \thanks{School of Mathematics and Statistics, Wuhan University, Wuhan, China. ({ningjf@whu.edu.cn}).} \and Fuqun Han \thanks{Department of Mathematics, University of Califonia, Los Angeles, USA.  ({fqhan@math.ucla.edu}).}  \and Jun Zou \thanks{Department of Mathematics, The Chinese University of Hong Kong, Shatin, N.T., Hong Kong. The work of this author was substantially supported by the Hong Kong RGC General Research Fund (projects
		14306623,  14308322 and 14306921). ({zou@math.cuhk.edu.hk}).}}
\date{}

\begin{document}
	\maketitle
	\begin{abstract}
		We consider in this work an inverse acoustic scattering problem when only phaseless data is available. 
		The inverse problem is highly nonlinear and ill-posed due to the lack of the phase information. 
		Solving inverse scattering problems with phaseless data is important in applications as the collection of physically acceptable phased data is usually difficult and expensive. A novel direct sampling method (DSM) will be developed to effectively estimate the locations and geometric shapes of the unknown scatterers from phaseless data generated by a very limited number of incident waves. With a careful theoretical analysis of the behavior of the index function and some representative numerical examples, the new DSM is shown to be computationally efficient, easy to implement, robust to large noise, and does not require any prior knowledge of the unknown scatterers. Furthermore, to fully exploit the index functions obtained from the DSM, we also propose 
		to integrate the DSM with a deep learning technique (DSM-DL) to achieve high-quality reconstructions. Several challenging and representative numerical experiments are carried out to demonstrate the accuracy and robustness of reconstructions by DSM-DL. The DSM-DL networks trained by phased data are further theoretically and numerically shown to be able to solve problems with phaseless data. Additionally, our numerical experiments also show the DSM-DL can solve inverse scattering problems with mixed types of scatterers, which renders its applications in many important practical scenarios. 
	\end{abstract}
	\section{Introduction}
	
	Inverse scattering problems involve the recovery of the locations, shapes, and physical properties of unknown scatterers with some measurement data. These problems have wide applications in various fields such as 
	identification of oil reservoirs,  
	medical imaging, geophysical prospection, and microwave remote sensing \cite{persico2014introduction,bulyshev2004three}. Numerous numerical methods have been developed in the literature over the past few decades to address inverse scattering problems with phased data. These methods include contrast source inversion method\cite{van2021forward}, recursive linearization methods \cite{bao2015inverse}, reverse time migration methods \cite{chen2013reverse,chen2013reverse2}, subspace optimization methods \cite{chen2009subspace}, linear sampling methods \cite{cakoni2011linear}, direct sampling methods \cite{ito2012direct,ito2013direct}, etc.
	
	However, phased data is usually difficult and expensive to obtain in real applications \cite{d2008phaseless,ivanyshyn2011inverse}, especially at frequencies beyond tens of gigahertz. 
	In these scenarios, since collecting high-accuracy phaseless data is much easier and more cost-effective, phaseless reconstruction is usually preferred in many important practical applications even the phaseless reconstruction is more ill-posed and non-linear compared to phased reconstruction. To tackle the inverse scattering problems with phaseless data, several optimization and iterative methods have been developed in the literature (see, e.g.,\cite{agaltsov2018iterative,bao2013numerical,pan2010subspace,zhang2017recovering}). While optimization and iterative methods can generally provide accurate reconstructions of unknown scatterers, they are mostly very expensive, require reasonably good initial guesses as well as some prior knowledge of the scatterers, which may not be easy to acquire in many scenarios. 
    To achieve fast reconstruction, several non-iterative methods have gained popularity recently in inverse scattering problems with phaseless data. In \cite{chen2017phaseless}, the reverse time migration for phased reconstruction was extended to reconstruct unknown scatterers from phaseless total field data. An approximate factorization method was proposed in \cite{zhang2020approximate} to recover the shapes and locations of scatterers from the phaseless total field data induced by incident plane waves. In \cite{zhang2018fast}, by superpositions of plane waves at a fixed frequency, a direct imaging approach was proposed to recover scattering obstacles from phaseless far-field data without suffering the translation-invariance obstruction.  Moreover, introducing a reference ball into the scattering system also offers an alternative method to break the translation-invariance for phaseless far-field data \cite{ dong2018reference}. By introducing the concept of scattering coefficients, the algorithms for phased and phaseless reconstructions in the linearized case were proposed and analyzed in \cite{ammari2016phased}. In \cite{klibanov2015two}, two reconstruction schemes were proposed for a 3D phaseless inverse scattering problem based on inverse Radon transform and integral geometry.   
Moreover, some inverse phaseless reconstruction problems were studied for the Maxwell's system 
in \cite{dedok2019numerical,romanov2020phaseless}. For some theoretical results including the uniqueness and stability concerning inverse scattering problems with phaseless data, we refer to \cite{klibanov2014phaseless, novikov2015formulas, klibanov2016reconstruction, xu2018uniqueness}.
	
	In addition to the severely non-linear and ill-posed challenges of phaseless reconstruction, another challenge in some applications is that measurement data is available only from one or a few incident waves. The direct sampling methods (DSMs) introduced in \cite{ito2012direct, ito2013direct, li2013direct} have been developed for phased reconstruction and can provide reasonable estimations of the locations and shapes of unknown scatterers using only limited incident waves, without requiring any prior knowledge of the scatterers. In addition to their data efficiency, the DSMs are also computationally efficient, involving only scalar products, completely parallel, and highly robust to large noise. Therefore, it is desirable to extend the DSMs to phaseless reconstruction. We also refer to some developments of the DSMs for other inverse problems in mathematical imaging \cite{chow2014direct,chow2021direct,chow2022direct,chow2021directRadon}. 
	
	The main aims of this work are fourfold. The first one is to develop a novel DSM for inverse scattering problems with only one set of phaseless total field data. The implementation of the new DSM is as simple as that of the phased DSMs while it is 
	still capable of providing a robust estimation of shapes and locations of scatterers. On the other hand, 
	it is known that the DSMs have a limitation that 
	they can not generate very accurate reconstructions in some important scenarios, 
	e.g., for scatterers with complex geometrical shapes. So the second aim of this work is to alleviate this main limitation, 
	by combining the DSMs with some special machine learning techniques. 
	Furthermore, we will show by theoretical analysis and numerical experiments 
	that the DSM-DL network trained by phased data can also be applied to solve problems 
	with only phaseless data available. Finally, we further exploit the advantages of our novel algorithm to demonstrate that without any prior knowledge of the scatterers and with only one trained neural network, the DSM-DL with phaseless data can also solve inverse problems with inhomogeneous media or obstacles in the computational domain, as well as their mixed problems.

	Machine learning techniques, especially deep learning ones, have shown promising potential recently for solving various PDEs \cite{raissi2019physics,li2020fourier,han2018solving} and inverse problems of PDEs \cite{arridge2019solving,tanyu2022deep}. In particular, there are several advantages of solving inverse problems by employing neural networks to learn the inverse maps. Firstly, once the network is trained, the forward pass of the network is very fast and a real-time reconstruction is usually achievable. Secondly, in classical iterative methods, a prior knowledge of the unknown targets is incorporated into a regularization term to make the inversion more stable, whereas it is often challenging to choose a proper regularization and its balancing parameter. On the other hand, neural networks can automatically learn the distribution of unknown targets from training data, which helps improve the reconstruction quality. 
	
	Over the last few years, we have witnessed many successes of deep learning methods in solving inverse scattering problems \cite{ning2023direct,wei2018deep,khoo2019switchnet,gao2022artificial}, and we refer to the review papers \cite{chen2020review,guo2023physics} for more comprehensive studies. Some deep-learning approaches have also been proposed for inverse scattering problems with phaseless data. For instance, a rough image is initially obtained from the phaseless total field using the contrast source inversion scheme with a few iterations \cite{xu2020deep}, followed by refinement using neural networks to produce a high-quality image. Another two-stage method presented in \cite{luo2021cascaded} involves a cascaded complex U-Net compromising a phase retrieval net and an image reconstruction net. In \cite{yin2020neural}, a two-layer sequence-to-sequence neural network is proposed to recover obstacles with limited phaseless far-field data, where parameters representing the boundary curve of the impenetrable obstacles are utilized. 
	Given the success of deep learning techniques in inverse scattering problems, we will first derive a novel DSM for phaseless reconstruction in this work, and then combine the new DSM with deep learning to produce more desirable reconstructions. This combination is a natural extension of our previous work DSM-DL \cite{ning2023direct}, where the inverse medium scattering problem with phased data was considered. The DSM-DL in \cite{ning2023direct} was shown to be very robust to noise, easy to implement, and computationally efficient.

	The remaining part of this paper is organized as follows. In Section \ref{Problem formula}, we introduce the inverse scattering problems addressed in this work. Section \ref{DSMPS} is dedicated to a review of the DSMs and the development of a novel DSM for phaseless reconstruction. The combination of the DSM with deep learning is discussed in Section \ref{sec:DSMDL}. In Section \ref{sec:numerical}, we present numerical experiments for phaseless reconstruction using the DSM approach with a limited number of incident fields, as well as the imaging results achieved through DSM-DL. Finally, Section \ref{sec:conclusion} concludes the paper with some closing remarks.
	
	\section{Problem formulations}
	\label{Problem formula}
	We shall consider the inverse scattering problems to recover the unknown obstacles or inhomogeneous media in the domain of interest with phaseless total fields. 
	Let $u^{i}(x,d)=e^{\mathbf{i}kx\cdot d}$ be the incident plane wave, with wavenumber $k$ and incident direction 
	$d\in \mathbb{S}^{N-1}$,  
	or $u^{i}(x,x_s)=G(x,x_s), x_s\in \mathbb{R}^{N}$ be the incident source wave where $G(x,y)$ is the free-space Green's function associated with the Helmholtz equation and is given by
 \begin{equation}
	   G(x,y)=\left\{
	   \begin{aligned}
		       \frac{\mathbf{i}}{4}H_{0}^{(1)}(k|x-y|), \quad N=2\,,\\
		       \frac{\exp(\mathbf{i}k|x-y|)}{4\pi|x-y|}, \quad N=3\,,
		   \end{aligned}
	   \right.
	\end{equation}
where $H_{0}^{(1)}$ refers to the zeroth-order Hankel function of the first kind.
Assuming that a bounded domain $D \subset \mathbb{R}^N$ is the support of obstacles or inhomogeneous media, then the total field $u=u^i+u^s$ induced by impenetrable obstacles satisfies the following Helmholtz equation\cite{colton2019inverse}
	\begin{equation}
		\Delta u + k^2 u = 0 \quad \text{in}\quad \mathbb{R}^N\backslash D,
	\end{equation}
	\begin{equation}\label{eq:radiation} 	\lim_{r\rightarrow\infty}r^{(N-1)/2}(\frac{\partial u^s}{\partial r}-\mathbf{i}ku^s)=0,\quad r=|x|,
	\end{equation}
	with a boundary condition
	\begin{equation}\left\{
		\begin{aligned}
			&u=0 \quad\qquad\qquad \text{on}\quad \partial D \quad \text{for sound-soft obstacles},\\
			&\frac{\partial u}{\partial \nu}=0 \qquad\qquad \text{on} \quad\partial D \quad\text{for sound-hard obstacles },		\\
			&\frac{\partial u}{\partial \nu} +\mathbf{i}k\lambda u =0  \quad \text{on} \quad\partial D  \quad \text{with} \quad \lambda \in C(\partial D)\quad\text{for impedance obstacles}.
		\end{aligned}
		\right.
	\end{equation}
	
	For the scattering by inhomogeneous media, the total field satisfies 
	\begin{equation}
		\Delta u + k^2n(x)u=0 \quad \text{in} \quad\mathbb{R}^N\,,
	\end{equation}
	with the same radiation \eqref{eq:radiation}, where $n(x)$ is the refractive index. 
	For the scattered field $u^s$, it holds that \cite{colton2019inverse}
	\begin{equation}
		u^{s}(x) = \frac{\exp(\mathbf{i}k|x|)}{|x|^{(N-1)/2}}\bigg\{u^{\infty}(\hat{x}) + \mathcal{O}(1/|x|)\bigg\}, \quad |x|\rightarrow\infty\,,
		\label{Asym1}
	\end{equation}
	which holds uniformly for all $\hat{x}=x/|x|\in \mathbb{S}^{N-1}$, where $u^{\infty}$ is called the far-field pattern of $u^{s}$.

The inverse scattering problem of our interest involves recovering unknown scatterers 
from noisy measurements of the phaseless total field, corresponding to one or a few incident fields. 
Specifically, we propose a direct sampling method to stably and efficiently recover the scatterer's geometry and then 
combine this method with a deep learning technique to identify the type of scatterer and recover the refractive index $n(x)$ of inhomogeneous media.
 	
\section{A direct sampling method for phaseless imaging}
\label{DSMPS}
\subsection{Direct sampling methods for phased data}
Consider a sampling domain $\Omega$ in $\mathbb{R}^N$, and a measurement surface $\Gamma_{r}$, which is assumed to be the circle or ball of radius $R_r$ in the sequel. We now first recall the direct sampling method (DSM) for the reconstruction of inhomogeneous medium inclusions and then extend it for the reconstruction of impenetrable scatterers. 
	
{\bf Reconstruction of inhomogeneous medium inclusions}. 
The DSM was proposed in \cite{ito2012direct} for the inverse medium scattering problem with phased data from one incident wave:
\begin{equation}
    \begin{split}
        I_{\text{DSM}}(z):=|\langle G(z,\cdot), u^s \rangle_{L^{2}(\Gamma_{r})}|
        =\bigg|\int_{\Gamma_{r}}G(z,x_r)\overline{u^{s}(x_r)}ds(x_r)\bigg|, \quad z \in \Omega .
    \end{split}
    \label{DSM}
\end{equation}

The inhomogeneous medium is recovered based on the values of the index function: 
if it attains an extreme value at a point $z$, the point lies most likely within the support of the inhomogeneous inclusion, 
whereas if the index function takes small values, the point $z$ likely lies outside the support. The DSM is shown to be computationally efficient, easy to implement, and does not require any a priori knowledge of the unknown inclusions. For the reconstruction using the far-field pattern, the index function in DSM is defined as  \cite{li2013direct}
\begin{equation}
    I_{DSM}^{\infty}(z):=|\langle G^{\infty}(z,\cdot), u^\infty \rangle_{L^{2}(\mathbb{S}^{N-1})}|,  \quad z \in \Omega\,,
\end{equation}
where $G^{\infty}(z,\cdot)$ is the far-field pattern associated with the fundamental solution $G(z,x_r)$, given by
\begin{equation}
    G^{\infty}(z,\hat{x}_r)=\left\{
    \begin{aligned}
        &\frac{\exp(\mathbf{i}\pi/4)}{\sqrt{8k\pi}}\exp(-\mathbf{i}k\hat{x}_r\cdot z), \quad &N=2,\\
        &\frac{1}{4\pi}\exp(-\mathbf{i}k\hat{x}_r\cdot z), \quad &N=3.
    \end{aligned}\right .
\end{equation}
And we have
\begin{equation}
    G(z,x_r)=\dfrac{e^{\mathbf{i}kR_r}}{R_r^{(N-1)/2}}\bigg\{G^{\infty}(z,\hat{x}_r) + \mathcal{O}(R_r^{-1})\bigg\}.
    \label{Asym2}
\end{equation}

	By combining \eqref{Asym1} \eqref{Asym2} and \eqref{DSM}, we see the asymptotic behavior of 
	the index function $I_{DSM}(z)$: 
	\begin{equation}
		\label{eqn_phaseless_far}
		I_{DSM}(z)=C_N\bigg|\int_{\mathbb{S}^{N-1}}G^{\infty}(z,\hat{x})\overline{u^{\infty}(\hat{x},d)}ds(\hat{x})\bigg| + \mathcal{O}(R_r^{-1})=C_{N}I_{DSM}^{\infty}(z)+\mathcal{O}(R_r^{-1})\,,
	\end{equation}
	where $I_{DSM}^{\infty}(z)$ is independent of $R_r$ and $C_{N}$ depends only on  the dimension $N$.
	
	\ss
	{\bf Reconstruction of impenetrable scatterers}. 
	We now extend the above index function to help reconstruct impenetrable scatterers. 
	This may be derived by using the fact that an impenetrable scatterer can be considered as the limit of the medium scatterer with vanishing or singular material properties \cite{li2013direct}. 
	Alternatively, we provide a simple and new derivation below by using the representation formula of the scattered wave.
	
	For the sound-soft obstacle $D$, the scattered field satisfies the Huygens’ principle \cite{colton2019inverse}
	\begin{equation}
		u^s(x)=-\mathcal{S}_D\bigg(\frac{\partial u}{\partial \nu}\bigg)(x):=-\int_{\partial D}\frac{\partial u}{\partial \nu}(y) G(x,y)dy\approx \sum_{j}\omega_j G(x,y_j), \quad x\in \mathbb{R}^{N}\backslash D\,,
		\label{equ:ussound}
	\end{equation}
	where $\mathcal{S}_D$ refers to the single-layer potential,  $\{y_j\}$ is a general set of discrete integration points 
	on $\partial D$. 
	When the radius of the measurement surface $\Gamma_{r}$ is large, we have the following approximation \cite{ito2012direct}
	\begin{equation}
		\int_{\Gamma_{r}}G(x_r,z)\overline{G(x_r,y_j)}ds(x_r)\approx k^{-1}\Im(G(z,y_j)) = k^{-1}\begin{cases}
			\frac{1}{4}J_0(k|z-y_j|)\,,\quad N = 2\,;\\
			\frac{1}{4\pi}\frac{\sin(k|z-y_j|)}{|z-y_j|}\,,\quad N =3\,
		\end{cases}  
	\end{equation}
	for  $z \in \Omega$,  $y_j\in \partial D.$
	Then by combining the above equation with \eqref{DSM} and \eqref{equ:ussound}, we derive an approximate relation for 
	the sound-soft obstacles: 
	\begin{equation}
		\int_{\Gamma_{r}}\overline{G(z,x_r)}u^{s}(x_r)ds(x_r) \approx k^{-1} \sum_{j}\omega_j \Im (G(y_j,z))\,.
		\label{equ:sum}
	\end{equation}
	We notice that  $\Im(G(z,y_j))$ ) takes a relatively large value if a point $z$ is close to some boundary point $y_j$ and decays quickly if $z$ moves away from the scatterers. This observation provides a heuristic justification of the DSM 
	for sound-soft obstacles. 
	
	For sound-hard and impedance obstacles, similar arguments to the sound-soft obstacles can be made by writing the scattered field as a single layer potential \cite{colton2013integral}:
	\begin{equation}
		u^s(x)=\int_{\partial D}\phi(y)G(x,y)dy, \quad x\in \mathbb{R}^{N}\backslash D,
	\end{equation}
	where the density $\phi\in C(\partial D)$ is a solution to the integral equation
	\begin{equation}
		\phi - \mathcal{K}^T_D\phi - \mathbf{i}k\lambda\mathcal{S}_D\phi=2\frac{\partial u^i}{\partial \nu}+2\mathbf{i}k\lambda u^i,
	\end{equation}
	with $\lambda=0$ for sound-hard obstacles, where $\mathcal{K}_D$ is the double-layer potential and $\mathcal{K}^T_D$ is the dual of $\mathcal{K}_D$ with respect to the dual form $\langle \cdot, \cdot \rangle_{\partial D}$. By conducting the same approximation as the case for sound-soft obstacles, we can verify the DSM for sound-hard and impedance obstacles.

	\subsection{Formulation and verification of DSM for phaseless imaging}
	We are now ready to develop a new DSM for reconstruction with phaseless data. Motivated by the corrected data adopted 
	in \cite{chen2017phaseless},  we propose the following novel index function to reconstruct obstacles or inhomogeneous media 
	with phaseless total field data	
	\begin{equation}
		\label{equ:DSMPS}
		I_{\text{DSM}}^{\text{phaseless}}(z): =\bigg|\int_{\Gamma_{r}}G(z,x_r)\Delta(x_r,d)ds(x_r)\bigg|, \quad \forall\, z\in \Omega,
	\end{equation}
	where $\Delta(x_r,d)$ is given by 
	\begin{equation}
		\Delta(x_r,d) = \frac{|u(x_r,d)|^2-|u^i(x_r,d)|^2}{u^{i}(x_r,d)}.
		\label{corrected data}
	\end{equation}
	The field $\Delta(x_r,d)$ above is called the corrected phaseless data. We remark that the work 
	in \cite{chen2017phaseless} concerns the phaseless imaging by making use of the reverse time migration 
	with a large number of point source incident fields in $\mathbb{R}^2$.
	Very different from \cite{chen2017phaseless}, we focus in this work on some special important physical scenarios,  
	namely when highly limited incident fields are available, e.g., 
	either a few incident source waves or a few incident plane waves for $\mathbb{R}^N, N=2,3$. 
	More importantly, we shall combine the proposed DSM and the deep learning strategy to achieve high-quality imaging reconstructions, even for scatterers with complicated geometric shapes and mixed types (i.e., 
	penetrable and impenetrable inhomogeneous inclusions coexist).
	
	The motivation for using the corrected phaseless data is that the leading order term in $I_{\text{DSM}}^{\text{phaseless}}(z)$ will be the same as $I_{\text{DSM}}(z)$ when the measurement surface is far away from $D$, as shown below. 
	
	We now make a resolution analysis on the index function $I_{\text{DSM}}^{\text{phaseless}}$ defined in \eqref{equ:DSMPS} and will show that
	\begin{equation}
		\label{eqn_phaseless_near}
		I_{\text{DSM}}^{\text{phaseless}}(z)= I_{\text{DSM}}(z)+\mathcal{O}(R_r^{(1-N)/2}) \quad \text{in} \quad \mathbb{R}^N\,.
	\end{equation}
	To do so, we first rewrite $\Delta(x_r,d)$ as 
	\begin{equation}
		\Delta(x_r,d)=\overline{u^s(x_r,d)}+\dfrac{|u^s(x_r,d)|^2}{u^i(x_r,d)}+\dfrac{u^s(x_r,d)\overline{u^i(x_r,d)}}{u^i(x_r,d)}\,, 
	\end{equation}
	using this relation, we can rewrite $I_{\text{DSM}}^{\text{phaseless}}(z)$ for any $z\in \Omega$ as 
	\begin{equation}
		\label{eqn:I_split}
		\begin{split}			I_{\text{DSM}}^{\text{phaseless}}(z)=& \bigg|\int_{\Gamma_{r}}G(z,x_r)\overline{u^s(x_r,d)}ds(x_r)\\
			&+\int_{\Gamma_{r}}G(z,x_r)\dfrac{|u^s(x_r,d)|^2}{u^i(x_r,d)}ds(x_r) + \int_{\Gamma_{r}}G(z,x_r)\dfrac{u^s(x_r,d)\overline{u^i(x_r,d)}}{u^i(x_r,d)}ds(x_r)\bigg |\,.
		\end{split}	
	\end{equation}
	
	{\bf Planar incident waves}. We first consider the case of the incident plane wave $u^{i}(x_r,d)=e^{\mathbf{i}kx\cdot d}$.
	In the following Lemmas \ref{Lemma_us2} to \ref{Lemma_usui}, we show that the last two terms at the right-hand side of the equation \eqref{eqn:I_split} are small when $R_r\gg 1$.
	For simplification, we use $C$ to denote a generic constant independent of $R_r$ 
	while its exact values may vary at different occasions.
	Then we first have an estimate of the second term at the right-hand side of \eqref{eqn:I_split}, 
	which can be derived directly by using the asymptotic formulas \eqref{Asym1}, \eqref{Asym2} and the fact that $|u^{i}(x_r,d)|=1$.
	
	\begin{lemma}
		\label{Lemma_us2}
		For any $z\in \Omega$, we have
		\begin{equation}
			\bigg|\int_{\Gamma_{r}}G(z,x_r)\dfrac{|u^s(x_r,d)|^2}{u^i(x_r,d)}ds(x_r)\bigg|\le C R_r^{(1-N)/2}\,.
		\end{equation}
	\end{lemma}

 \begin{proof}
     By the asymptotic formulas \eqref{Asym1} and \eqref{Asym2}, we have
     \begin{equation*}
         |G(z,x_r)|=\mathcal{O}(R_r^{(1-N)/2}), \quad |u^s(x_r,d)|=\mathcal{O}(R_r^{(1-N)/2}),
     \end{equation*}
     which, together with the fact that $|u^{i}(x_r,d)|=1$ and $|\Gamma_r|=2^{N-1}\pi R_r^{N-1},$ we derive 
     \begin{equation*}  \bigg|\int_{\Gamma_{r}}G(z,x_r)\dfrac{|u^s(x_r,d)|^2}{u^i(x_r,d)}ds(x_r)\bigg|\le C R_r^{(1-N)/2}\,,
     \end{equation*}
     where $C$ is a constant independent of $R_r$.
 \end{proof}

	Now we turn to estimating the third term at the right-hand side of \eqref{eqn:I_split}. For this purpose, we need the following 
	estimates of some oscillatory integrals from \cite[Lemma~3.9]{chen2017phaseless}.
	\begin{lemma}
		For any $-\infty <a<b<\infty$ and real-valued function $u\in C^2[a,b]$ thats satisfies $|u'(t)|\ge 1$ for $t\in (a,b)$. Assume that $a=x_0<x_1<\cdots<x_M=b$ is a division of $(a,b)$ such that $u'$ is monotone in each interval $(x_{i-1},x_i),i=1,...,M$. Then it holds 
		for any function $\phi$ defined on $(a, b)$ with integrable derivative and for any $\lambda>0$ that 
		\begin{equation}
			\bigg|\int_{a}^{b}e^{\mathbf{i}\lambda u(t)}\phi(t)dt\bigg|\le (2M+2)\lambda^{-1}\bigg[|\phi(b)|+\int_{a}^{b}|\phi'(t)|dt\bigg].
		\end{equation}
		\label{oscillatory}
	\end{lemma}
	
	With the above Lemma, we can have the following estimate for the last term in \eqref{eqn:I_split}.  	
	\begin{lemma}
		\label{Lemma_usui}
		For any $z\in \Omega$, we have
		\begin{equation}
			\bigg|\int_{\Gamma_{r}}G(z,x_r)\dfrac{u^s(x_r,d)\overline{u^i(x_r,d)}}{u^i(x_r,d)}ds(x_r)\bigg|\le C R_r^{(1-N)/2}.
		\end{equation}
	\end{lemma}
	\begin{proof}
		By the asymptotic relations in \eqref{Asym1}, \eqref{Asym2} and change of variables, we obtain  
		\begin{equation}\label{eq:green_estimate}
			\bigg |\int_{\Gamma_{r}}G(z,x_r)\dfrac{u^s(x_r,d)\overline{u^i(x_r,d)}}{u^i(x_r,d)}ds(x_r)\bigg|=C\bigg|\int_{\mathbb{S}^{N-1}}G^{\infty}(z,\hat{x})u^{\infty}(\hat{x},d)e^{-2\mathbf{i}kR_r\hat{x}\cdot d}ds(\hat{x})\bigg|+\mathcal{O}(R_r^{-1}).
		\end{equation}
		
		We now estimate the first term in \eqref{eq:green_estimate} for the cases of $N=2$ and $N=3$ separately. 
		For simplicity, we write $\phi(\hat{x})=G^{\infty}(z,\hat{x})u^{\infty}(\hat{x},d)$.
		First for $N=2$, without loss of generality we take $d=(1,0)$ and let $\hat{x}=(\cos\theta,\sin\theta)$, then we can write 
		\begin{equation}
			\int_{\mathbb{S}^{1}}\phi(\hat{x})e^{-2\mathbf{i}kR_r\hat{x}\cdot d}ds(\hat{x})=\int_{0}^{2\pi}\phi(\hat{x})e^{-2\mathbf{i}kR_r\cos\theta}d\theta\,.
		\end{equation}
		Denoting $\delta=(kR_r)^{-1/2}$ and $Q_\delta=\{\theta\in (0,2\pi):|\theta-m\pi|<\delta,m=0,1,2\}$, it is straightforward to derive
		\begin{equation}
			\bigg|\int_{Q_\delta}\phi(\theta)e^{-2\mathbf{i}kR_r\cos\theta}d\theta\bigg|\le CR_r^{-1/2}\,.
		\end{equation}
		For $\theta\in (0,2\pi)\backslash Q_\delta$, we can easily see 
		\begin{equation}
			2(kR_r)^{1/2}|(\cos\theta)'|=	2(kR_r)^{1/2}|\sin\theta|\ge 	2(kR_r)^{1/2}|\sin\delta|\ge 1\,,
		\end{equation}
		and obviously $(\cos\theta)'$ is piecewise monotone in $(0,2\pi)$. Then it follows from Lemma \ref{oscillatory} that
		\begin{equation}
			\bigg|\int_{(0,2\pi)\backslash Q_\delta}\phi(\theta)e^{-2\mathbf{i}kR_r\cos\theta}d\theta\bigg|\le CR_r^{-1/2}\,.
		\end{equation}
		
		Next, for $N=3$, without loss of generality, we assume $d=(0,0,1)$ and $\hat{x}=(\sin\theta\cos\varphi,\sin\theta\sin\varphi,\cos\theta)$. Then we can write 
		\begin{equation}
			\int_{\mathbb{S}^2}\phi(\hat{x})e^{-2\mathbf{i}kR_r\hat{x}\cdot d}ds(\hat{x})=\int_{0}^{2\pi}\int_{0}^{\pi}\phi(\theta,\varphi)e^{-2\mathbf{i}kR_r\cos\theta}\sin\theta d\theta d\varphi.
		\end{equation}
		Since $(\frac{1}{2\mathbf{i}kR_r}e^{-2\mathbf{i}kR_r\cos\theta})'= -e^{-2\mathbf{i}kR_r\cos\theta}\sin\theta$,
		by integration by parts, we obtain that 
		\begin{equation}
			\int_{0}^{\pi}\phi(\theta,\varphi)e^{-2\mathbf{i}kR_r\cos\theta}\sin\theta d\theta=-\dfrac{1}{2\mathbf{i}kR_r}\bigg[e^{-2\mathbf{i}kR_r\cos\theta}\phi(\theta,\varphi)\bigg]_{0}^{\pi}+\dfrac{1}{2\mathbf{i}kR_r}\int_{0}^{\pi}\frac{\partial \phi(\theta,\varphi)}{\partial \theta}e^{-2\mathbf{i}kR_r\cos\theta}d\theta.
		\end{equation}
		The proof is then completed by observing that $\phi$ and $\frac{\partial \phi}{\partial\theta}$ are bounded and independent of $R_r$.
	\end{proof}
	
	Finally, we summarize the main results in this section by combining the above lemmas and \eqref{eqn:I_split}. 
	These results indicate that when $R_r$ is large enough, the index function with phaseless data can reconstruct the scatterers accurately and stably as $I_{DSM}$ for the inverse scattering problems with phased data.
	\begin{theorem}
		With the incident plane wave $u^i(x,d)=e^{\mathbf{i}kx\cdot d}$ and $z\in \Omega$, it holds that 
		\begin{equation}
			I_{\text{DSM}}(z) = C_{N}I_{DSM}^{\infty}(z)+\mathcal{O}(R_r^{-1}) \,,
		\end{equation}
		\begin{equation}
			I_{\text{DSM}}^{\text{phaseless}}(z)=I_{\text{DSM}}(z)+\mathcal{O}(R_r^{(1-N)/2})=C_{N}I_{\text{DSM}}^{\infty}(z)+\mathcal{O}(R_r^{(1-N)/2})\,.
		\end{equation}
		\label{theorem1}
	\end{theorem}
	
	{\bf Incident point source waves}. We now consider the case of incident point source waves 
	$u^{i}(x,x_s)=G(x,x_s)$
	for $x_s\in \mathbb{R}^{N}$ with $R_s=|x_s|$ as the radius of the surface where point sources are located. Without loss of generality, we assume $R_r=\tau R_s$ for $\tau$$\ge$ 1 where $R_r$ is the radius of the measurement surface. 
	For the forward scattering problems, it is known that the solution to the scattering problem and its derivatives depend continuously on the incident field in the maximum norm 
	\cite[Theorems~3.11, 3.12, 3.16, 8.7]{colton2019inverse}. With these well-posedness results for different types of scatterers, the following estimates can be derived and are used repeatedly in our subsequent analysis.
  
 \begin{lemma}
     For the incident point source wave
	$u^{i}(x,x_s)=G(x,x_s)$, we have the following estimates for $u^s$ and $u^\infty$:
 \begin{align}
		\Vert u^\infty(\hat{x},x_s)\Vert_{\infty,\mathbb{S}^{N-1}}&\le \tilde{C} \Vert u^i\Vert_{\infty,\Omega}\le CR_s^{(1-N)/2}\,,
		\label{A1}\\
		u^s(x_r,x_s)&= \dfrac{e^{\mathbf{i}kR_r}}{R_r^{(N-1)/2}}\bigg\{u^\infty(\hat{x}_r,x_s)+\mathcal{O}(R_s^{(1-N)/2}R_r^{-1})\bigg\}\,,
		\label{A2}\\
		|u^s(x_r,x_s)|&\le C (R_s R_r)^{(1-N)/2}\,,
		\label{A3}
	\end{align}
 where $\tilde{C}$ and $C$ are constants independent on $R_s$ and $R_r$.
 \end{lemma}
 \begin{proof}
     It is known that the scattered field has the following Green's formula \cite{colton2019inverse}:
\begin{equation}
    u^s(x) = \int_{\partial B}\bigg\{ u^s(y)\frac{\partial G(x,y)}{\partial \nu}-\frac{\partial u^s}{\partial\nu}(y)G(x,y) \bigg\} ds(y),
\end{equation}
where $B$ is any bounded domain of class $C^2$ containing the whole scatterer, and we assume $B$ is the unit ball without loss of generality. Combining the Green's formula and the following two asymptotic expressions
\begin{equation}
		G(x,y)=\dfrac{e^{\mathbf{i}k|x|}}{|x|^{(N-1)/2}}\bigg\{G^{\infty}(\hat{x},y) + \mathcal{O}(|x|^{-1})\bigg\},
	\end{equation}

 \begin{equation}
     \frac{\partial G(x,y)}{\partial \nu}=\dfrac{e^{\mathbf{i}k|x|}}{|x|^{(N-1)/2}}\bigg\{\frac{\partial G^\infty(x,y)}{\partial \nu} + \mathcal{O}(|x|^{-1})\bigg\},
 \end{equation}
we have
\begin{equation}
    u^{s}(x) = \frac{e^{\mathbf{i}k|x|}}{|x|^{(N-1)/2}}\bigg\{u^{\infty}(\hat{x}) + \mathcal{O}\big((\Vert u^s\Vert_{\infty,\partial B}+\Vert \frac{\partial u^s}{\partial\nu}\Vert_{\infty,\partial B})|x|^{-1}\big)\bigg\}\,,
\end{equation}
where
\begin{equation}
     u^\infty(\hat{x}) = \int_{\partial B} \bigg\{u^s(y)\frac{\partial G^\infty(\hat{x},y)}{\partial \nu}-\frac{\partial u^s}{\partial\nu}(y)G^\infty(\hat{x},y)\bigg\} ds(y).
\end{equation}
By the well-posedness results \cite[Theorems~3.11, 3.12, 3.16, 8.7]{colton2019inverse}, we then have
\begin{equation}
    \Vert u^s\Vert_{\infty,\partial B}+\Vert \frac{\partial u^s}{\partial\nu}\Vert_{\infty,\partial B} =\mathcal{O}(\Vert u^i\Vert_{\infty,\Omega})=\mathcal{O} (R_s^{(1-N/2)}),
\end{equation}
hence 
\begin{equation}
    \Vert u^\infty(\hat{x},x_s)\Vert_{\infty,\mathbb{S}^{N-1}}=\mathcal{O}(\Vert u^i\Vert_{\infty,\Omega})=\mathcal{O} (R_s^{(1-N/2)}),
\end{equation}

\begin{equation}
    u^s(x_r,x_s)= \dfrac{e^{\mathbf{i}kR_r}}{R_r^{(N-1)/2}}\bigg\{u^\infty(\hat{x}_r,x_s)+\mathcal{O}(R_s^{(1-N)/2}R_r^{-1})\bigg\}\,,
\end{equation}
which have proved the two estimates \eqref{A1}-\eqref{A2}. Estimate \eqref{A3} can be directly derived from \eqref{A1} and \eqref{A2}.

 \end{proof}
	To further the resolution analysis for the phaseless index function with incident source waves, we now show a helpful estimate. 
	\begin{lemma}
		\label{Lemma_us2_s}
		The following estimate holds for the source incident wave $u^i(x,x_s)=G(x,x_s)$: 
		\begin{equation}
			\bigg|\int_{\Gamma_{r}}G(z,x_r)\dfrac{|u^s(x_r,x_s)|^2}{u^i(x_r,x_s)}ds(x_r)\bigg|\le C R_s^{1-N}, \quad z\in \Omega\,.
		\end{equation}
	\end{lemma}
	\begin{proof}
		By \eqref{Asym2} and \eqref{A3}, we have
		\begin{equation}\label{eq:gammar}
			\bigg|\int_{\Gamma_{r}}G(z,x_r)\dfrac{|u^s(x_r,x_s)|^2}{u^i(x_r,x_s)}ds(x_r)\bigg|\le C R_s^{1-N}R_r^{3(1-N)/2}\int_{\Gamma_{r}}\dfrac{1}{|u^i(x_r,x_s)|}ds(x_r)\,.
		\end{equation}
		
		Next, we consider the case $N=2$ and $N=3$ separately. 
		Firstly, when $N=2$, since $|H_0^{(1)}(t)|^{-1}$ is bound in $(0,1/2)$ and $t^{-\frac{1}{2}}|H_0^{(1)}(t)|^{-1}$ is decreasing for $t>0$, there exist constants $C_1$ and $C_2$ such that
		\begin{equation}
			|H_0^{(1)}(t)|^{-1}\le C_1 t^{1/2} +C_2, \quad t \in(0,\infty).
		\end{equation}
		Then by the condition $R_r\ge \tau R_s$ for $\tau>1$ and $|u^i(x_r,x_s)|=|\frac{\textbf{i}}{4}H_0^{(1)}(k|x_r-x_s|)|$, we have $|x_r-x_s|\le 2R_r$, and 
		\begin{equation}
  \begin{split}
      	\int_{\Gamma_{r}}\dfrac{1}{|u^i(x_r,x_s)|}ds(x_r)&\le 4\int_{\Gamma_{r}} \bigg\{C_1(k|x_r-x_s|)^{1/2} +C_2\bigg\}ds(x_r)\\
       &\le 4\int_{\Gamma_{r}} \bigg\{C_1(2kR_r)^{1/2} +C_2\bigg\}ds(x_r) \le C R_r^{3/2},
  \end{split}		
		\end{equation}
		which, along with \eqref{eq:gammar}, yields 
		\begin{equation}
			\bigg|\int_{\Gamma_{r}}G(z,x_r)\dfrac{|u^s(x_r,x_s)|^2}{u^i(x_r,x_s)}ds(x_r)\bigg| \le CR_s^{-1}\,.
		\end{equation}
		
		Finally, for $N=3$, we can see that $|u^i(x_r,x_s)|^{-1}= 4\pi|x_r-x_s|\le CR_r $. Then we can similarly obtain    
		\begin{equation}
			\bigg|\int_{\Gamma_{r}}G(z,x_r)\dfrac{|u^s(x_r,x_s)|^2}{u^i(x_r,x_s)}ds(x_r)\bigg| \le CR_s^{-2}.
		\end{equation}
	\end{proof}
	
	Next, we derive some results similar to those in Lemma \ref{Lemma_usui}, but for the incident source waves. For this purpose, we first recall the following important mixed reciprocity relation\cite{colton2019inverse,potthast2001point}
	\begin{lemma}
		Denote $u^{\infty}(d,x_s)$ as the far field pattern of the scattering problem with incident wave $u^i(x)=G(x,x_s)$, and $w^s(x,-d)$ as the scattering solution with incident wave $w^i(x)=e^{-\mathbf{i}kx\cdot d}$, then we have 
		\begin{equation}
			u^{\infty}(d,x_s)=\gamma_N w^s(x_s,-d),\quad x_s \in \mathbb{R}^{N}\backslash D, d\in \mathbb{S}^{N-1},
		\end{equation}
		where $\gamma_2=\frac{1}{4\pi}$ and  $\gamma_3=\frac{e^{\mathbf{i}\pi/4}}{\sqrt{8k\pi}}$. 
		\label{lemma:mix}
	\end{lemma}
	
	With Lemma \ref{lemma:mix}, equations \eqref{Asym1} and \eqref{A2}, we then have the following estimate for $u^s(x_r,x_s)$:
	
	\begin{equation}
		u^s(x_r,x_s)=\frac{\gamma_N e^{\mathbf{i}k(R_r+R_s)}}{R_r^{(N-1)/2}R_s^{(N-1)/2}}w^{\infty}(\hat{x}_s,-\hat{x}_r)+\mathcal{O}(R_s^{(1-N)/2}R_r^{-(1+N)/2}+R_s^{-(1+N)/2}R_r^{(1-N)/2}),
		\label{equa:mix2}
	\end{equation}
	where $w^{\infty}(\hat{x}_s,-\hat{x}_r)$ is the far field pattern corresponding to the incident wave $w^i(x)=e^{-\mathbf{i}kx\cdot \hat{x}_r}$.

	\begin{lemma}
		\label{Lemma_usui_s}
		With source incident wave $u^i(x,x_s)=G(x,x_s)$, we have
		\begin{equation}
			\bigg|\int_{\Gamma_{r}}G(z,x_r)\dfrac{u^s(x_r,x_s)\overline{u^i(x_r,x_s)}}{u^i(x_r,x_s)}ds(x_r)\bigg|\le C R_s^{(1-N)}\,.
		\end{equation}
	\end{lemma}
	\begin{proof}
		We carry out the proof in two separate cases, $N=2$ and $N=3$. 
		
		First for $N=2$, without loss of generality, we assume $x_s=(R_s,0)$. Denoting $\delta=(kR_s)^{-1/2}$, $\Theta_\delta=\{\theta\in (0,2\pi): |\theta \pm m\pi|\le\delta, m=0,1,2\}$ and $Q_\delta=\{x_r\in \Gamma_{r}: |\theta_r\in \Theta_\delta\}$, then it follows easily from \eqref{Asym2}, \eqref{A3} and  $|Q_\delta|\le CR_rR_s^{-1/2} $ that
		\begin{equation}
			\bigg|\int_{Q_\delta}G(z,x_r)\dfrac{u^s(x_r,x_s)\overline{u^i(x_r,x_s)}}{u^i(x_r,x_s)}ds(x_r)\bigg|\le C R_s^{-1}\,.
		\end{equation}
		For $x_r\in \Gamma_{r}\backslash Q_\delta$, since $\sin t\ge t/2$ for $t\in(0,\pi/2)$, we have
		\begin{equation}
			k|x_r-x_s|\ge 2kR_s\bigg|\sin\frac{\theta_r-\theta_s}{2}\bigg|\ge 2kR_s\sin\frac{\delta}{2}\ge\frac{1}{2}(kR_s)^{1/2}\,.
		\end{equation}
		Thus, by the asymptotic behavior of the Hankel functions \cite{colton2019inverse}
		\begin{equation}
			H_n^{(1)}(t)=\bigg(\frac{2}{\pi t}\bigg)^{1/2}e^{\mathbf{i}(t-\frac{n\pi}{2}-\frac{\pi}{4})}+R_n(t),\quad|R_n(t)|\le Ct^{-3/2},\quad t>0\,,
		\end{equation}
		we can write 
		\begin{equation}
			\frac{\overline{u^i(x_r,x_s)}}{u^i(x_r,x_s)}=e^{-2\mathbf{i}k|x_r-x_s|+\mathbf{i}\frac{\pi}{2}}+\rho_1(x_r,x_s)\,,
			\label{A15}
		\end{equation}
		where $|\rho_1(x_r,x_s)|\le CR_s^{-1/2}$. Then, by combining \eqref{Asym2}, \eqref{A15}, and \eqref{equa:mix2}, we obtain 
		\begin{equation}
			\begin{split}
				&\bigg|\int_{\Gamma_r\backslash Q_\delta}G(z,x_r)\dfrac{u^s(x_r,x_s)\overline{u^i(x_r,x_s)}}{u^i(x_r,x_s)}ds(x_r)\bigg|\\
				\le&C R_s^{-1/2}\bigg| \int_{(0,2\pi)\backslash \Theta_\delta}G^{\infty}(\hat{x}_r,z)w^{\infty}(\hat{x}_s,-\hat{x}_r)e^{-2\mathbf{i}k|x_r-x_s|}d\theta_r\bigg |+\mathcal{O}(R_s^{-1})\\
				=& C R_s^{-1/2}\bigg|\int_{(0,2\pi)\backslash \Theta_\delta}G^{\infty}(\hat{x}_r,z)w^{\infty}(\hat{x}_s,-\hat{x}_r)e^{2\mathbf{i}kR_sv(\theta_r)}d\theta_r\bigg|+\mathcal{O}(R_s^{-1})\,,
			\end{split}
		\end{equation}
		where $v(\theta_r)=-\sqrt{1+\tau^2-2\tau\cos\theta_r}$. For small $\delta$, The derivative of $v$ can be computed explicitly as
		\begin{equation}
			|v'(\theta_r)|\ge \tau|\sin\delta|/|v(\theta_r)|\ge\frac{\tau}{1+\tau}\frac{\delta}{2}\ge \frac{1}{4}(kR_s)^{-1/2}\,.
		\end{equation}
		Employing the above two estimates and utilizing Lemma \eqref{oscillatory}, we have 
		\begin{equation}
			\bigg|\int_{\Gamma_r\backslash Q_\delta}G(z,x_r)\dfrac{u^s(x_r,x_s)\overline{u^i(x_r,x_s)}}{u^i(x_r,x_s)}ds(x_r)\bigg| \le C R_s^{-1}\,.
		\end{equation}
		
		Now we analyze for the case $N=3$. We first see from \eqref{Asym2} and \eqref{equa:mix2} that 
		\begin{equation}\label{eq:casen=3}
			\begin{split}
				&\bigg|\int_{\Gamma_r}G(z,x_r)\dfrac{u^s(x_r,x_s)\overline{u^i(x_r,x_s)}}{u^i(x_r,x_s)}ds(x_r)\bigg|\\
				=&CR_s^{-1}\bigg|\int_{\Gamma_{r}}G^{\infty}(\hat{x}_r,z)w^{\infty}(\hat{x}_s,-\hat{x}_r)e^{-2\mathbf{i}k|x_r-x_s|}d(x_r)\bigg|+\mathcal{O}(R_s^{-2}).
			\end{split}
		\end{equation}
		Without loss of generality, we assume $x_s=(0,0,R_s)$ and let $x_r=R_r(\sin\theta\cos\varphi,\sin\theta\sin\varphi,\cos\theta)$, $\phi(\hat{x}_r)=G^{\infty}(\hat{x}_r,z)w^{\infty}(\hat{x}_s,-\hat{x}_r)$, we can then rewrite the leading order term in \eqref{eq:casen=3} as
		\begin{equation}
			\int_{\Gamma_{r}}\phi(\hat{x}_r)e^{-2\mathbf{i}k|x_r-x_s|}d(x_r)=\int_{0}^{2\pi}\int_{0}^{\pi}\phi(\hat{x}_r)e^{-2\mathbf{i}kR_s\sqrt{\tau^2+1-2\tau\cos\theta}}\sin\theta d\theta d\varphi\,.
		\end{equation}
		Denoting $v(\theta)= \sqrt{\tau^2+1-2\tau\cos\theta}$ and knowing $v'(\theta)=\tau\sin\theta/v(\theta)$, 
		by integration by parts, we obtain that
		$$ 
		\int_{0}^{\pi}\phi(\hat{x}_r)e^{-2\mathbf{i}kR_s v(\theta)}\sin\theta d\theta
		=-\dfrac{1}{2\mathbf{i}k\tau R_s}\bigg(\bigg[\phi(\hat{x}_r)v(\theta)e^{-2\mathbf{i}kR_sv(\theta)}\bigg]_{\theta=0}^{\theta=\pi}
		-\int_{0}^{\pi}e^{-2\mathbf{i}kR_sv(\theta)}\frac{\partial(\phi(\hat{x}_r)v(\theta))}{\partial\theta}d\theta\bigg)\,.
		$$ 
		We then have
		\begin{equation}
			\bigg|\int_{\Gamma_{r}}G^{\infty}(\hat{x}_r,z)w^{\infty}(\hat{x}_s,-\hat{x}_r)e^{-2\mathbf{i}k|x_r-x_s|}d(x_r)\bigg|\le CR_s^{-1}\,.
		\end{equation}
		Then the desired estimate follows from this and \eqref{eq:casen=3}.
	\end{proof}
	
	By substituting the estimates in Lemmas \ref{Lemma_us2_s} and \ref{Lemma_usui_s} into the decomposition of the index function in \eqref{eqn:I_split}, we come to the following conclusions, 
	where the second result comes from replacing $u^i(x,d)$ with $u^i(x,x_s)$ in \eqref{equ:DSMPS} and \eqref{corrected data}.
	
	\begin{theorem}
		\label{theorem:source}
		With the incident source wave $u^i(x,x_s)=G(x,x_s)$, and assuming $R_r=\tau R_s$ for $\tau$$\ge$ 1, we then have for $z\in \Omega$:
		\begin{equation}
			I_{\text{DSM}}(z) =  C_NI_{DSM}^{\infty}(z)+\mathcal{O}(R_s^{(1-N)/2}R_r^{-1})= \mathcal{O}(R_s^{(1-N)/2})\,.
		\end{equation}
		and
 		\begin{equation}
			I_{\text{DSM}}^{\text{phaseless}}(z)=I_{\text{DSM}}(z)+ \mathcal{O}(R_s^{1-N})=C_NI_{DSM}^{\infty}(z)+\mathcal{O}(R_s^{1-N})\,.
		\end{equation}
	\end{theorem}

\begin{remark}
    Theorems \ref{theorem1} and \ref{theorem:source} indicate that $I_{\text{DSM}}^{\text{phaseless}}$ converges to $I_{\text{DSM}}(z)$ and $I_{\text{DSM}}^\infty(z)$ when $R_r$ (and $R_s$ for incident point source waves) is sufficiently large. On the other hand, it was theoretically and numerically shown in \cite{ito2012direct,li2013direct} that $I_{\text{DSM}}(z)$ and $I_{\text{DSM}}^\infty(z)$ can take large values when $z$ is near the scatterer and small values when $z$ is away from the scatterer so that they can provide an approximation for the location and shape of the scatterer. Thus, by the convergence properties stated in Theorems \ref{theorem1} and \ref{theorem:source}, $I_{\text{DSM}}^{\text{phaseless}}$ can also be applied to recover the geometry of the scatterer for sufficiently large $R_r$ and $R_s$.
\end{remark}
	
\begin{remark}
As we have seen from the above two theorems that $I_{\text{DSM}}^{phaseless}$ converges to  $I_{\text{DSM}}$ as $R_r$ (and $R_s$ for point incident waves) increases. However, increasing $R_r$ will pose challenges in practice. Firstly, take the incident plane wave as an example, we observe that $|u^s(x_r,d)|\le CR_r^{(1-N)/2}$, while $\lim_{R_r\rightarrow\infty}|u(x_r, \theta)|=1$. 
Thus, as $R_r$ increases, the incident field dominates in the total field, and the phaseless reconstruction would be more sensitive to noise. Secondly, as $R_r$ increases, the phaseless data $|u(x_r,\theta)|$ when $|x_r| = R_r$ become more oscillatory in terms of $\theta$ and thus more receivers are required to collect the data. 
 Hence, it is important to choose an appropriate $R_r$ in applications. In section \ref{sec:numerical}, our numerical experiments show that the reconstructions are still quite acceptable with a reasonable range of choices $R_r$, e.g., for wavelength $\lambda=0.75$, the DSM can still provide reasonable reconstructions with $R_r=4$.
\label{remark}
\end{remark}

	\section{The direct sampling-based deep learning approach}
	\label{sec:DSMDL}
		
	In this section, we present a direct sampling-based deep learning approach (DSM-DL) that combines DSM for phaseless data, as described in section \ref{DSMPS}, with deep learning techniques to perform phaseless reconstruction for both penetrable and impenetrable scatterers, as well as their mixed problems. 
	We apply deep learning techniques to address a common trade-off in classical methods for solving inverse scattering problems, which often involves a compromise between computation time and accuracy. Deep learning has the potential to overcome this limitation by leveraging large datasets and optimization algorithms during the training process. A well-trained neural network can offer rapid, stable, and highly accurate image reconstructions. Additionally, a key advantage of the deep learning approach is its ability to alleviate the need for prior knowledge about unknown scatterers, as it can learn the distribution information from training data.
	
	The DSM-DL consists of two steps: we first train a neural network, denoted as $\mathcal{G}_\Theta$, to learn the mapping between the index functions  $\{I_{\text{DSM}, i}^{\text{phaseless}}\}_{i=1}^{N_i}$ and the medium profile $n(x)$, where $N_i$ is the number of incidences. This mapping is relatively easier to learn compared to the highly nonlinear inverse mapping from phaseless scattered data to $n(x)$ because  $\{I_{\text{DSM}, i}^{\text{phaseless}}\}_{i=1}^{N_i}$ provides a stable and rough estimate for $n(x)$. Furthermore, both the index function and $n(x)$ may be regarded to be in the same function space, e.g., $L^2(\mathbb{R}^n)$.
	
	Subsequently, to solve inverse problems, we will first compute $\{I_{\text{DSM},i}^{\text{phaseless}}\}_{i=1}^{N_i}$ by \eqref{equ:DSMPS} and feed them into the neural network $\mathcal{G}_\Theta$ that we obtained above to have an image $\tilde{n}(x)=\mathcal{G}_\Theta(\{I_{\text{DSM},i}^{\text{phaseless}}\}_{i=1}^{N_i})$. The detailed algorithm for training is summarized in Algorithms \ref{AlgorithmDLDSM} .

	\begin{algorithm}[t]
		\caption{DSM-DL for Phaseless Data (Training)}
		\label{AlgorithmDLDSM}
		\KwIn{ \\
			\quad$\bullet$ Given fixed $N_i$ incident fields $\{u_{p}^{i}\}_{p=1}^{N_i}$.\\
			\quad$\bullet$ $N$ true refractive functions$\{n_{j}(x)\}_{j=1}^{N}$(for impenetrable scatterer $D$, we set $n(x)=0$ for point $x$ inside the scatterer $D$  ) defined in a selected sampling domain $\Omega$ and their corresponding phaseless total-field data$\{|u|_{p,j}, p=1,\cdots,N_{i}\}_{j=1}^{N}$ measured on a surface $\Gamma_r$.\\
			\quad$\bullet$ The free-space Green's function $\{G(x,y),x\in \Omega, y\in \Gamma_r\}$.\\
			\quad$\bullet$ Number of epochs: $L$; Batch number: $Q$;  Learning rates: $\{\tau_{q}\}_{q=1}^{L}$.
		}
		Compute the index functions $I_{p,j}^{Phaseless}(z)=\bigg|\int_{\Gamma_{r}}G(z,x_r)\Delta(x_r,d_p)ds(x_r)\bigg|, \quad \forall\, z\in \Omega; j=1,\cdots,N;p=1,\cdots,N_{i}$. where $\Delta(x_r,d_p)$ is given by 
		$\Delta(x_r,d_p) = \frac{|u(x_r,d_p)|^2_{p,j}-|u^i(x_r,d_p)|^2}{u^{i}(x_r,d_p)}.$
		\\Introduce a loss function $Loss$ such as \eqref{equa:lossfunc}.
		\\
		Construct and initialize a network $\mathcal{G}_{\Theta}$, where $\Theta$ denotes the parameters in the neural network.\\
		\For{$q=1,2,\cdots,L$}{
			\For{$l=1,2,\cdots,Q$}{
				Update $\Theta\leftarrow\Theta-\frac{1}{|\mathcal{I}_l|}\sum_{j\in \mathcal{I}_l}\tau_q\nabla_{\Theta}Loss\big(\mathcal{G}_\Theta(\{I_{p,j}^{Phaseless}\}_{p=1}^{N_{i}}),n_{j}\big)$.  
				\\
				\tcc{ $\mathcal{I}_{l}$ refers to the index set for the $l_{th}$ batch, $|\mathcal{I}_{l}|$ denotes the number of elements in $\mathcal{I}_{l}$, and $\tau_{q}$ refers to learning rate for the $q_{th}$ epoch.}
			}
		}
	\end{algorithm}

	\begin{figure}[htp]
		\centering
		\includegraphics[width=1.0\linewidth]{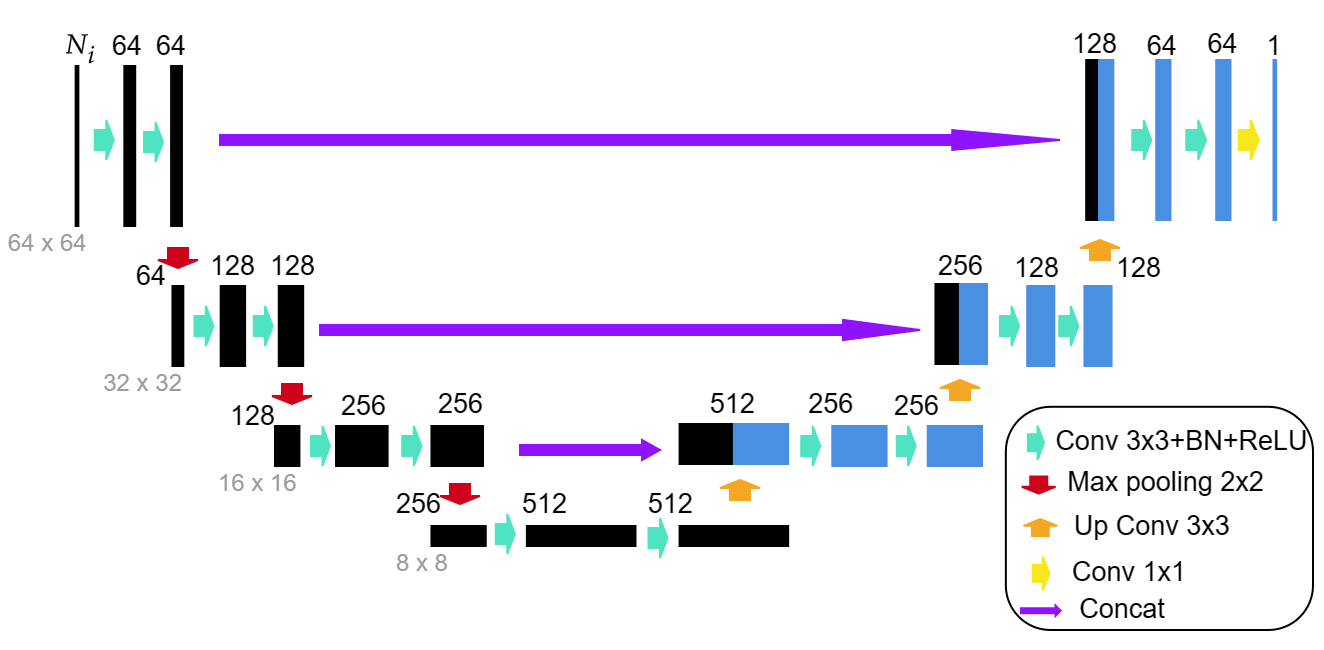}
		\caption{The architecture of the U-Net used in the numerical experiments.}
		\label{UNET}
	\end{figure}
	
	In our DSM-DL approach, we employ the U-Net architecture \cite{ronneberger2015u} as the neural network, as shown in Fig. \ref{UNET}. The U-Net architecture is suitable for our task because it can handle multi-channel input, which in our case corresponds to the index functions. The U-Net follows a symmetric U-shaped structure. The left-hand side, known as the contracting path, involves repeated convolutions, batch normalization, rectified linear unit (ReLU), and max pooling operations. And the right-hand side of the U-Net forms the expansive path, which is similar to the contracting path but employs $3\times3$ up-convolutions instead of $2\times2$ max pooling. Skip connections are used to recover lost spatial information during down-sampling by transferring features from the contracting path to the expansive path. Since the index functions have larger values near the scatterers and relatively smaller values away from the scatterers, the relationship between the index functions and the true contrast images is likely to be mainly local which is rather consistent with the local nature of CNNs. 
	
	While it is possible to employ extended versions of the classical U-Net to potentially enhance the results, it is important to note this is an endless pursuit of improvement and is not the primary focus of this work. Therefore, we chose the standard U-Net architecture for our neural networks, which proved to be able to provide quite satisfactory reconstructions, as shown by our numerical experiments in the sequel.
	
	We propose the following loss function to generate an effective network for numerical reconstructions: 
	\begin{equation}
		Loss = \frac{1}{T_n}\sum_{i=1}^{T_n}\bigg(\Vert X_i-Y_i\Vert^2_2 + \alpha_1 \text{TV}(X_i) + \alpha_2(1-\text{SSIM}(X_i,Y_i)\bigg)\,,
		\label{equa:lossfunc}
	\end{equation}
	where $T_n$ is the total number of the training data, $X_i$ and $Y_i$ are the output of the neural network and the true image, respectively, while $\alpha_1$ and $\alpha_2$ are the regularization parameter and weight coefficient.
	The first term in \eqref{equa:lossfunc} is a mean square error to ensure the accuracy of reconstruction, the second term is a total variation regularization term, with 
	$\text{TV}(u):=\int_{\Omega}|\nabla u| dx$,  
	that is added to detect edges of images more efficiently which is critical for many practical situations as the boundary information of scatters is usually most demanding.  
	The third term is a structural similarity function (SSIM) \cite{wang2003multiscale} 
	to help enhance the accuracy of the reconstruction regarding its luminance and contrast of true and reconstructed images.

	\ss
	We now list several attractive advantages of the DSM-DL approach for solving inverse scattering problems. 
	\begin{itemize}
		\item Firstly, once the neural network is trained, the DSM-DL method requires only inexpensive computations of the index functions and a forward pass through the network, enabling real-time estimation of the unknown scatterers. 
		\item  The DSM-DL method is very robust to large noise, as the noise is generally significantly smoothed out during the DSM stage.
		\item  Moreover, the DSM-DL method does not require a prior knowledge of the physical properties of the scatterers. This feature enables the applications of the DSM-DL to the mixed problem, where both medium scatterers and impenetrable scatterers exist within the domain of interest.
		\item Importantly, Theorems \ref{theorem1} and \ref{theorem:source} demonstrate that with sufficient large $R_r$ a DSM-DL network trained using phased data can be also applied to problems with only phaseless data available which will be illustrated in numerical experiments in the next section. 
		\item  Additionally, numerical demonstrations confirm that the DSM-DL method possesses a high potential for generalization, allowing it to solve problems that differ significantly from the training data, which is crucial in the applications of deep learning to general inverse scattering problems.
		\item The DSM-DL can generally do a much better reconstruction than the classical DSM , namely the latter can only recover the locations and shapes of the scatterers, while 
		the former is also able to recover the refractive index values of the scatterers. 
	\end{itemize}
	
	\section{Numerical Experiments}
	\label{sec:numerical}
	We present several representative numerical examples to illustrate the performance of the DSM and DSM-DL for reconstructing the unknown scatterers from noisy phaseless data. The wavelength is chosen as $\lambda=0.75$ and the wavenumber $k$ will be $2\pi/\lambda \approx 8.37$. The noisy data $|u_\delta|$ are generated point-wisely by the formula
	\begin{equation}
		|u_\delta(x)|=|u(x)|+\delta\zeta(x)\Vert u\Vert_2, \quad x \in \Gamma_r,
	\end{equation}
	where $\Vert u\Vert_2=(\frac{1}{N}\sum_{i=1}^N|u(x_i)|^2)^{1/2}$,   $\delta$ refers to the relative noise level, and $\zeta(x)$ follows the standard normal distribution. 
	
	We take a sampling domain to be $\Omega=[-1,1]^2$, which may contain both inhomogeneous media 
	and impenetrable scatterers. 
	Moreover, to reconstruct impenetrable scatterers with DSM-DL, a fundamental question is a proper representation of scatterers using pixel values. Since the scattering property of impenetrable scatterers depends solely on their boundaries, we can assign a pixel value of $0$ to points inside the scatterers and $1$ to points outside the scatterers to represent them. 
	Remarkably, our numerical experiments demonstrate that a single trained DSM-DL network can simultaneously solve 
	both inverse medium and obstacle scattering problems, as well as their mixed problems (see section \ref{NE_mixed}).
	We first present the numerical results for the new DSM with phaseless data in Section \ref{numeriDSM}, 
	then the numerical results for DSM-DL in Section \ref{numeriDSMDL}.
	
	\subsection{Numerical results for phaseless reconstruction by DSM}\label{numeriDSM}
	
	For phaseless reconstruction by DSM, we are interested in the important case when very limited data 
	is available, e.g., only one or a few incident fields. Unless otherwise specified, one incident plane wave with incident direction $d=(\cos(\frac{\pi}{4}),\sin(\frac{\pi}{4}))$ is employed, and the data $|u|$ is measured at 100 points uniformly distributed on a circle of radius 4. We normalize the index function $I(z)$ by $\hat{I}(z)=I(z)/\max_{z}I(z)$ so that its maximum value is 1. 
	
	\textbf{Example 1.} This first example recovers a circular sound-hard scatterer of radius 0.15 located at the origin.
	
	The numerical reconstructions are presented in Fig.\,\ref{fig:one scatterer}. We observe that for sampling points near the scatterer, the index function is indeed relatively larger than points outside the scatterer. From this example, we can see the index function can provide a reliable and visible indicator for locating the scatterer, even when the noise level is $\delta=5\%$ or $\delta=10\%$ in the phaseless data.
	
	\begin{figure}[htp]
		\centering
		\includegraphics[width=1.0\linewidth]{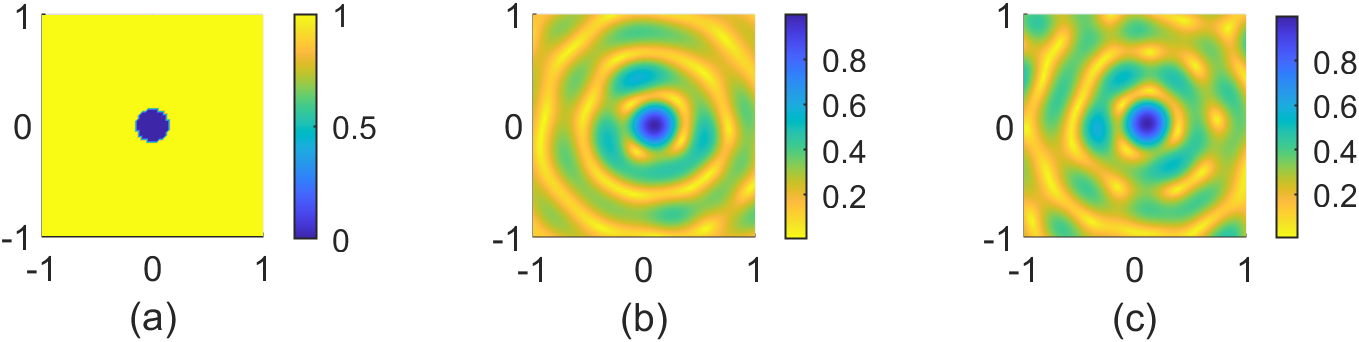}
		\caption{Example   1: (a) ground truth,  (b) reconstruction using noisy data with $\delta=5\%$ and (c)  reconstruction using noisy data with $\delta=10\%$. }
		\label{fig:one scatterer}
	\end{figure}

	\textbf{Example 2.} In this example, we consider two small square sound-soft scatterers of width 0.15. The two scatterers are located at $(-0.5,0.5)$ and $(0.5,-0.5)$ which are well separated, respectively; see Fig.\,\ref{fig:twoscatterer1}.
	
	From the plots in Fig.\,\ref{fig:twoscatterer1}, we can see that the index function can provide satisfactory reconstructions, with only one incident field, and is quite robust under the noise level $\delta = 5\%$ or and $10\%$.
	
	\begin{figure}[htp]
		\centering
		\includegraphics[width=1.0\linewidth]{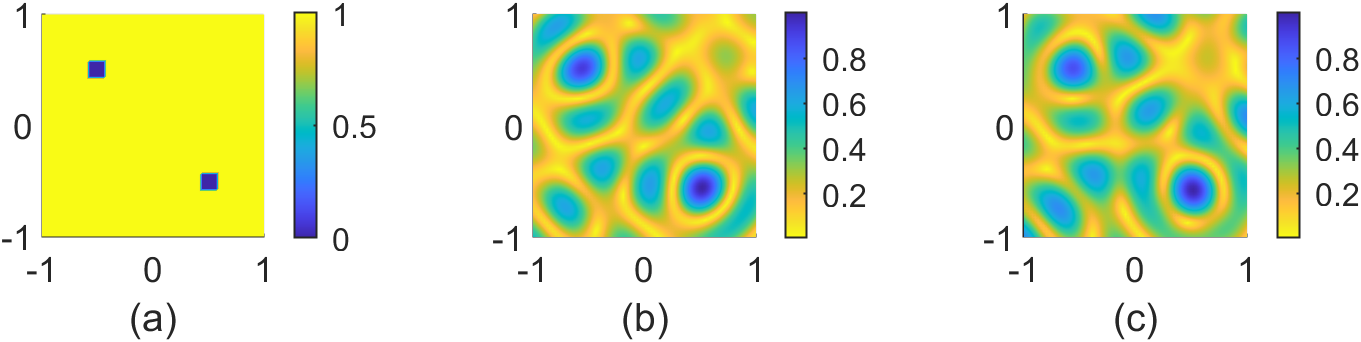}
		\caption{Example   2: (a) ground truth, (b) reconstruction using noisy data with $\delta=5\%$ and (c)  reconstruction using noisy data with $\delta=10\%$.}
		\label{fig:twoscatterer1}
	\end{figure}
	
	\textbf{Example 3.} This example considers two inhomogeneous scatterers with width $0.15$ located at $(0.1,0.1)$ and $(-0.1,-0.1)$, respectively. The distance between them is less than one-half of the wavelength ($\lambda/2$), which is known to be very challenging for imaging reconstruction. Nonetheless, the index function can still allow us to distinguish two close scatterers (cf. Fig.\,\ref{fig:twoscatterer2}).
	
	\begin{figure}[htp]
		\centering
		\includegraphics[width=1.0\linewidth]{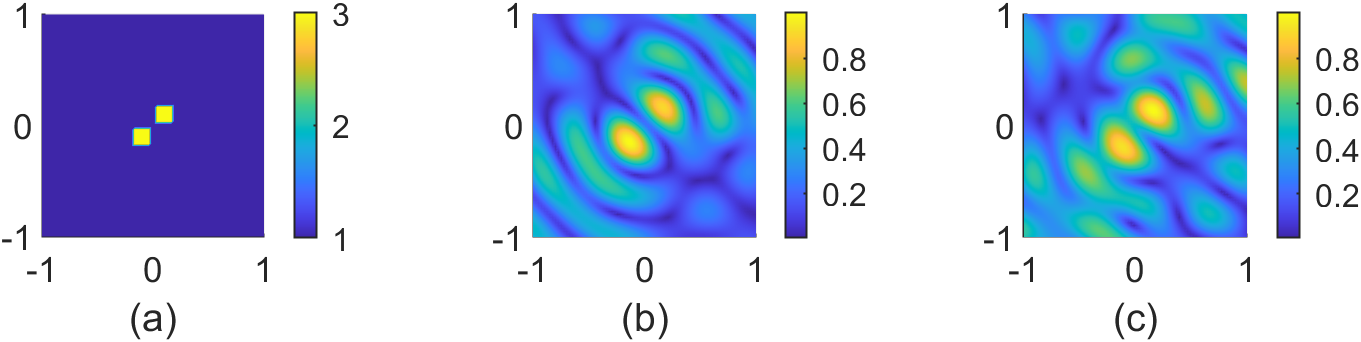}
		\caption{Example   3: (a) ground truth,   (b) reconstruction using noisy data with $\delta=5\%$, and (c)  reconstruction using noisy data with $\delta=10\%$. }
		\label{fig:twoscatterer2}
	\end{figure}
	
	\textbf{Example 4.} In this example, we consider an inhomogeneous scatterer which is a ring located at the origin, with the outer and the inner boundaries of the ring very close, with the radius being 0.3 and 0.25, respectively. The coefficient $n(x)$ of the scatterer is 3.
	
	The ring-shaped scatterer is known to be one of the most challenging objects to recover in inverse scattering problems, especially in the case when the thickness of the ring is very small. From Fig.\,\ref{fig:Ring}, we notice that employing one single incident wave is only able to recover some parts of the ring structure which is reasonable as the lack of measurement data.
	We then try to use more incident waves $\{e^{ikx\cdot d_j}\}_{j=1}^{N_i}$, where $d_j = (\cos(\theta_j), \sin(\theta_j)), \theta_j=\frac{2\pi (j-1)}{N_{i}}+\frac{\pi}{4}$, with the following average index function
	\begin{equation}
		I(z)=\dfrac{1}{N_{i}}\sum_{j=1}^{N_{i}}I_j(z),
	\end{equation}
	where $I_{j}(z)$ denotes the index function corresponding to the $j$th incident wave. We investigate $N_{i}=1,3$ and $5$ and observe that with more incident waves the method can produce more accurate and stable reconstructions for this challenging case, as shown in Fig.\,\ref{fig:Ring}.
	
	\begin{figure}[h]
 \vspace{-1.0cm}
		\centering
		\includegraphics[width=1.0\linewidth]{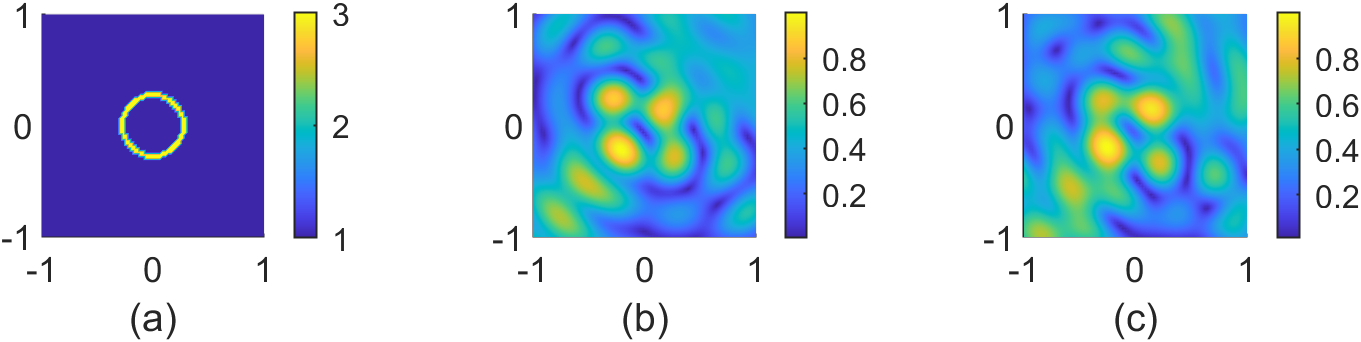}
		\includegraphics[width=1.0\linewidth]{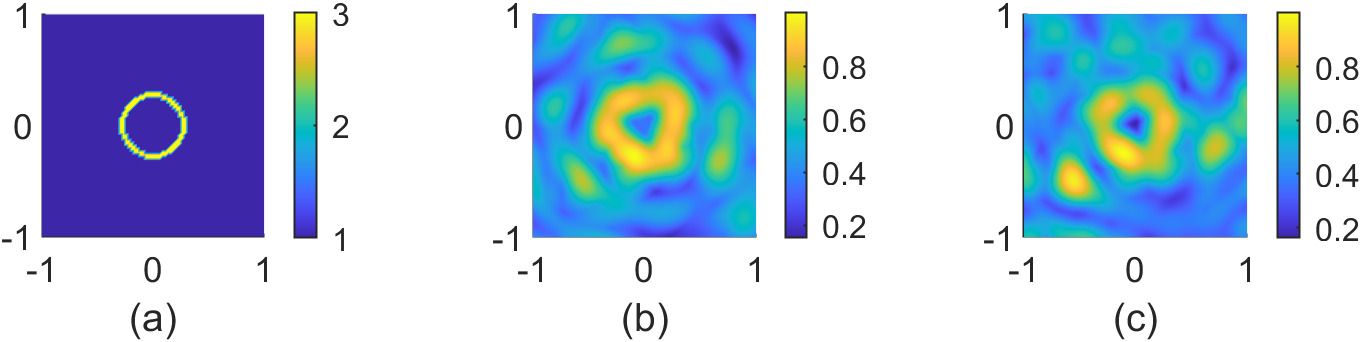}
		\includegraphics[width=1.0\linewidth]{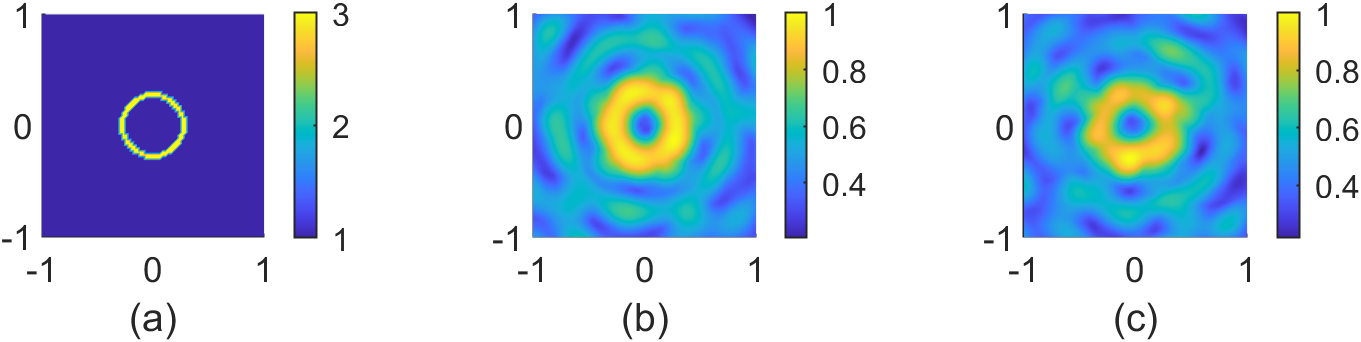}
		\caption{Example   4: (a) ground truth,  (b) reconstruction using $5\%$ noisy data and (c)  reconstruction using 
			$10\%$ noisy data. The first, second, and third rows are reconstructions with 1, 3, and 5 incidences respectively.}
		\label{fig:Ring}
	\end{figure}

\subsection{Numerical results of phaseless reconstruction by DSM-DL}\label{numeriDSMDL}

In this subsection, we evaluate the performance of the DSM-DL scheme by training the neural networks with three very different datasets: a polygon dataset, a MNIST dataset, and a mixed circle dataset. We discretize the index functions and the true images by $64\times 64$ pixels.  We scale all the inputs (index functions) of the neural networks by multiplying a constant $2/W$ in the numerical experiments, where $W$ is the maximum value of the index functions among all training data. 
The Adam optimizer is used to minimize the loss function \eqref{equa:lossfunc} and $\alpha_1 = \alpha_2 =0.5$ which is determined by grid-search as described in section \ref{sec:DSMDL}. We add $1\%$ Gaussian noise to the phaseless data in the training stage, while in the testing stage, much higher noise levels are considered. A server with GeForce GTX 1080Ti and 256 GiB RAM is used in the training stage. It takes about 0.5 hours for training the network, and less than 1.0 seconds for reconstructing the scatterers using the trained network. To evaluate the quality of the recovered images, we employ the relative L2 error:
\begin{equation}
	R_e(X,Y) = \frac{\Vert X-Y\Vert_2}{\Vert Y\Vert_2},
\end{equation}
where $X$ is the output of the network and $Y$ is the true image. We also employ the following accuracy metric for examples in \ref{examp:polygon}:
\begin{equation}
	Acc(X,Y):=\frac{1}{N_{d_x}N_{d_y}}\sum_{i=1}^{N_{d_x}}\sum_{j=1}^{N_{d_y}}a_{i,j},\quad \text{where}\quad a_{i,j}=1 \quad \text{if}\quad X_{i,j}=Y_{i,j}, \quad \text{else} \quad a_{i,j}=0.
\end{equation}
The metric $Acc(X, Y)$ is suitable for impenetrable scatterers as it is independent of the pixel values we choose, and is effective when we are concerned with only the locations of the unknown scatterers.

\subsubsection{Training with polygon dataset}
In the first example, we use a regular polygon with a random location to simulate either a medium scatterer or a sound-soft obstacle in each image. The number of the sides of each polygon is randomly set between 3 and 6, and the circumradius of each polygon is taken from the uniform distribution $U(0.3,0.5)$. For medium scatterers, the coefficient $n(x)$ inside the scatterer is set to 3. For the sound-soft obstacles, the pixel values of the points inside the obstacles are set to 0. In addition, we randomly rotate the polygon to enhance the diversity of the data. In this example, We employ 7000 images as the training data and 200 images as the testing data. The batch size is set to 10 and a total of 30 epochs is used in the training. For each batch of training data, there are 5 examples concerning medium scatterers and 5 examples concerning sound-soft obstacles. The learning rate starts at 0.001 and decreases by a factor of 0.5 every 3 epochs. We will employ $N_i$ incident plane waves with $N_i=1$ or $4$, and employ $100$ receivers that are equally distributed on a circle of radius 4 centered at the origin.

\subsubsubsection{Tests with Polygon testing data}
\label{examp:polygon}
After the training, we use the networks to test the examples from the testing data of the Polygon dataset. In this example, to compute the metric $Acc(X,Y)$, we further apply cutoff values 0.5 and 2.0 to the output of neural network so that the pixel values of the images only take from $\{0,1,3\}$. In Fig.\,\ref{Polygon}, we present reconstructed images with different noise levels and number of incidences for 5 examples, where 3 images are for medium scatterers and 2 images are for sound-soft obstacles. We observe that for $\delta=2\%$, even with only one single incident wave, the DSM-DL can still exactly identify the type of the scatterer and provide a very accurate estimation for the shapes and locations of the scatterers. As we increase the noise level to $10\%$, the method with one incidence can still identify the type and location of the scatterers, while the shapes of reconstructed scatterers are distorted, especially for medium scatterers. But when we use multiple incidents, even still a small number, e.g., 4 incident fields, 
higher-quality reconstructions can be obtained. The average accuracy metric $Acc$ over the testing data is listed in Table.\,\ref{tab:Polygon}. We notice that for $N_i=1$, the accuracy is decreased by $1.77\%$ when the noise level is increased from $2\%$ to $10\%$, while for $N_i=4$ the accuracy is only decreased by $0.61\%$. This indicates that employing more incidences can make the reconstruction more robust against noise.

\subsubsection{Training with MNIST dataset}
In this subsection, we employ a modification of the well-known MNIST dataset to model the inhomogeneous medium and train the neural networks. The resolution of each image in the MNIST dataset is $28 \times 28$ and the pixel values range from 0 and 1. The resolution is then rescaled to $64 \times 64$ in our numerical experiments and a threshold is applied with value 0.3 so that the pixel values only take 0 and 1. To increase the diversity of the data, the digits are randomly rotated, and a circle with a radius from $U(0.2,0.4)$ is added to each image. The coefficient $n(x)$ of the digits and the circles are randomly taken from $U(1.2, 1.7)$. In this example, We employ 10000 images as the training data and 200 images as the testing data. The batch size is set to 10 and a total of 30 epochs is used in the training. The learning rate starts at 0.001 and decreases by a factor of 0.5 every 3 epochs. We will employ $N_i$ incident plane waves with $N_i=4$ or $16$, and employ $100$ receivers that are equally distributed on a circle of radius 4 centered at the origin.
\begin{figure}[htbp]\small
	\begin{center}
		\begin{tblr}
			{colspec = {X[-1,m]X[c,h]X[c,h]X[c,h]X[c,h]X[c,h]},
				stretch = 0,
				rowsep = 0pt,}
			{Ground\\ Truth}&
			\includegraphics[width=0.18\textwidth]{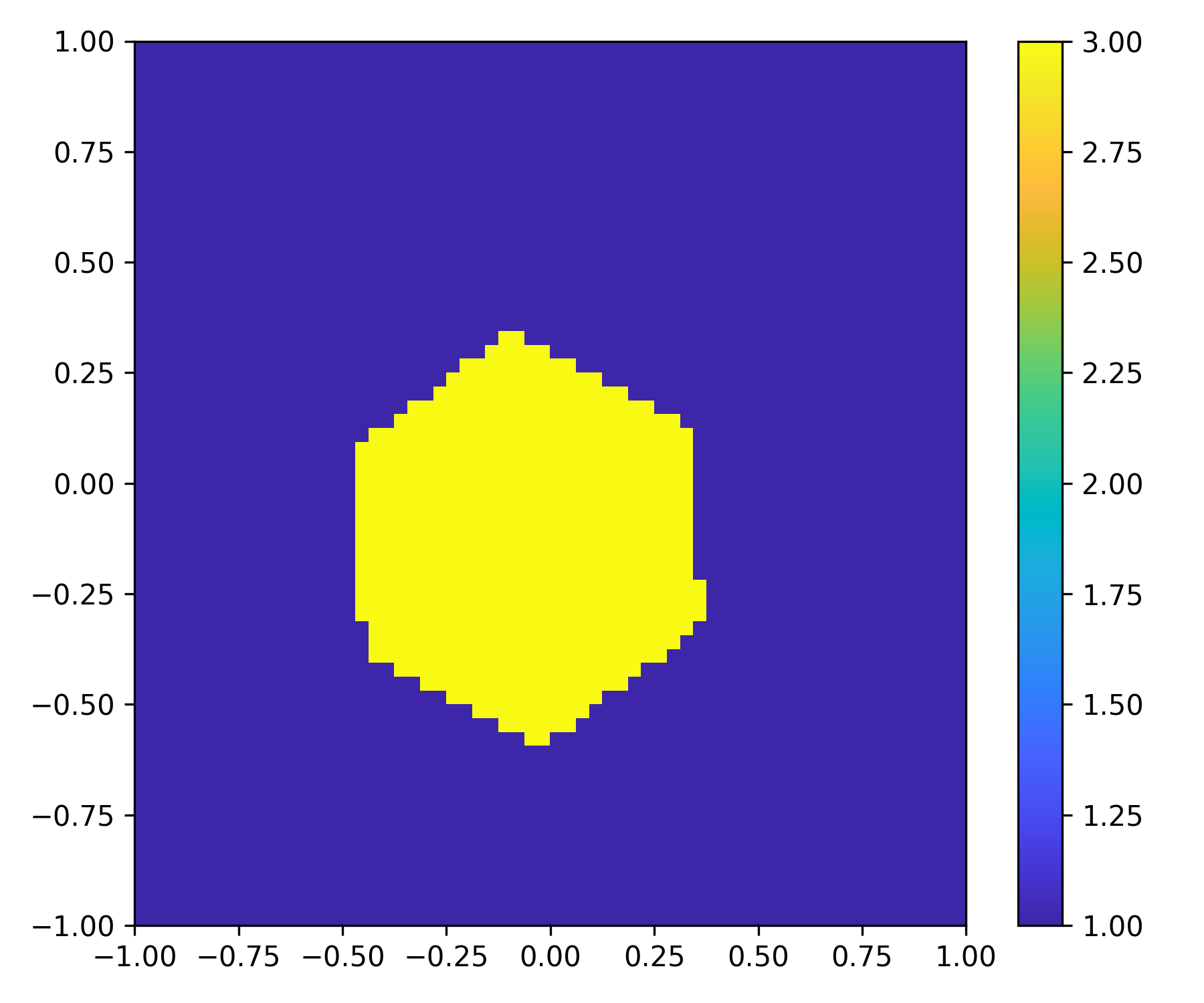}&
			\includegraphics[width=0.18\textwidth]{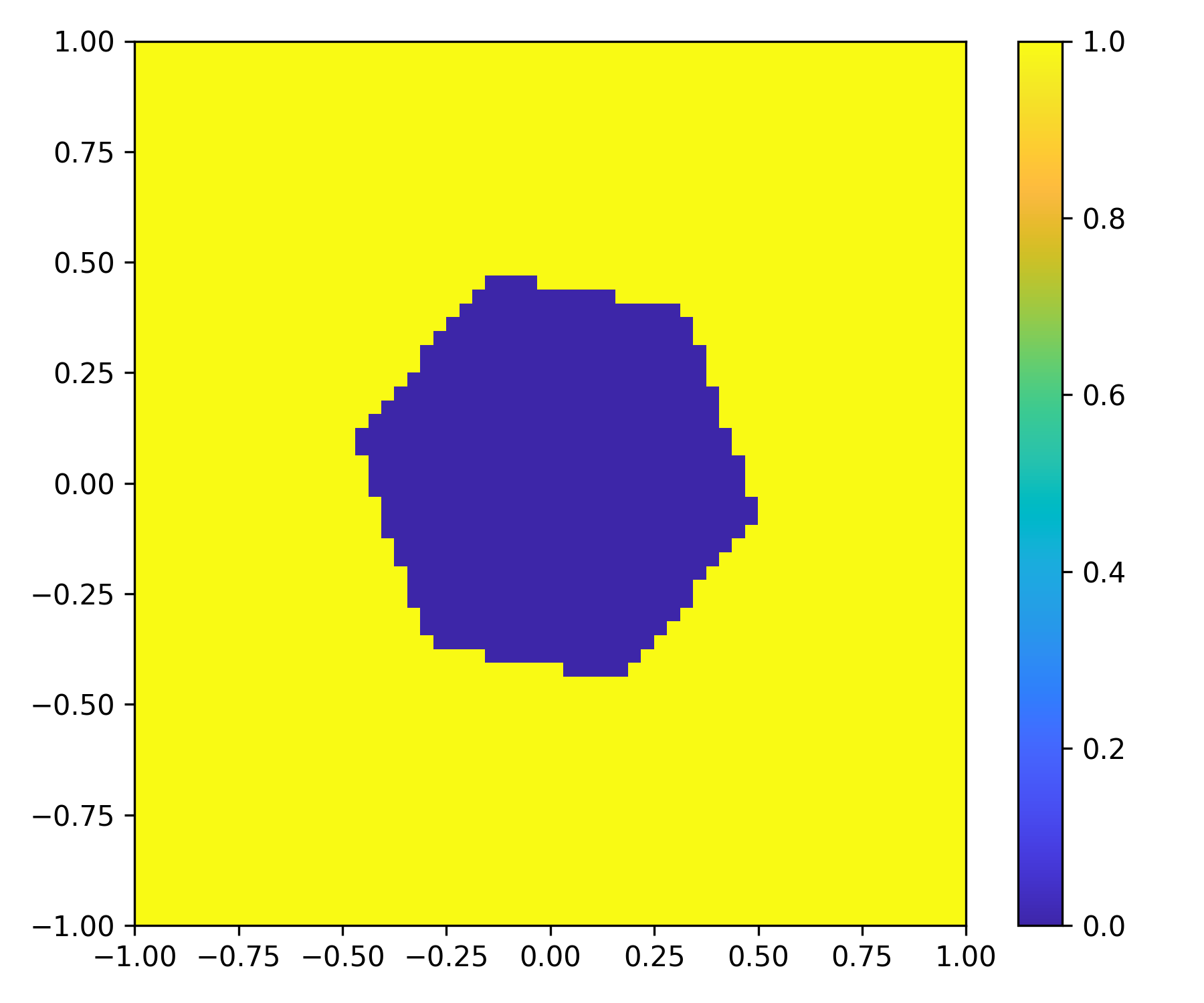} &\includegraphics[width=0.18\textwidth]{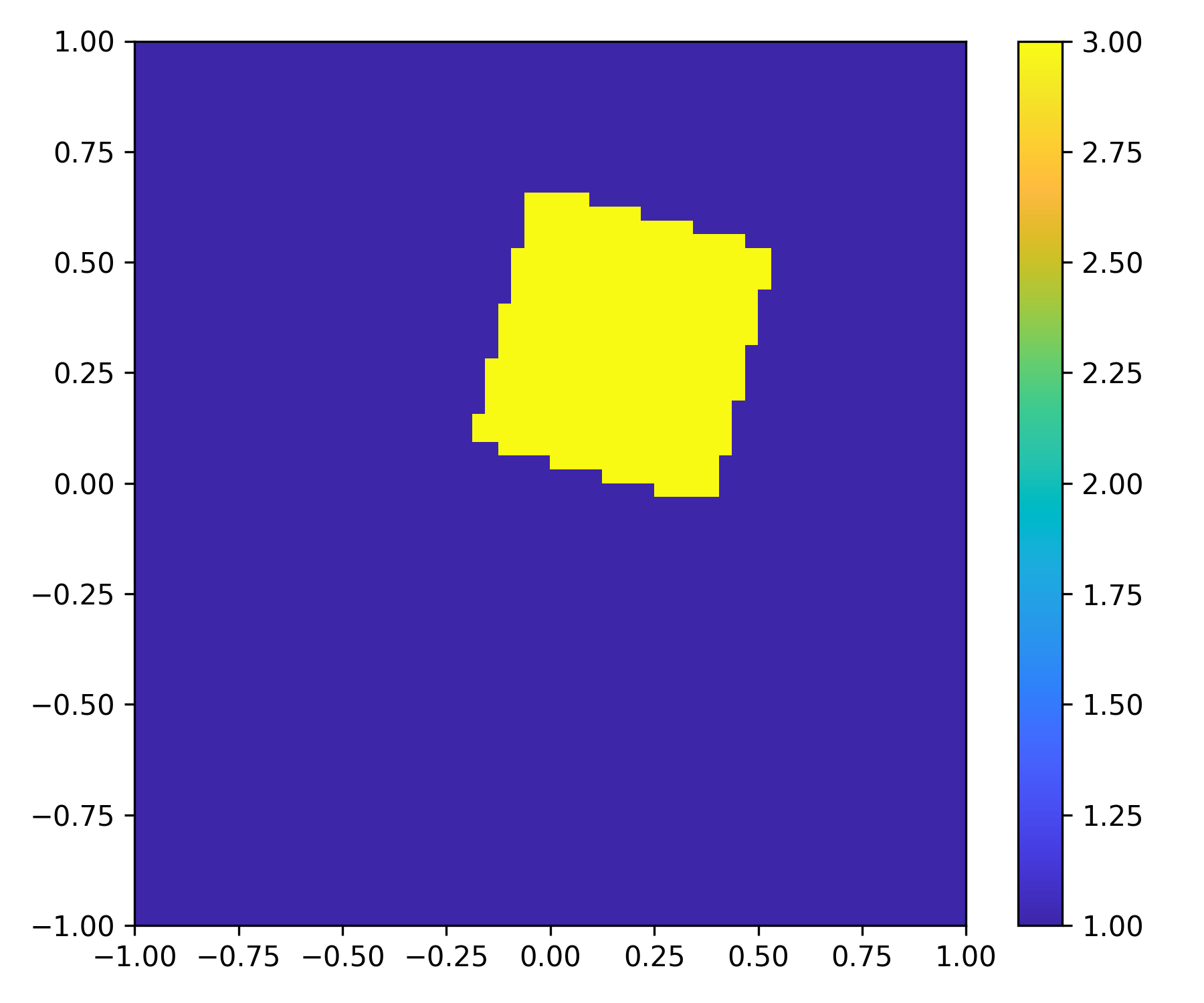}
			&\includegraphics[width=0.18\textwidth]{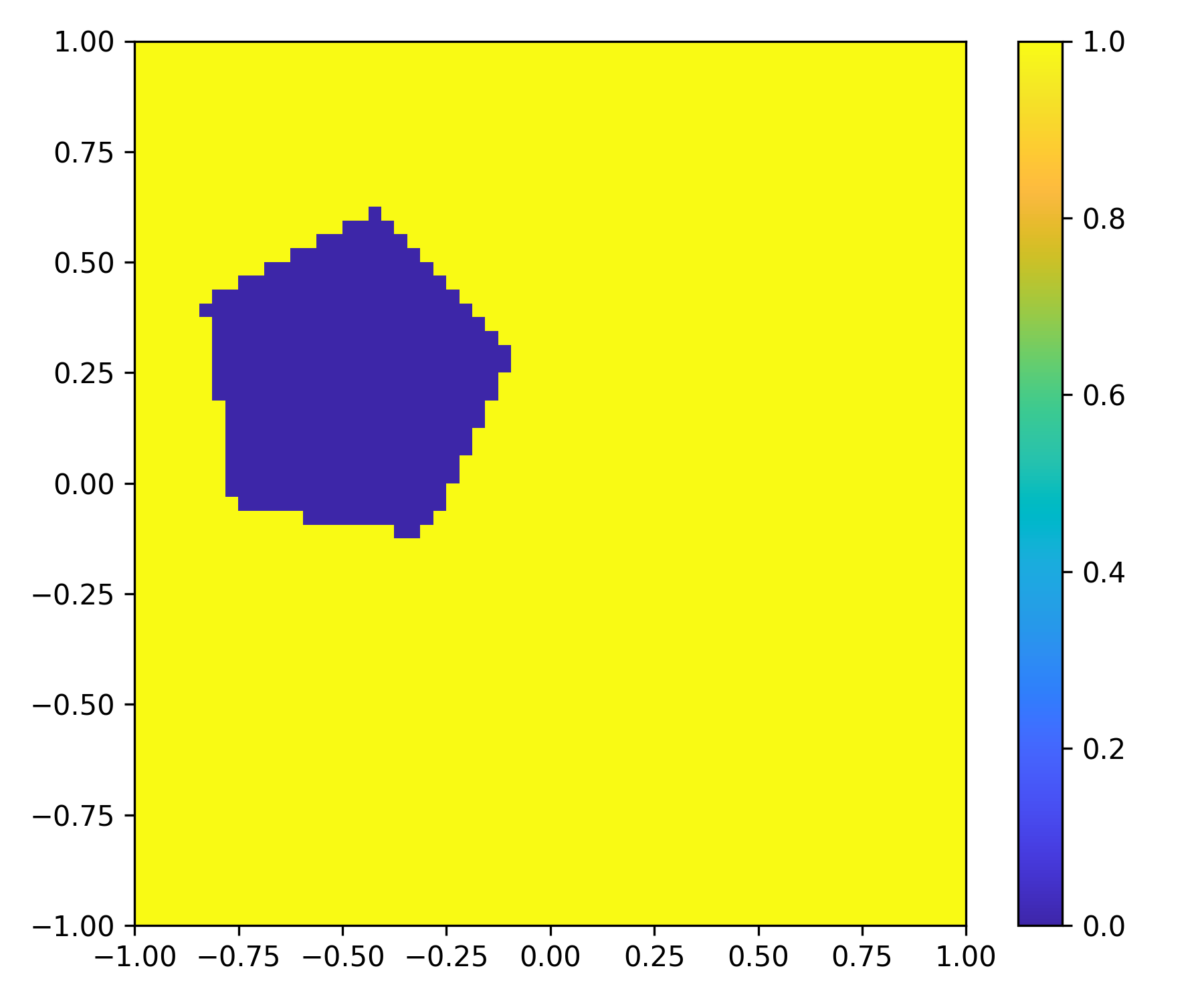}
			&\includegraphics[width=0.18\textwidth]{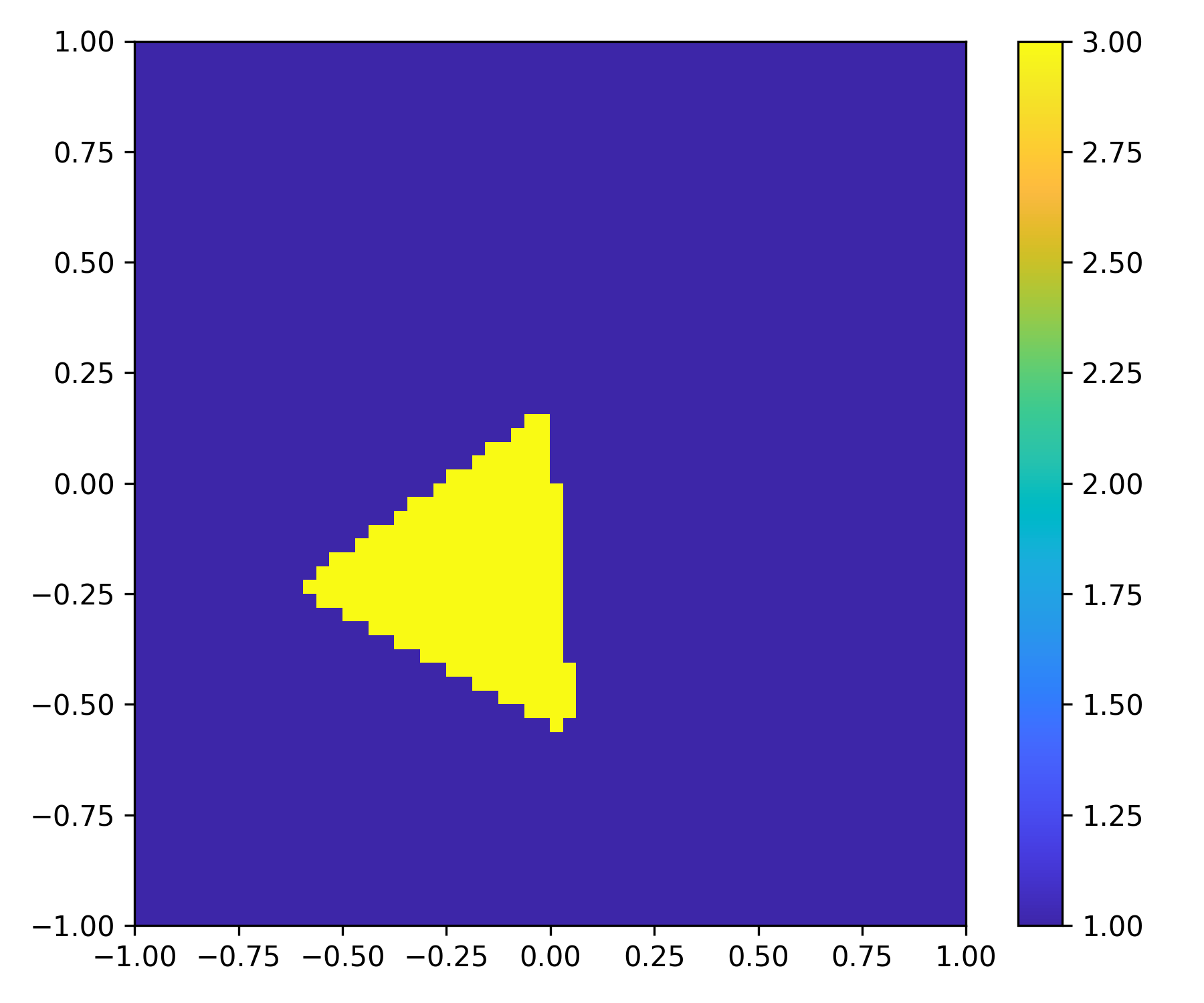}\\
			{$N_i=1$\\ $\delta=2\%$}&
			\includegraphics[width=0.18\textwidth]{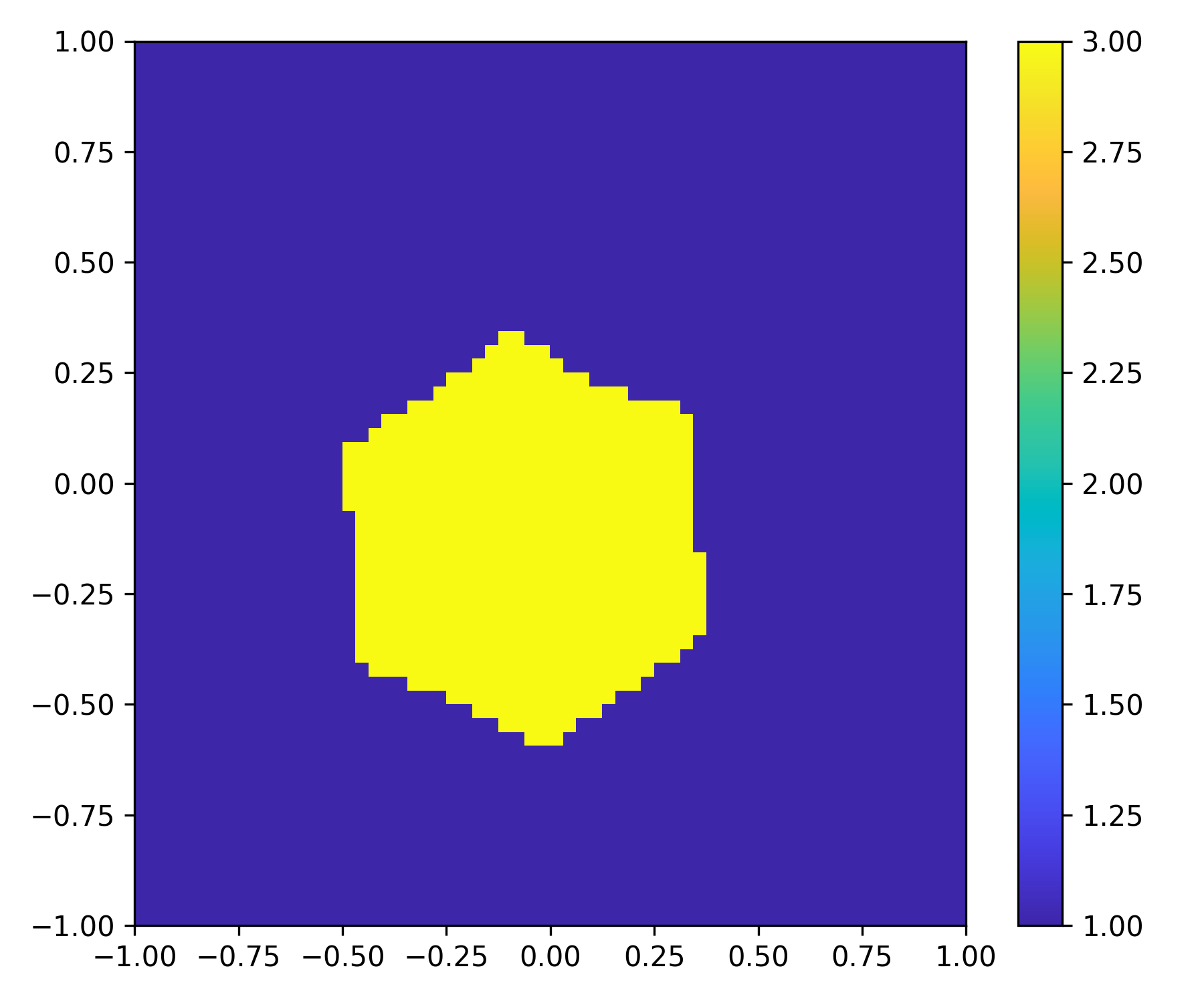}&\includegraphics[width=0.18\textwidth]{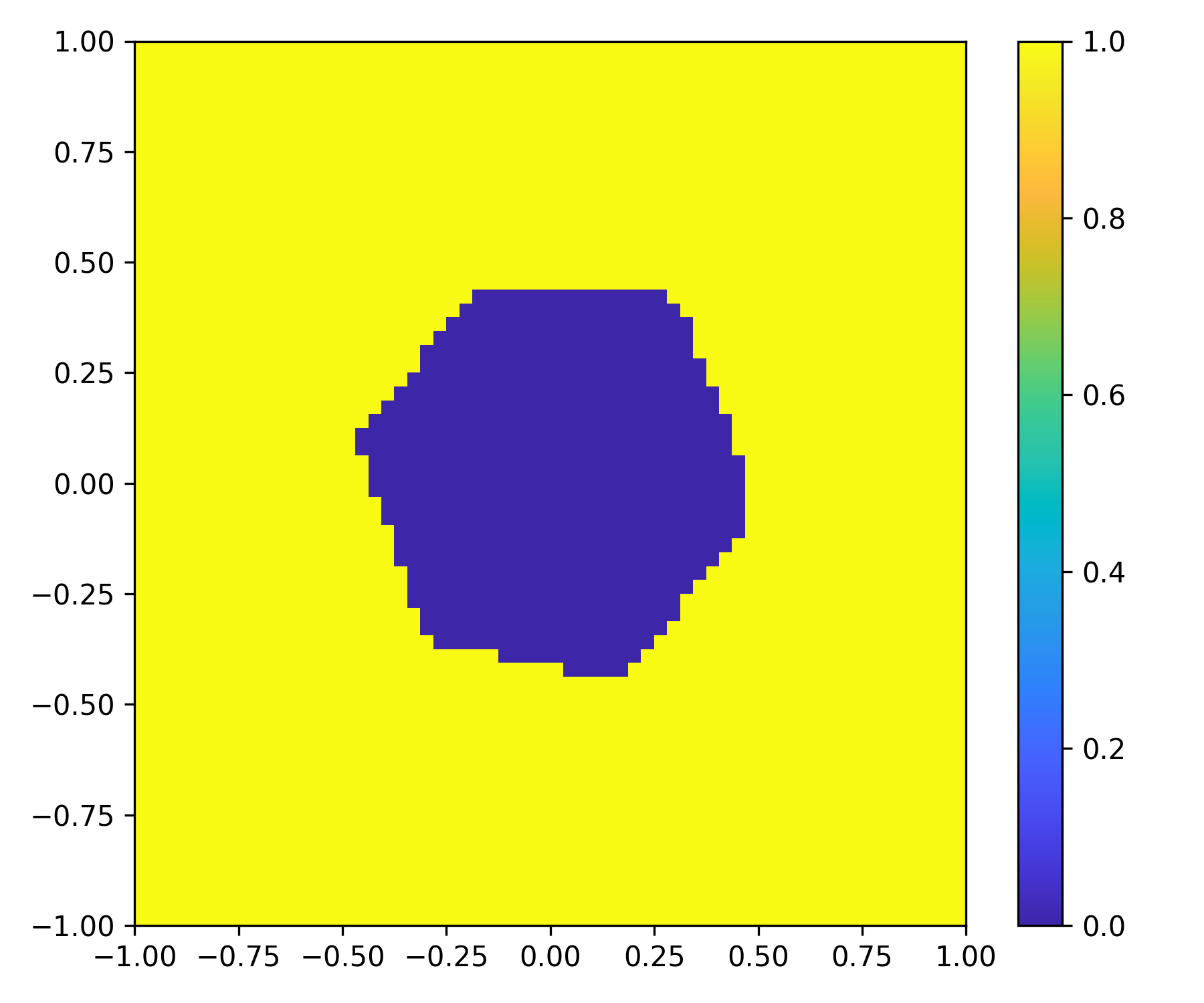} &\includegraphics[width=0.18\textwidth]{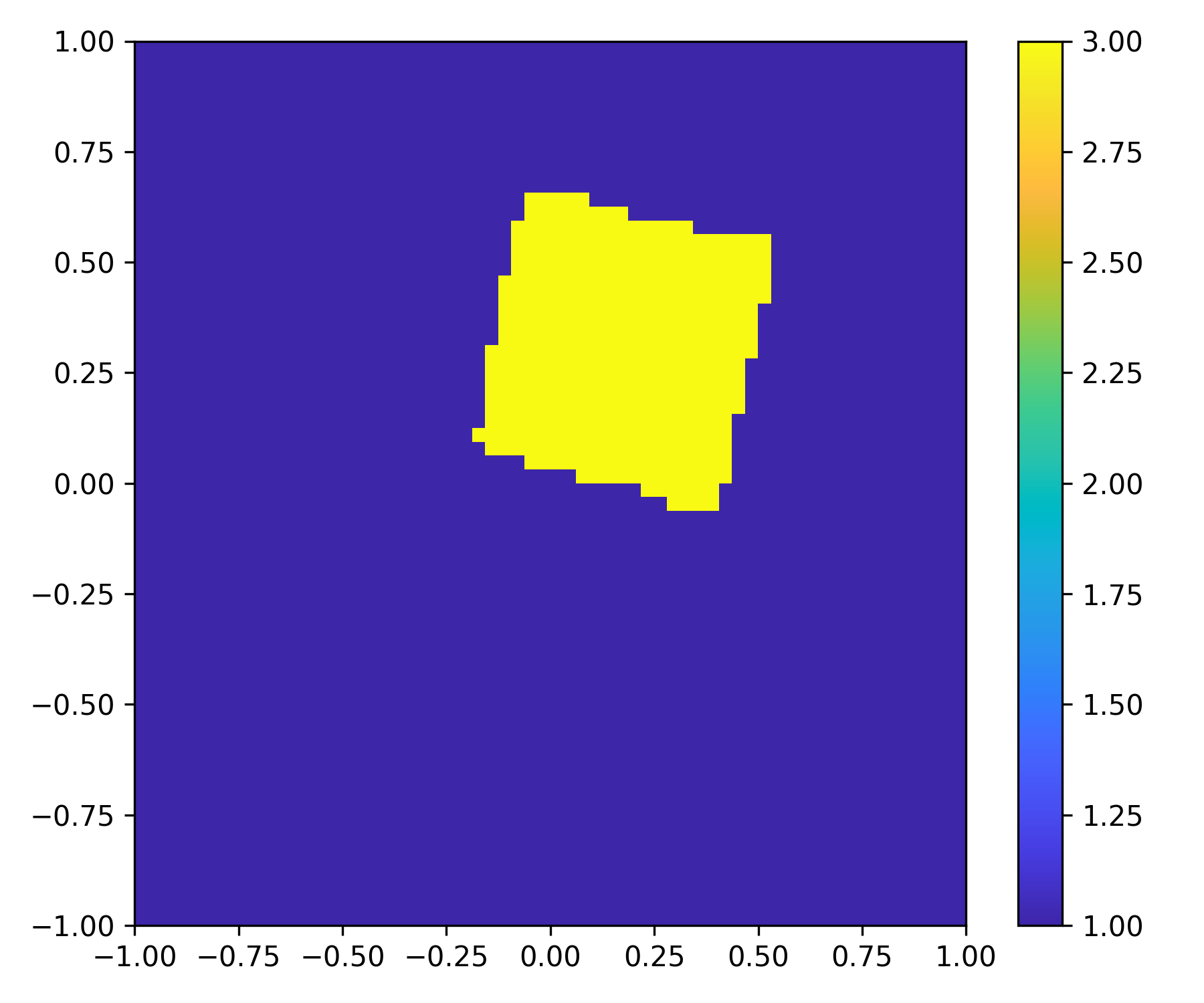}
			&\includegraphics[width=0.18\textwidth]{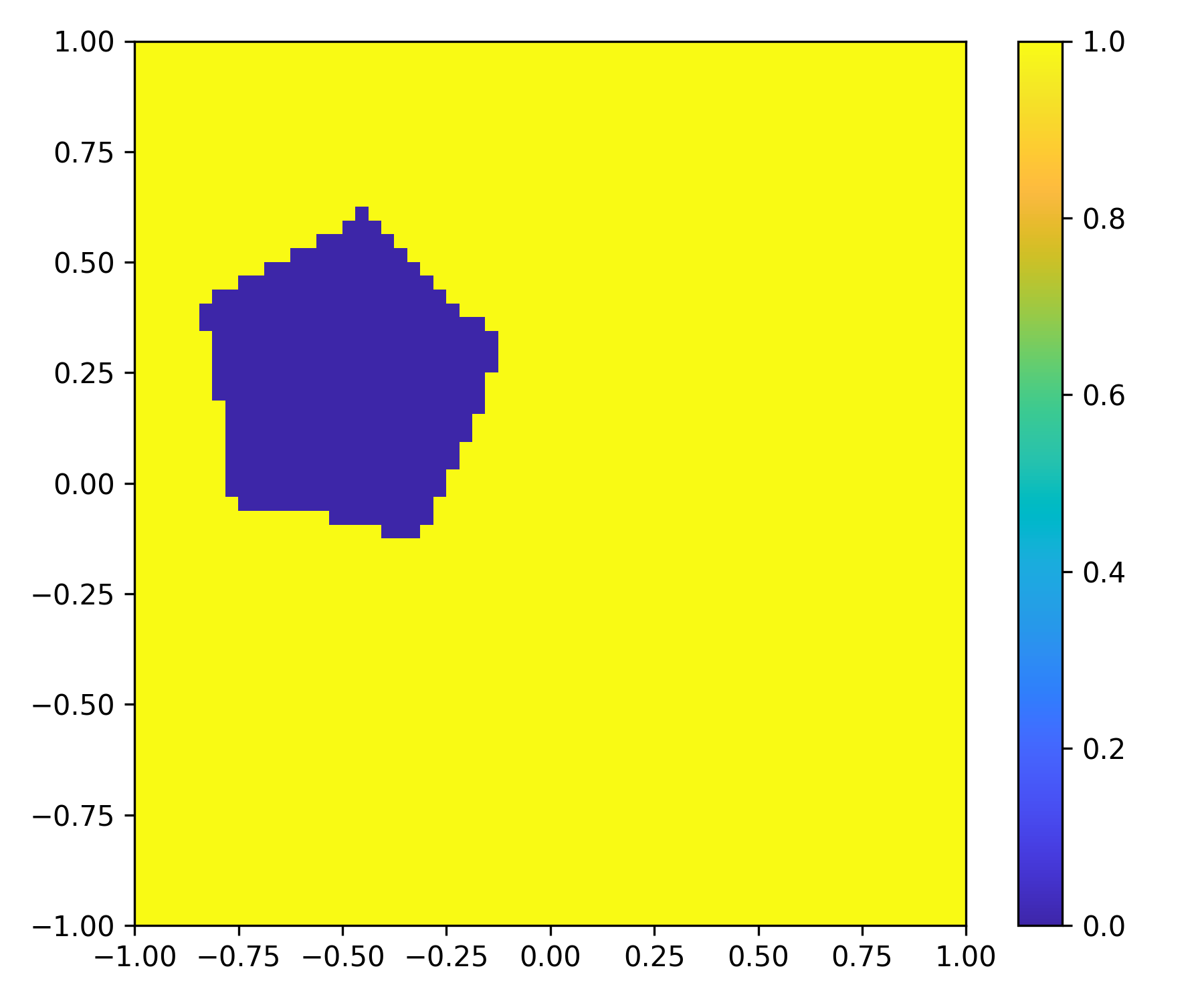}
			&\includegraphics[width=0.18\textwidth]{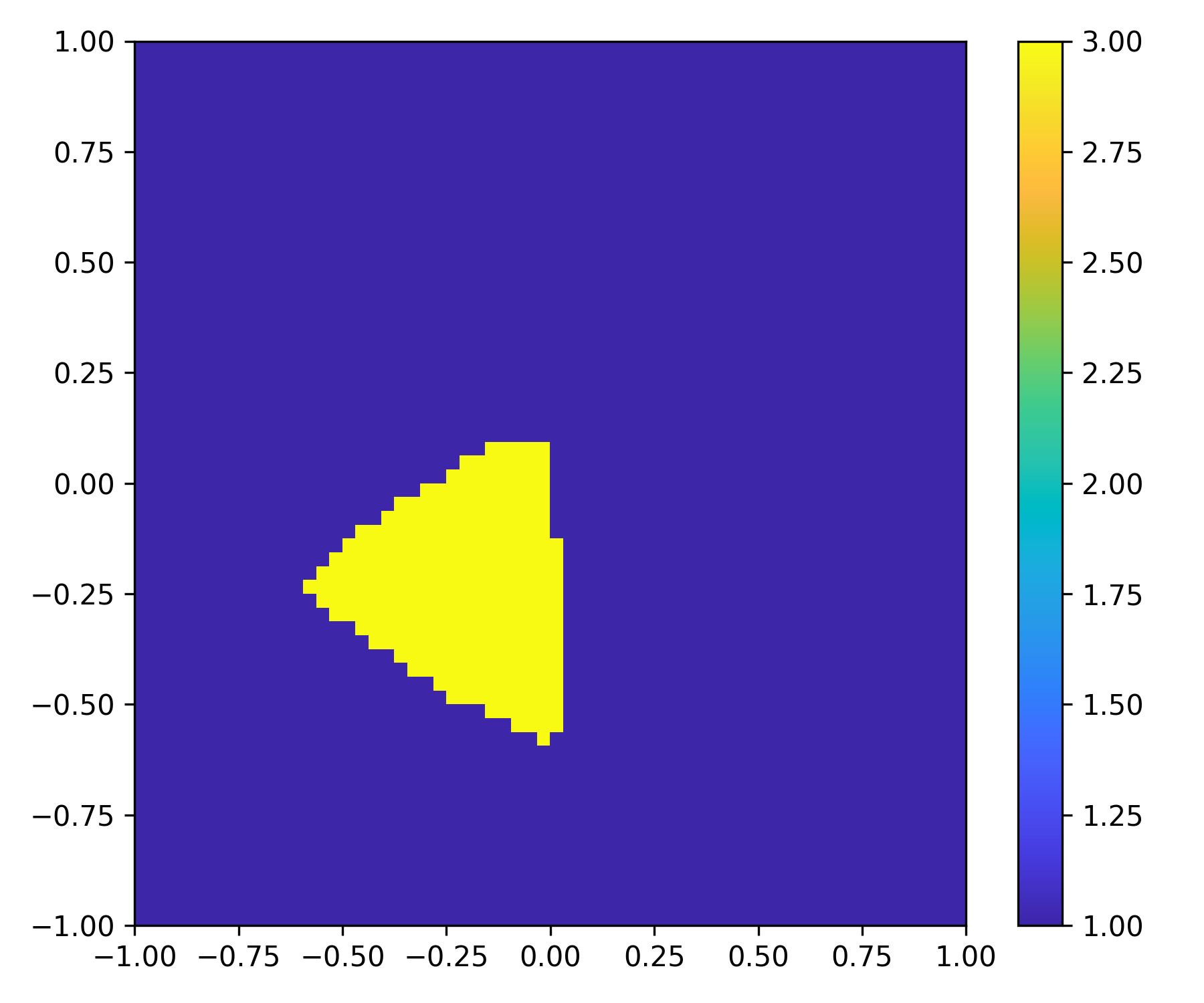}\\
			{$N_i=1$\\ $\delta=10\%$}&
			\includegraphics[width=0.18\textwidth]{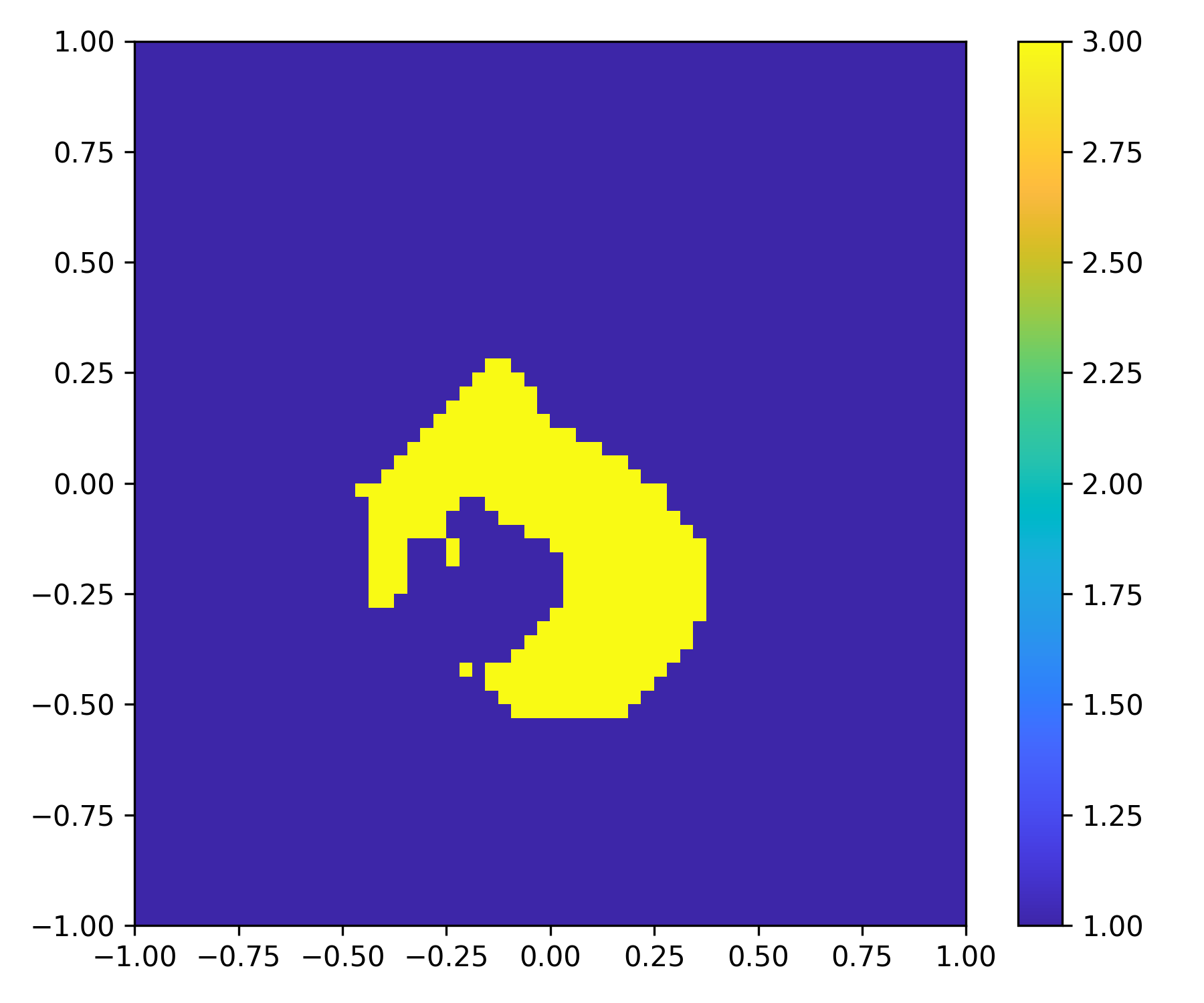}&\includegraphics[width=0.18\textwidth]{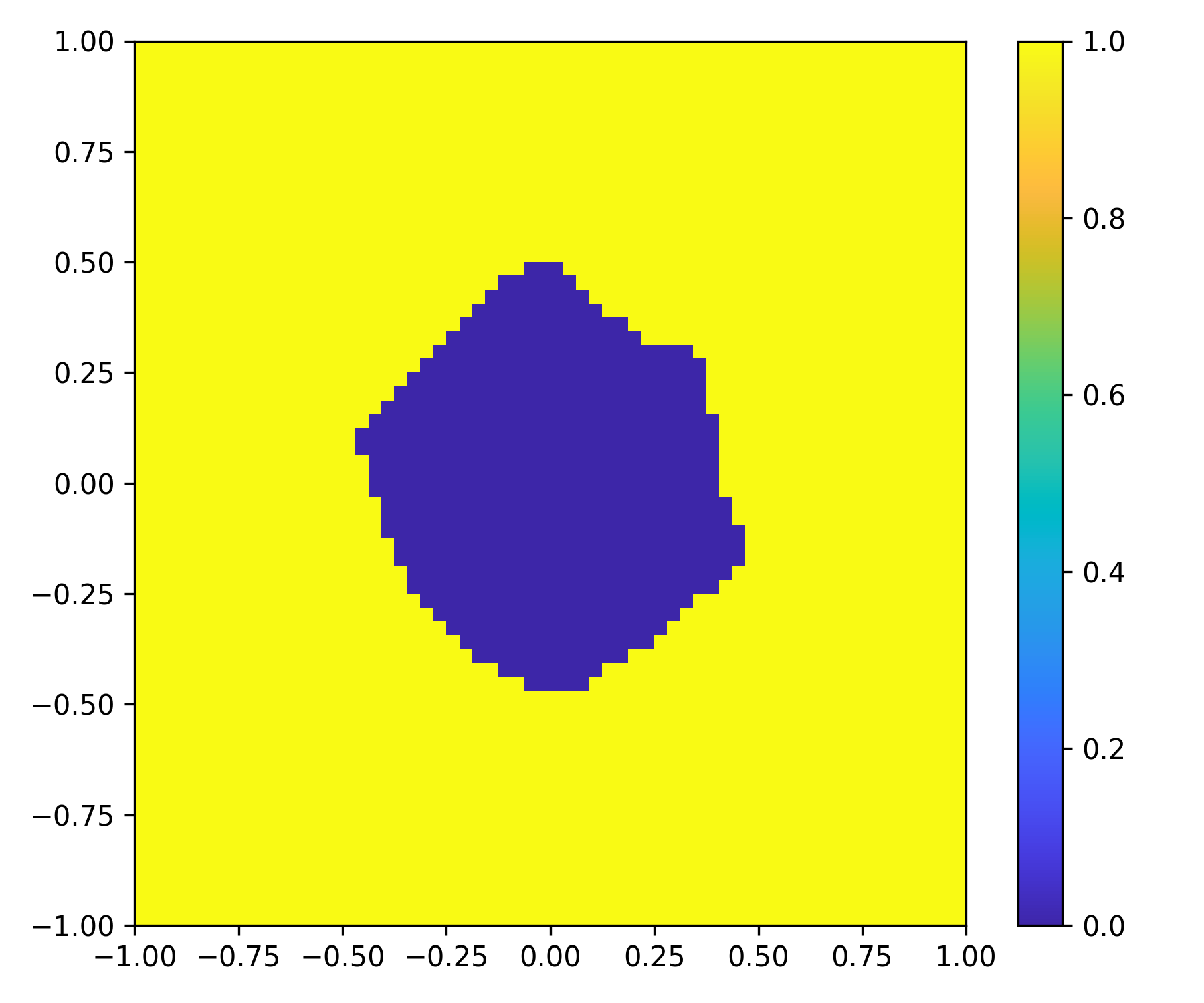} &\includegraphics[width=0.18\textwidth]{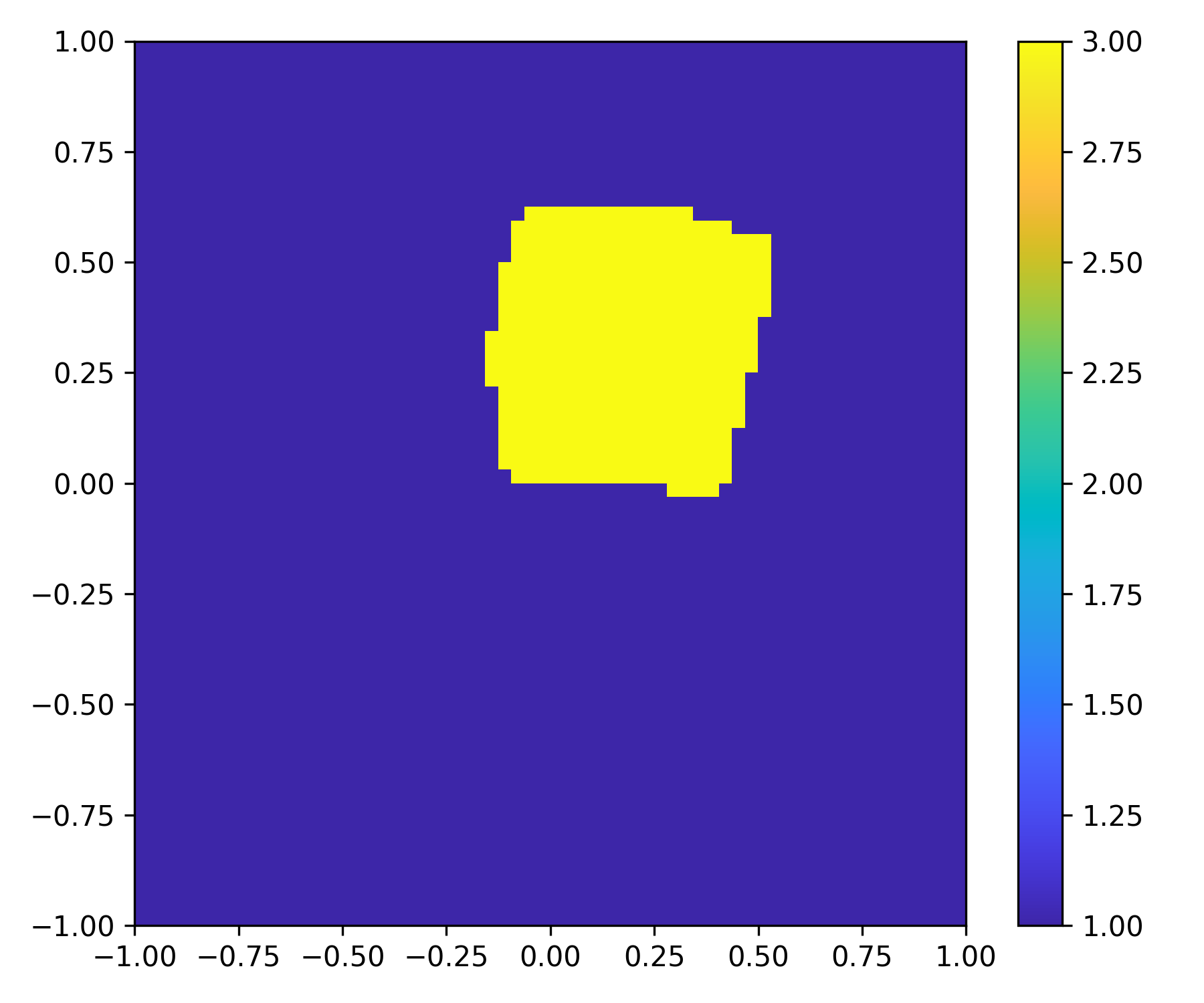}
			&\includegraphics[width=0.18\textwidth]{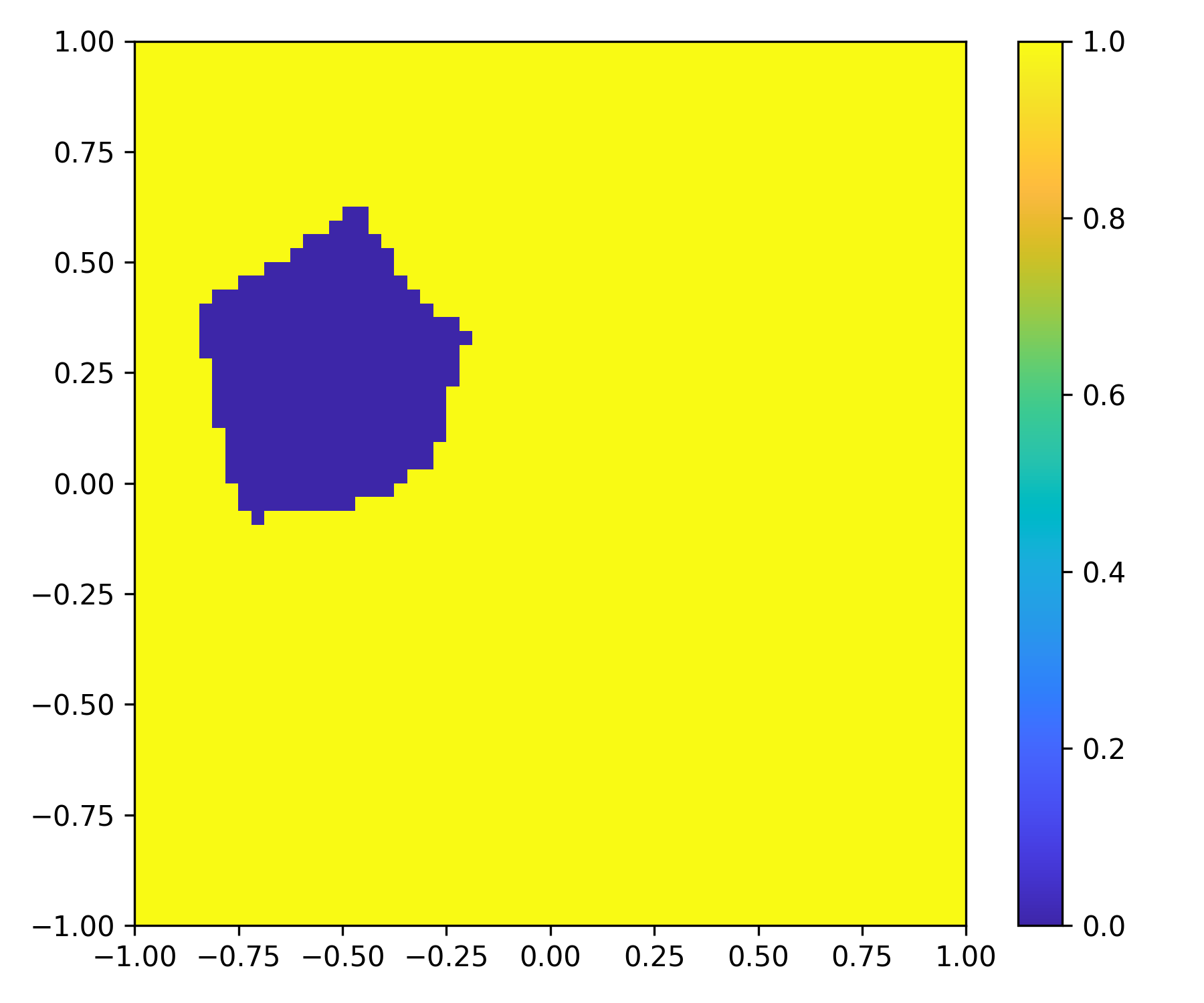}
			&\includegraphics[width=0.18\textwidth]{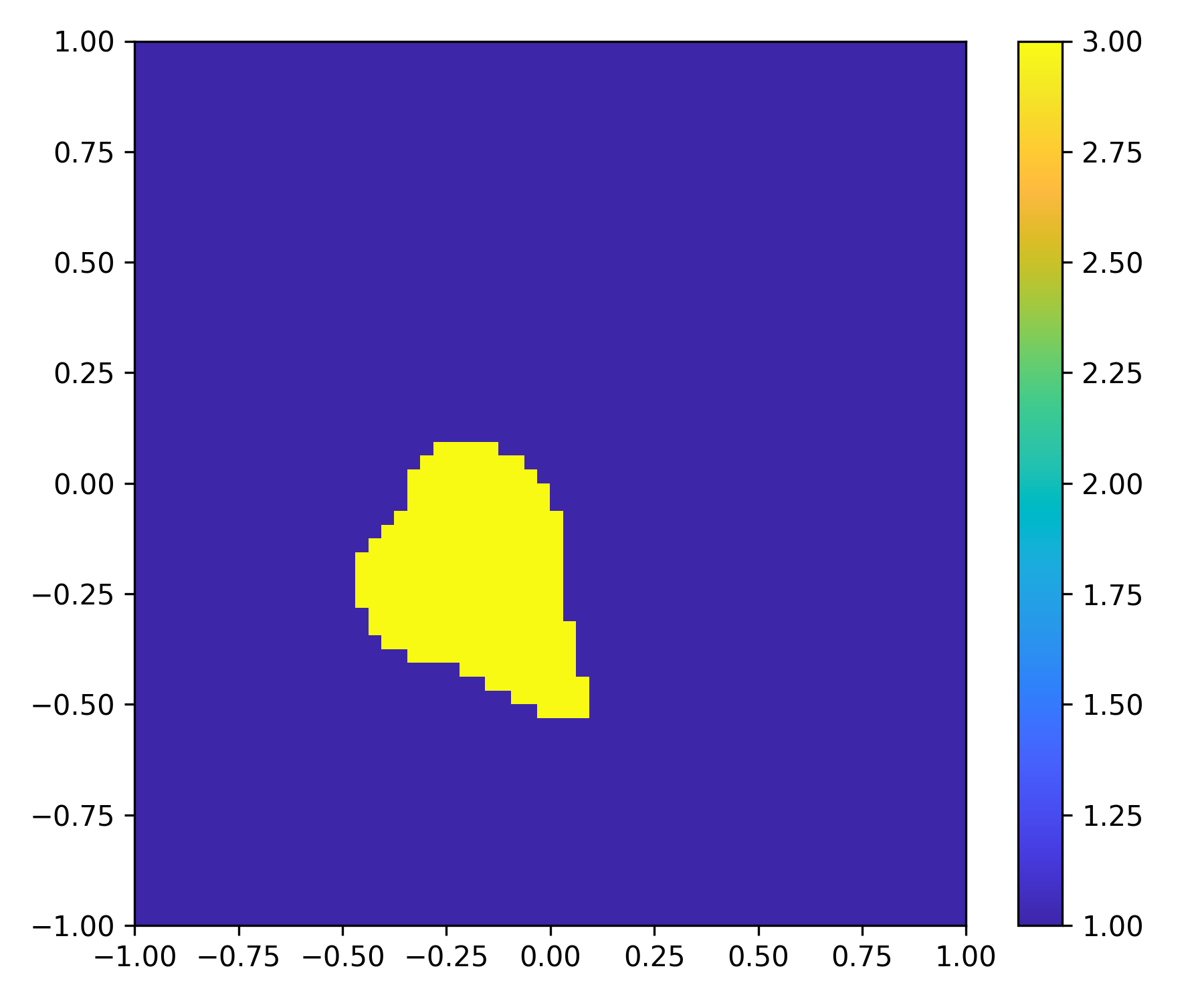}\\
			{$N_i=4$\\ $\delta=2\%$}&
		\includegraphics[width=0.18\textwidth]{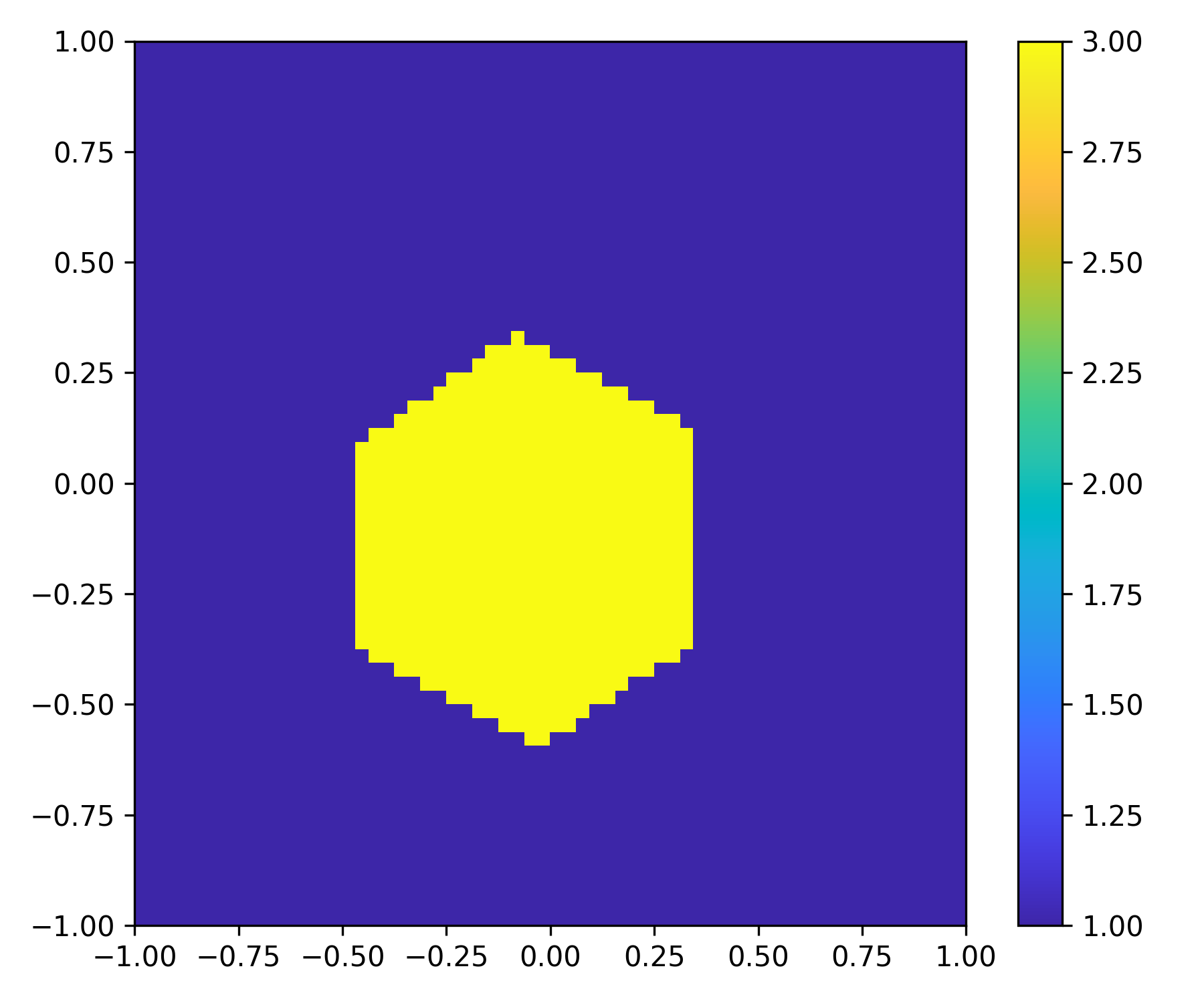}&\includegraphics[width=0.18\textwidth]{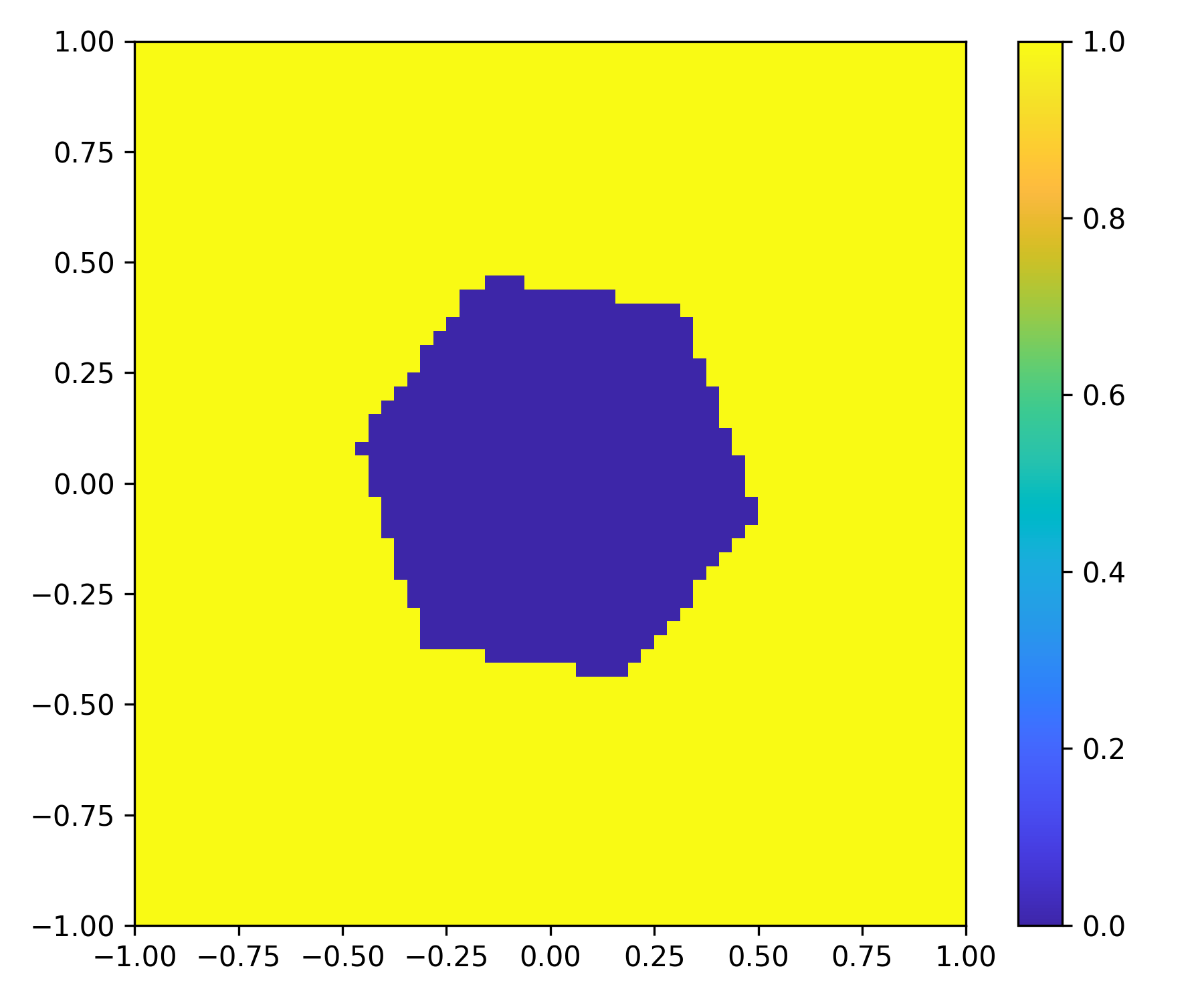} &\includegraphics[width=0.18\textwidth]{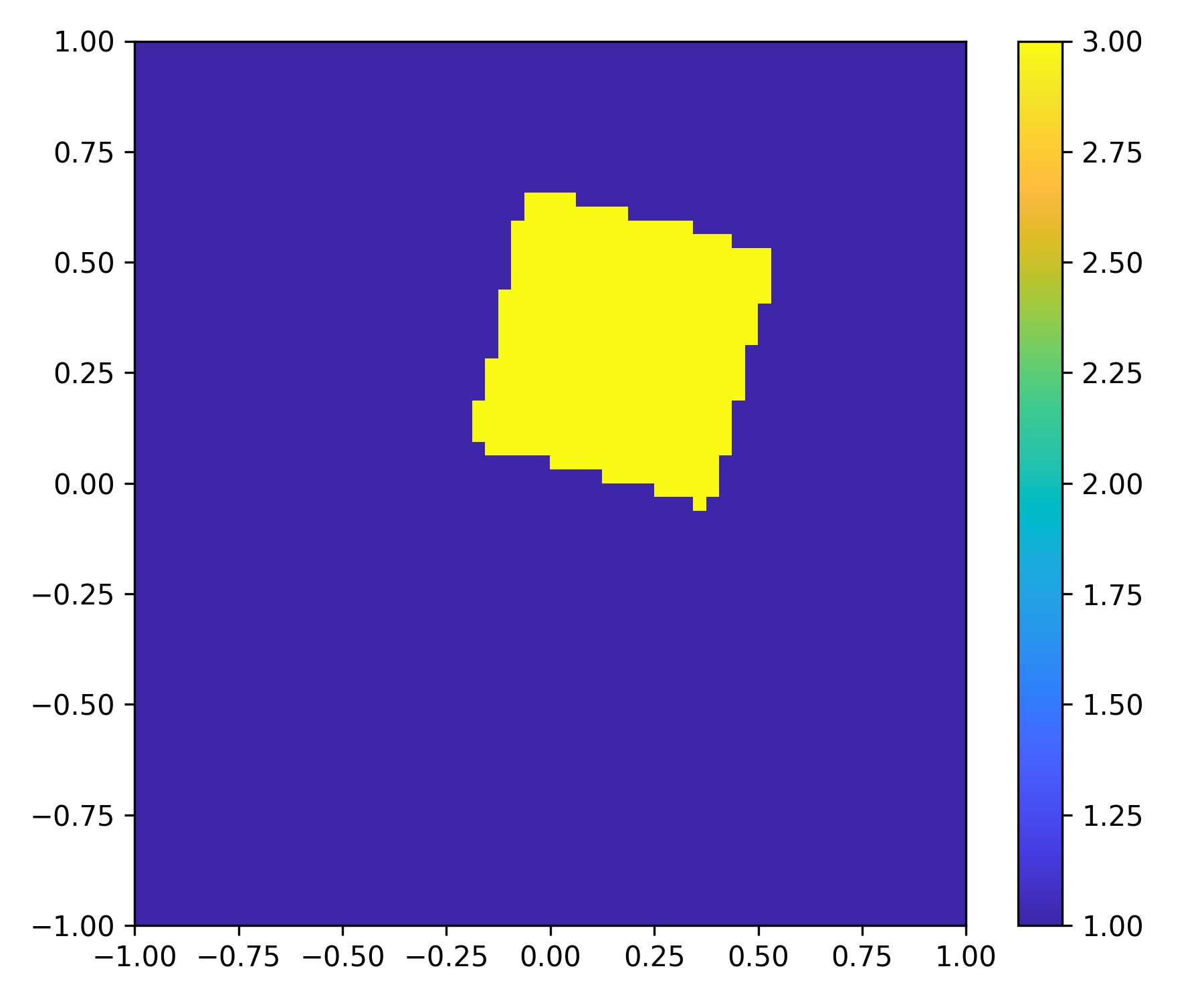}
		&\includegraphics[width=0.18\textwidth]{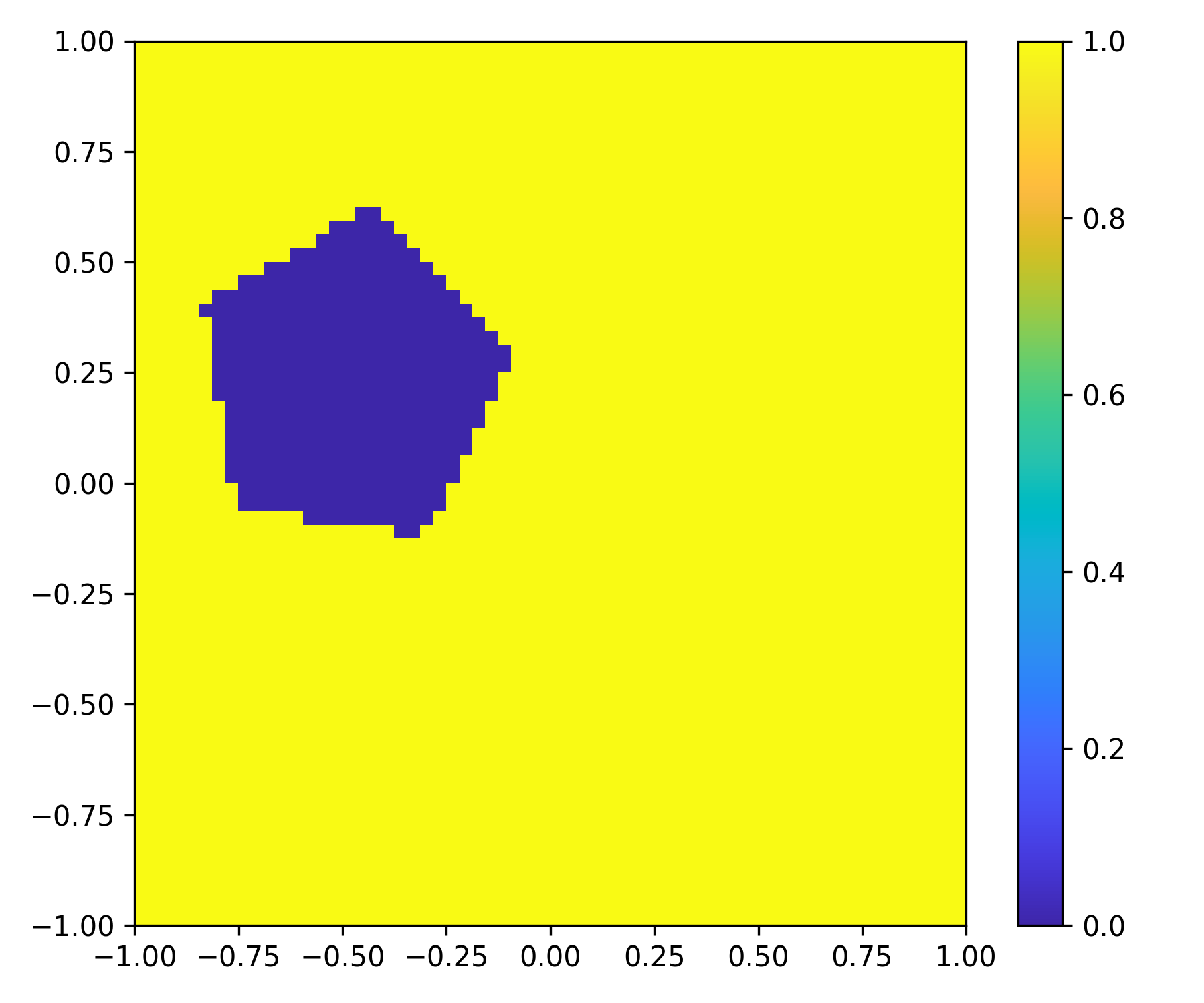}
		&\includegraphics[width=0.18\textwidth]{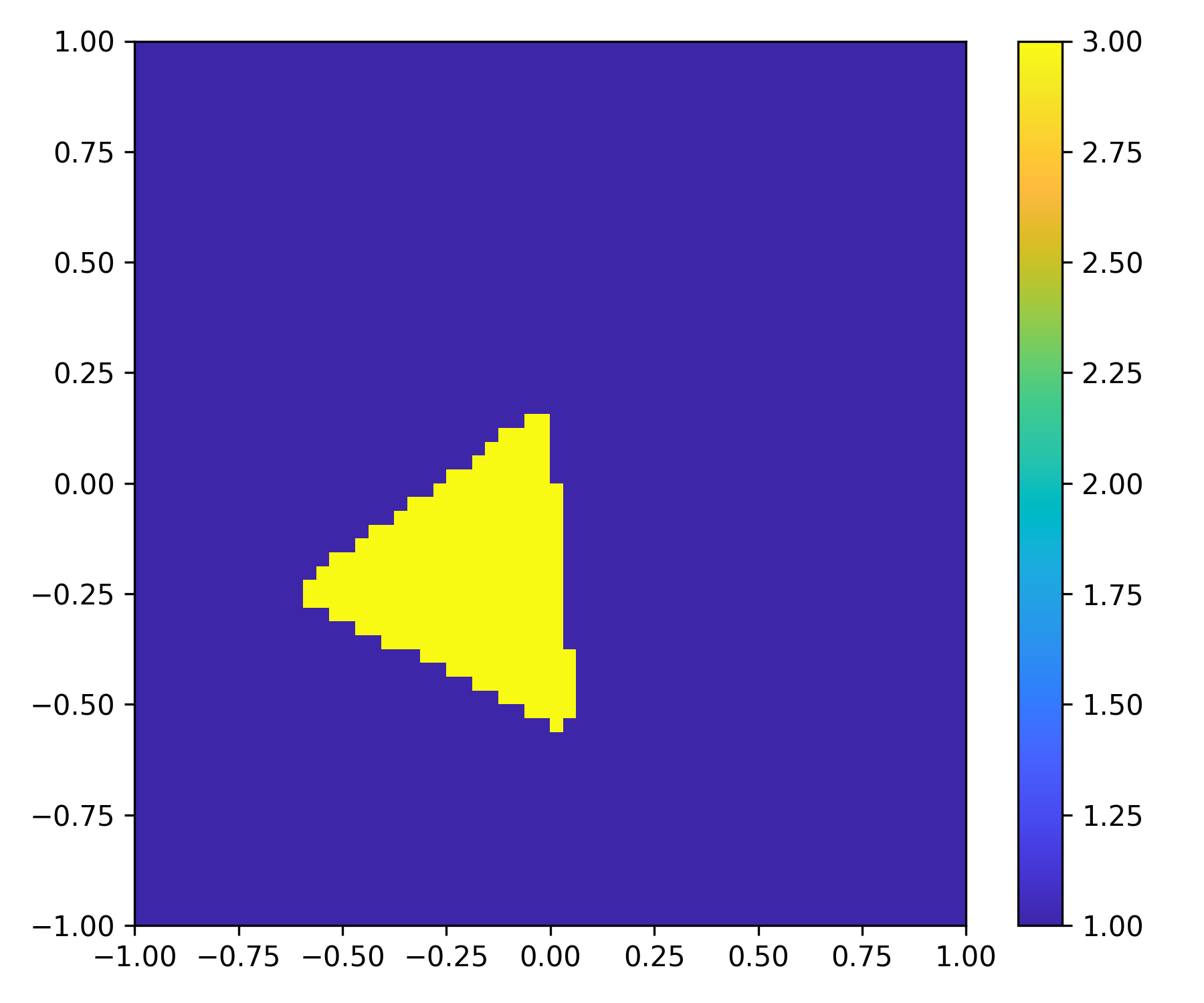}\\
		{$N_i=4$\\ $\delta=10\%$}&
		\includegraphics[width=0.18\textwidth]{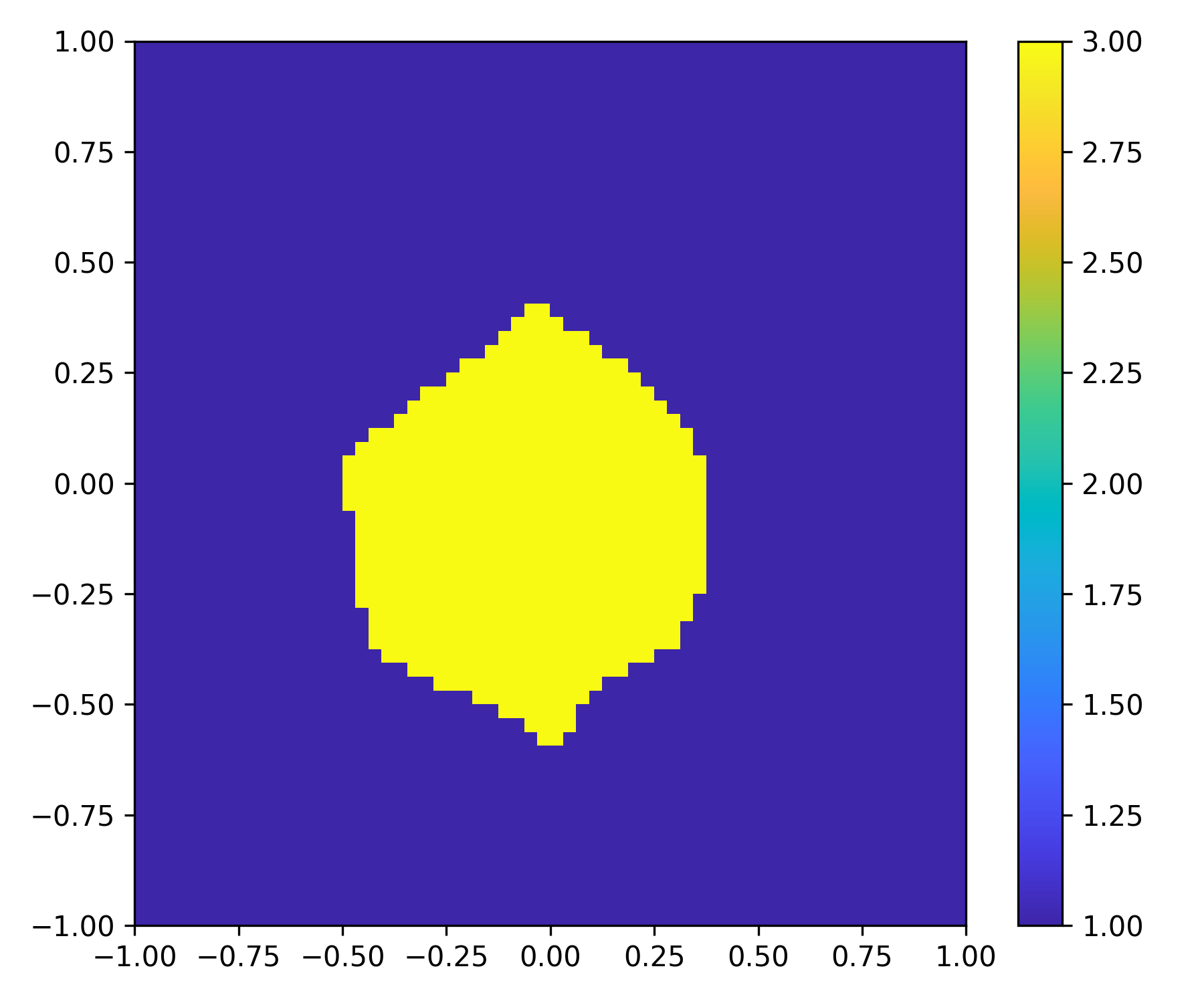}&\includegraphics[width=0.18\textwidth]{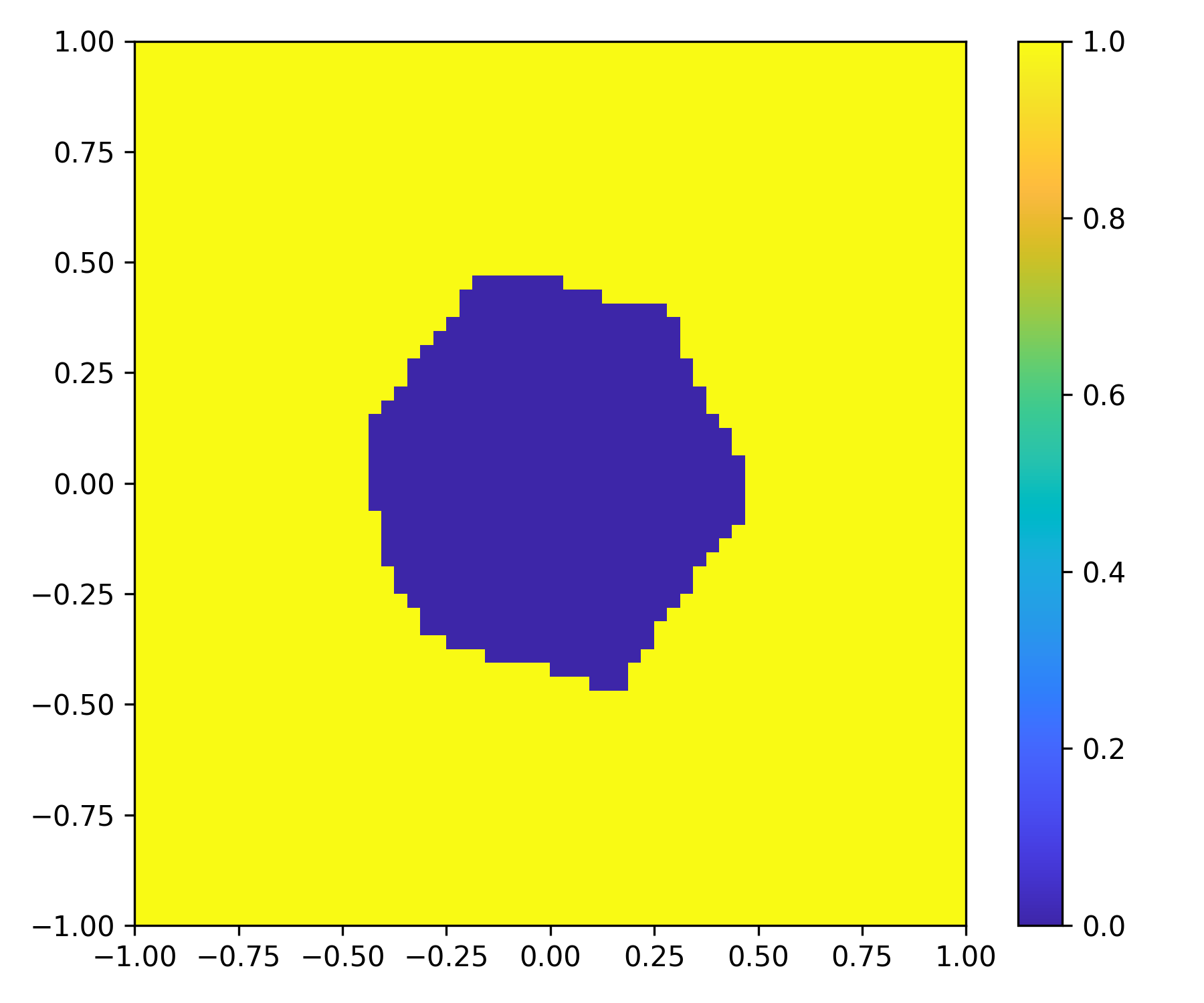} &\includegraphics[width=0.18\textwidth]{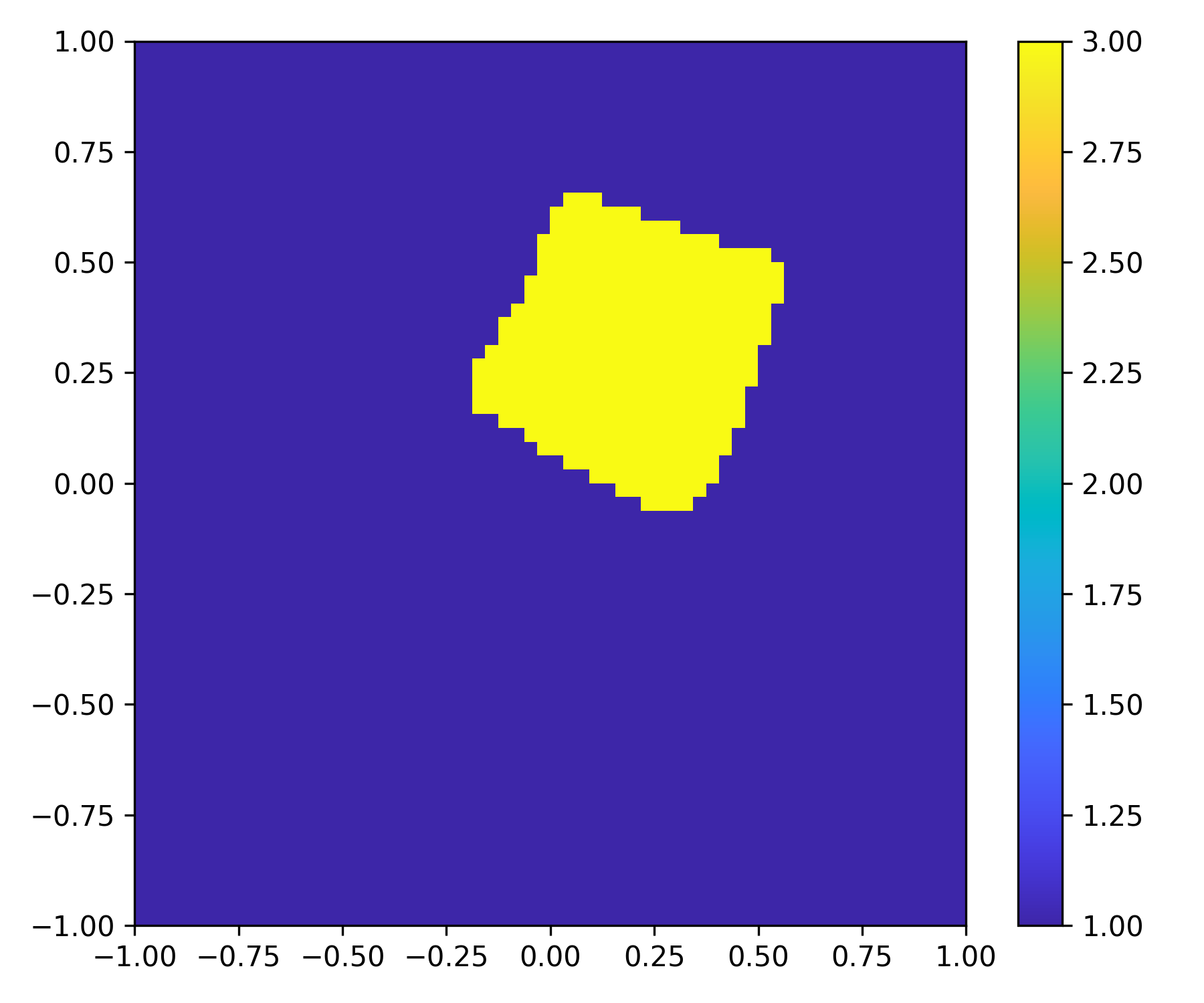}
		&\includegraphics[width=0.18\textwidth]{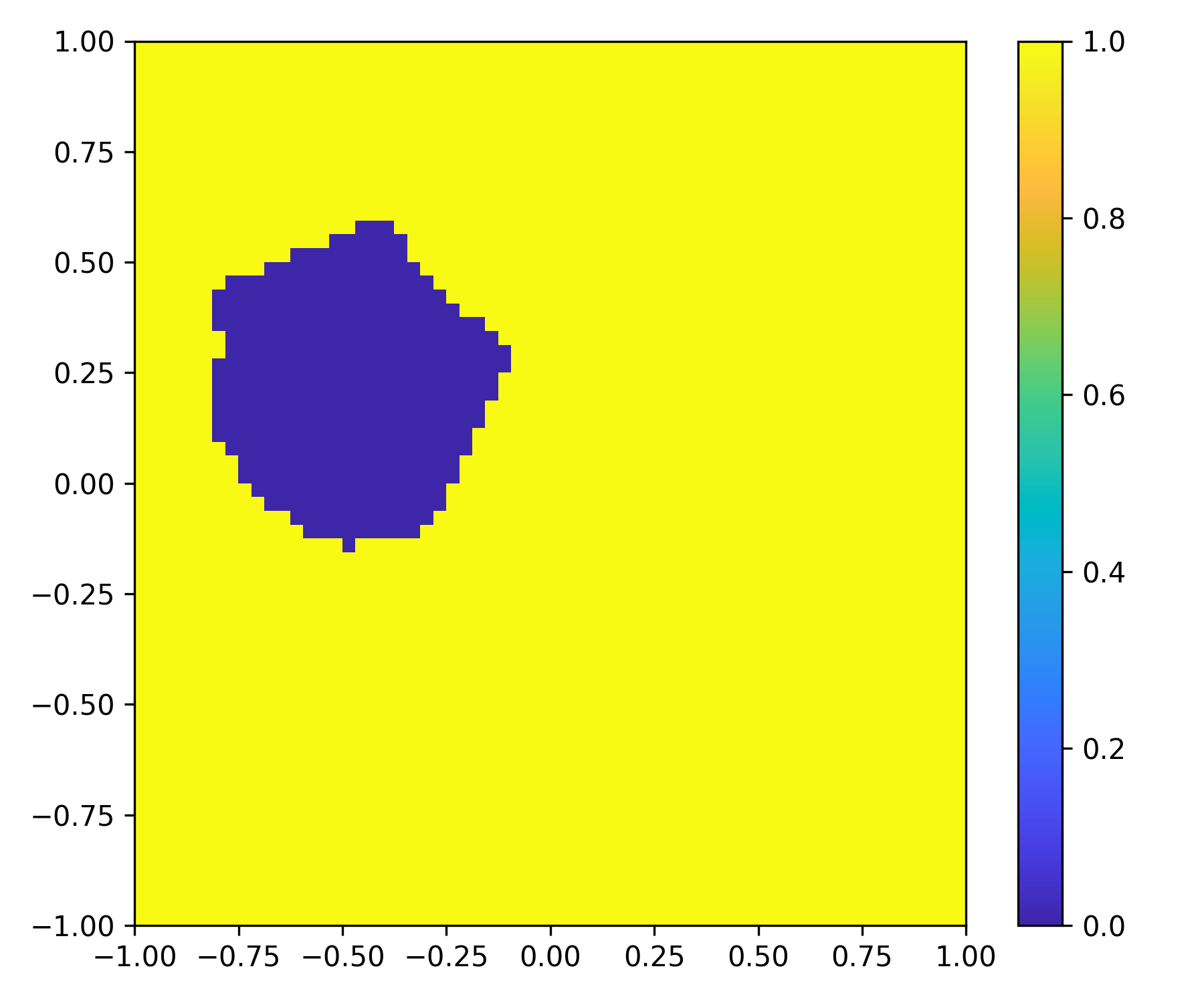}
		&\includegraphics[width=0.18\textwidth]{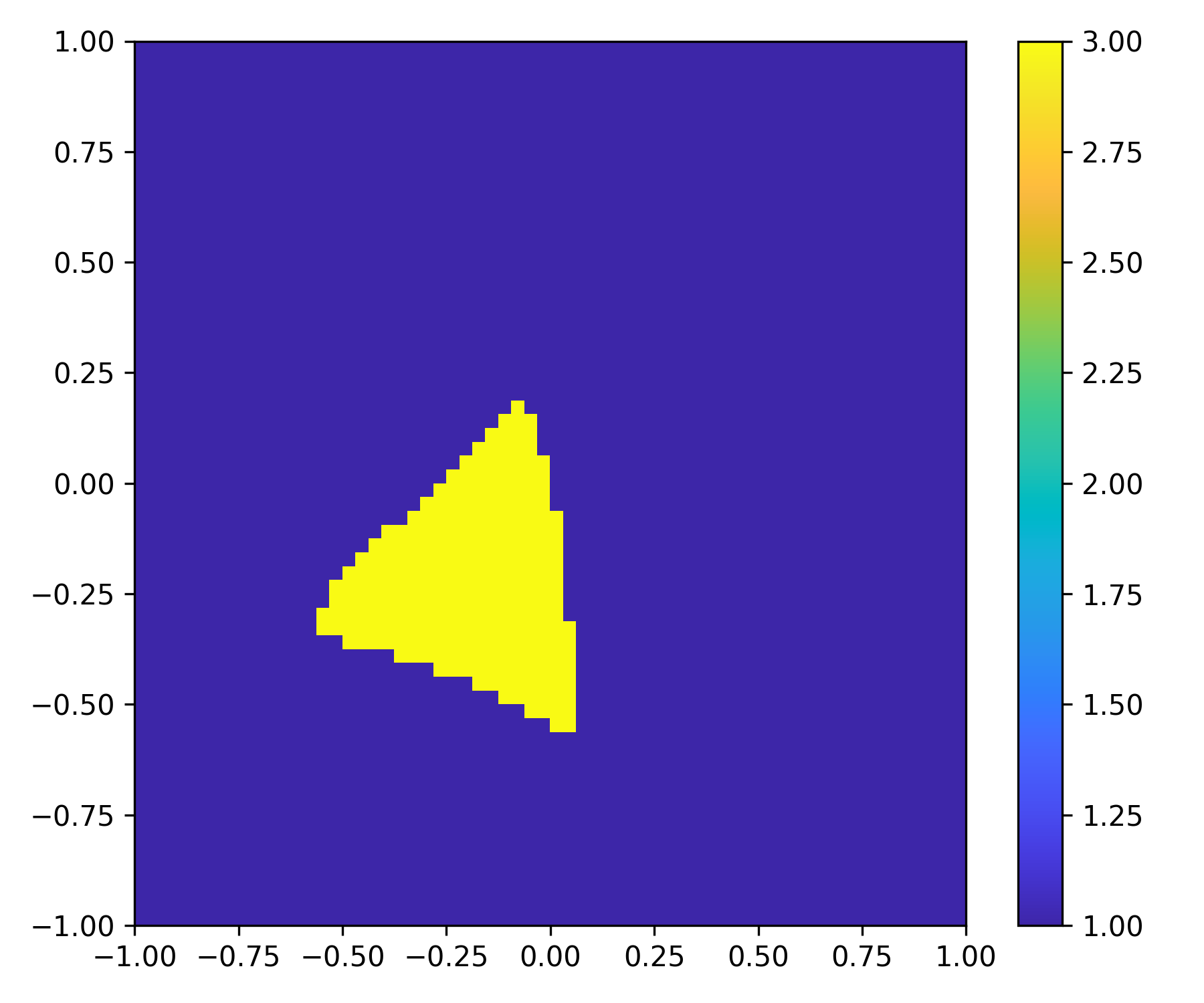}\\
		\end{tblr}	
		\caption{Example  \ref{examp:polygon}: Reconstructed images by using the networks trained by the Polygon dataset. 
		Row 1: true images; other rows: reconstructions with different incidences and noise levels.}
		\label{Polygon}
	\end{center}
\end{figure}

	\begin{table}[htbp]
	\begin{center}
		\begin{tabular}{c|c|c|c|c}
			\hline
			\textbf{Example} & $\mathbf{N_i=1,\delta=2\%}$ & $\mathbf{N_i=1,\delta=10\%}$ &
			$\mathbf{N_i=4,\delta=2\%}$ &
			$\mathbf{N_i=4,\delta=10\%}$	
			\\			
			\hline
			\textbf{Polygon} & $99.49\%$ & $97.72\% $& $99.77\%$ &$99.16\%$ \\
			\bottomrule
		\end{tabular}
	\end{center}
	\captionsetup{skip=1pt}
	\caption{Examples \ref{examp:polygon}: The accuracy $Acc$ for the polygon example.}
	\label{tab:Polygon}
\end{table}

\subsubsubsection{Tests with MNIST testing data.} 
\label{examp:mnist}
To evaluate the performance of the trained networks, we first test the examples from MNIST testing data. In Fig.\,\refeq{Mnist}, reconstructions of five examples from the MNIST testing data with different noise levels and number of incidences are presented. With $N_i=4$, the method can produce some reasonable reconstructions and can distinguish the hand-written digit and the circle while the accuracy is not very high. However, if 16 incidences are employed, high-quality and stable reconstructions can be achieved. This shows that multiple data are important for recovering complicated scatterers and the DSM-DL can fully make use of multiple data. The average relative error for the testing data is presented in Table.\,\ref{tab:MNIST}.
	
\begin{figure}[h]\small
		\begin{center}
			\begin{tblr}
				{colspec = {X[-1,m]X[c,h]X[c,h]X[c,h]X[c,h]X[c,h]},
					stretch = 0,
					rowsep = 0pt,}
				{Ground\\Truth}&
				\includegraphics[width=0.18\textwidth]{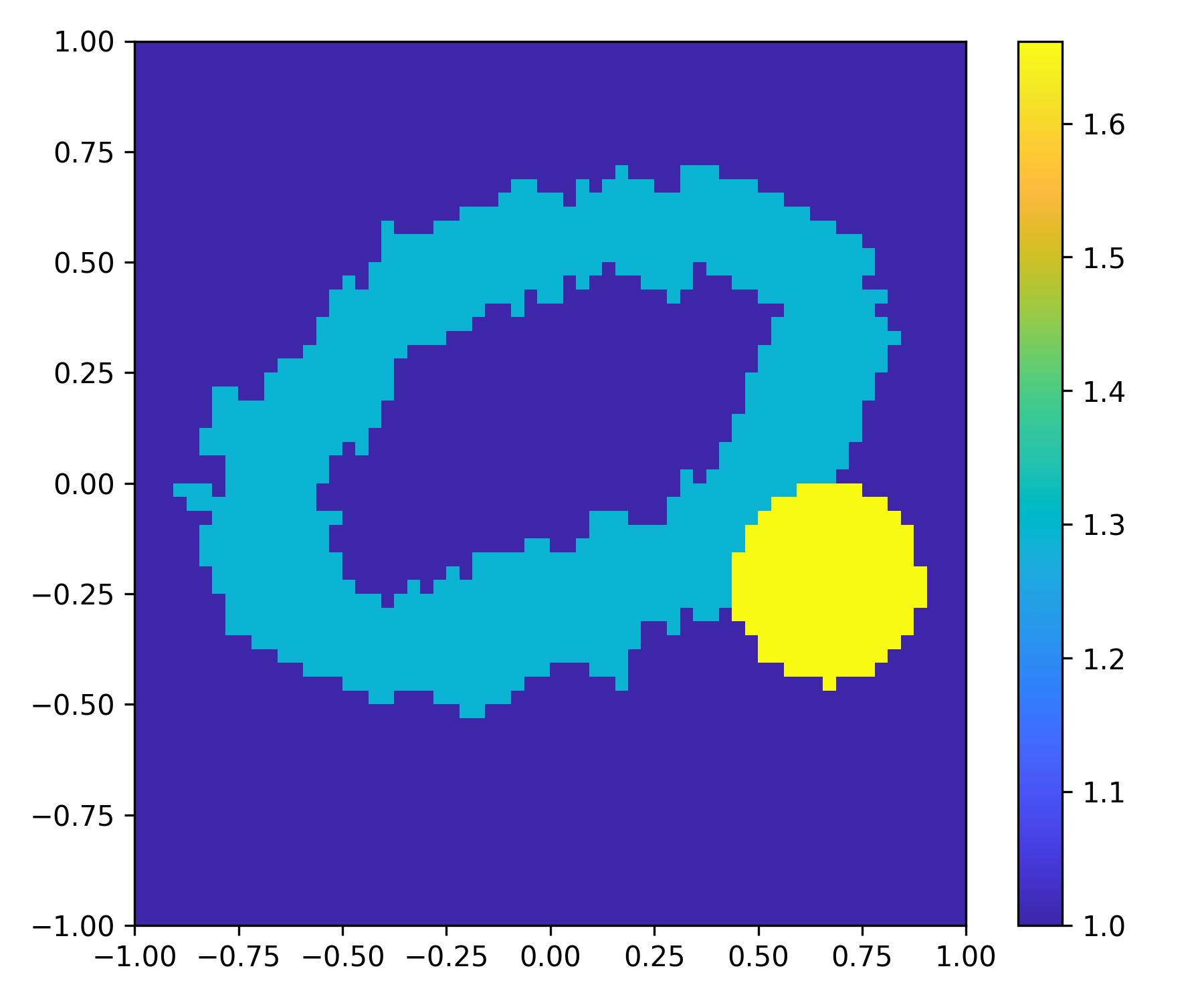}&
				\includegraphics[width=0.18\textwidth]{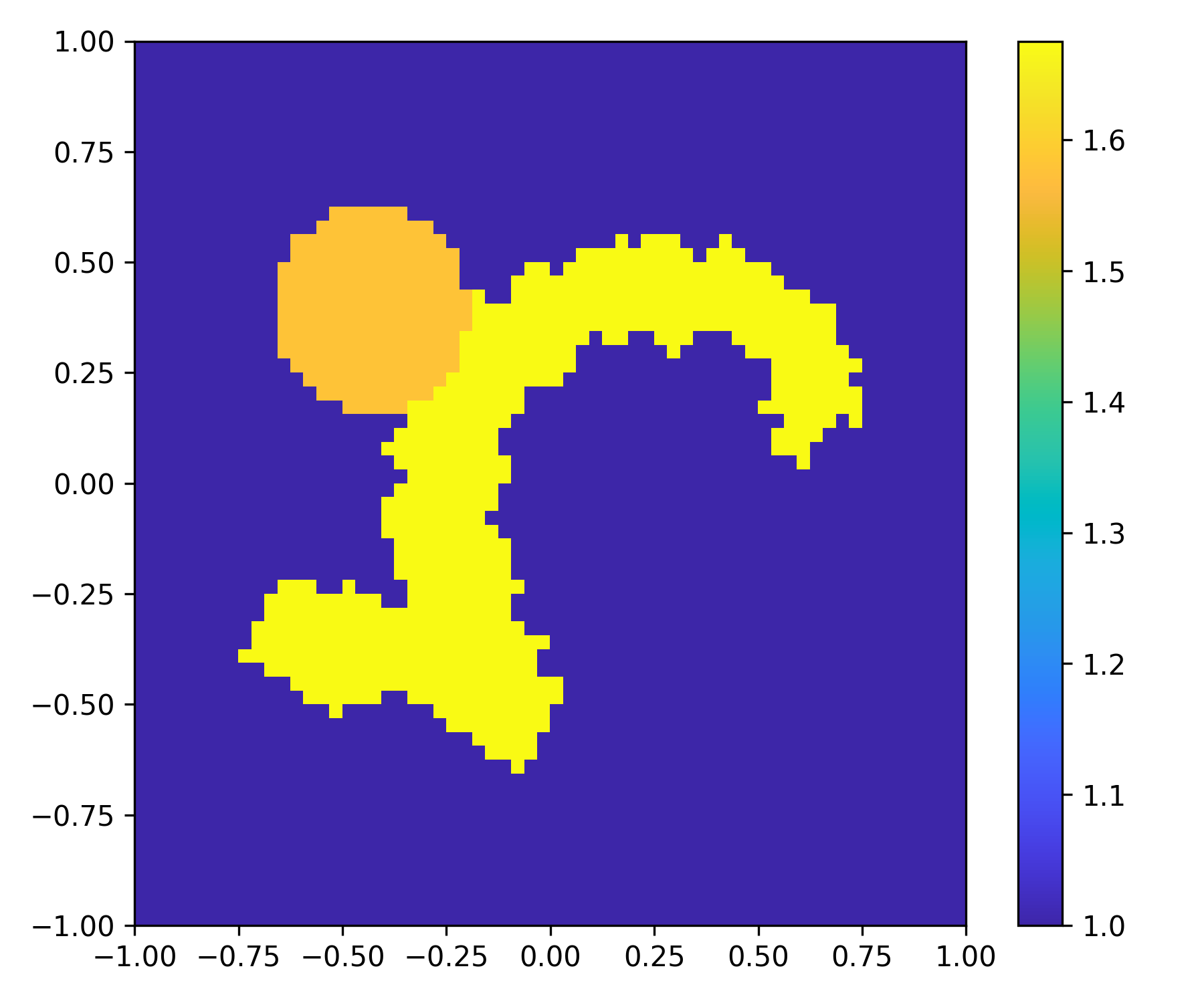} &\includegraphics[width=0.18\textwidth]{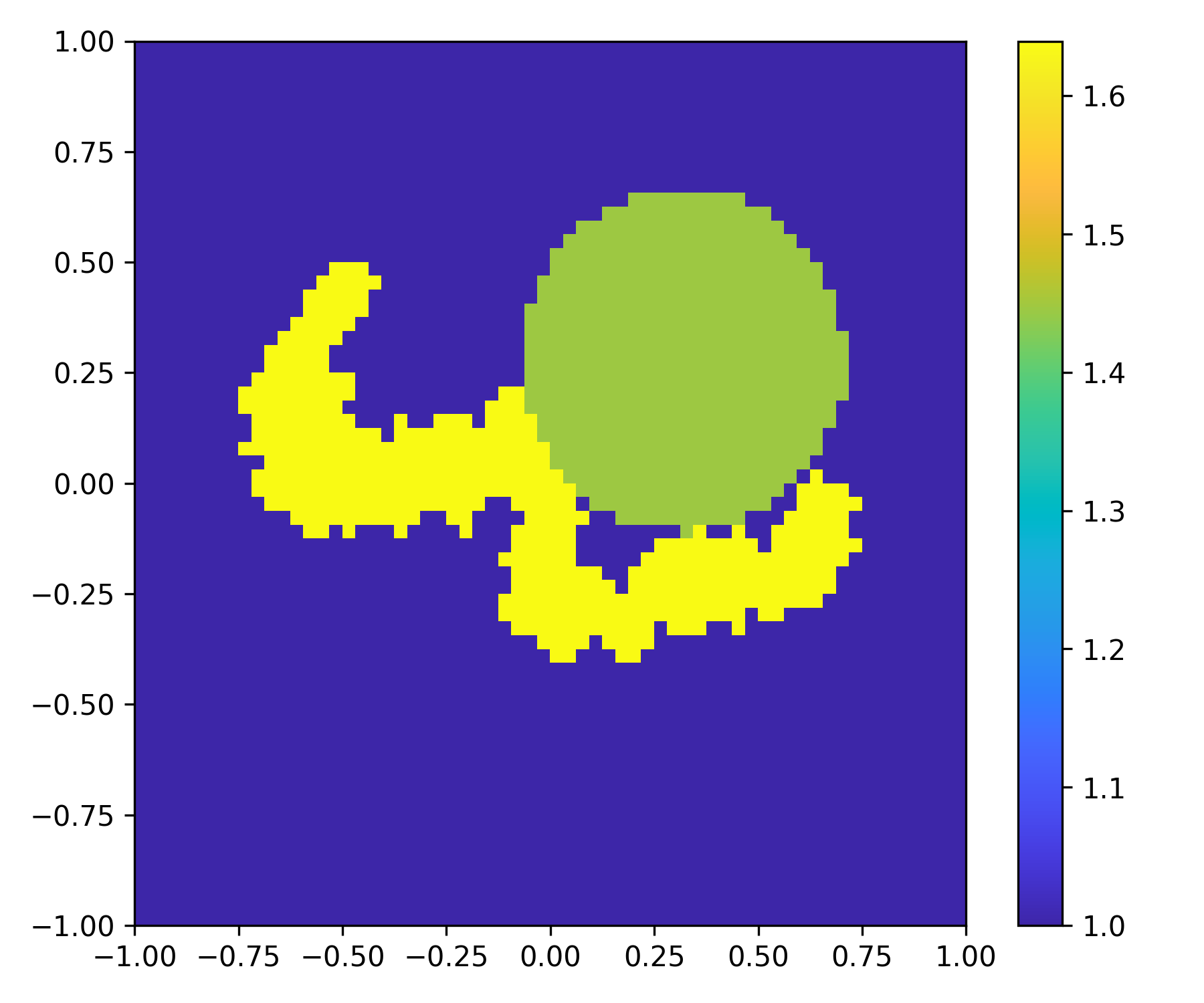}
				&\includegraphics[width=0.18\textwidth]{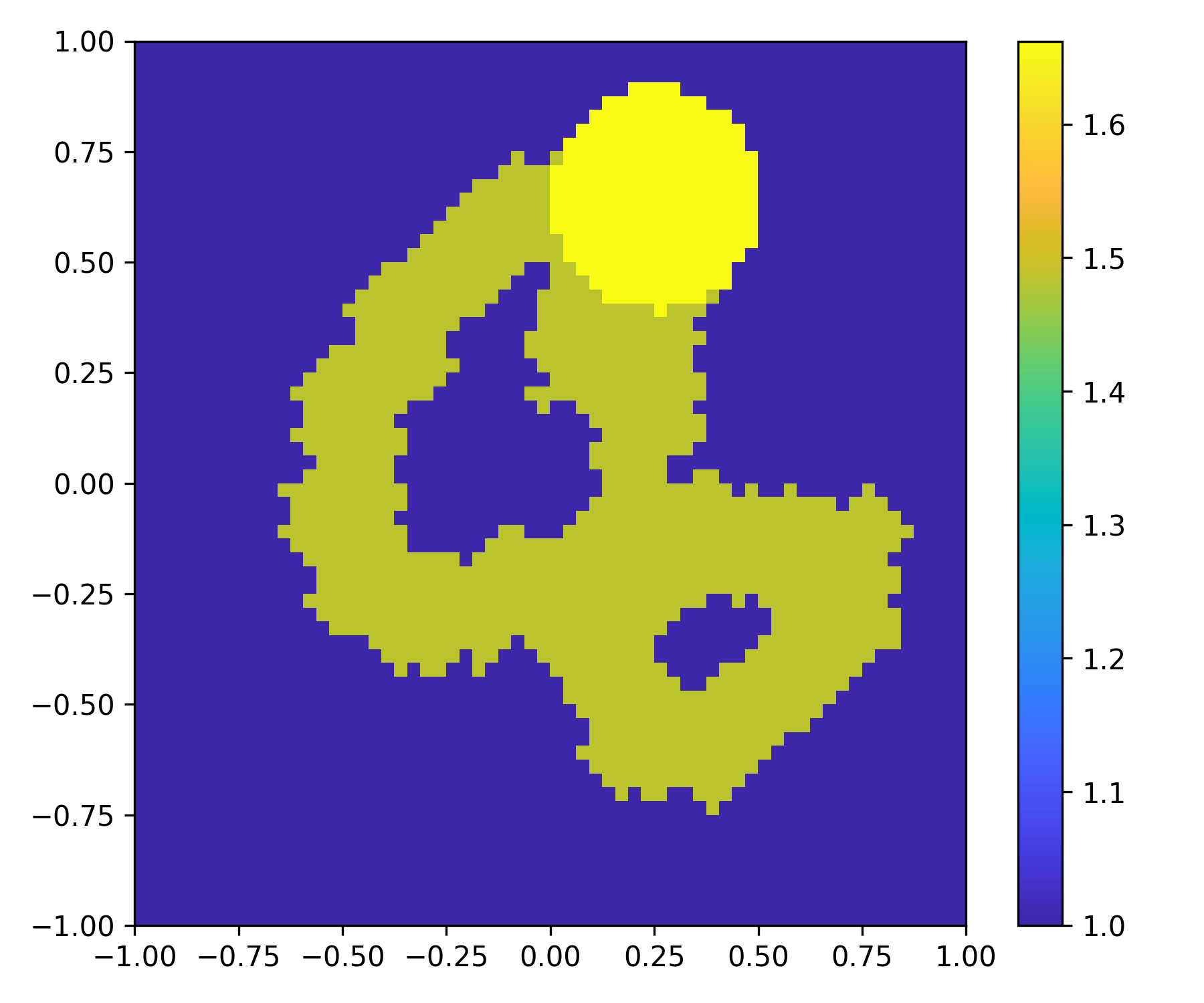}
				&\includegraphics[width=0.18\textwidth]{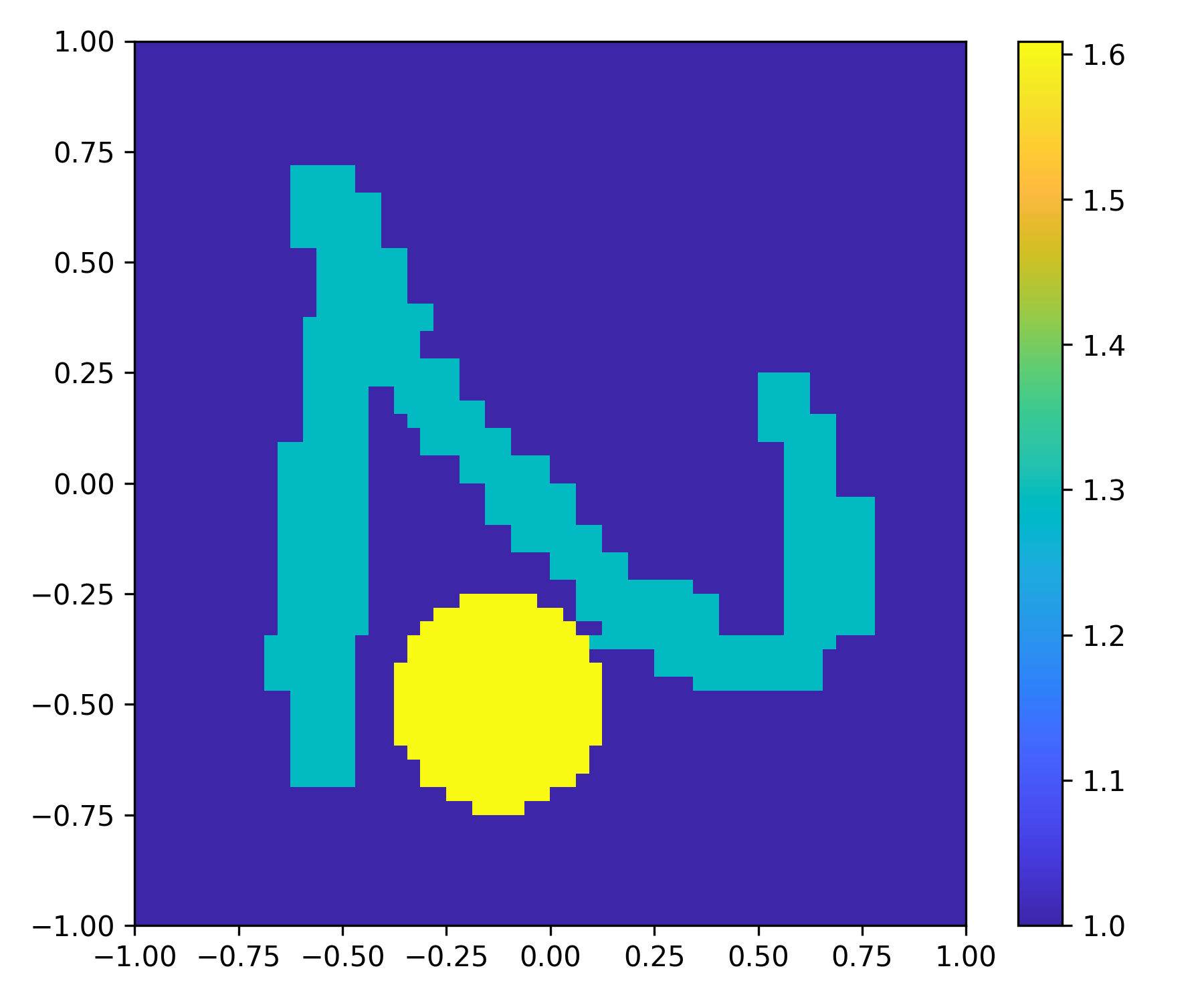}\\
				{$N_i=4$\\ $\delta=5\%$}&
				\includegraphics[width=0.18\textwidth]{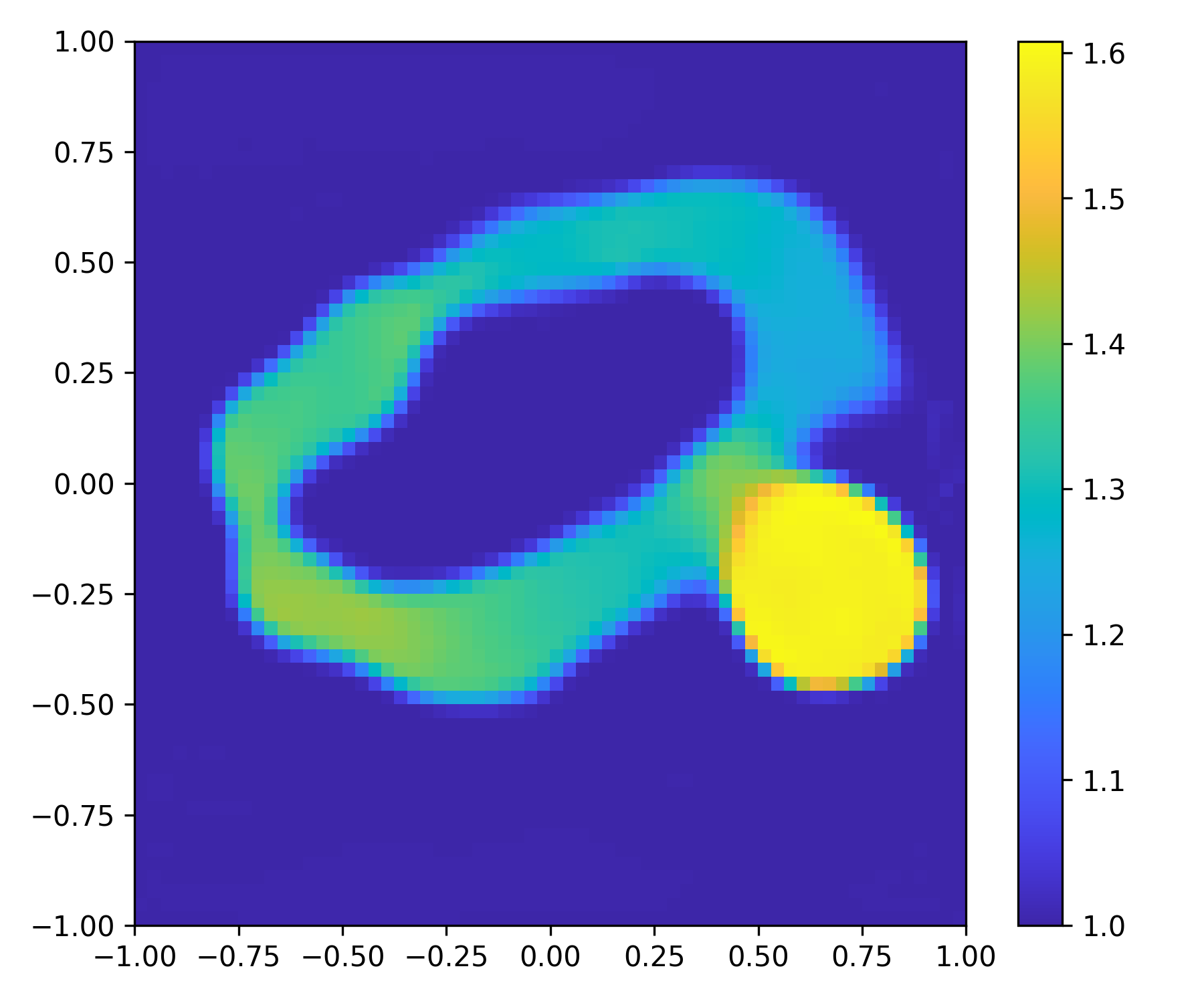}&\includegraphics[width=0.18\textwidth]{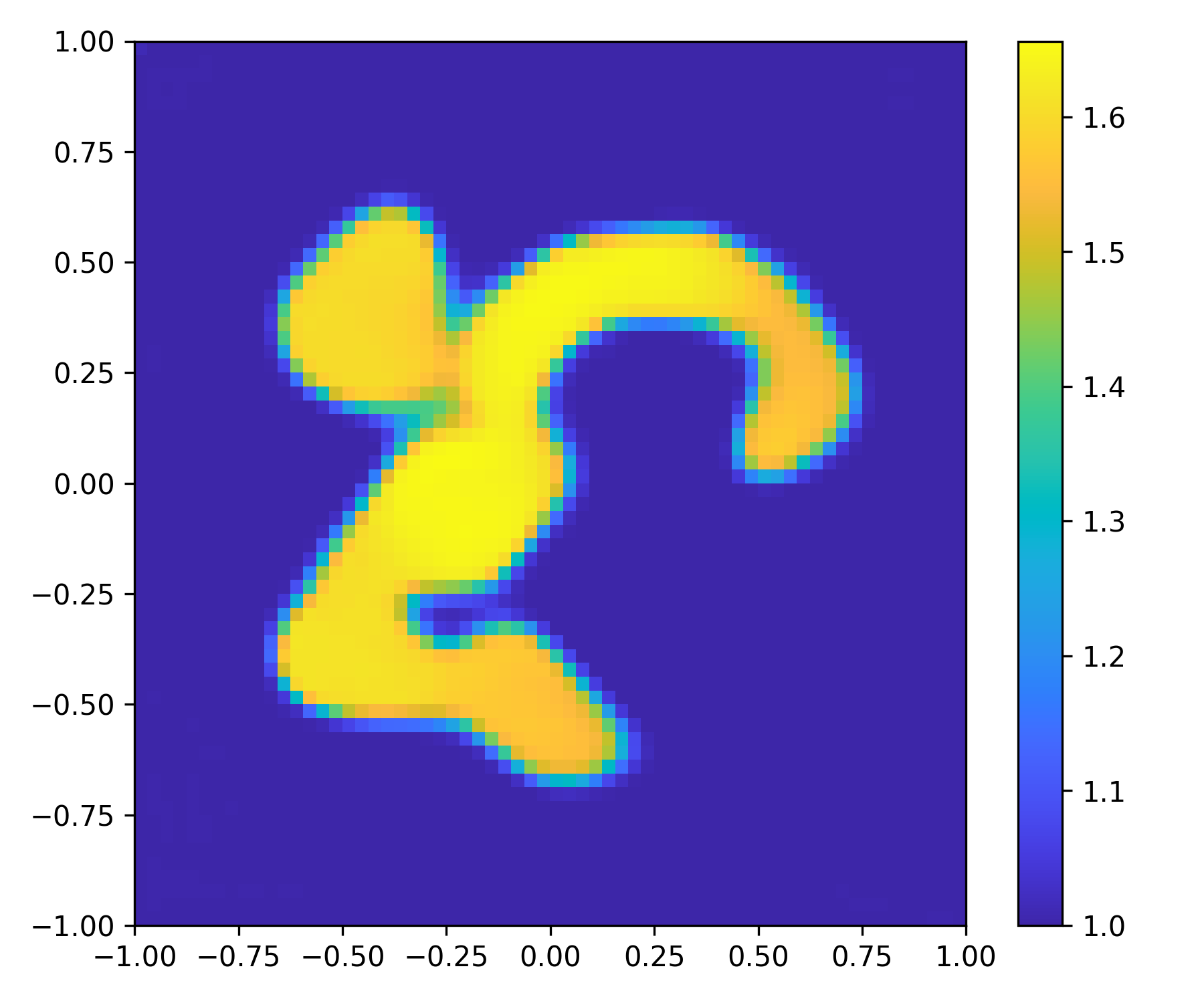} &\includegraphics[width=0.18\textwidth]{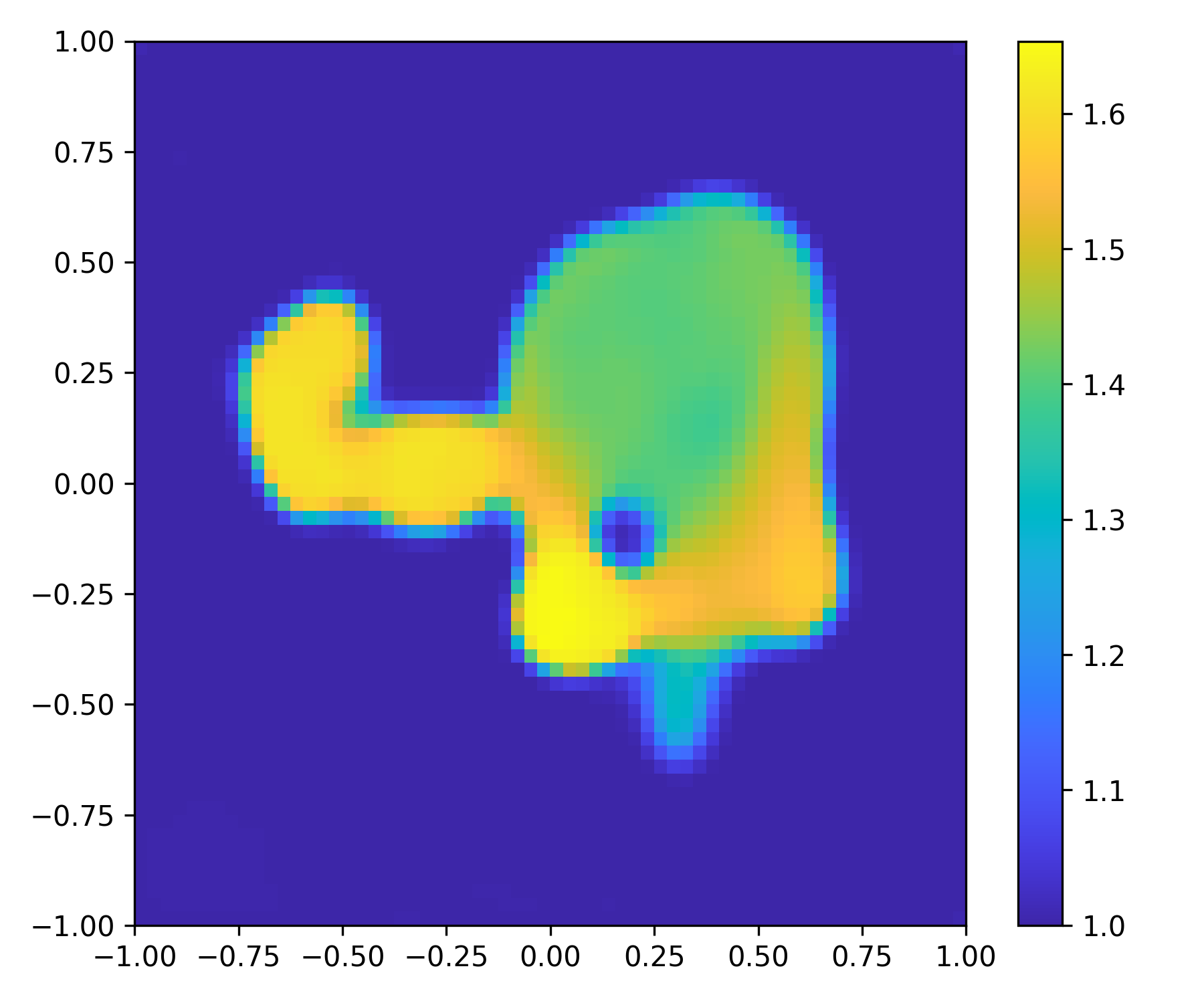}
				&\includegraphics[width=0.18\textwidth]{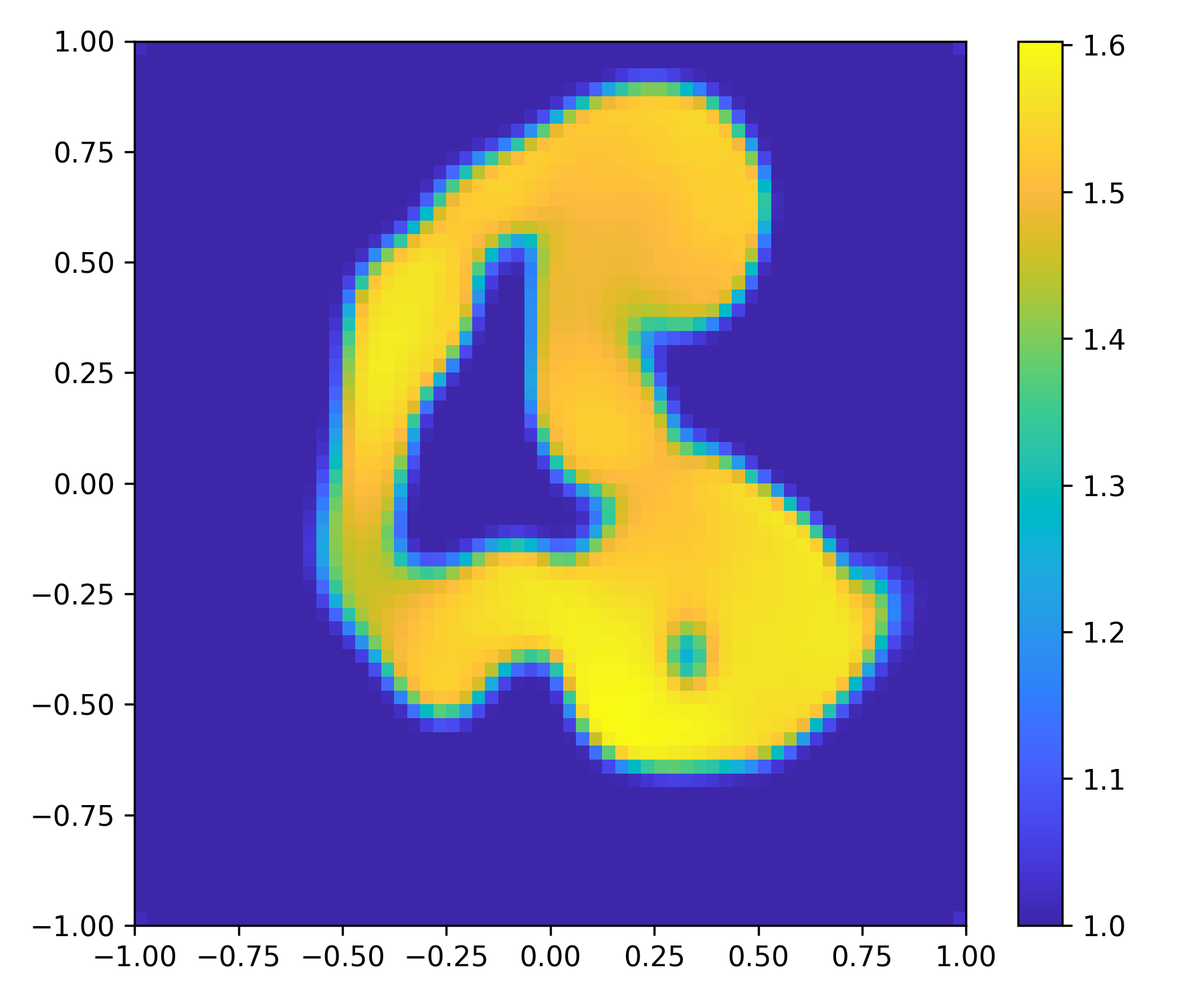}
				&\includegraphics[width=0.18\textwidth]{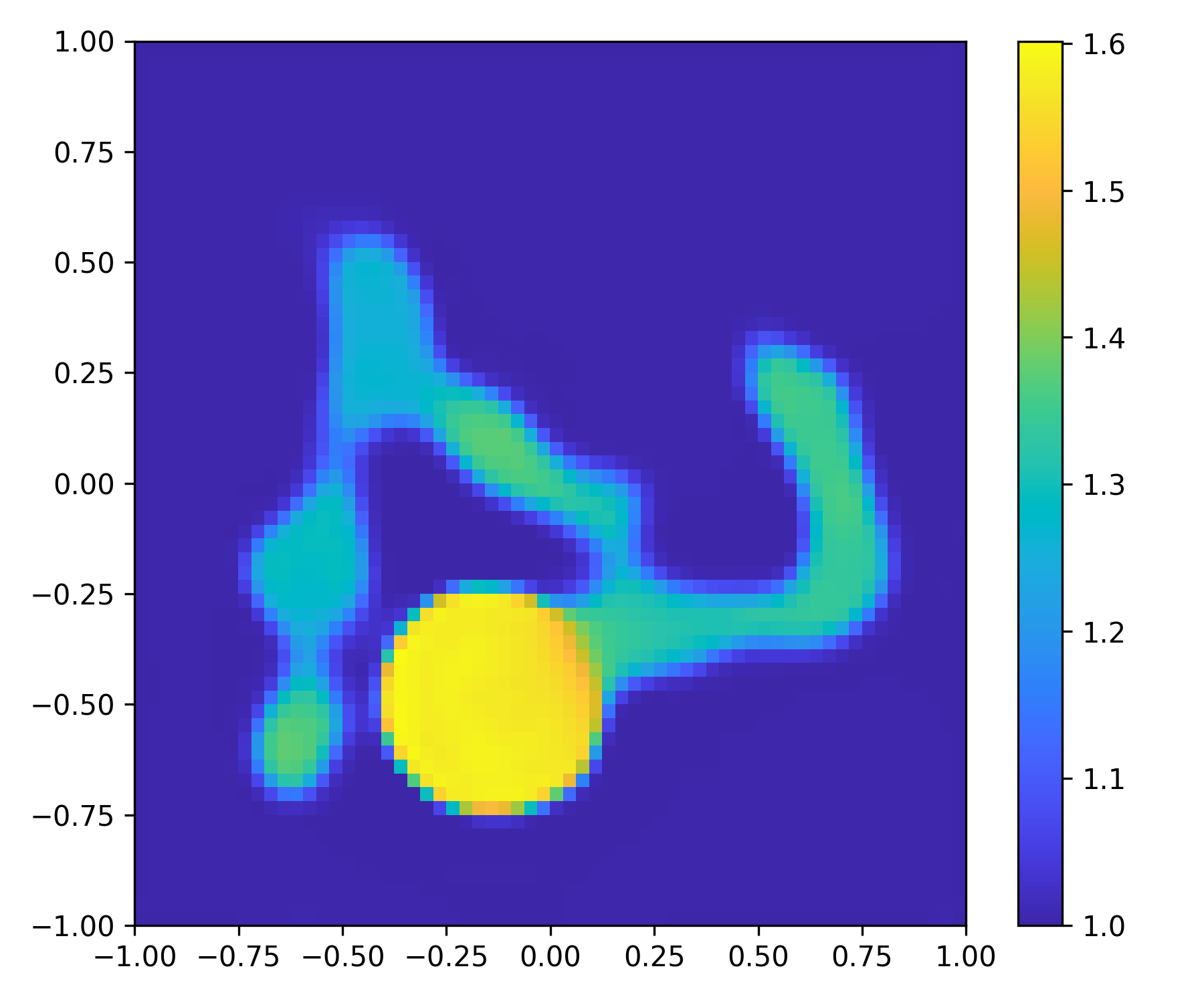}\\
				{$N_i=4$\\ $\delta=10\%$}&
				\includegraphics[width=0.18\textwidth]{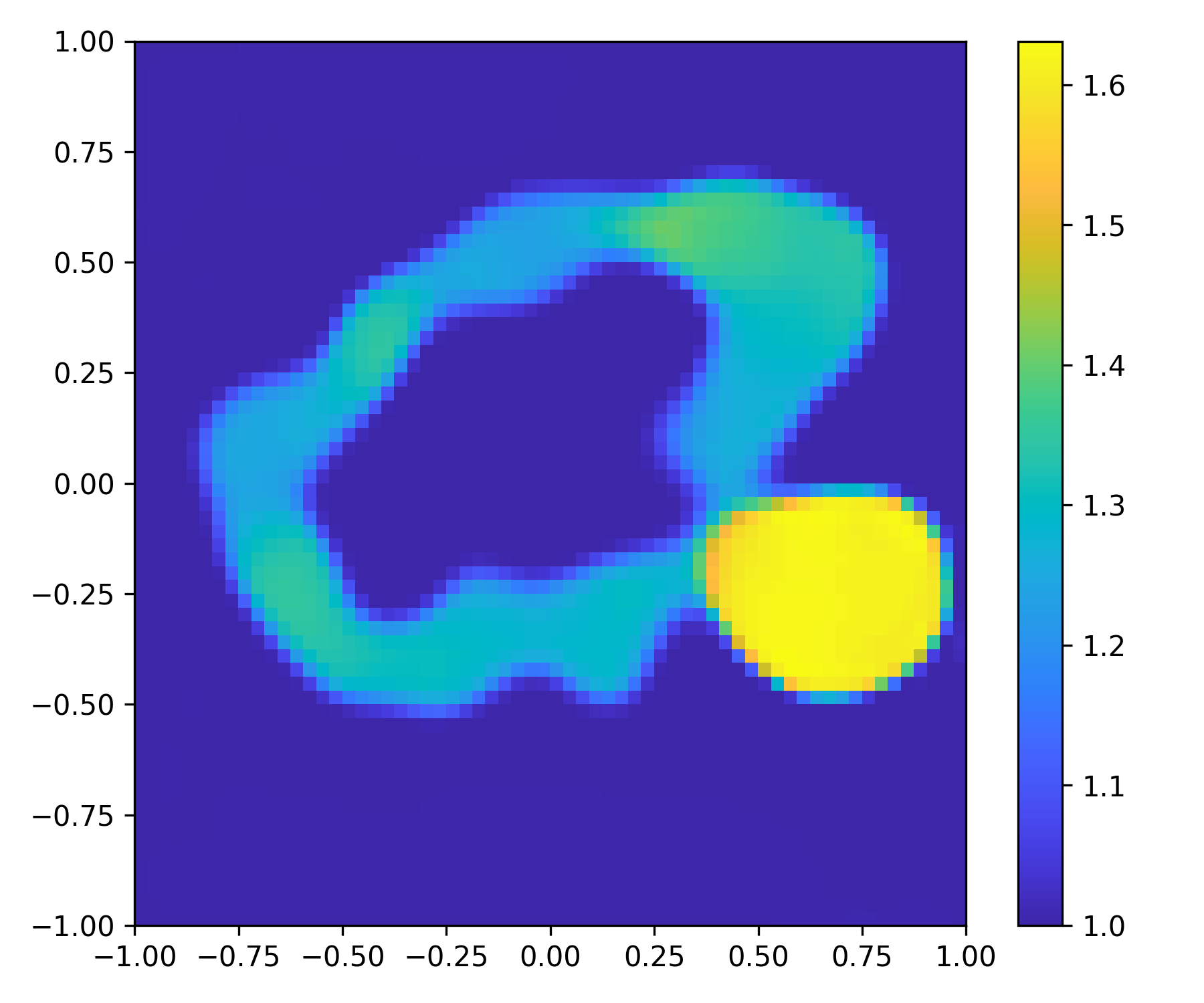}&\includegraphics[width=0.18\textwidth]{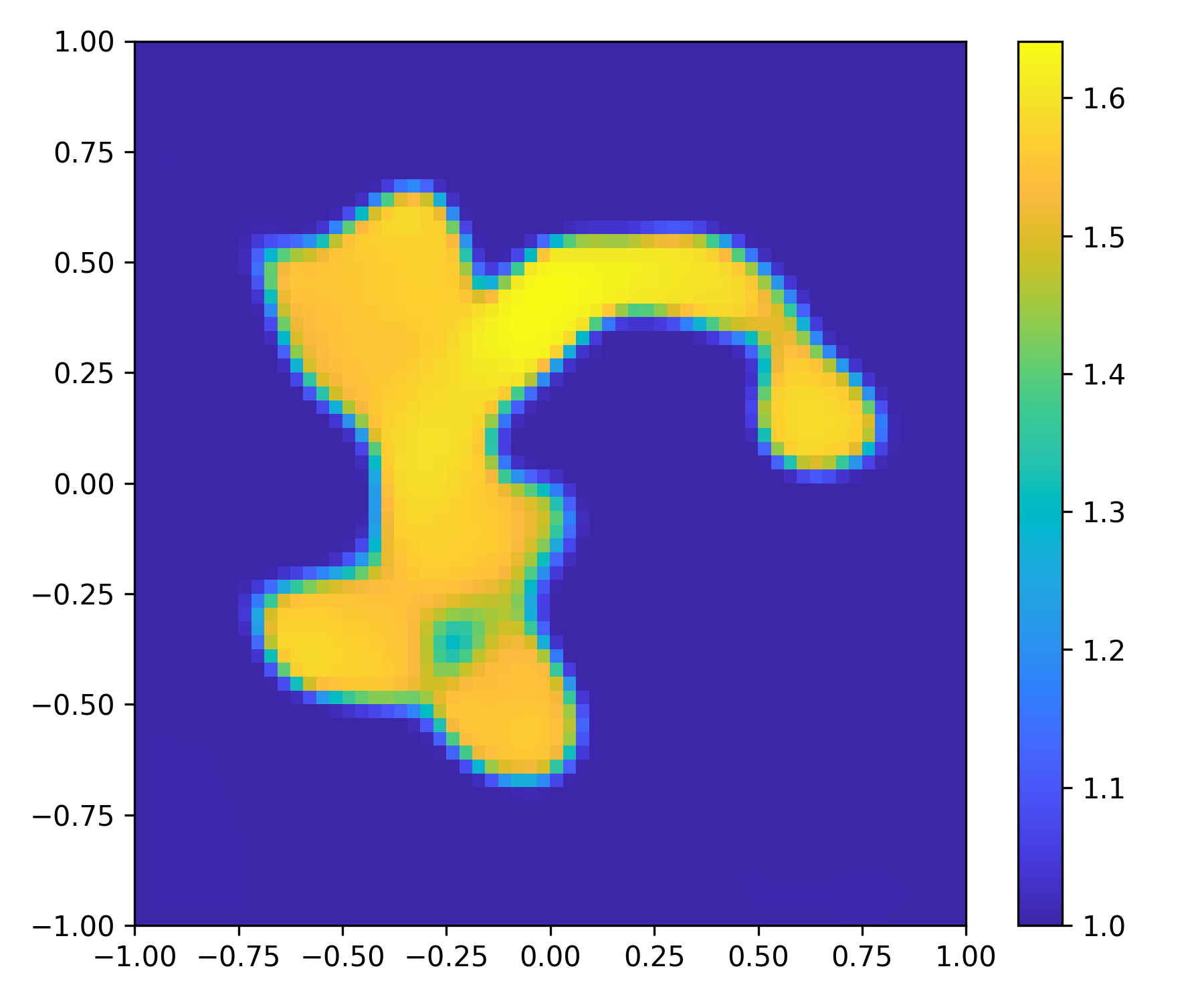} &\includegraphics[width=0.18\textwidth]{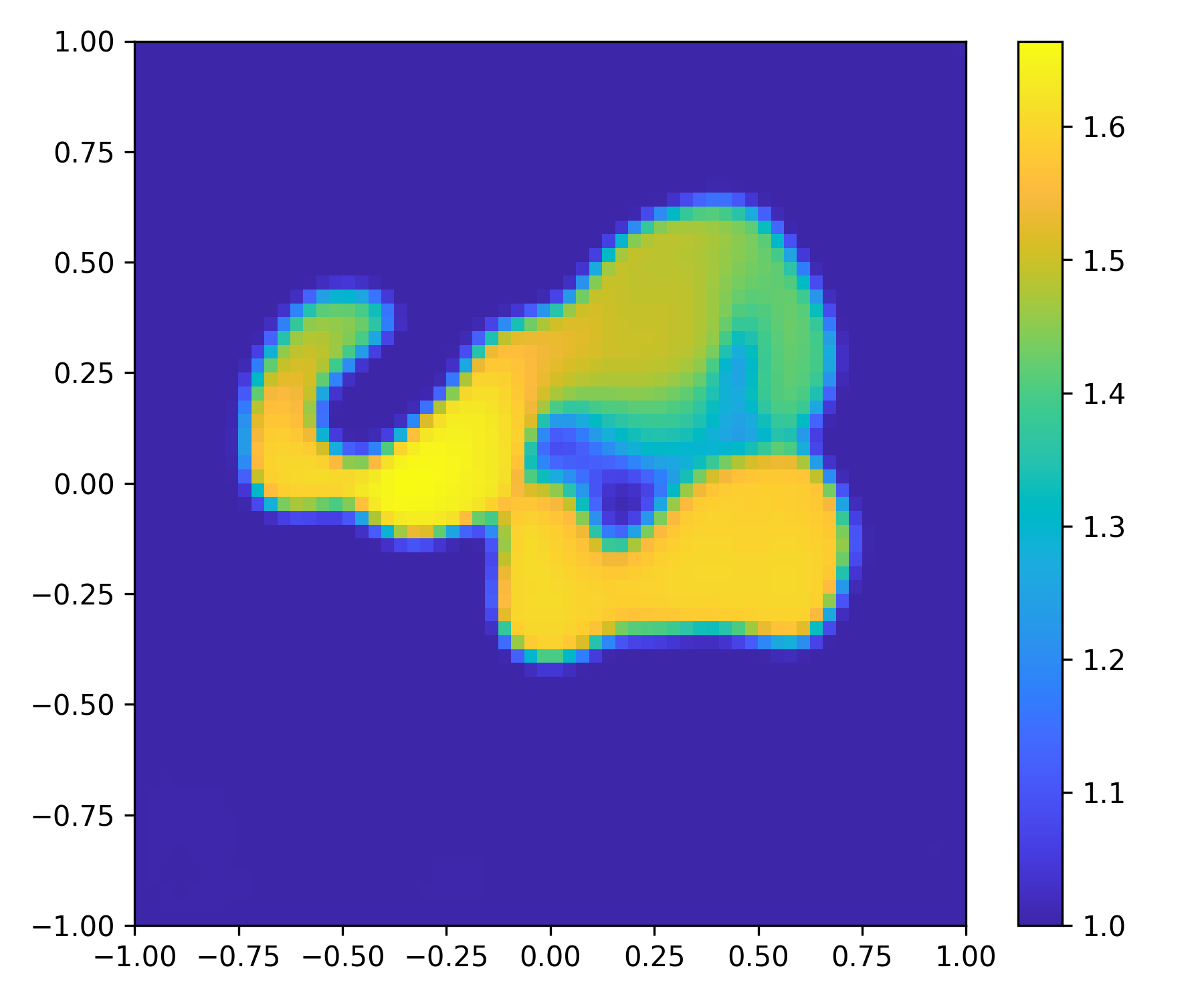}
				&\includegraphics[width=0.18\textwidth]{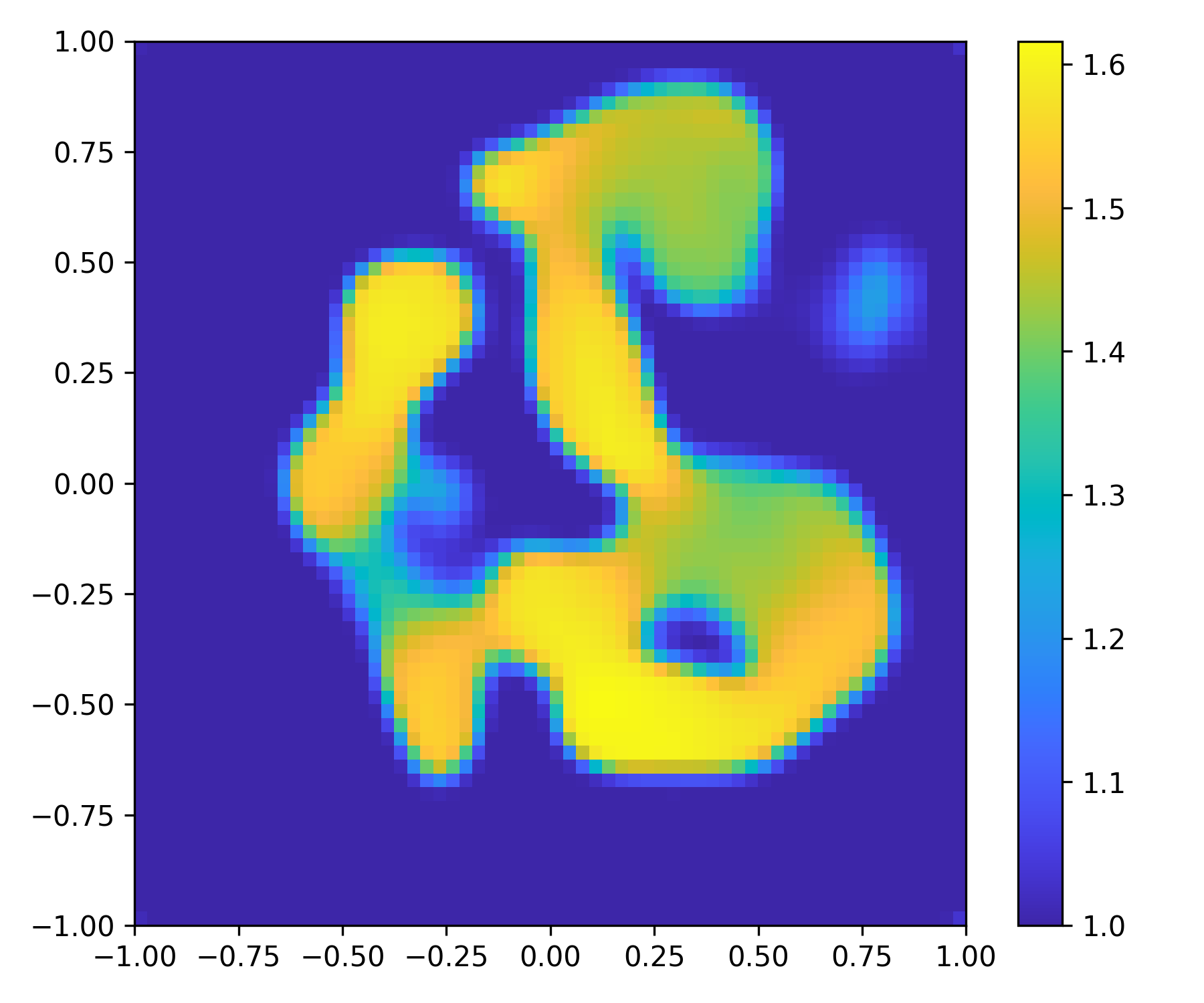}
				&\includegraphics[width=0.18\textwidth]{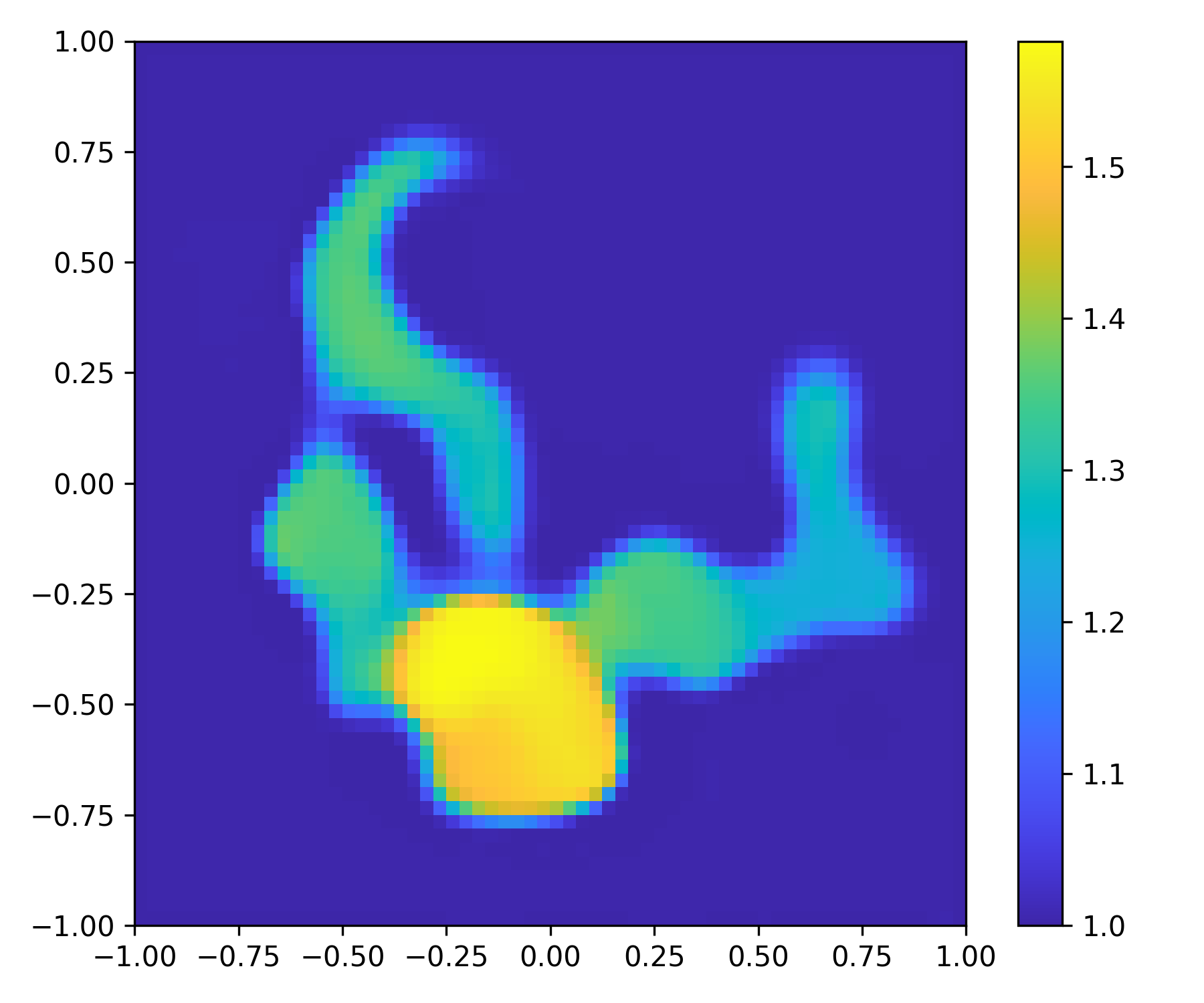}\\
				
			{$N_i=16$\\ $\delta=5\%$}&
			\includegraphics[width=0.18\textwidth]{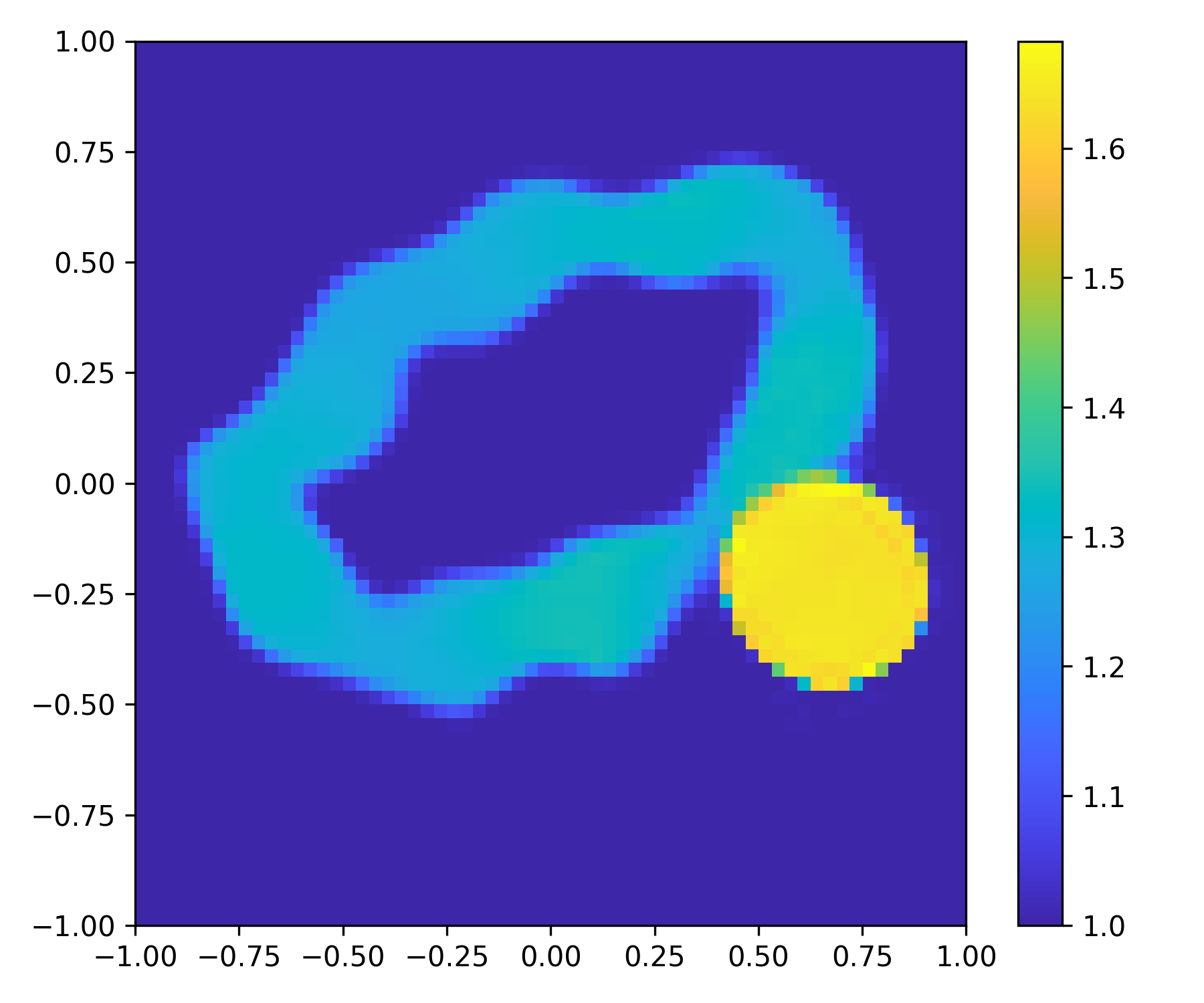}&\includegraphics[width=0.18\textwidth]{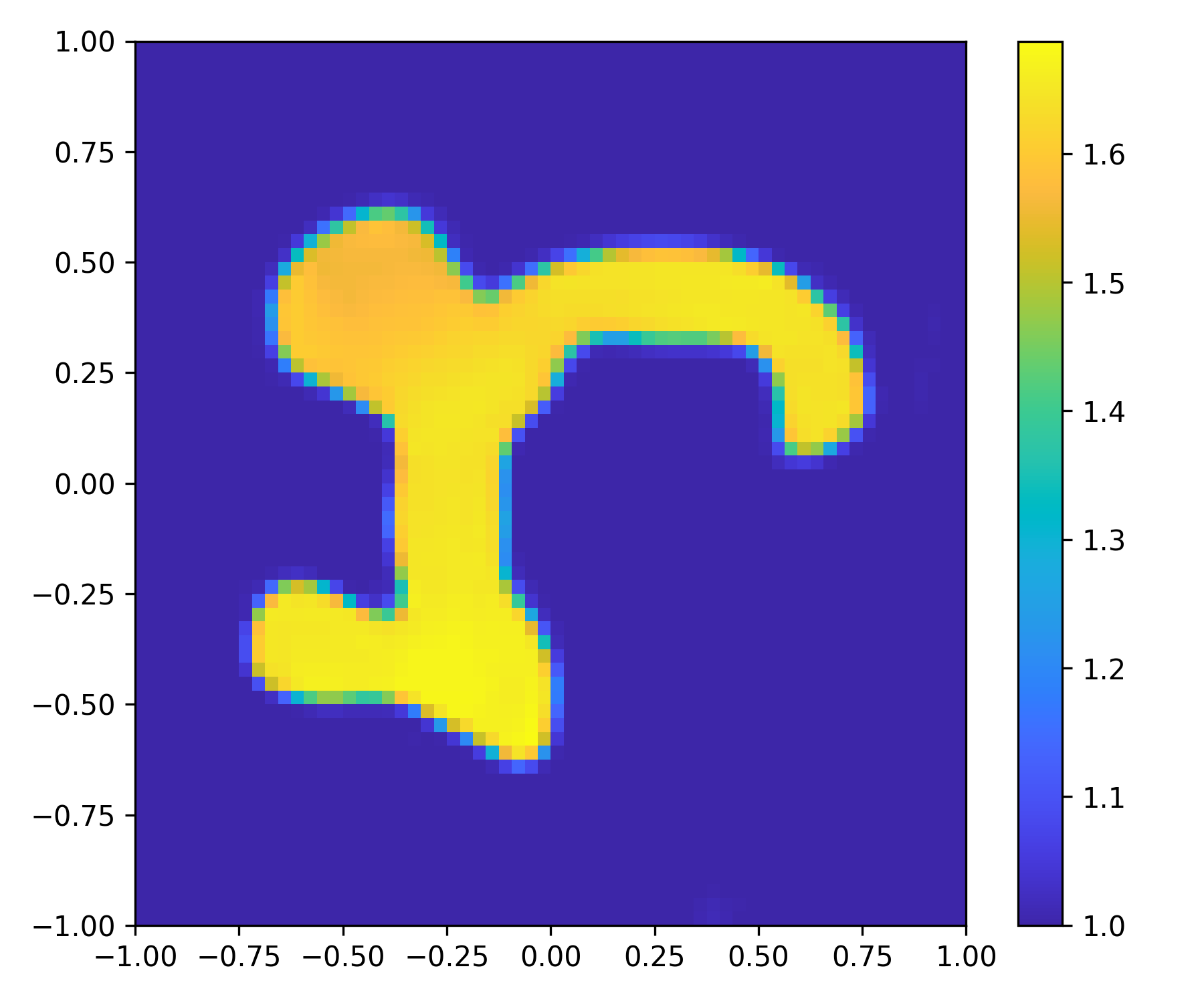} &\includegraphics[width=0.18\textwidth]{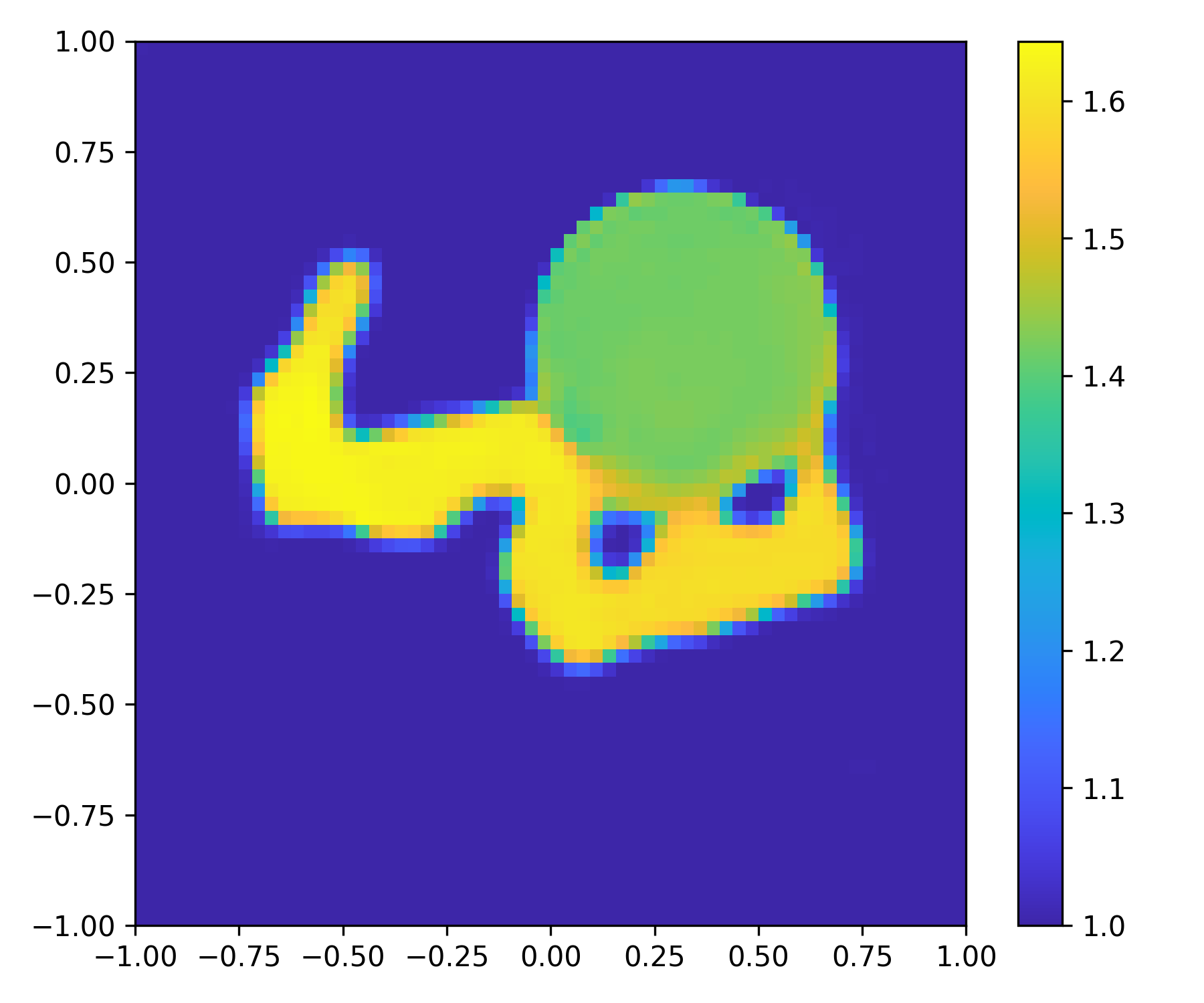}
			&\includegraphics[width=0.18\textwidth]{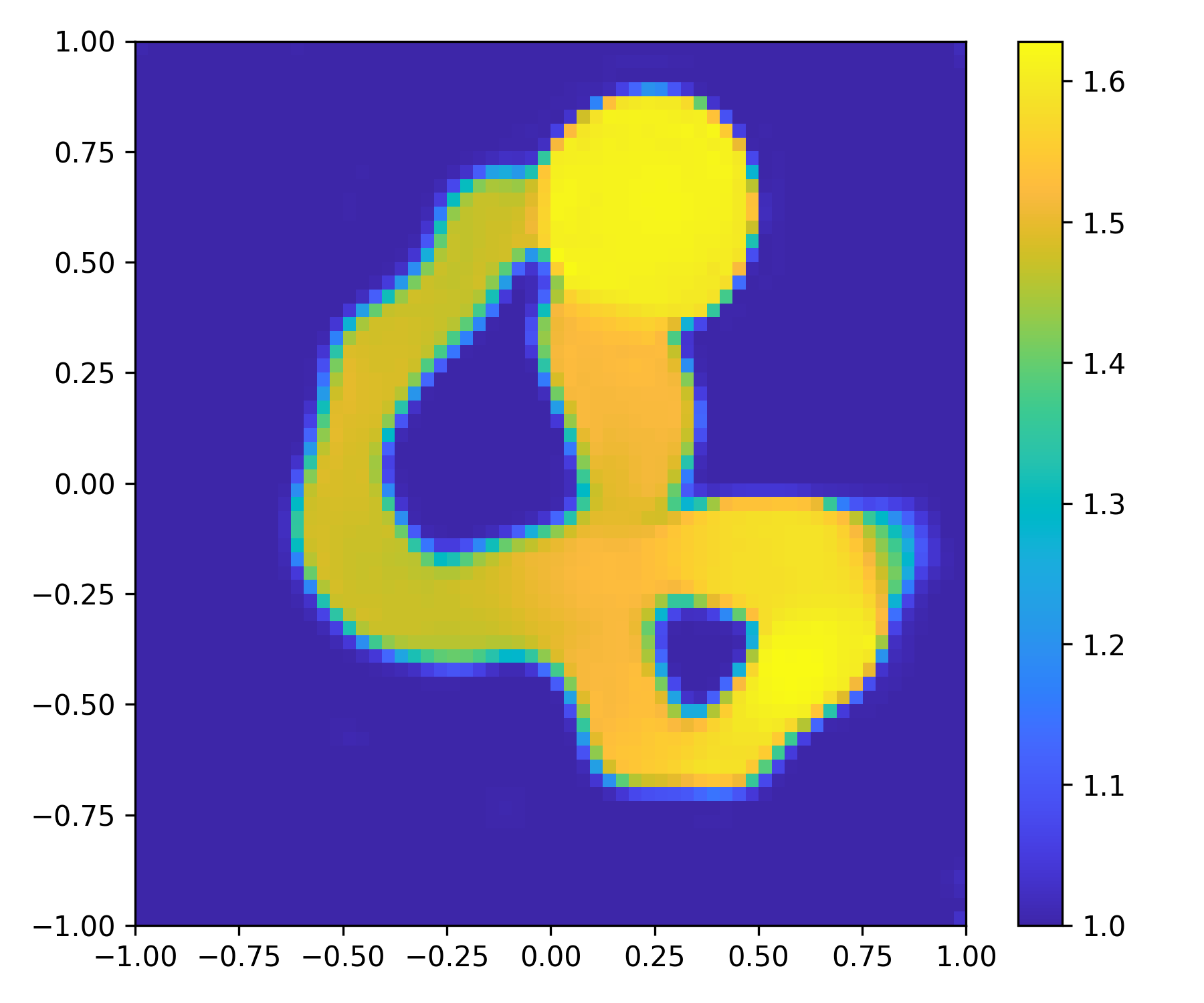}
			&\includegraphics[width=0.18\textwidth]{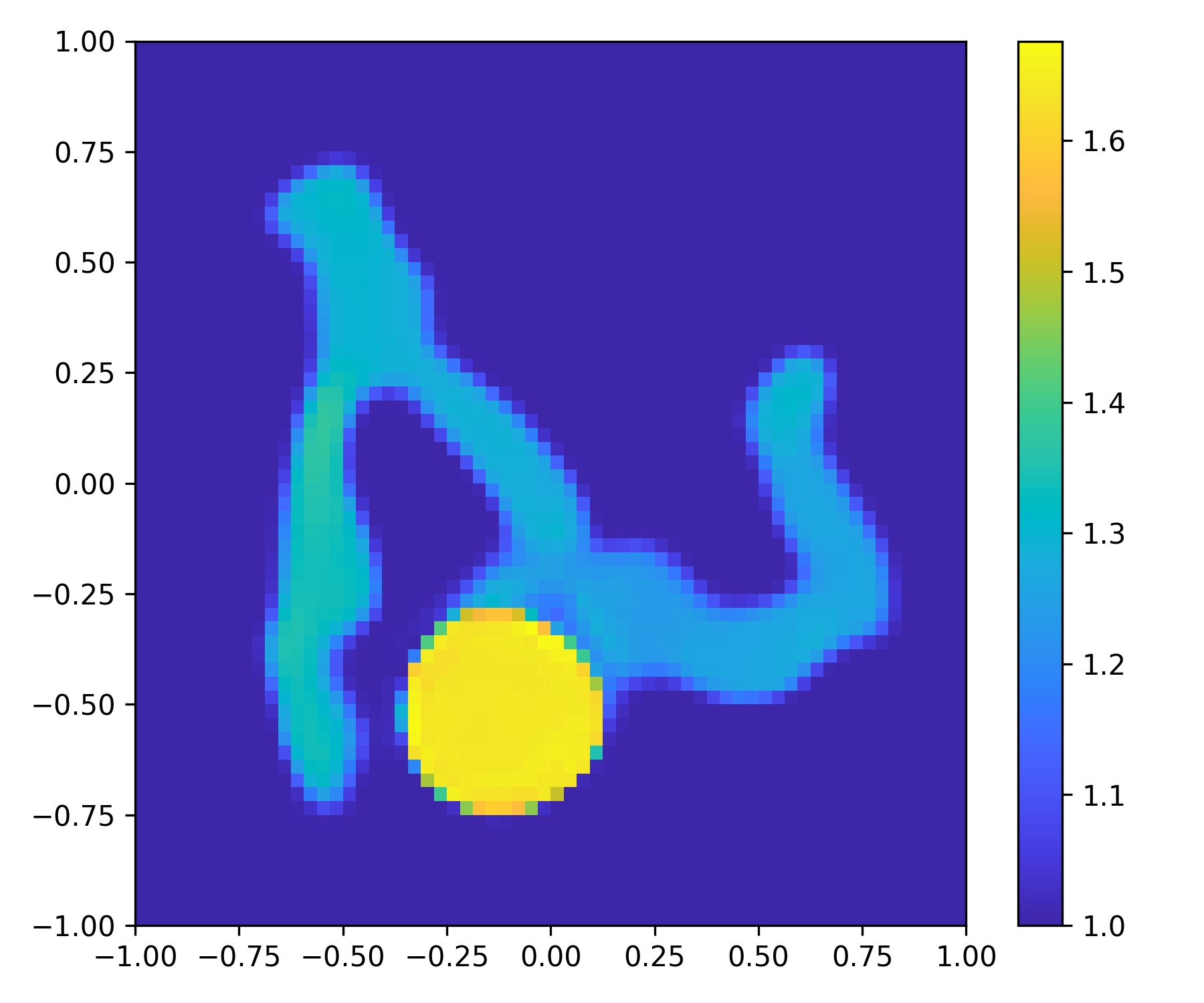}\\
			{$N_i=16$\\ $\delta=10\%$}&
			\includegraphics[width=0.18\textwidth]{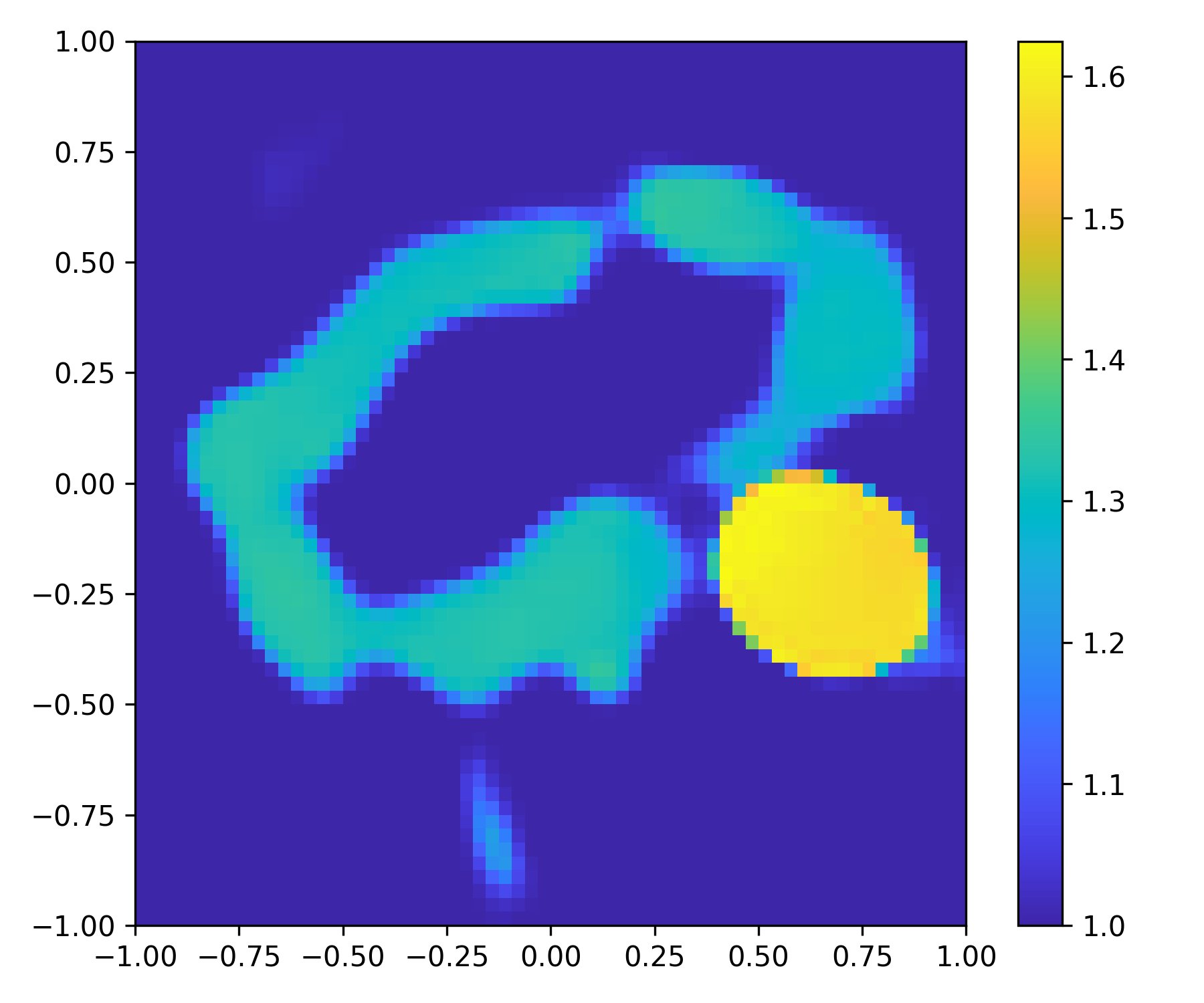}&\includegraphics[width=0.18\textwidth]{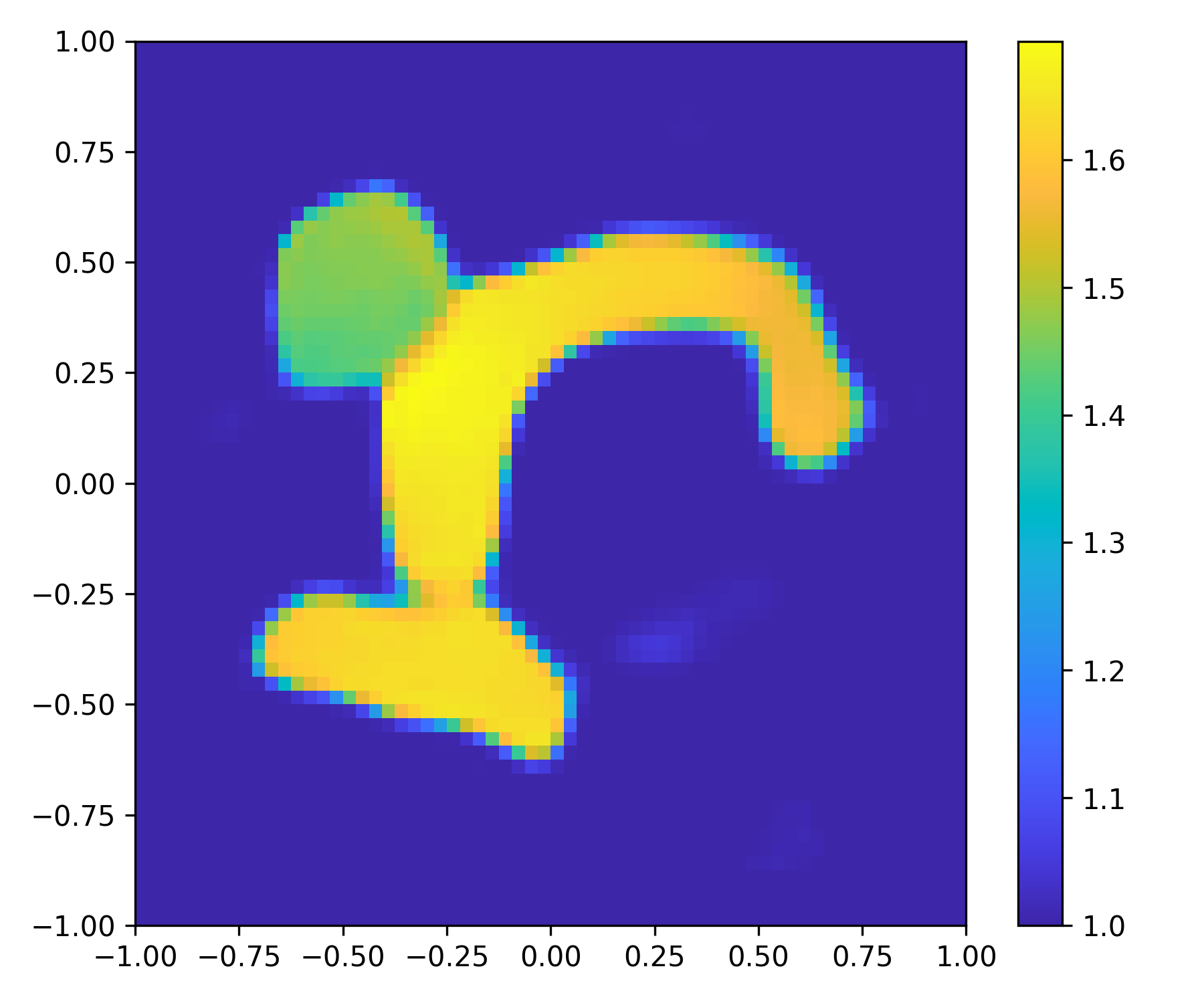} &\includegraphics[width=0.18\textwidth]{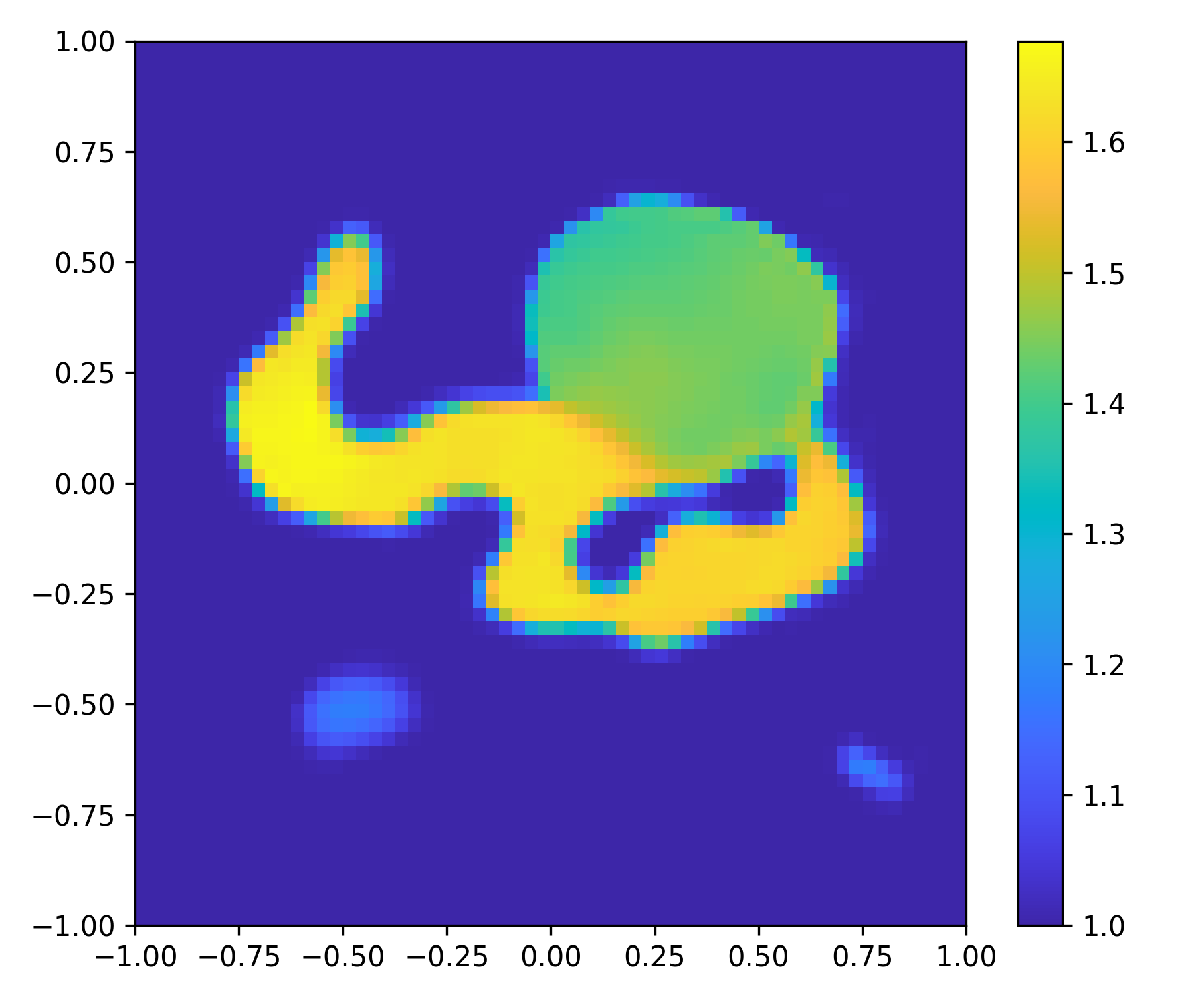}
			&\includegraphics[width=0.18\textwidth]{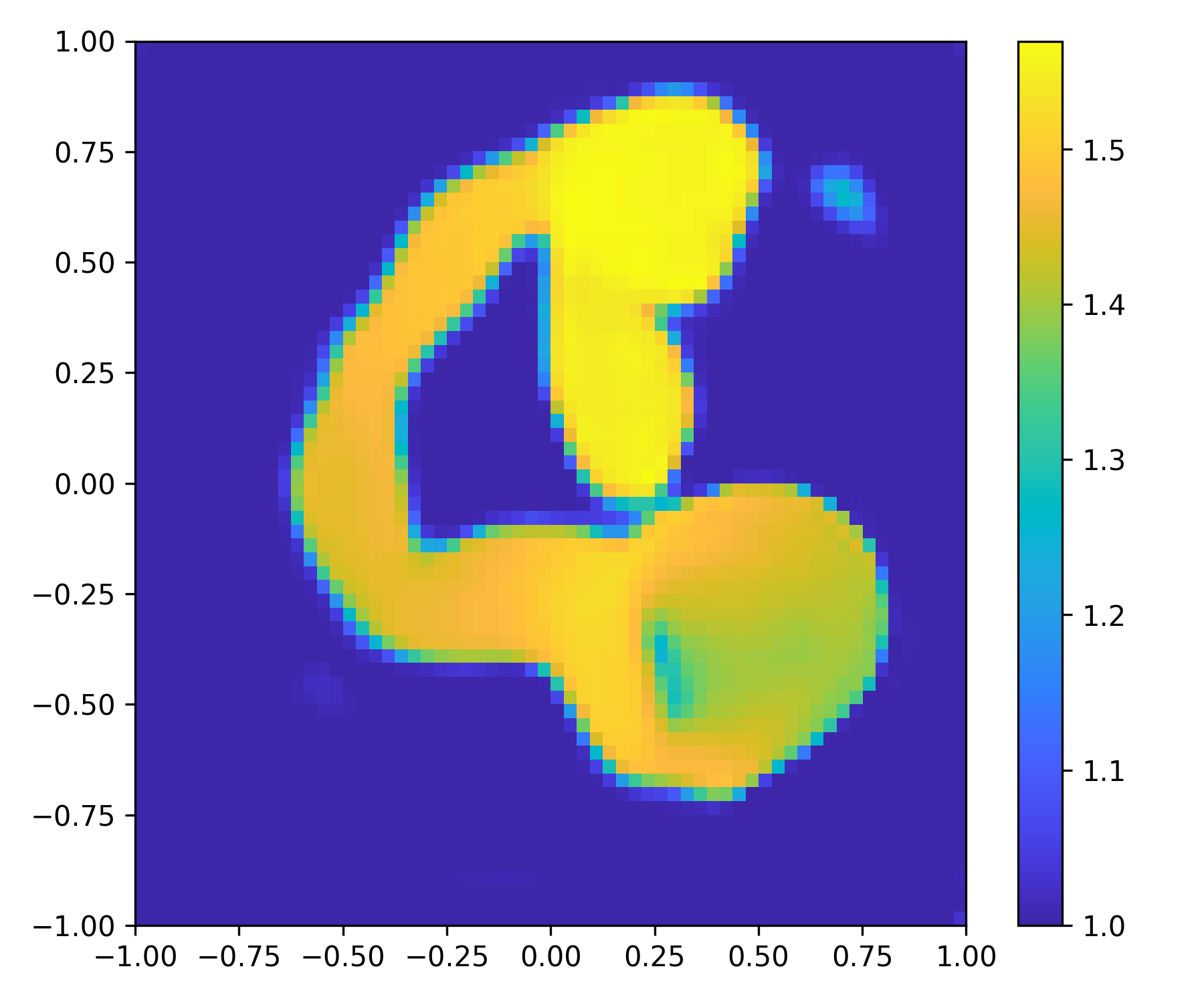}
			&\includegraphics[width=0.18\textwidth]{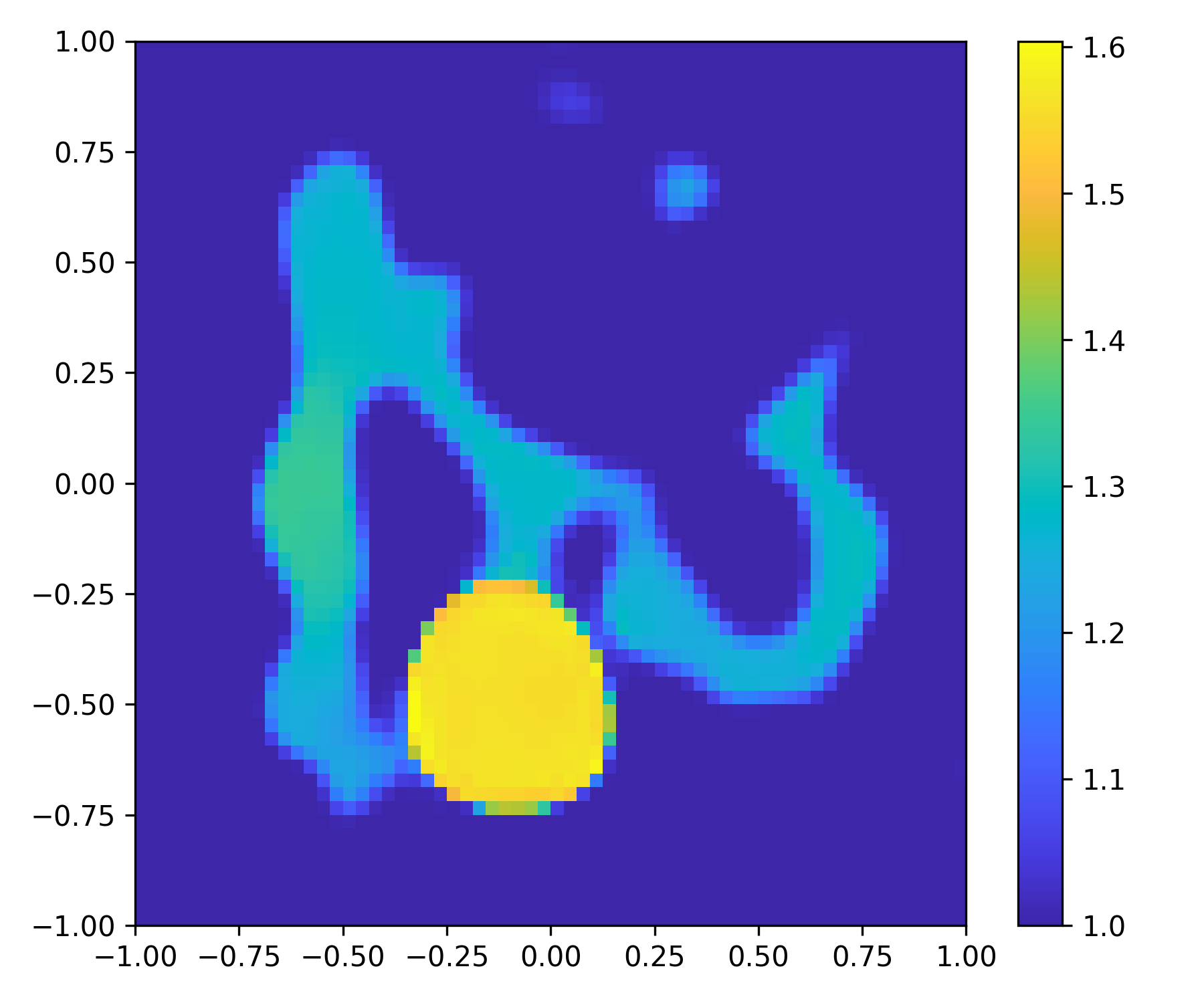}\\
			\end{tblr}
		
			\caption{Example  \ref{examp:mnist}: Reconstructed images by using the networks trained by the MNIST dataset. Row 1: true images; other rows: reconstructions with different incidences and noise levels.}
			\label{Mnist}
		\end{center}
\end{figure}

\subsubsubsection{Tests with Chinese characters.} 
\label{examp:chinese}To further test the generalization ability of the trained models, we consider the reconstruction of five Chinese characters as shown in Fig.\,\refeq{Chinese}, whose shapes are very different from those of the MNIST dataset. The coefficient value $n(x)$ of the scatterers is set to 1.5. The recovered images presented in  Fig.\,\refeq{Chinese} also show the ability of the DSM-DL to fully extract the hidden information in the measurement data and thus provide more accurate and stable reconstructions with more incidences used. The average relative error for the five examples is presented in Table.\,\ref{tab:MNIST}

\begin{figure}[htbp]\small
	\begin{center}
		\begin{tblr}
			{colspec = {X[-1,m]X[c,h]X[c,h]X[c,h]X[c,h]X[c,h]},
				stretch = 0,
				rowsep = 0pt,}
			{Ground\\ Truth}&
			\includegraphics[width=0.18\textwidth]{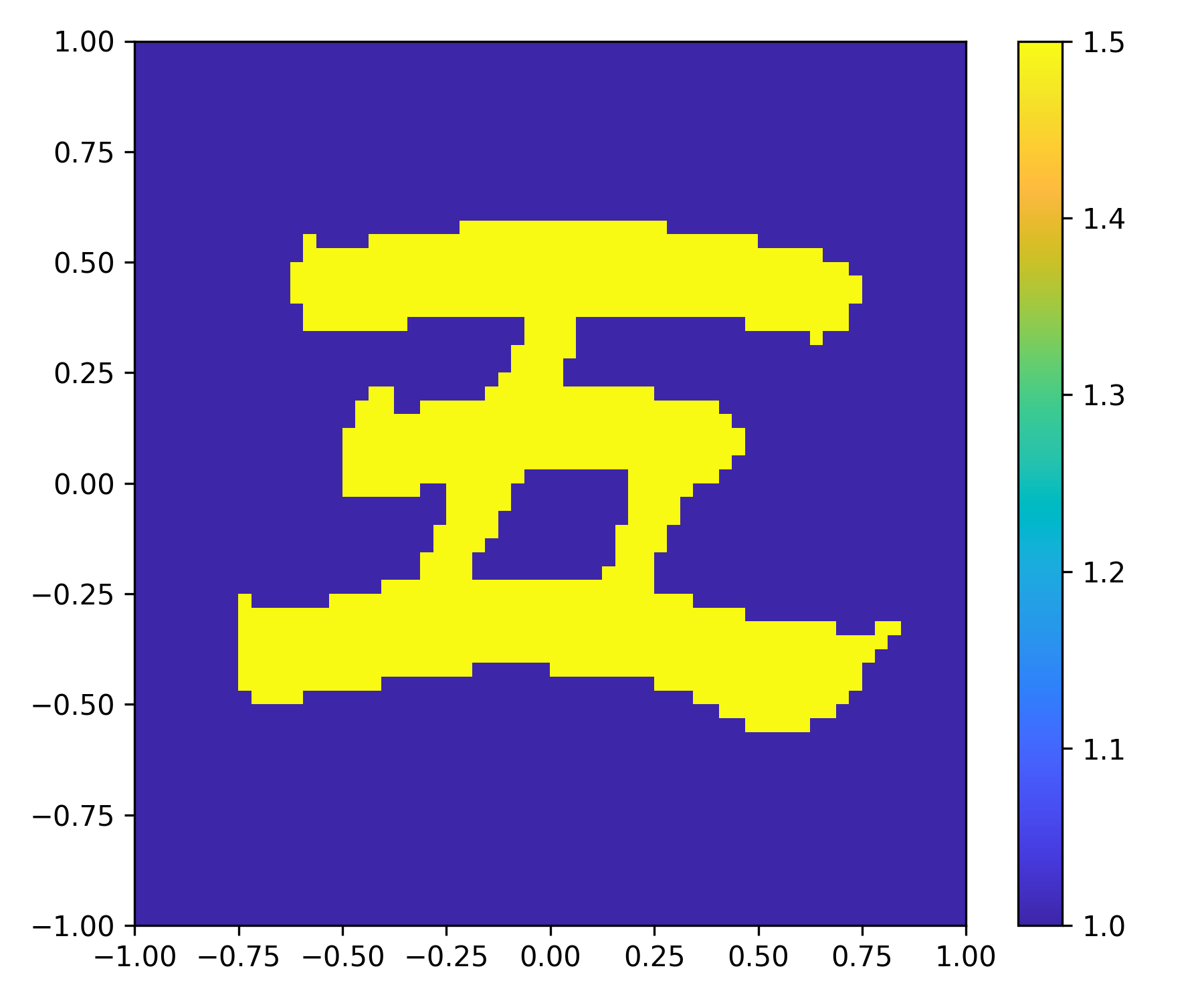}&\includegraphics[width=0.18\textwidth]{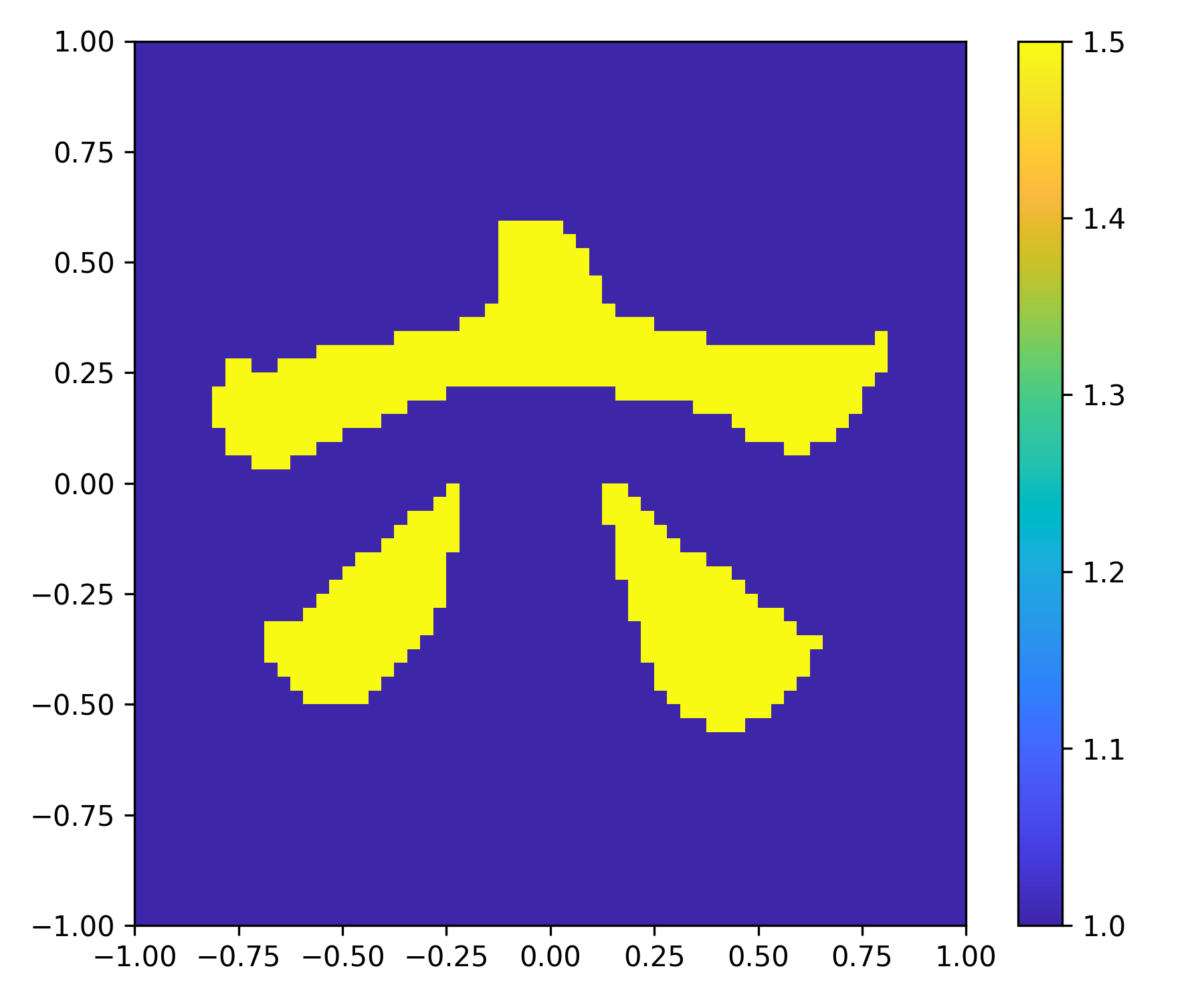} &\includegraphics[width=0.18\textwidth]{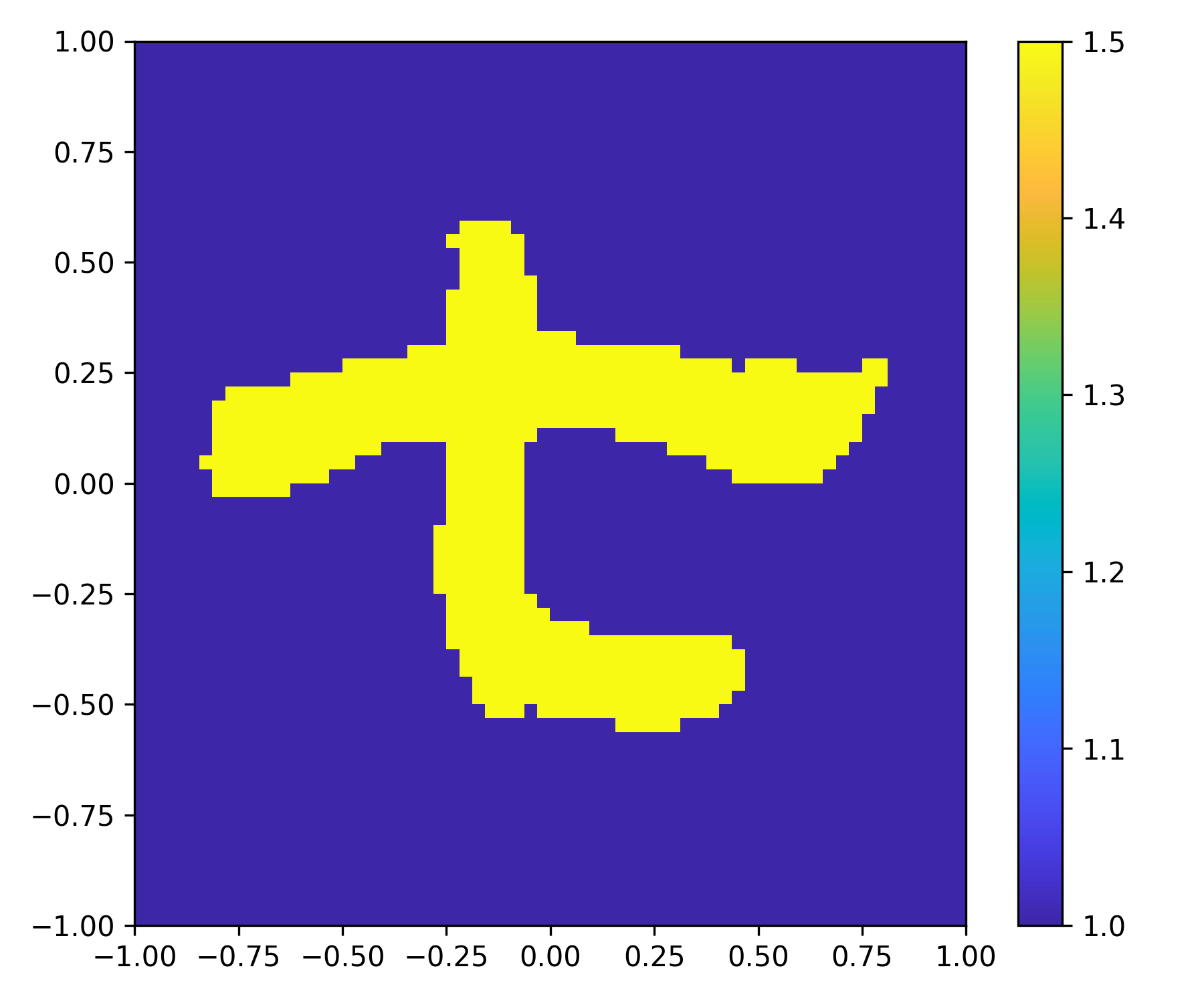}
			&\includegraphics[width=0.18\textwidth]{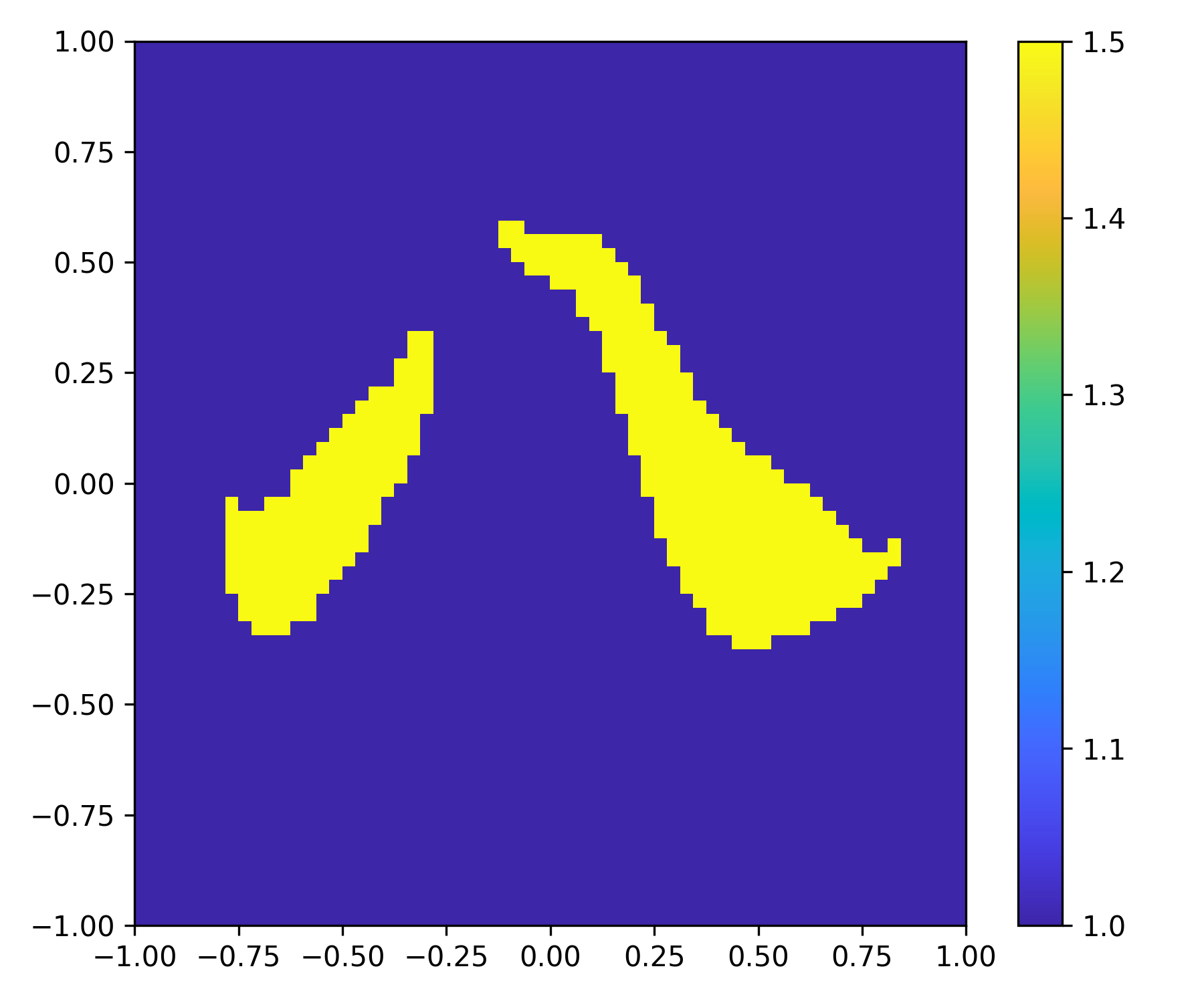}
			&\includegraphics[width=0.18\textwidth]{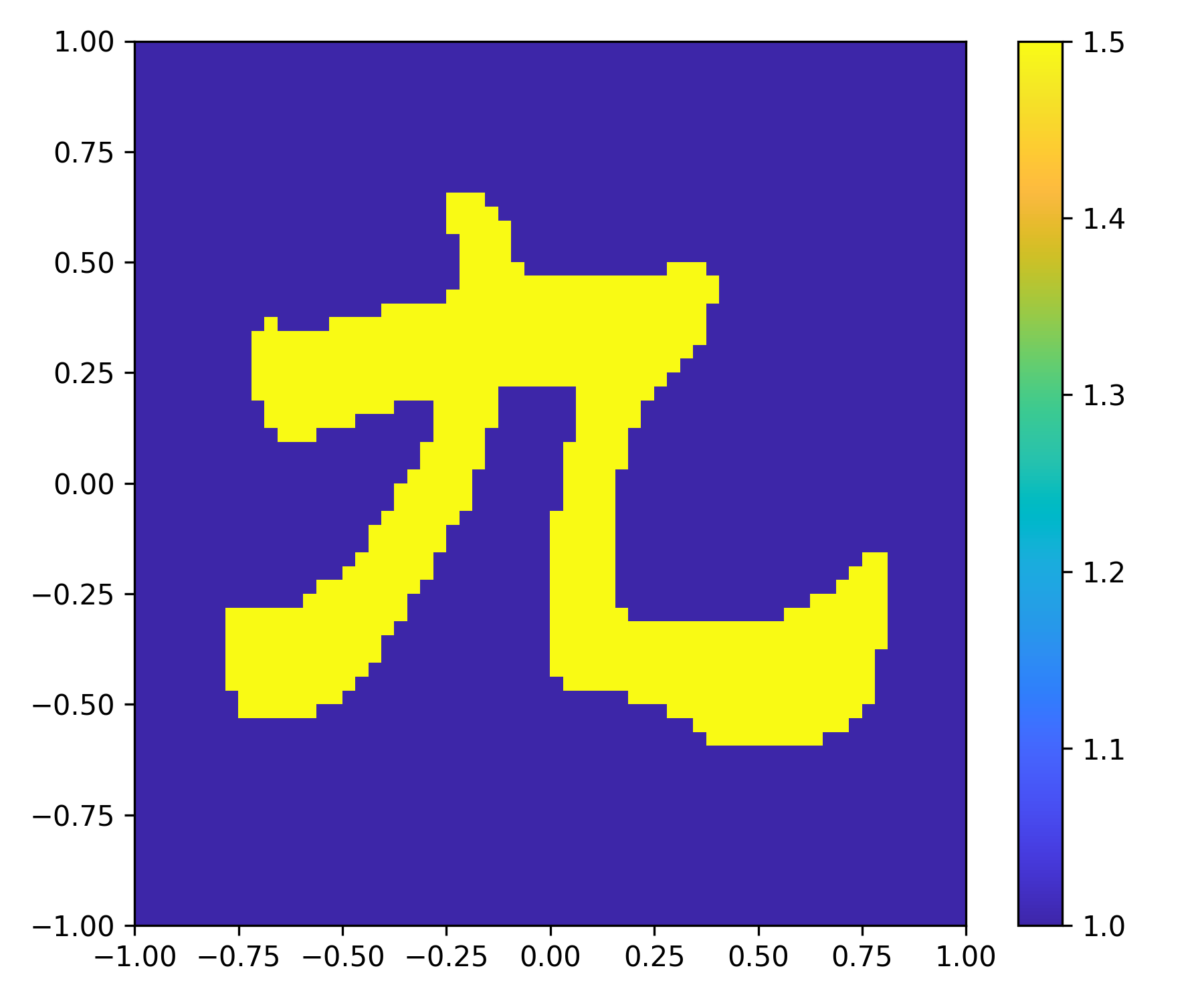}\\
			{$N_i=4$\\ $\delta=5\%$}&
			\includegraphics[width=0.18\textwidth]{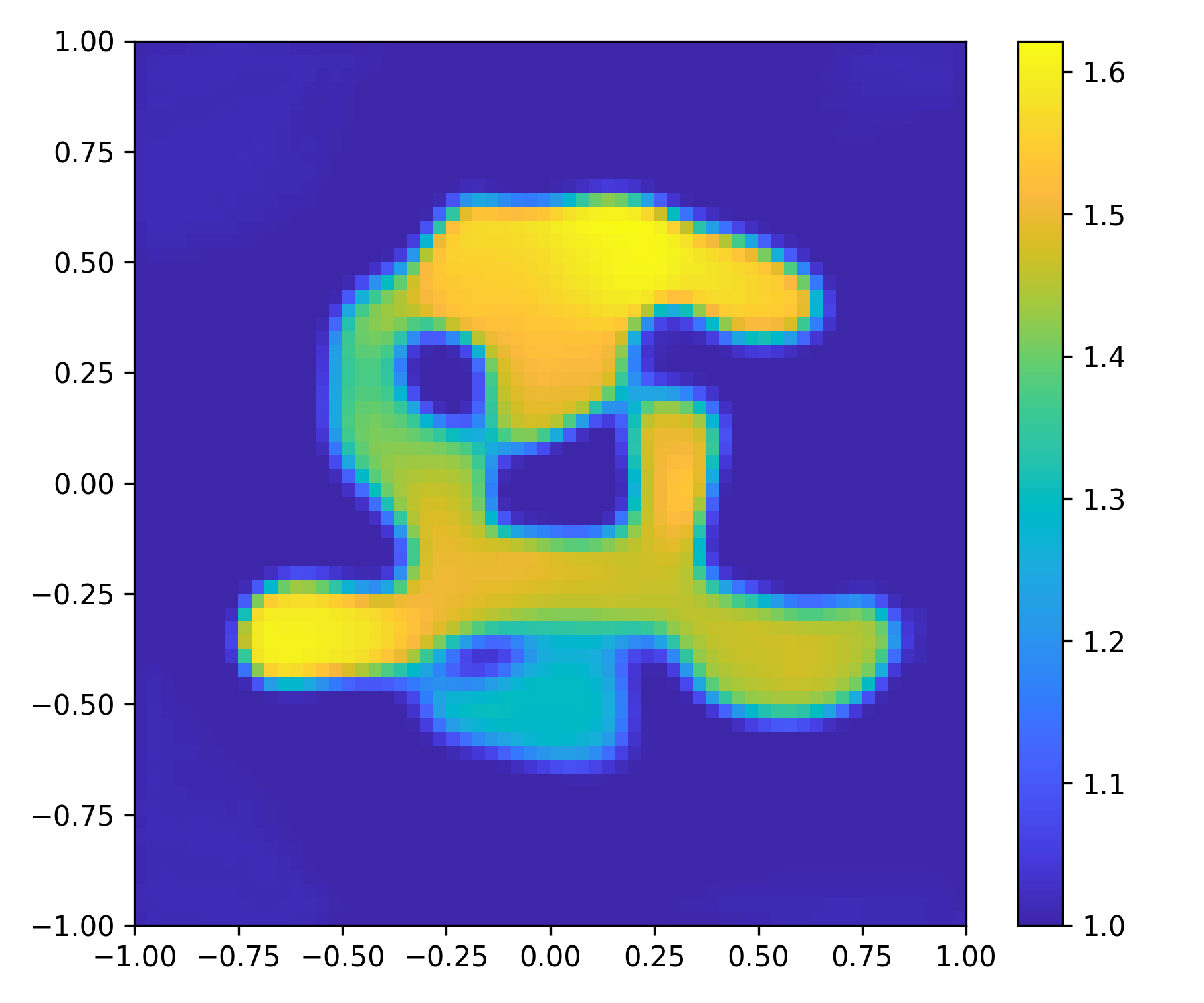}&\includegraphics[width=0.18\textwidth]{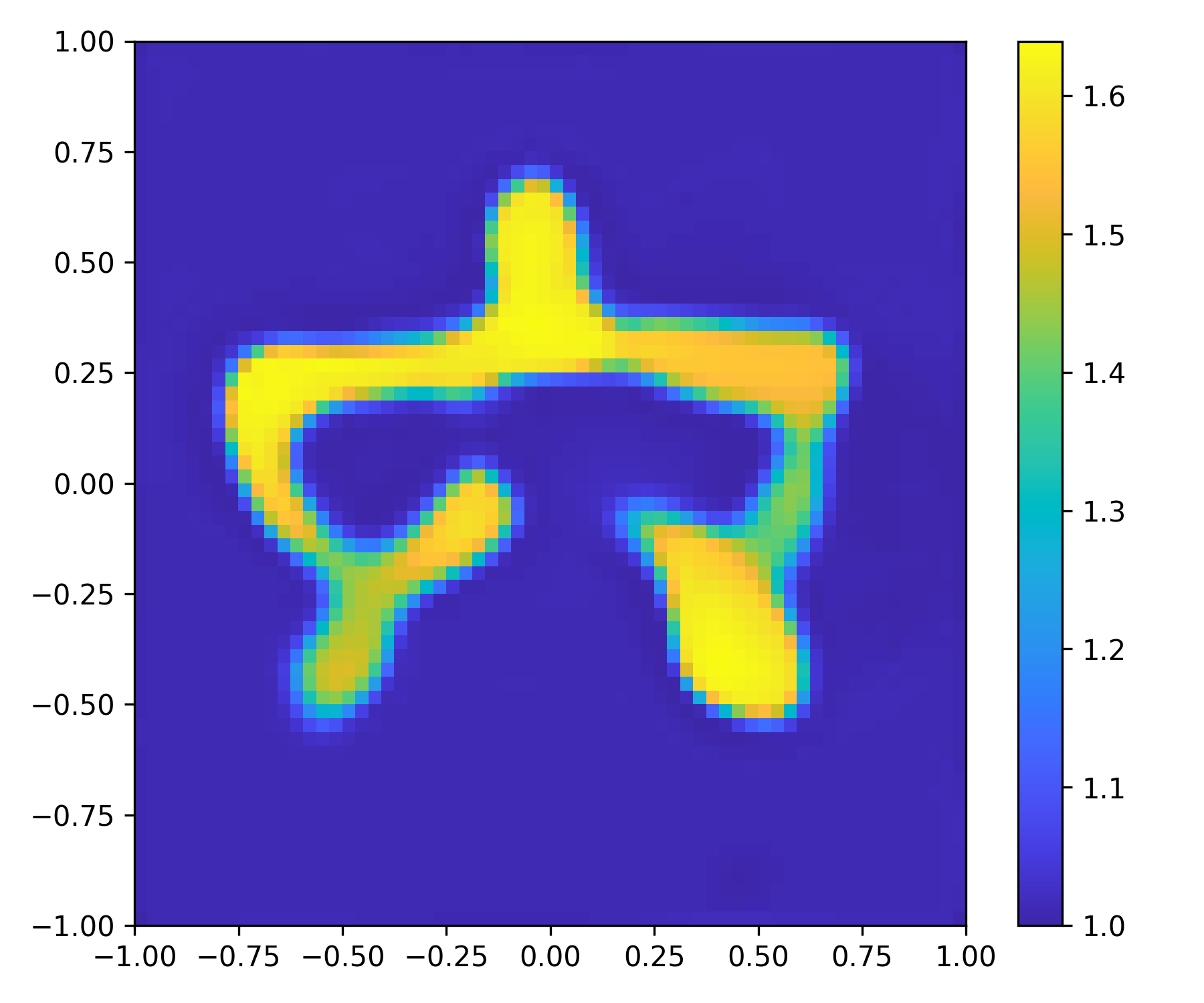} &\includegraphics[width=0.18\textwidth]{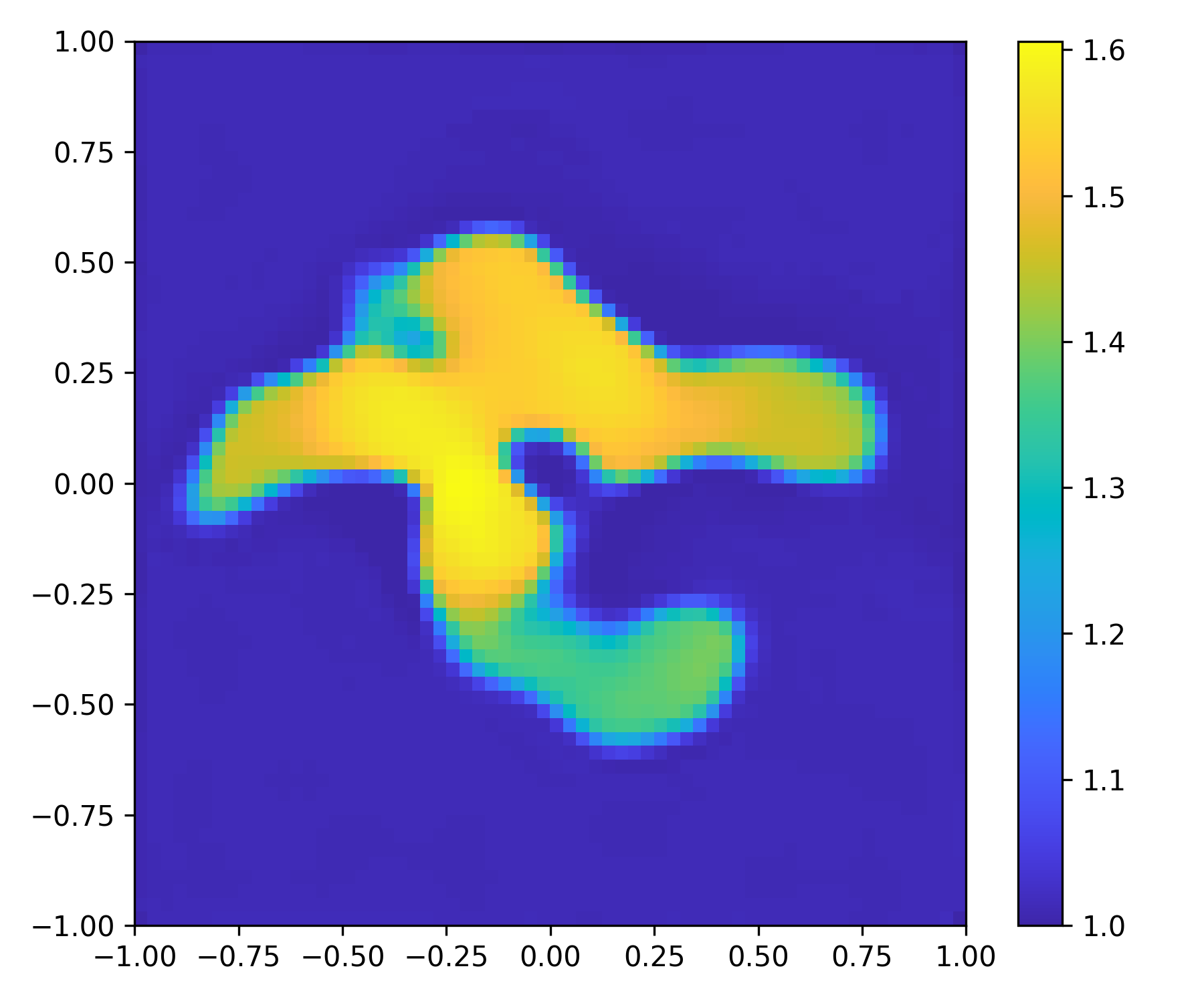}
			&\includegraphics[width=0.18\textwidth]{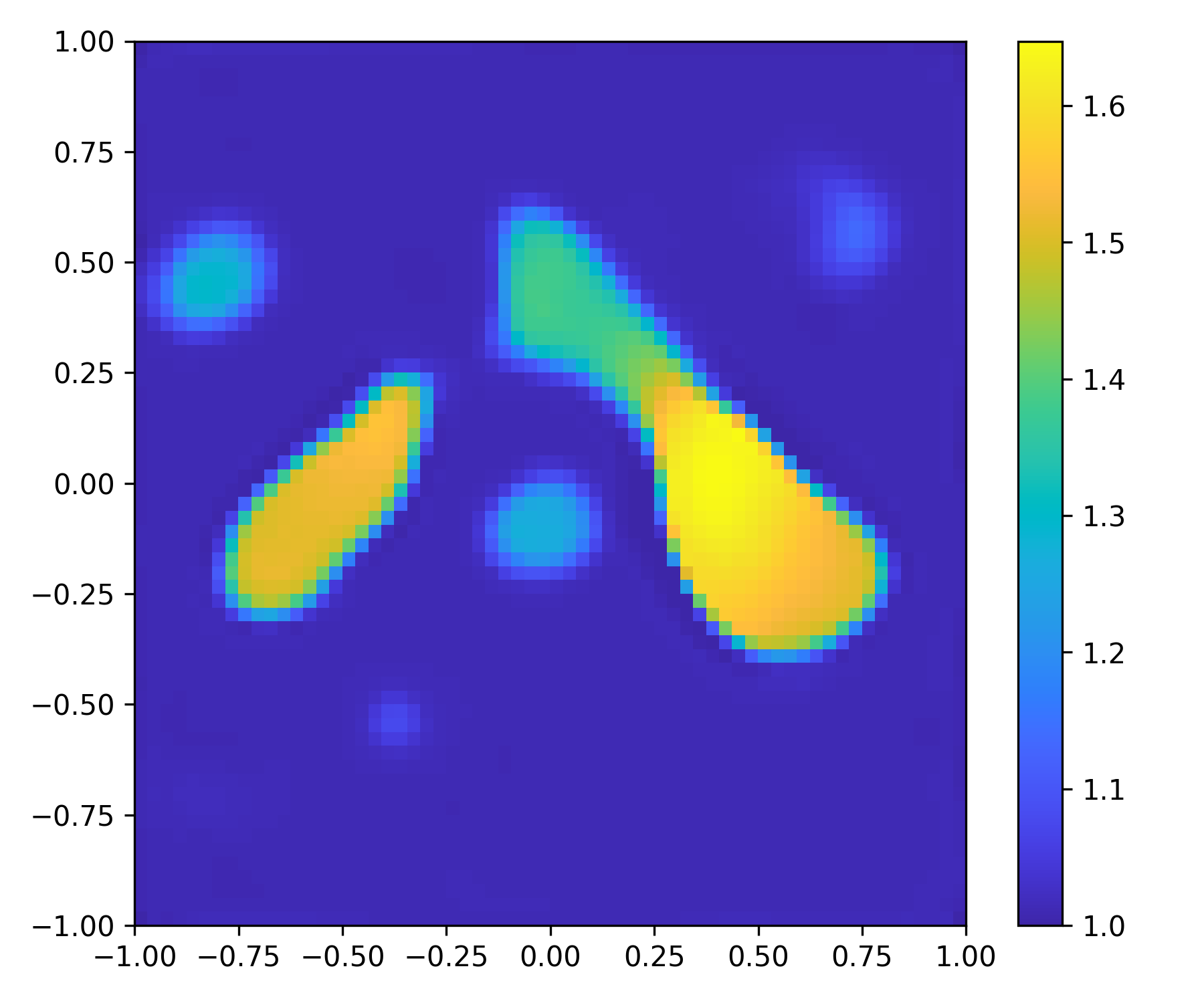}
			&\includegraphics[width=0.18\textwidth]{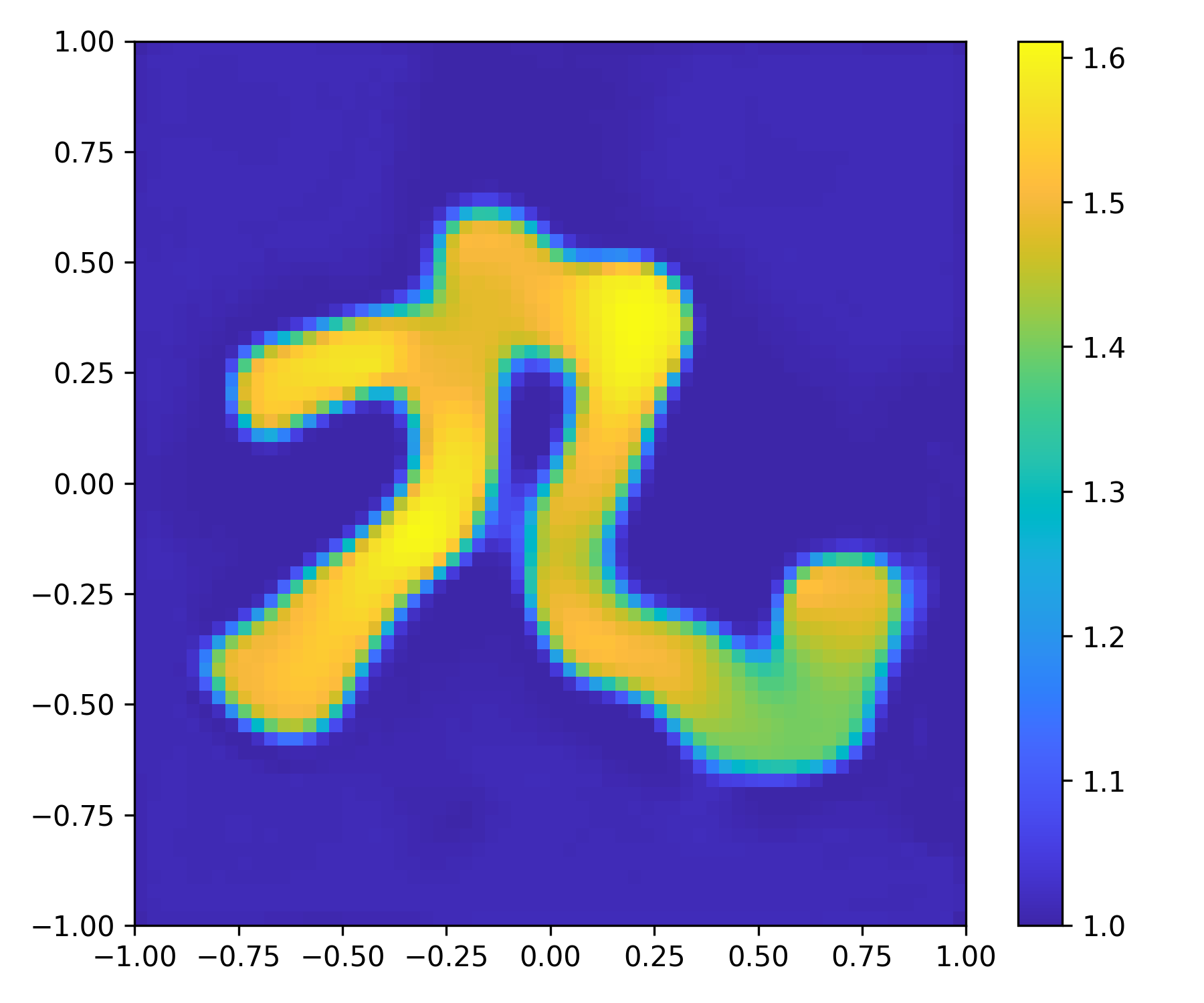}\\
			{$N_i=4$\\ $\delta=10\%$}&
			\includegraphics[width=0.18\textwidth]{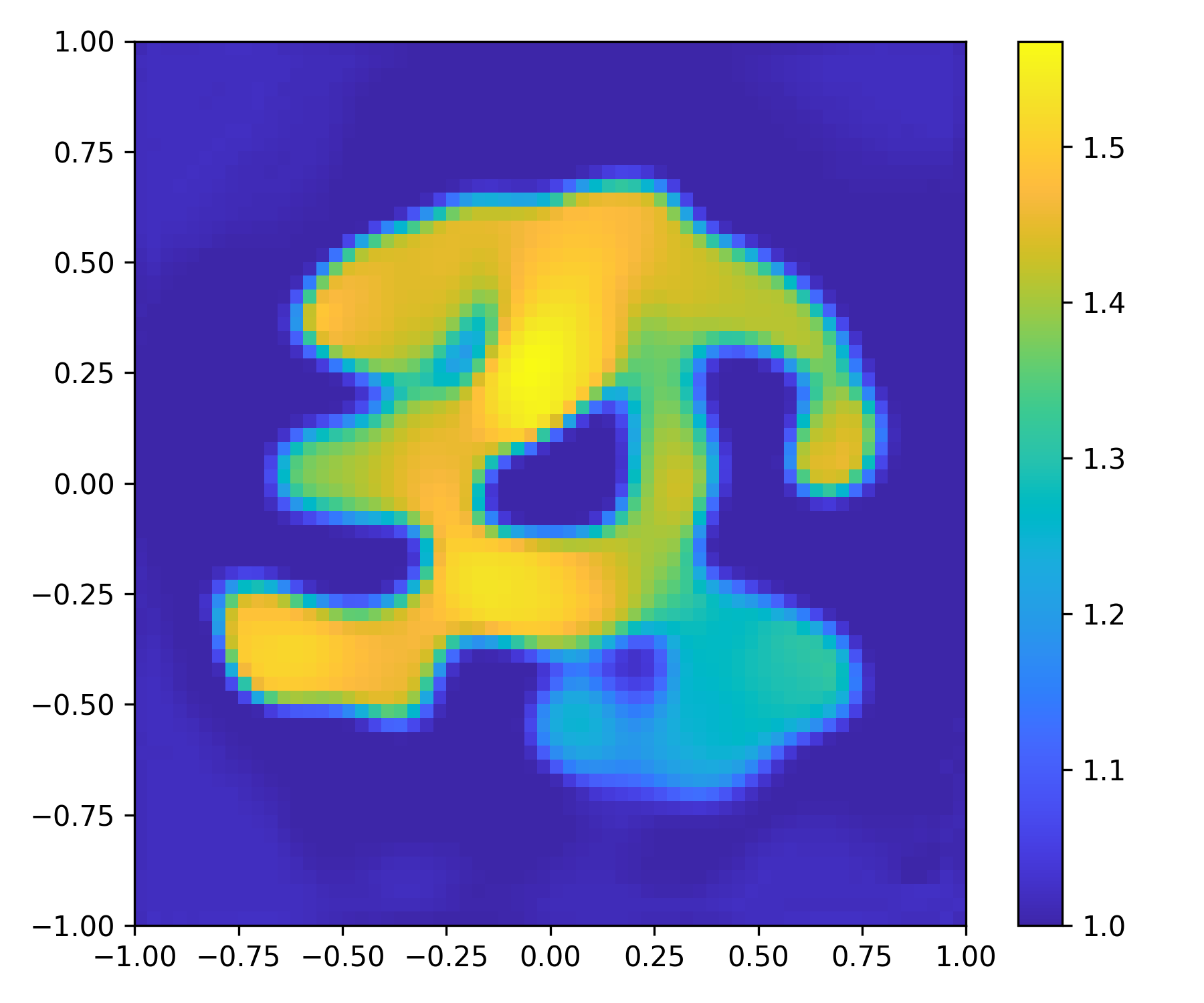}&\includegraphics[width=0.18\textwidth]{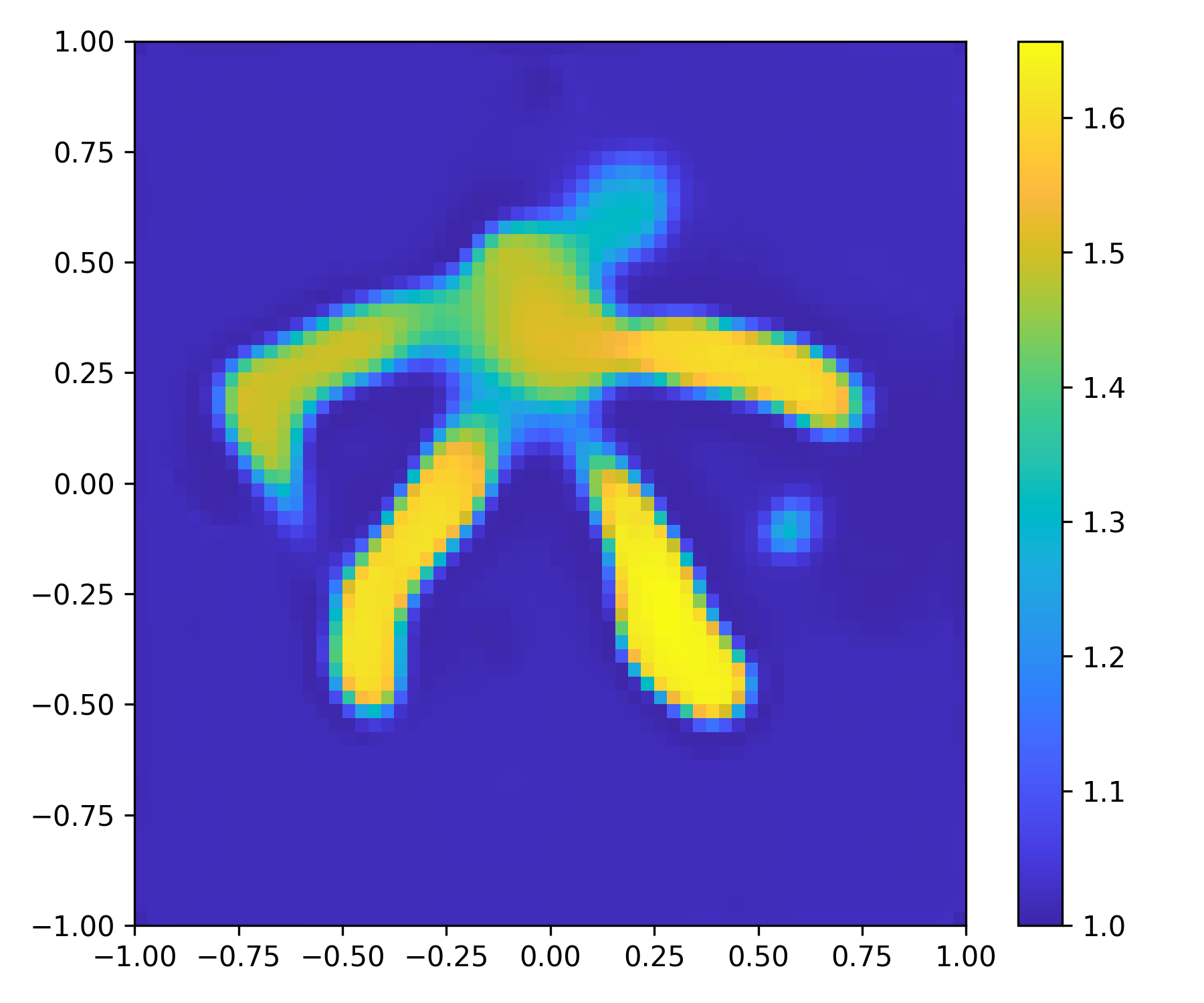} &\includegraphics[width=0.18\textwidth]{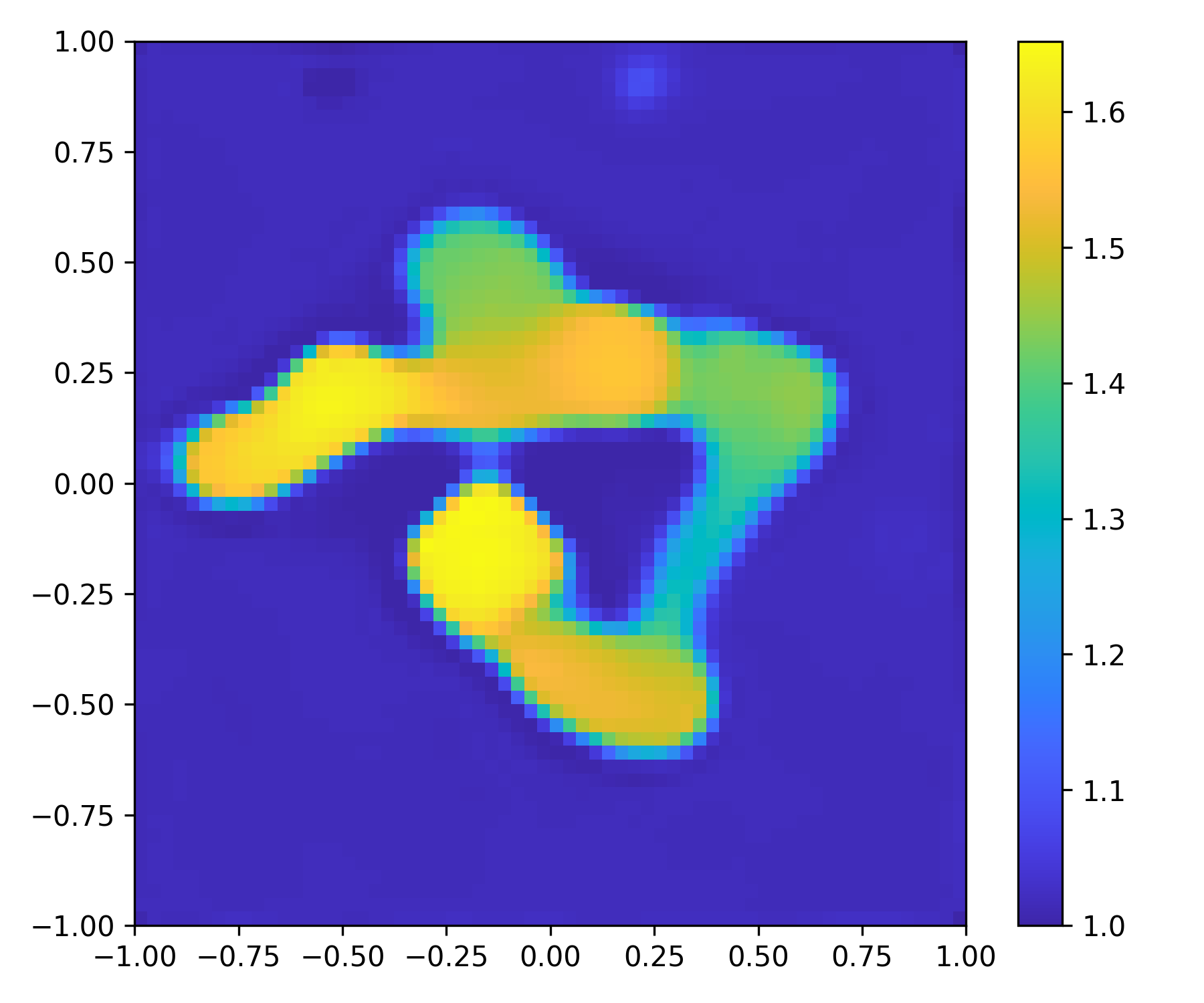}
			&\includegraphics[width=0.18\textwidth]{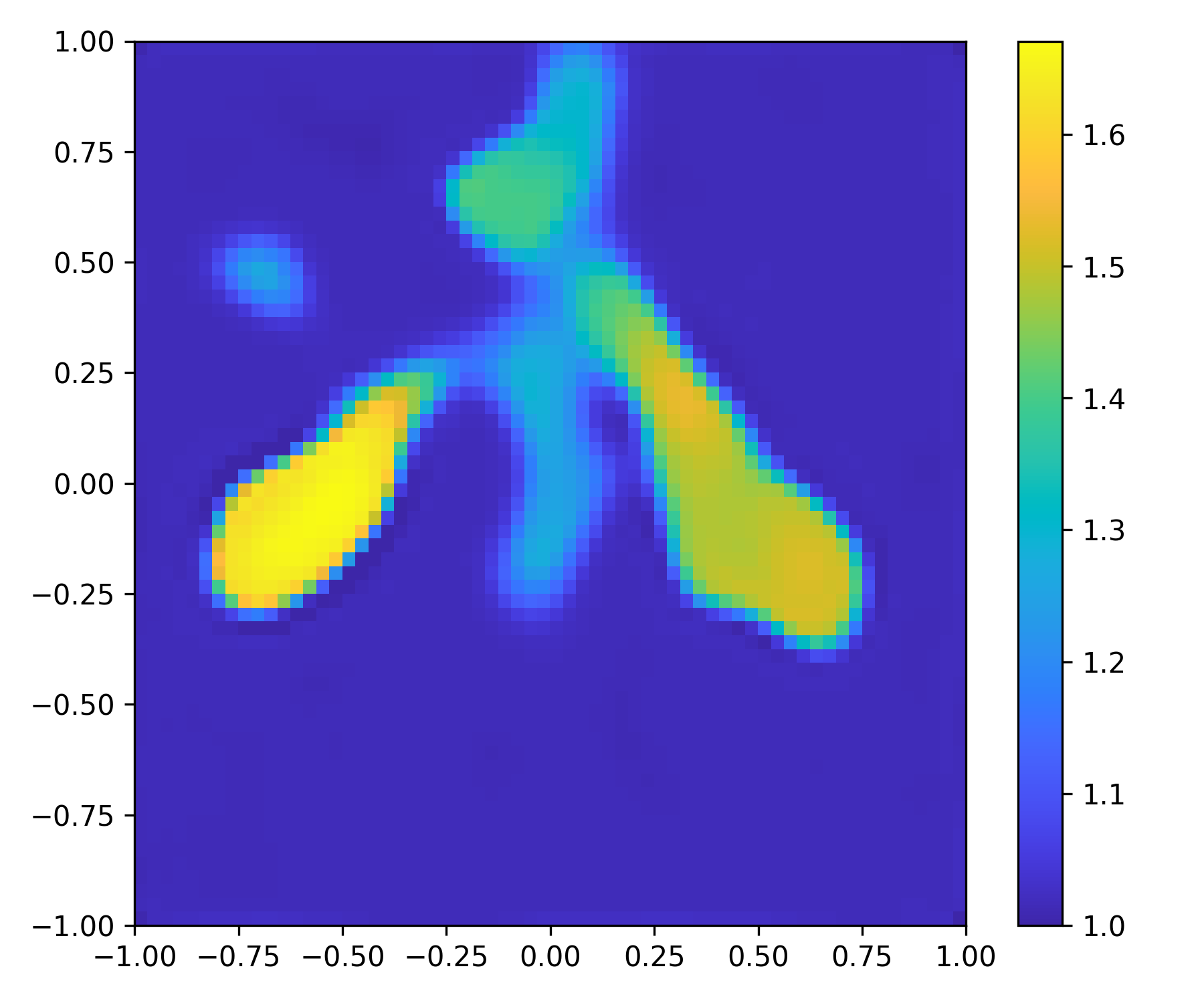}
			&\includegraphics[width=0.18\textwidth]{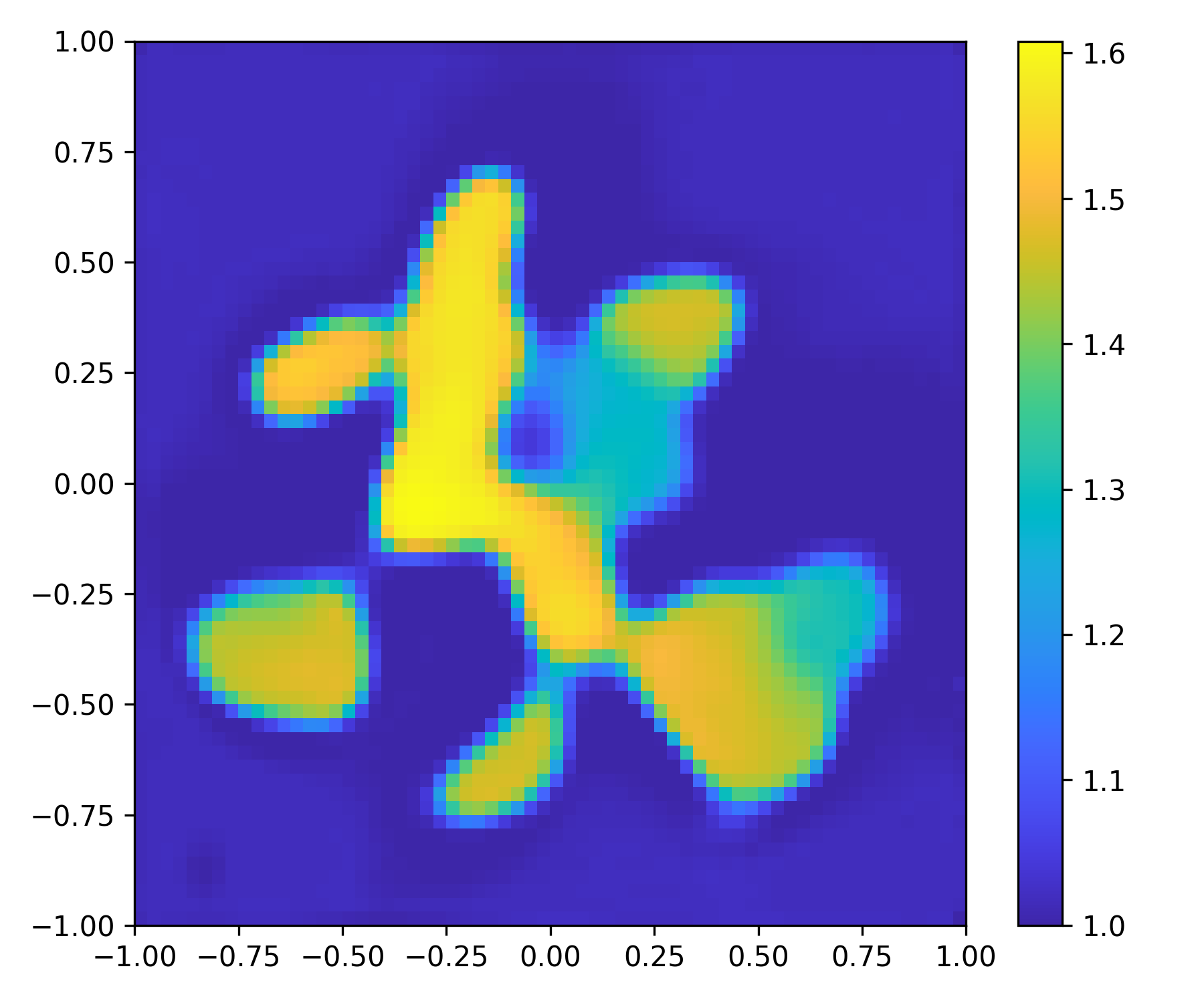}\\
			{$N_i=16$\\ $\delta=5\%$}&
			\includegraphics[width=0.18\textwidth]{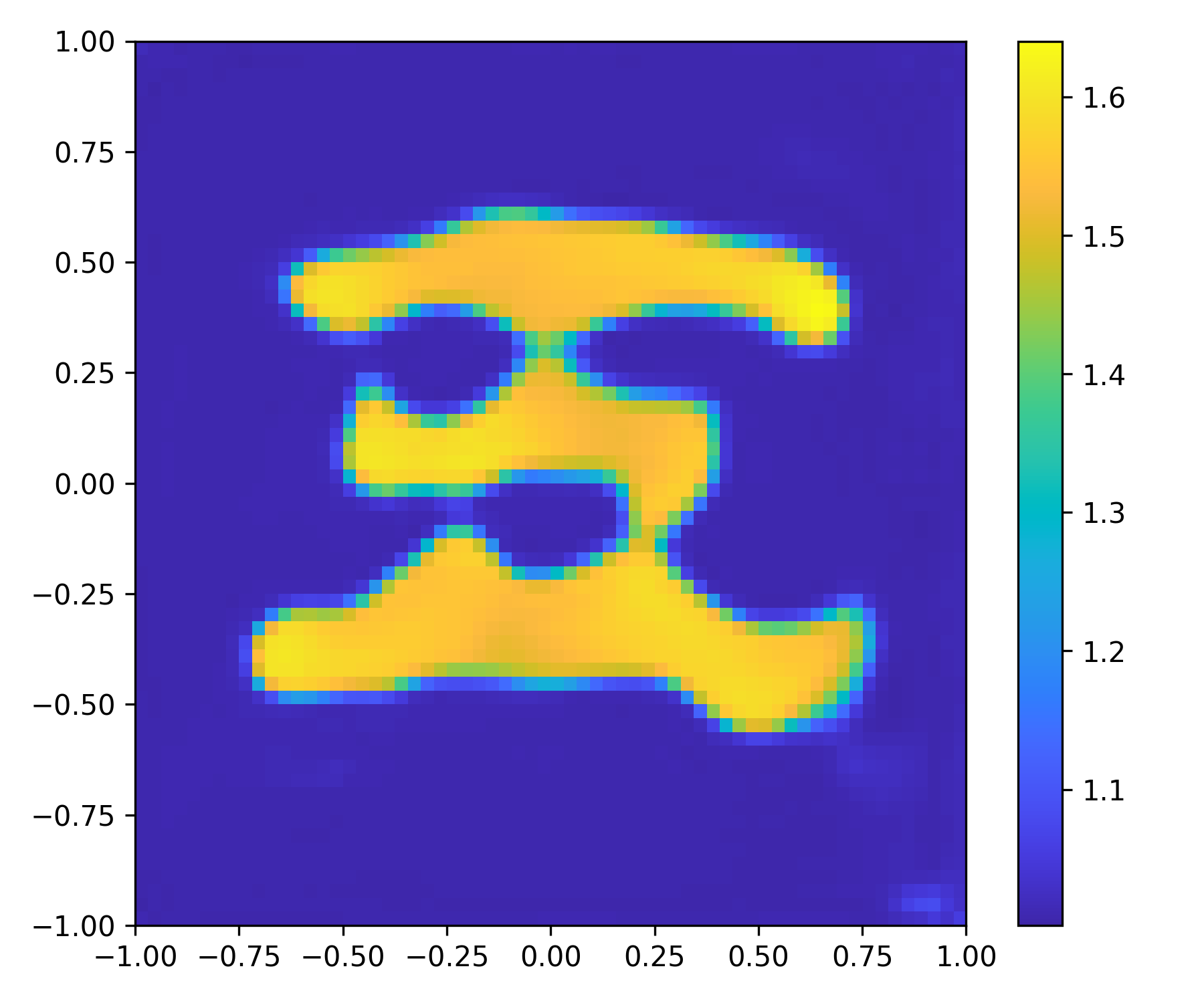}&\includegraphics[width=0.18\textwidth]{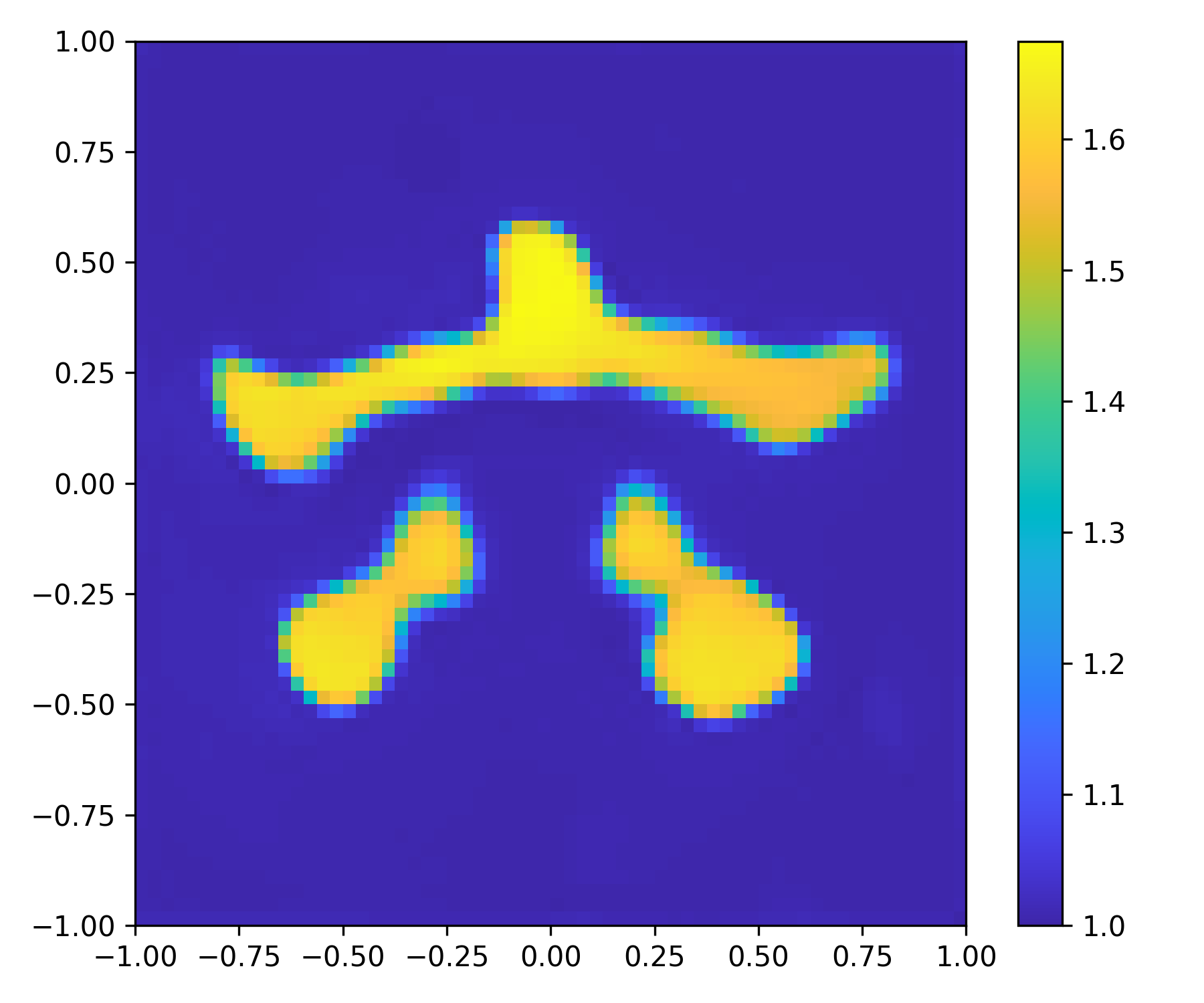} &\includegraphics[width=0.18\textwidth]{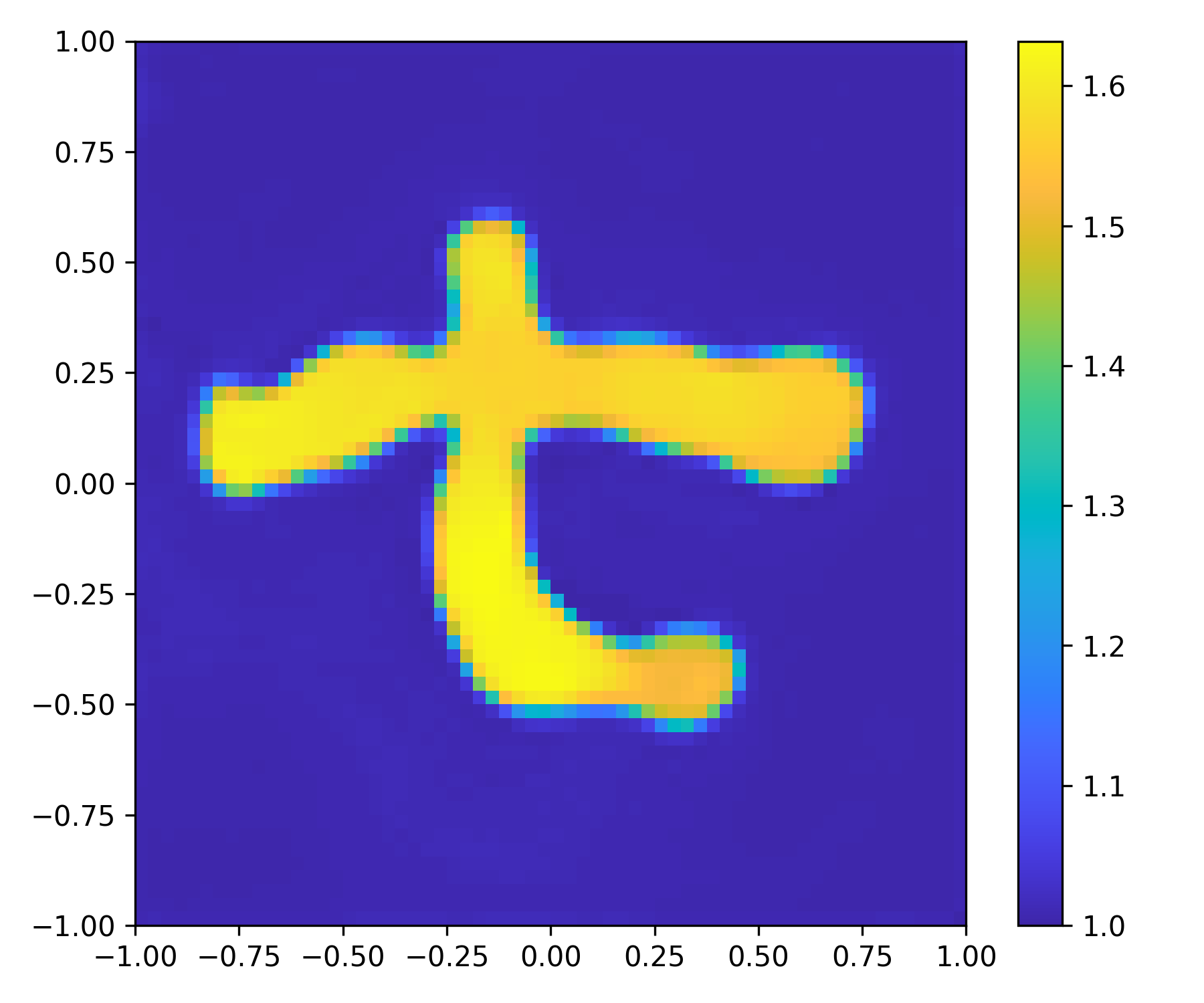}
			&\includegraphics[width=0.18\textwidth]{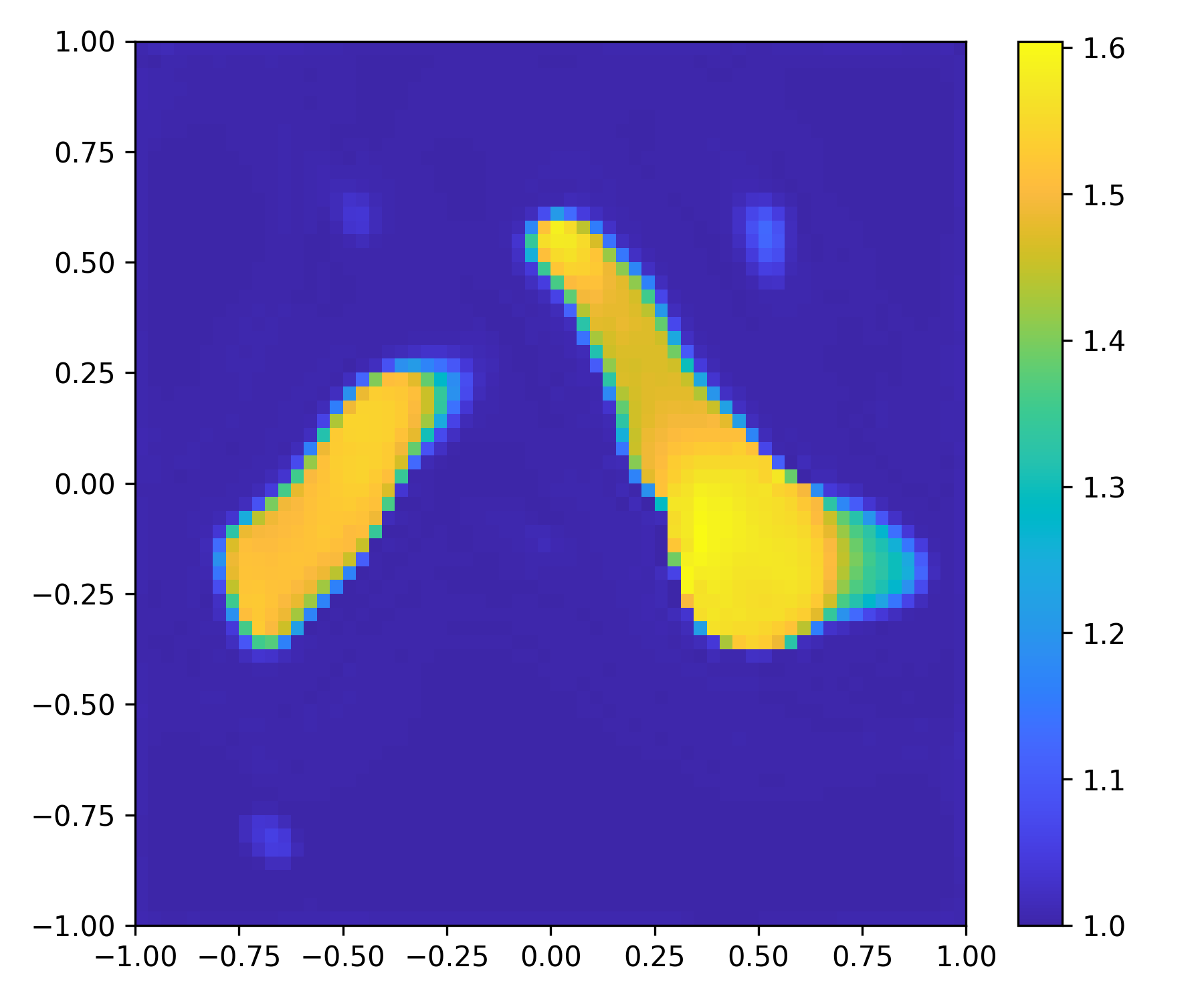}
			&\includegraphics[width=0.18\textwidth]{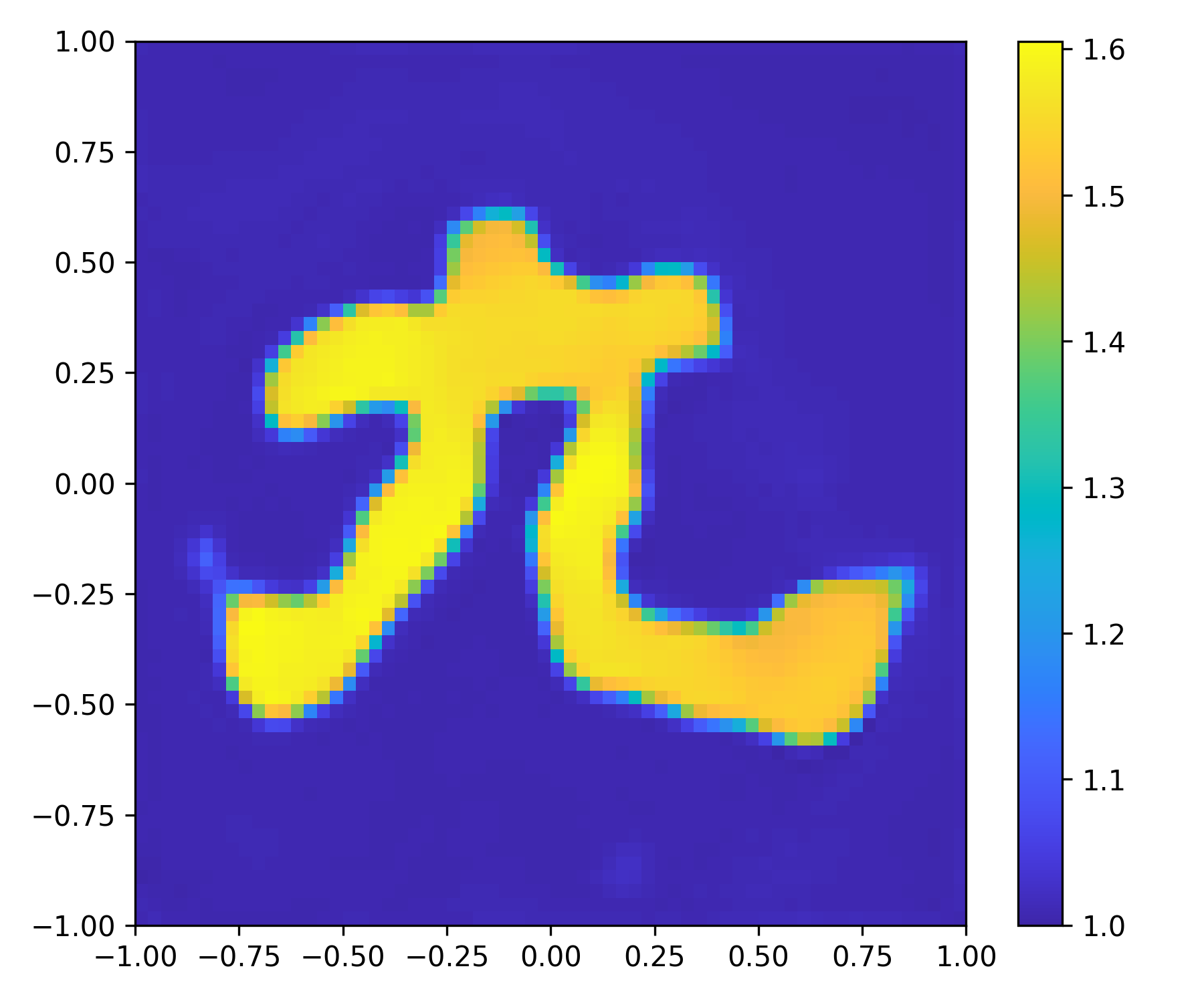}\\
			{$N_i=16$\\ $\delta=10\%$}&
			\includegraphics[width=0.18\textwidth]{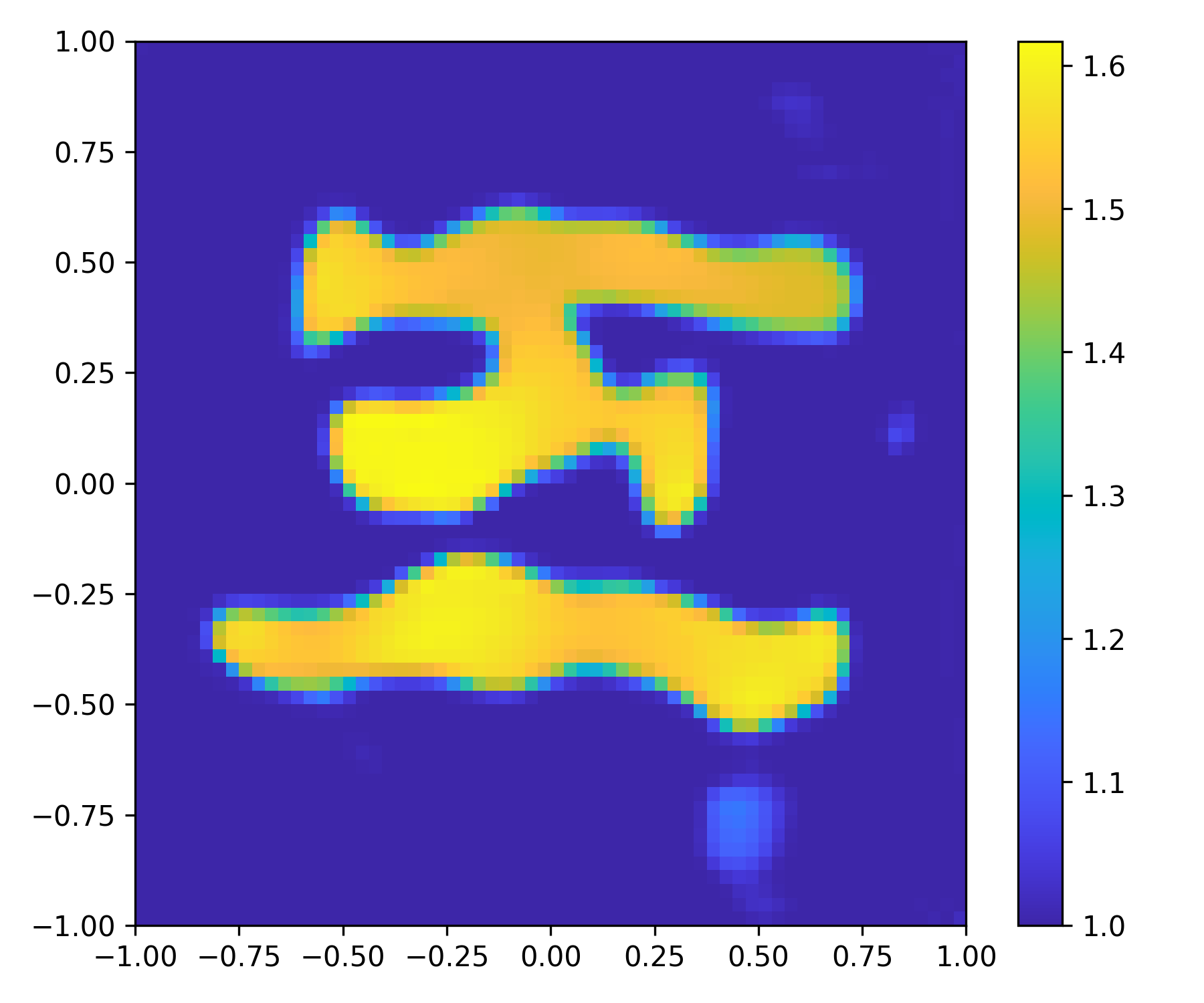}&\includegraphics[width=0.18\textwidth]{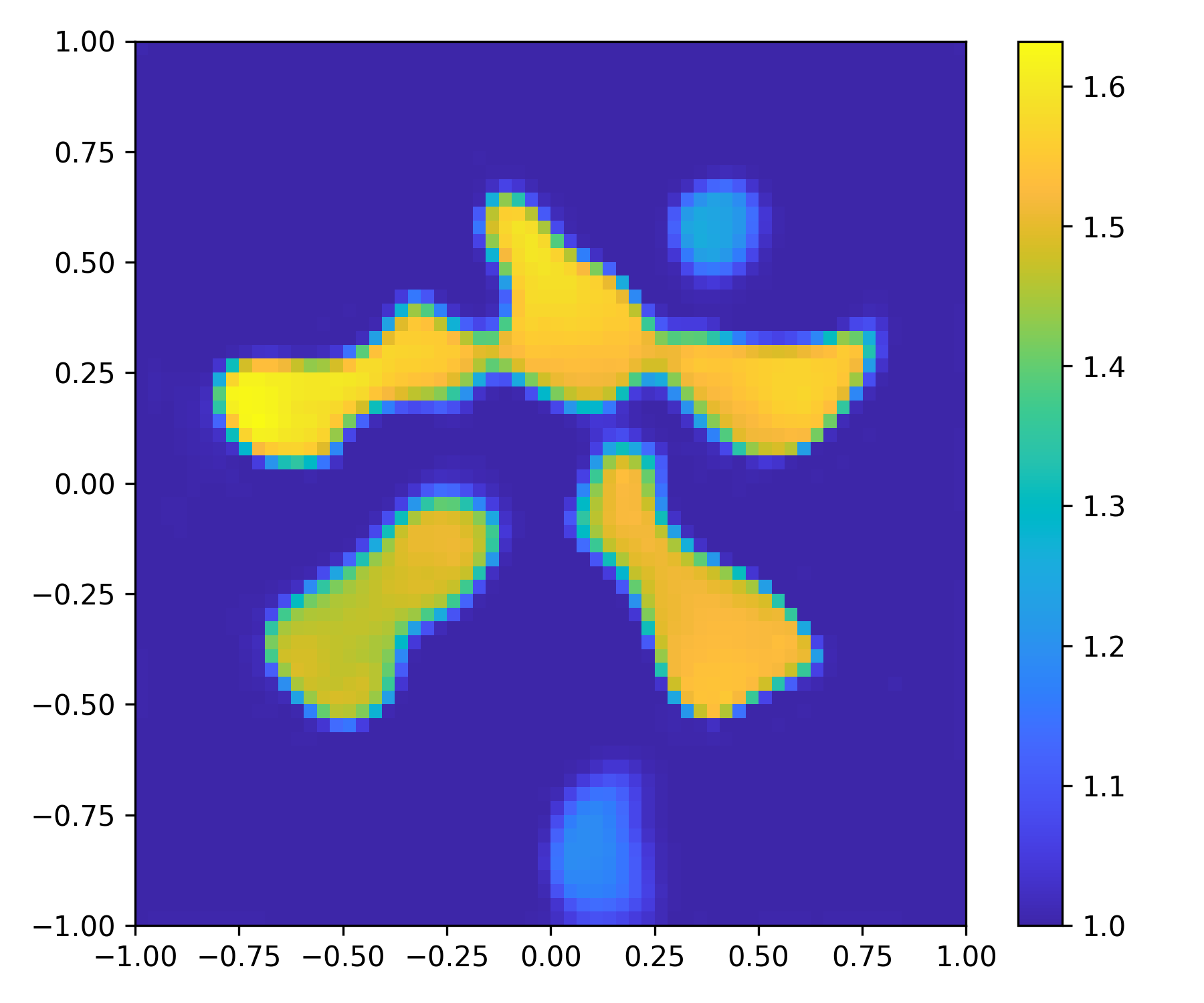} &\includegraphics[width=0.18\textwidth]{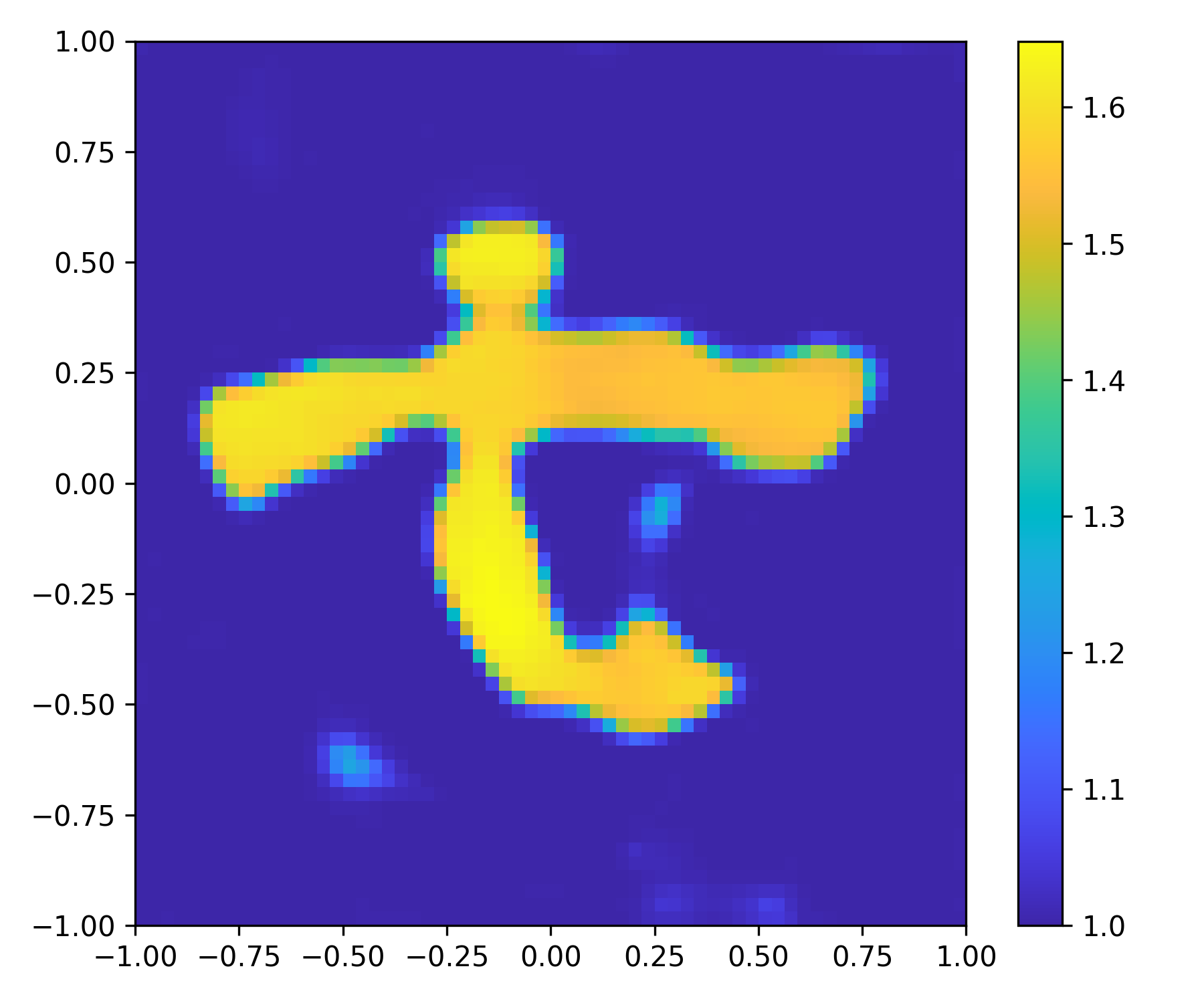}
			&\includegraphics[width=0.18\textwidth]{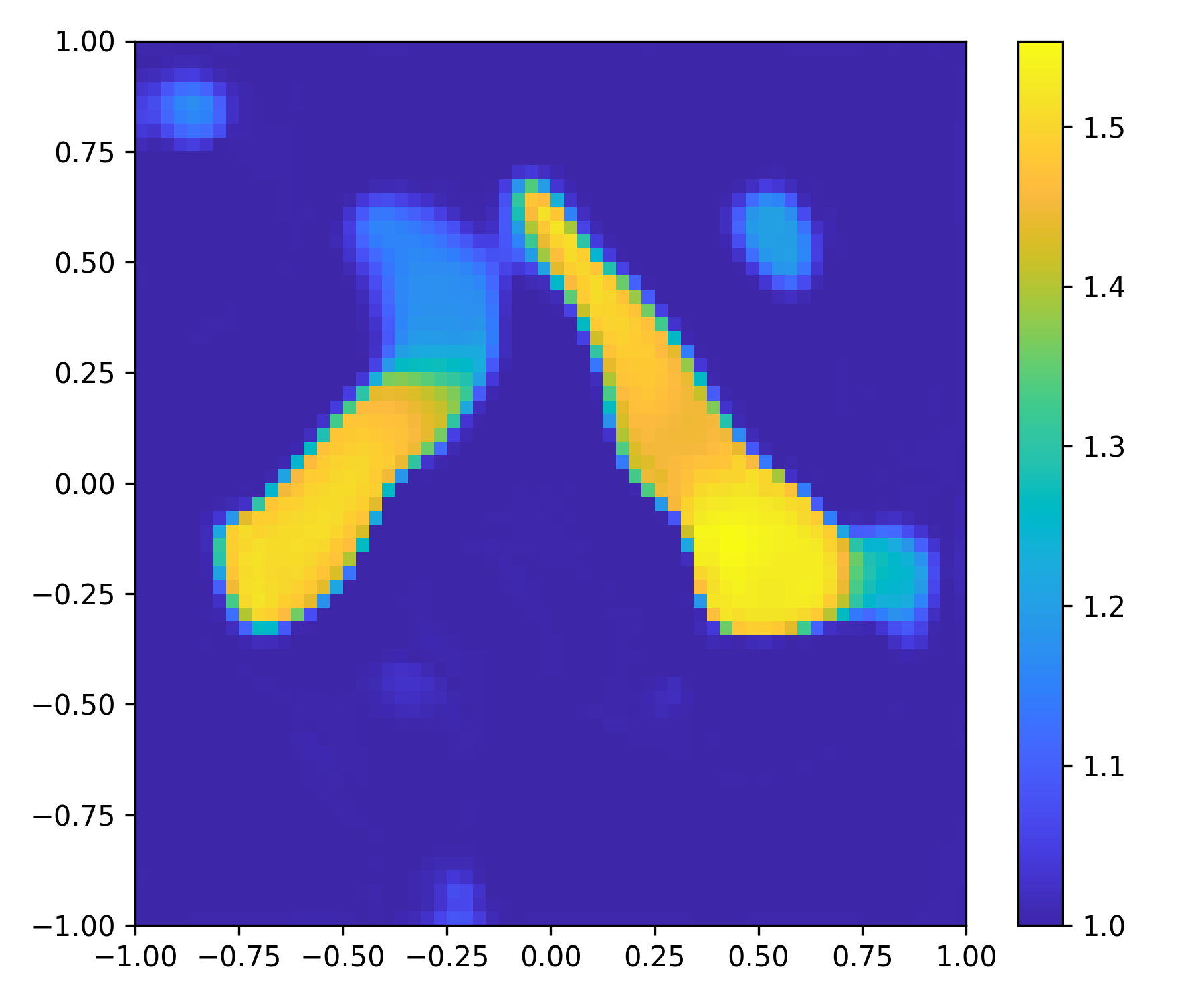}
			&\includegraphics[width=0.18\textwidth]{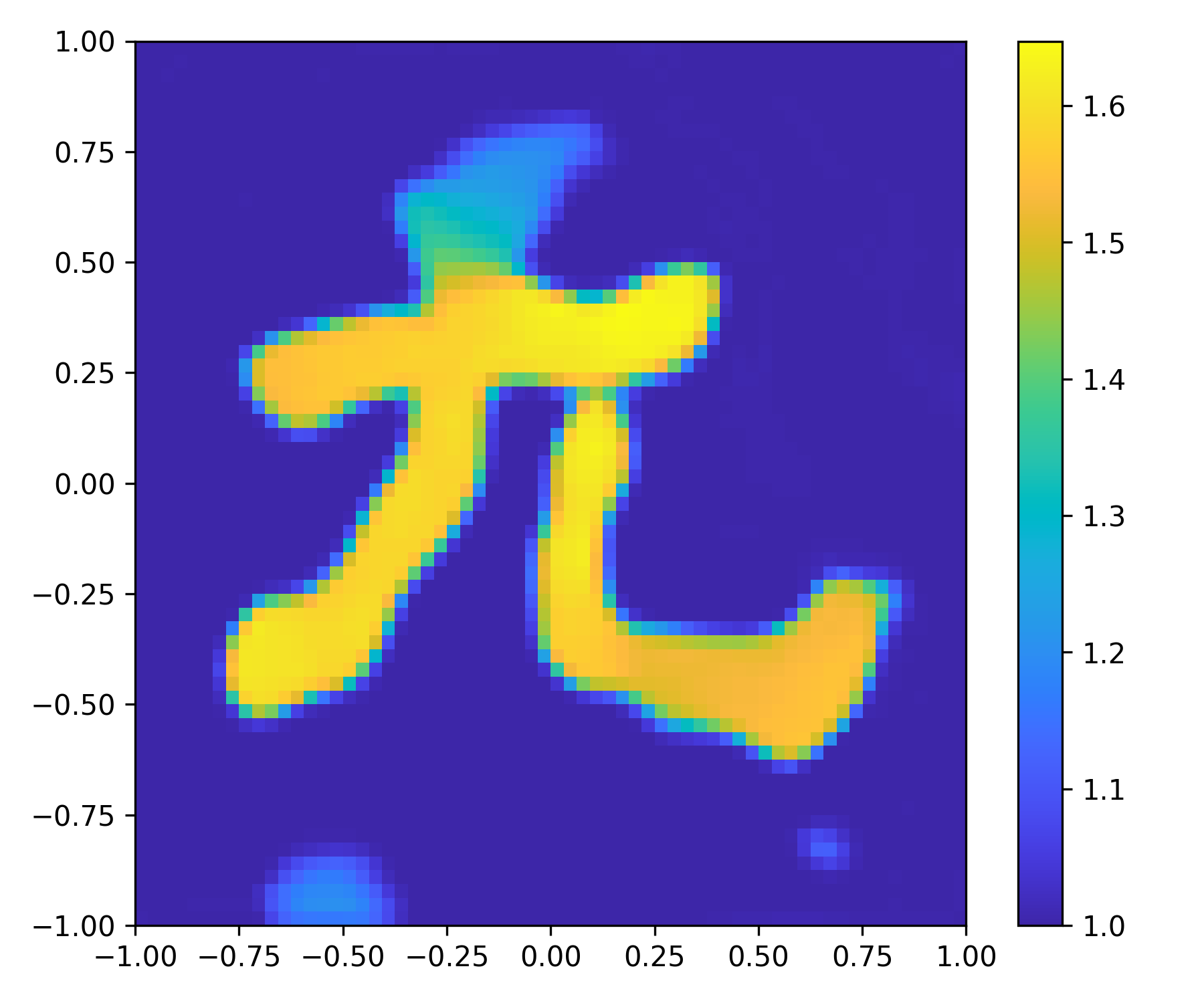}\\
			
		\end{tblr}
		
		\caption{Example  \ref{examp:chinese}: Reconstructed images for 5 Chinese characters by using the networks trained by the MNIST dataset. Row 1: true images; other rows: reconstructions with different incidences and noise levels.}
		\label{Chinese}
	\end{center}
\end{figure}

\subsubsubsection{Tests with “Austria Rings”.}
\label{examp:austria}
Finally,  we apply the trained neural networks to test two different “Austria rings” as shown in Fig.\,\ref{Austria}. The  “Austria ring” is known to be a challenging profile in inverse scattering problems. For the first “Austria” example, the coefficient $n(x)$ of the scatterers is set to 1.5. While for the second “Austria” example, the coefficient $n(x)$ of the two circles is changed to 2.0, which is out of the range in the training data. When $N_i=4$, the method can only provide very rough reconstructions, and the recovered images are heavily affected by high-level noises. By using $N_i=16$ incidences, the method can produce more accurate reconstructions and can distinguish the strong scatterers and the weak scatterer in the second example.

\begin{figure}[htbp]\small
	\begin{center}
		\begin{tblr}
			{colspec = {X[c,m]X[c,h]X[c,h]X[c,h]X[c,h]X[c,h]},
				stretch = 0,
				rowsep = 0pt,}
			Ground Truth& $N_i=4,\delta=5\%$ & $N_i=4,\delta=10\%$ &$N_i=16,,\delta=5\%$ &$N_i=16,,\delta=10\%$\\
			\includegraphics[width=0.18\textwidth]{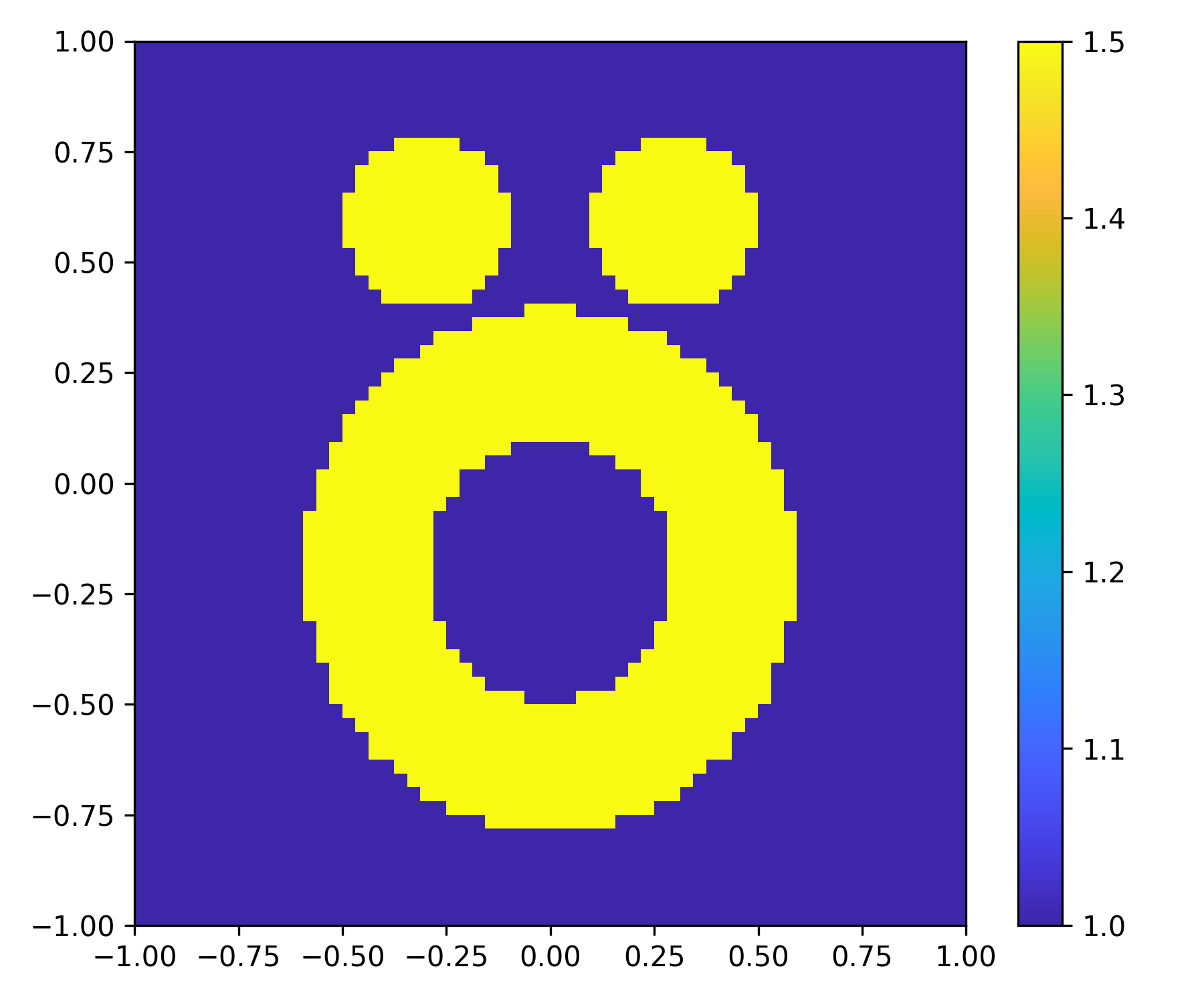}&
			\includegraphics[width=0.18\textwidth]{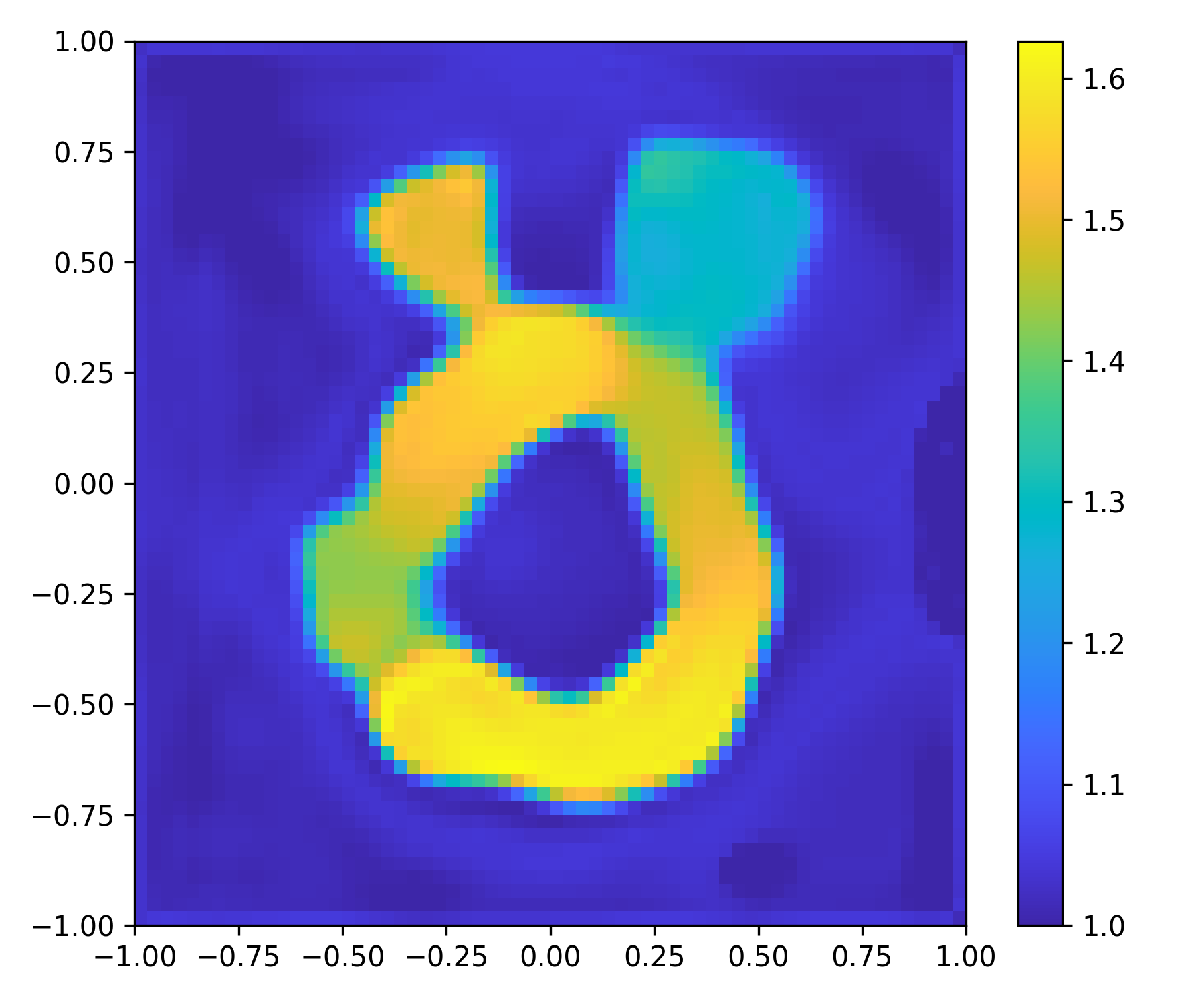}&
			\includegraphics[width=0.18\textwidth]{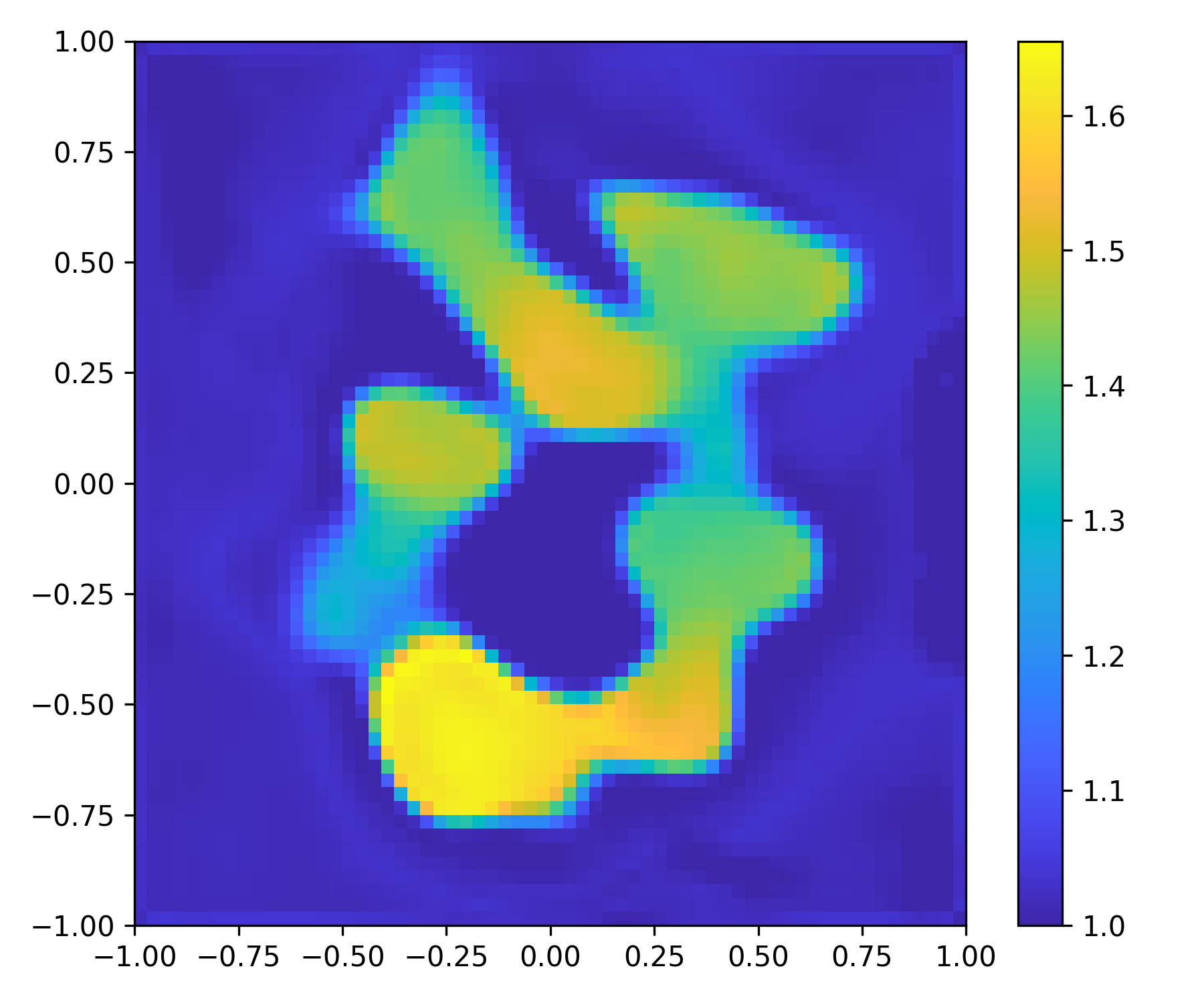}&
			\includegraphics[width=0.18\textwidth]{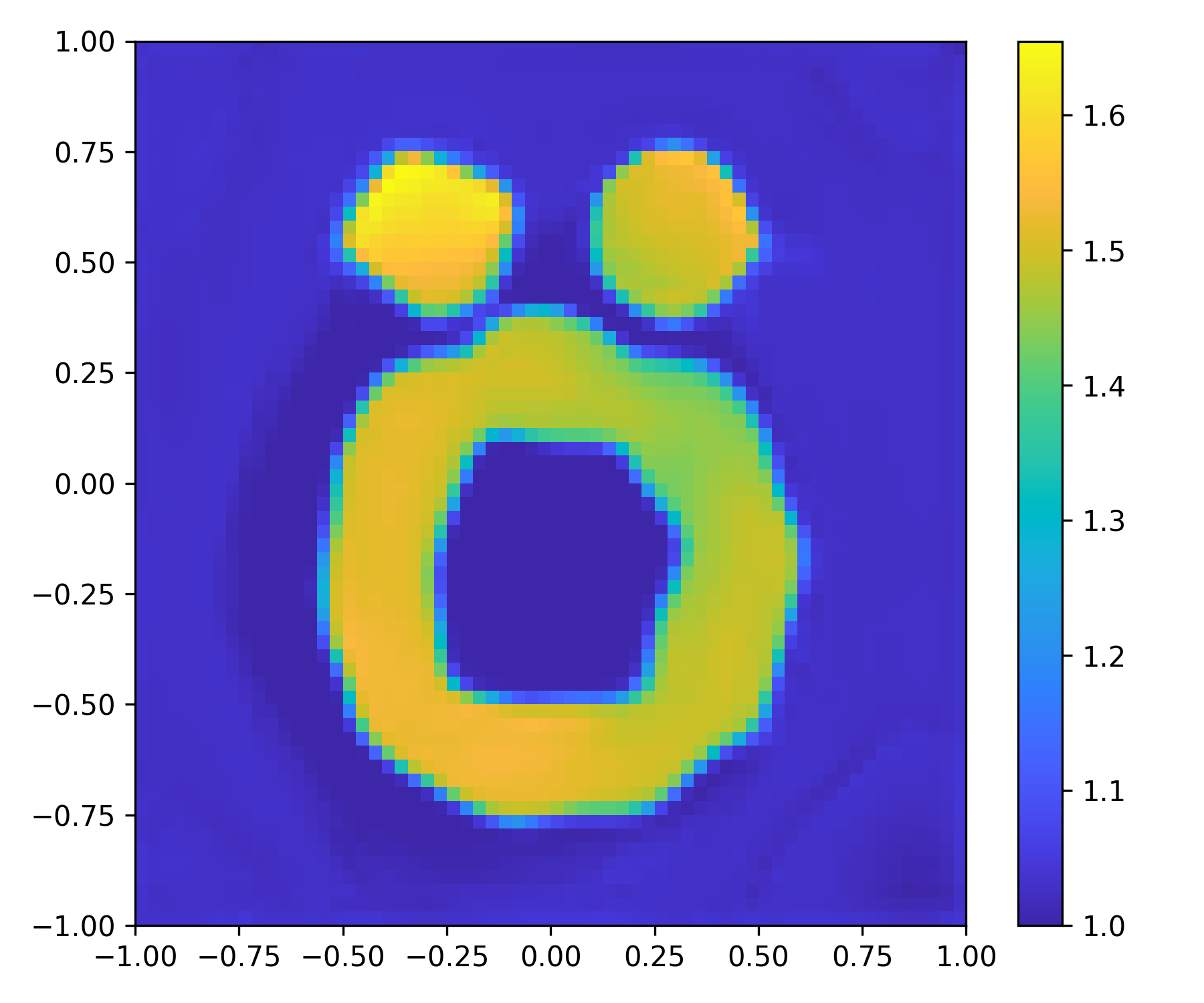}&
			\includegraphics[width=0.18\textwidth]{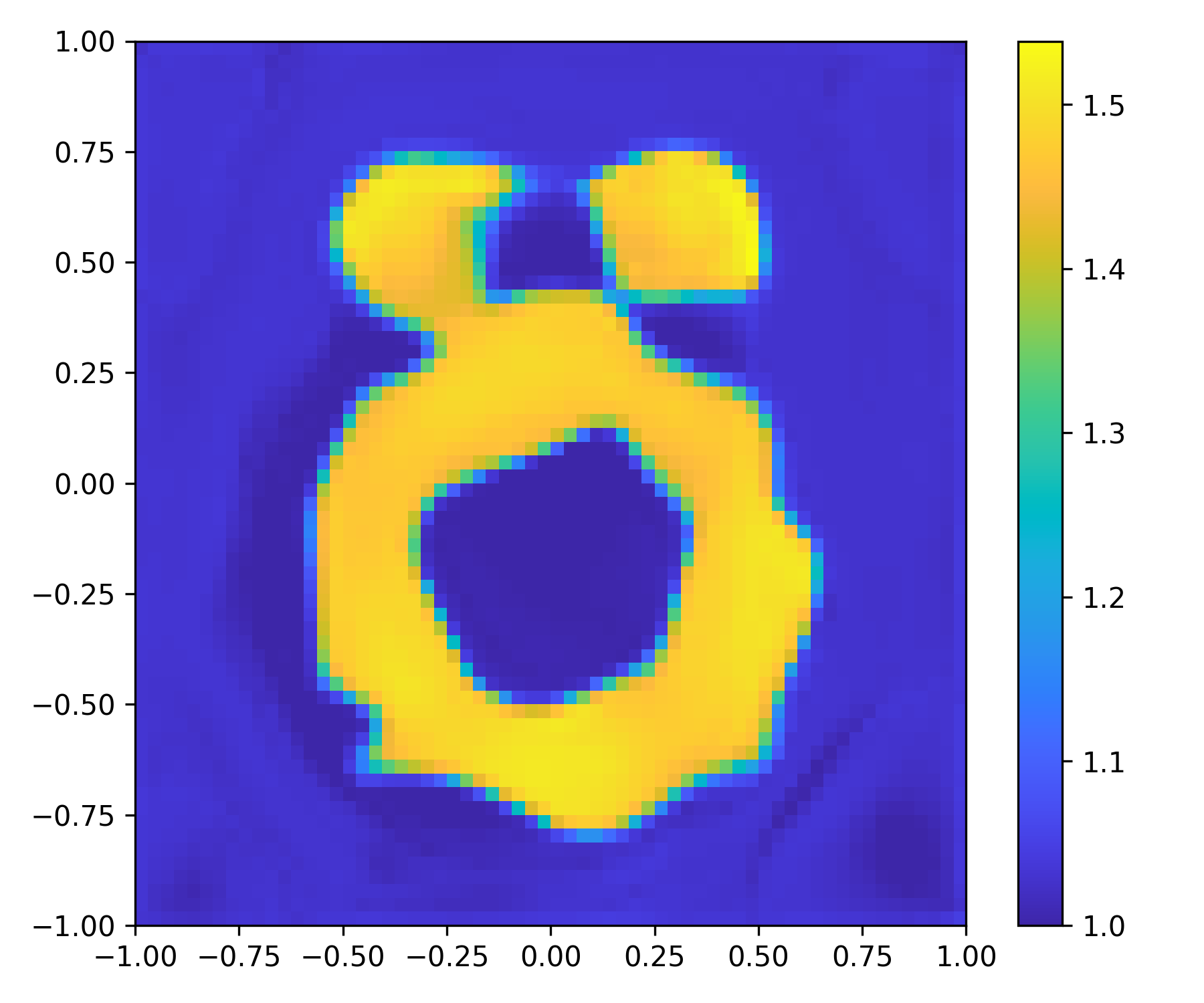}\\
			\includegraphics[width=0.18\textwidth]{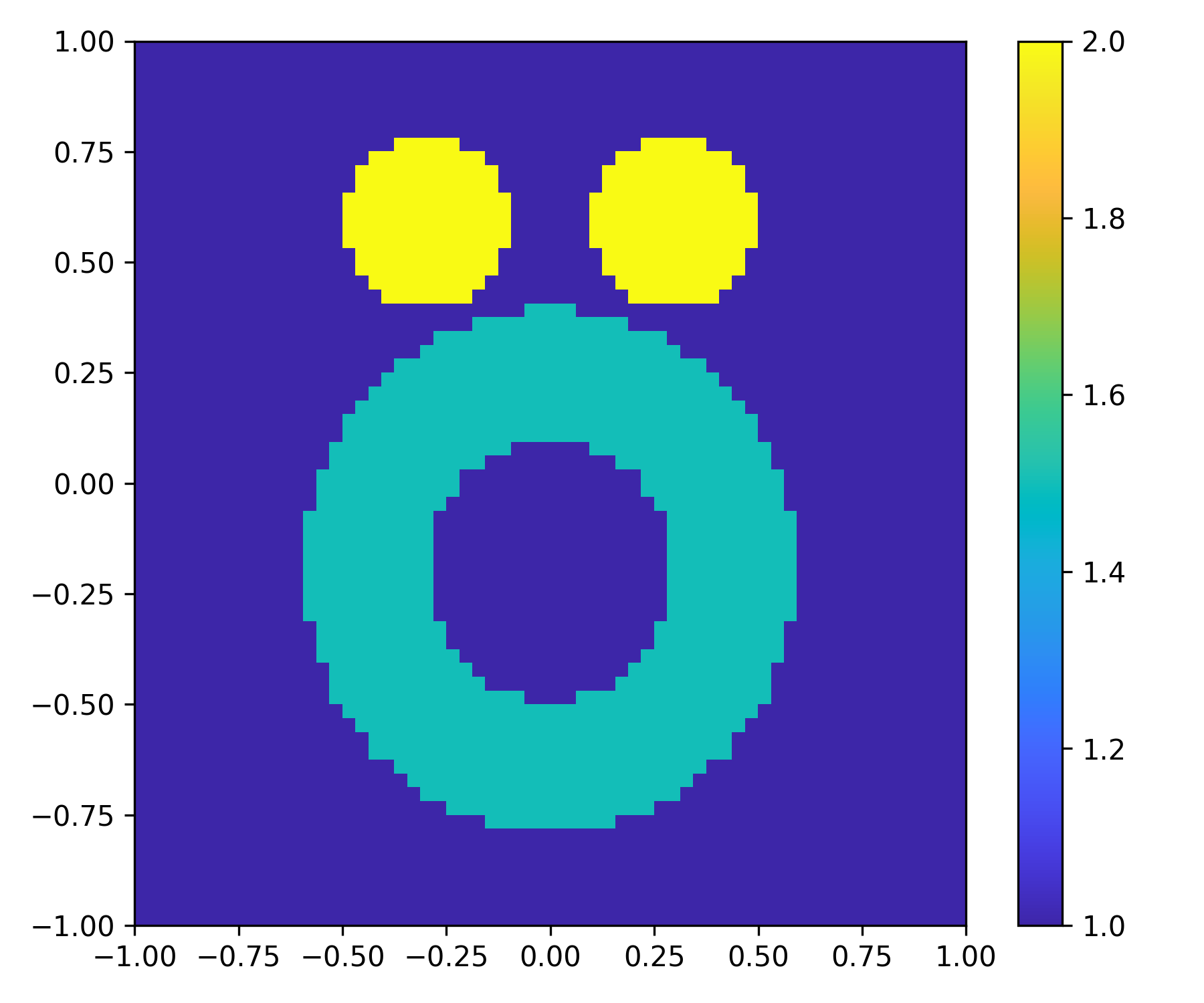}&
			\includegraphics[width=0.18\textwidth]{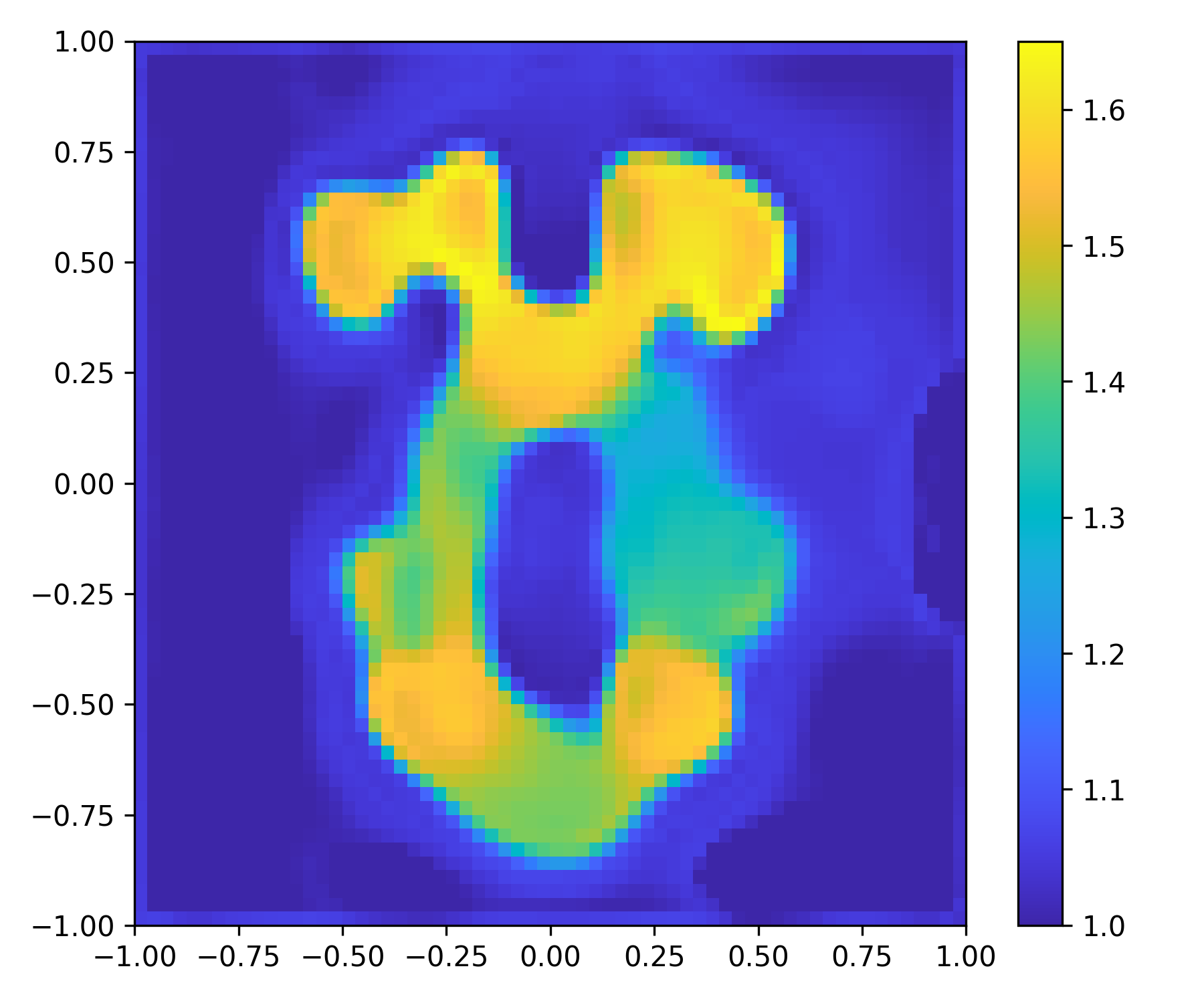}&
			\includegraphics[width=0.18\textwidth]{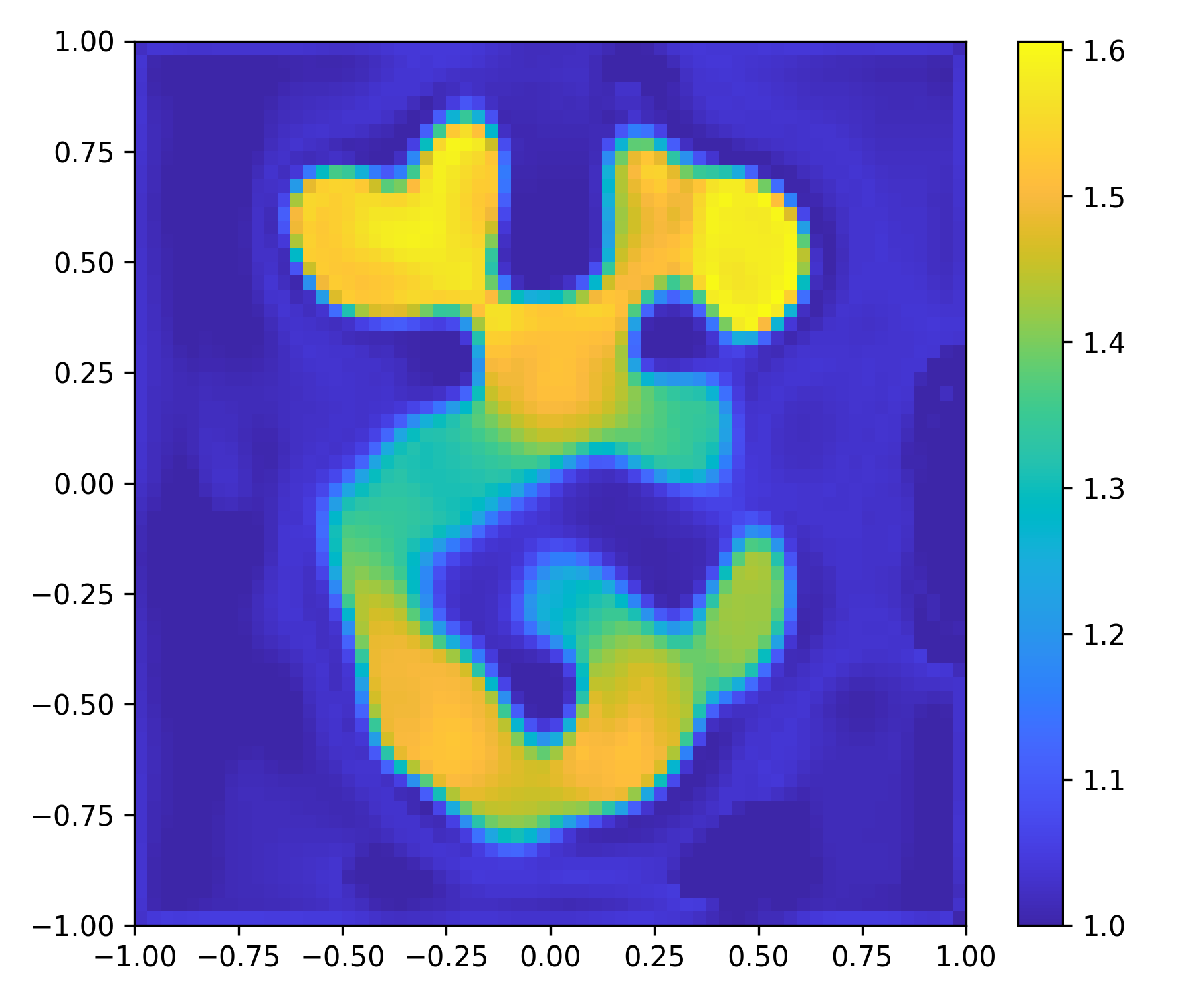}&
			\includegraphics[width=0.18\textwidth]{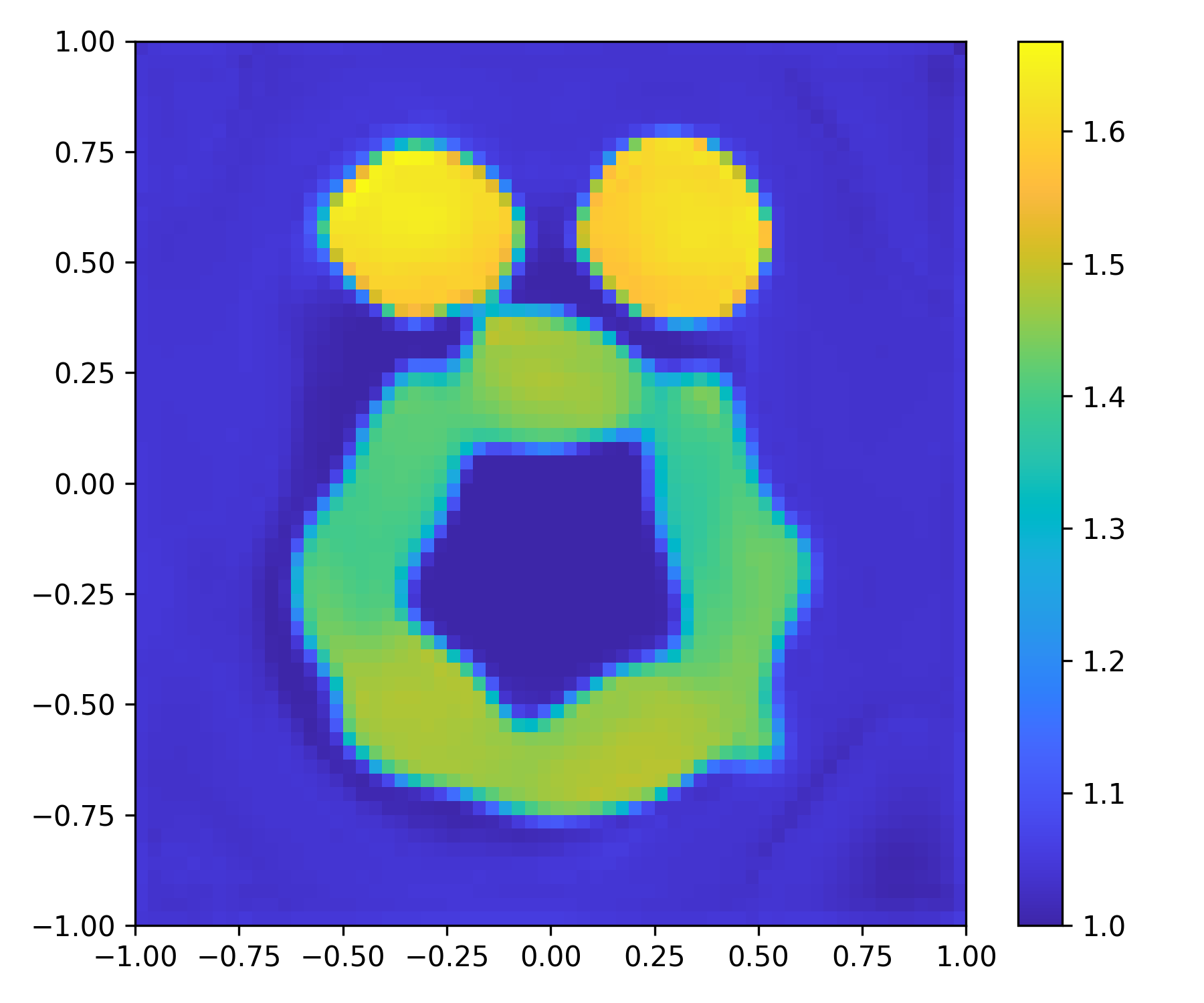}&
			\includegraphics[width=0.18\textwidth]{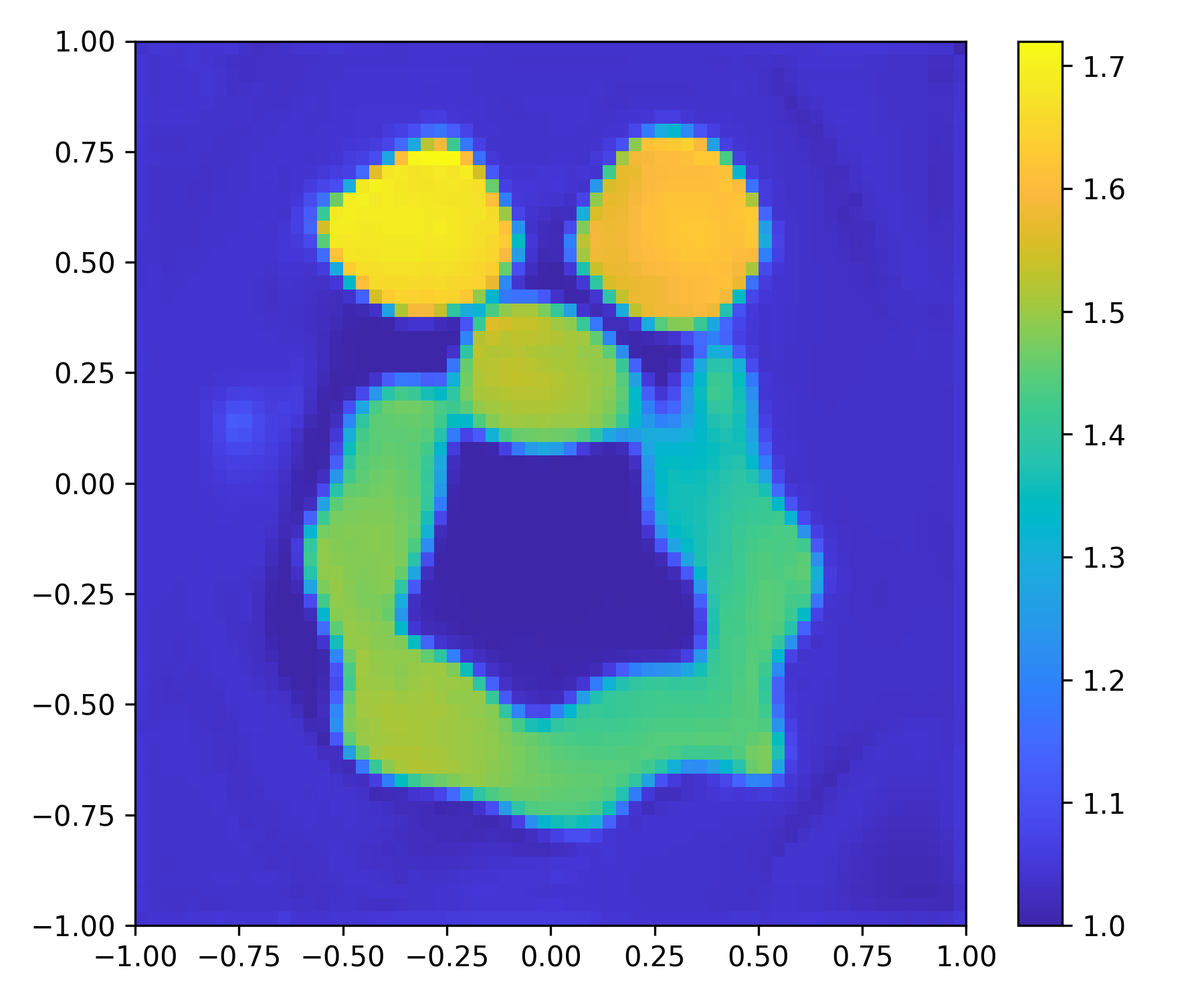}
			
		\end{tblr}
		
		\caption{Example \ref{examp:austria}: Reconstructed images for two “Austria profiles” by using the networks trained by the MNIST dataset. Column 1: true images; other columns: reconstructions with different incidences and noise levels.}
		\label{Austria}
	\end{center}
\end{figure}
\begin{table}[hbtp]
	\begin{center}
		\begin{tabular}{c|c|c|c|c}
			\hline
			\textbf{Example} & $\mathbf{N_i=4, \delta=5\%}$ & $\mathbf{N_i=4, \delta=10\%}$ &
			$\mathbf{N_i=16,\delta=5\%}$ &
			$\mathbf{N_i=16, \delta=10\%}$	
			\\			
			\hline
			\textbf{MNIST} & 0.0827 & $0.1043$& $0.0617$ &$0.0755$ \\
			\textbf{Chinese characters} & $0.1096$ & $0.1252 $& $0.0721$ &$0.0854$ \\
			\textbf{Austria Ring 1} & $0.1163$ & $0.1258$& $0.0851$ &$0.0922$ \\
			\textbf{Austria Ring 2} & $0.1897$ & $0.1810 $& $0.1260$ &$0.1367$ \\	
            \bottomrule
		\end{tabular}
	\end{center}
  \captionsetup{skip=1pt}
\caption{Examples \ref{examp:mnist}, \ref{examp:chinese} and \ref{examp:austria}: The relative L2 errors for different examples with the networks trained by the MNIST dataset.}
	\label{tab:MNIST}
\end{table}

\subsubsection{Training with a mixed circle dataset by phased data}\label{NE_mixed}

We now consider a mixed problem in which the domain of interest may consist of both medium scatterers and sound-soft obstacles. To further show the convergence property of the index functions, the network is trained by phased data, i.e, with the index functions\eqref{DSM} computed form phased data as the input of network, and then applied to solve problems with phaseless data. This is a very important advantage of the DSM-DL as it is usually difficult to recover phased information from phaseless data. In the training dataset, $1\sim 3$ circles with varying radius drawn from the uniform distribution $U(0.2,0.3)$ are randomly added to the domain to simulate scatterers. Scatterers are not allowed to overlap and each scatterer has an equal probability of being a medium or sound-soft obstacle. The coefficient value of the medium scatterer is randomly set from $U(1.5,3.0)$, and we still set the pixel value for points inside the sound-soft obstacles to be 0.  We will employ 10 incident plane waves and employ $180$ receivers that are equally distributed on a circle of radius 8 centered at the origin to collect both phased data and phaseless data. We use 20000 datasets to train the network and 200 datasets for testing.  The batch size is taken as 20 and a total of 30 epochs in the training stage. The learning rate starts at 0.001 and decreases by a factor of 0.5 every 3 epochs.

\subsubsubsection{Tests with mixed circle testing data by phaseless data.} 
\label{examp:mixed}
In this example, we employ the model trained by phased data to solve 200 mixed circle examples with phaseless data, and compare the results by adding different noises to the measurement data. The reconstructed images of five typical cases are presented in Fig.\,\ref{fig:MixCir}. The shown results not only prove the capability of the DSM-DL for dealing with inverse scattering problems with mixed scatterers in a unified framework, but also show that the DSM-DL network trained by phased data can also be applied to phaseless data. We observe that the method can accurately identify the boundaries, locations, sizes, and coefficient values of the scatterers. In particular, the results are still quite satisfactory when high-level noises are added to the data. We remark that although the DSM-DL network trained by phased data can be applied to phaseless data, the DSM-DL network trained by phaseless data actually has better performance for phaseless testing data. This is due to the difference between the index functions computed from phased data and phaseless data but the difference is minor when $R_r$ is sufficiently large.

\begin{figure}[htbp]\small
	\begin{center}
		\begin{tblr}
			{colspec = {X[-1,m]X[c,h]X[c,h]X[c,h]X[c,h]X[c,h]},
				stretch = 0,
				rowsep = 0pt,}
			{Ground\\ Truth}&
			\includegraphics[width=0.18\textwidth]{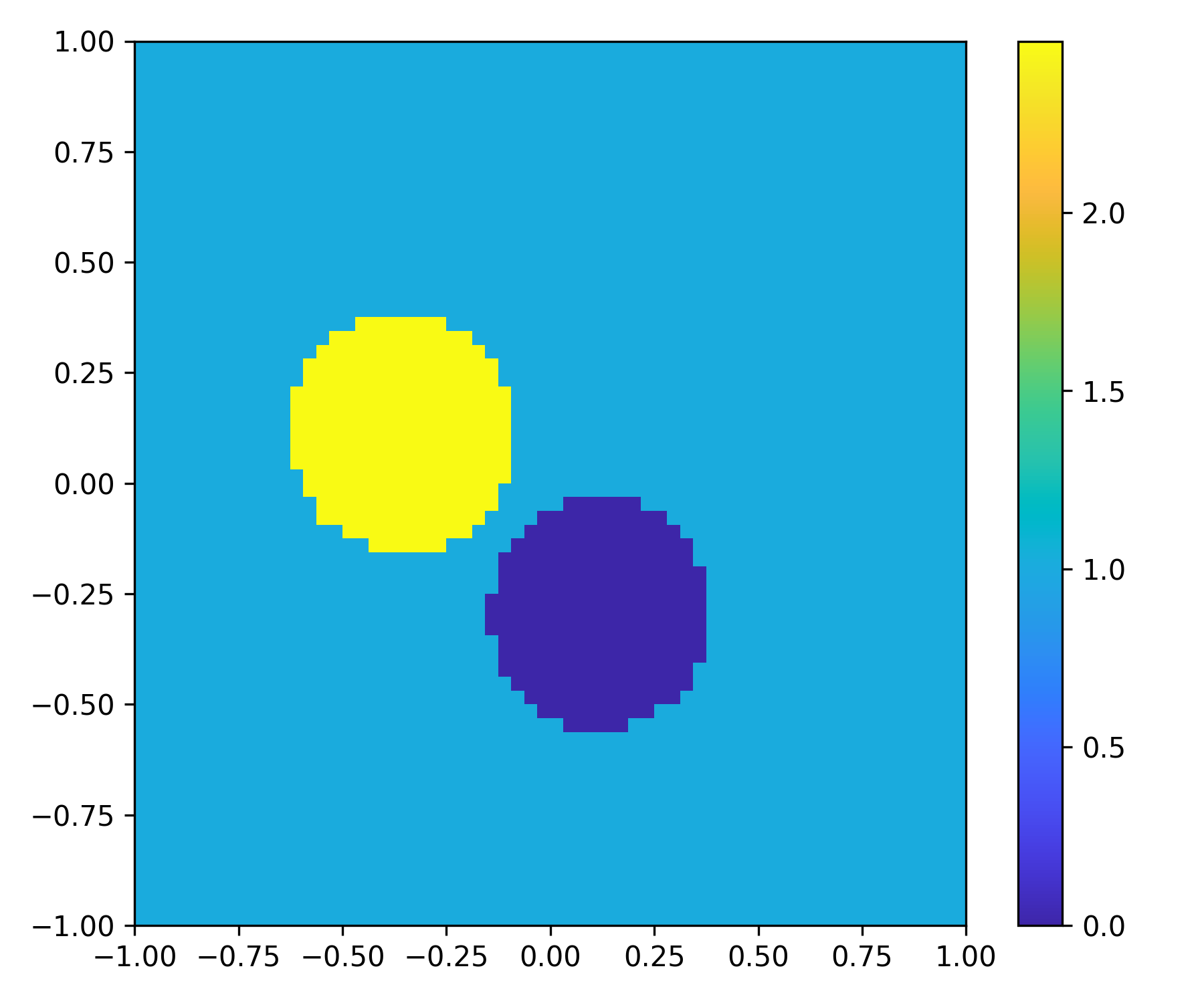}
			&\includegraphics[width=0.18\textwidth]{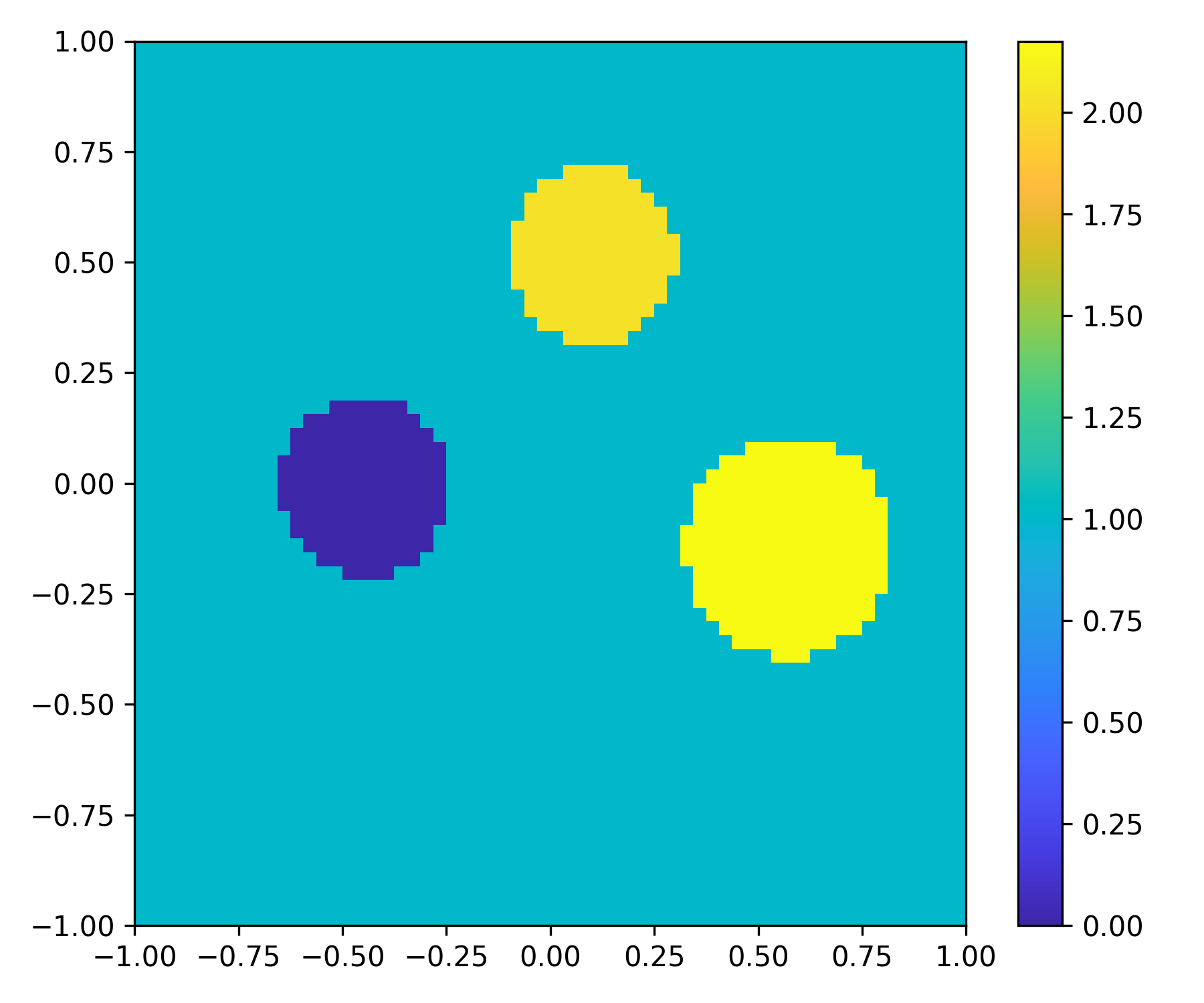} &\includegraphics[width=0.18\textwidth]{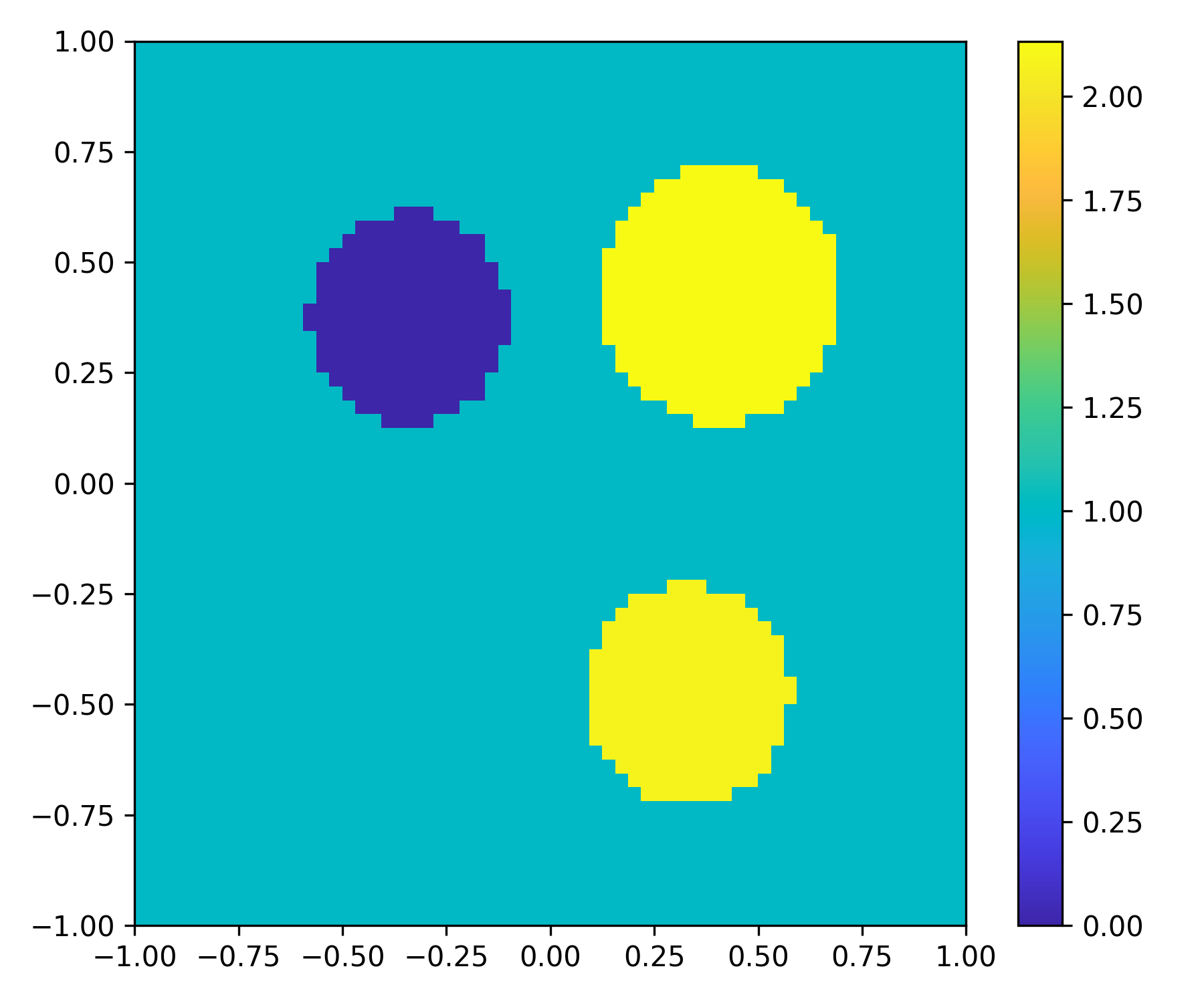}
			&\includegraphics[width=0.18\textwidth]{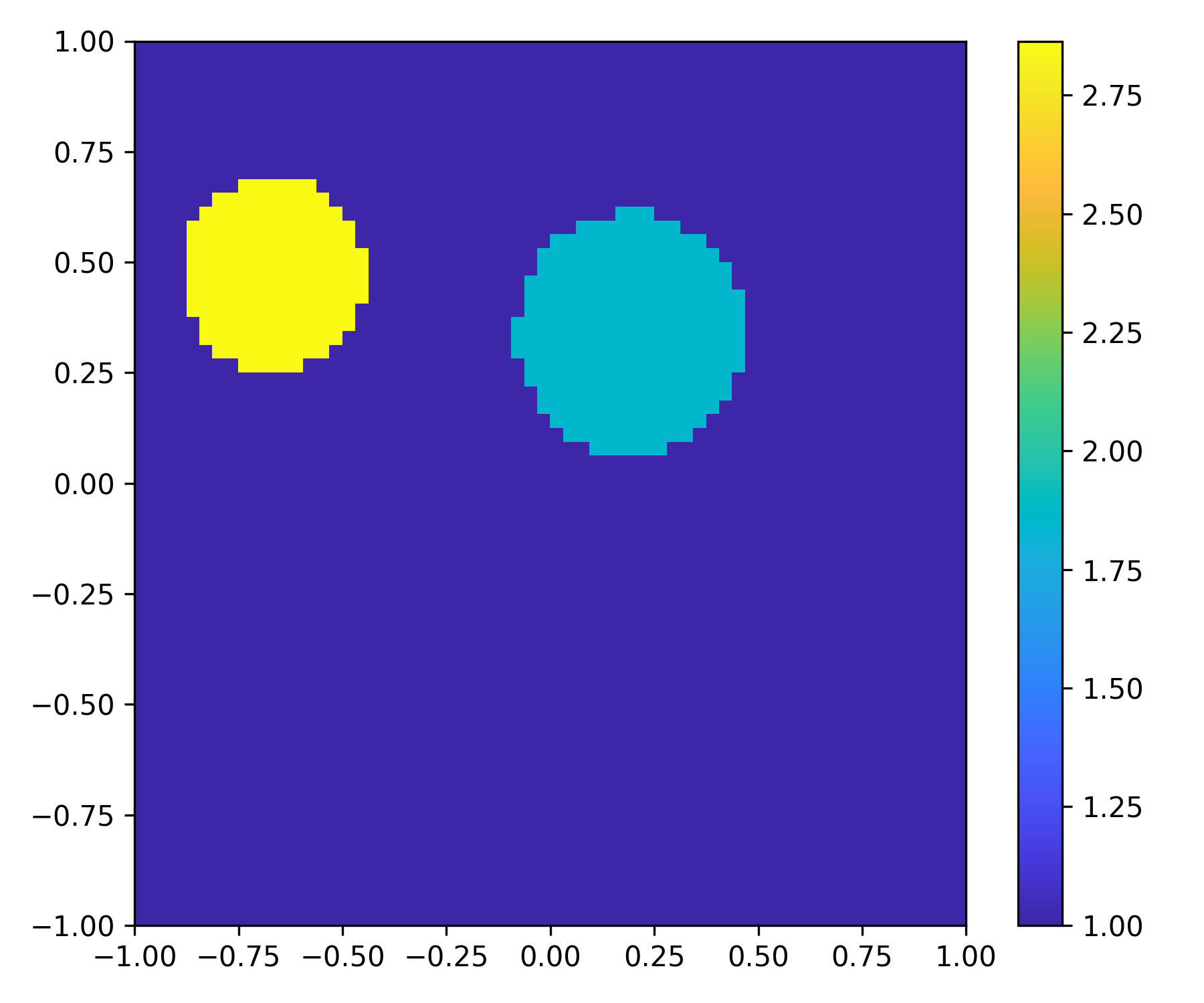}
			&\includegraphics[width=0.18\textwidth]{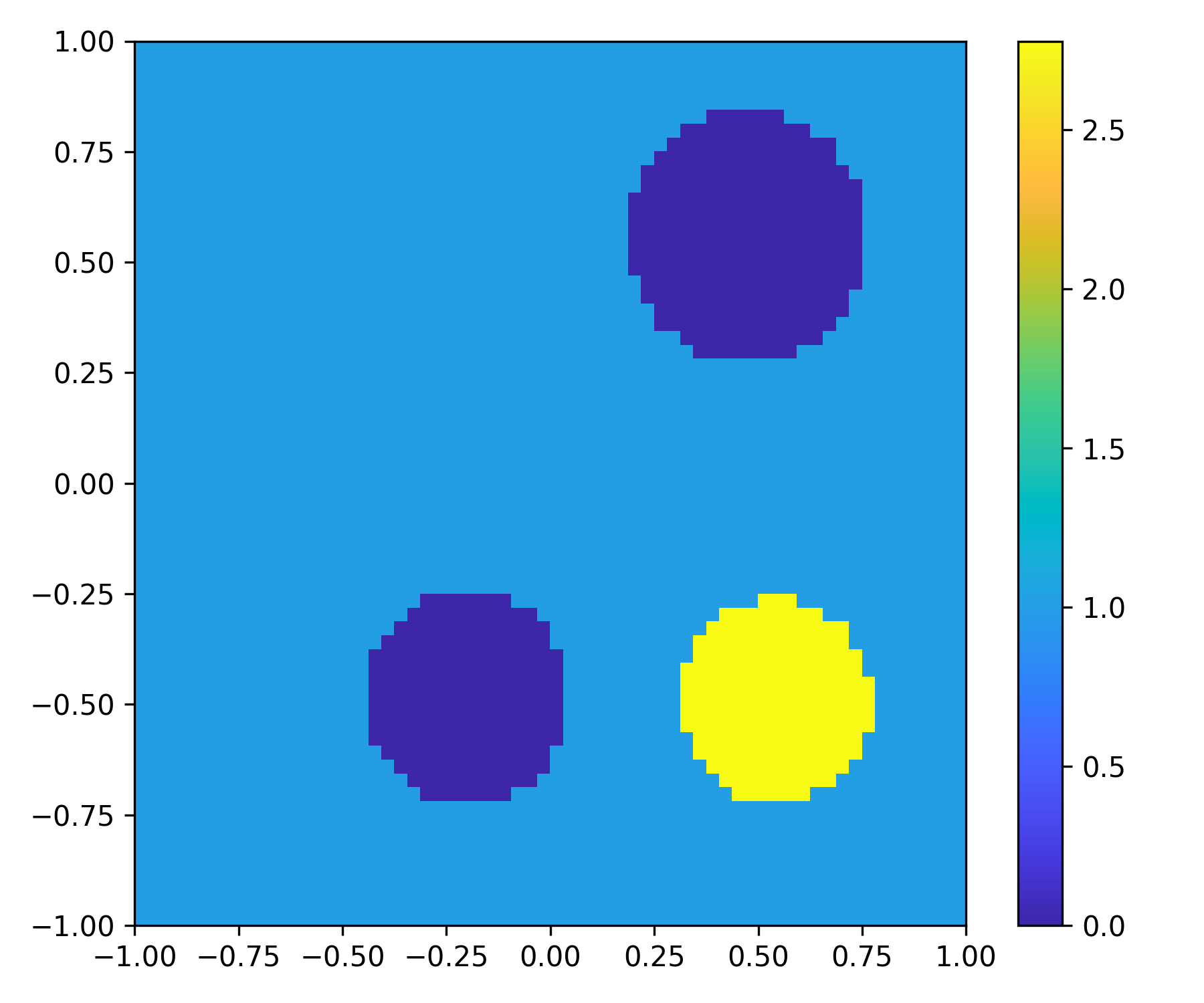}\\
			{$N_i=10$\\ $\delta=5\%$}&
			\includegraphics[width=0.18\textwidth]{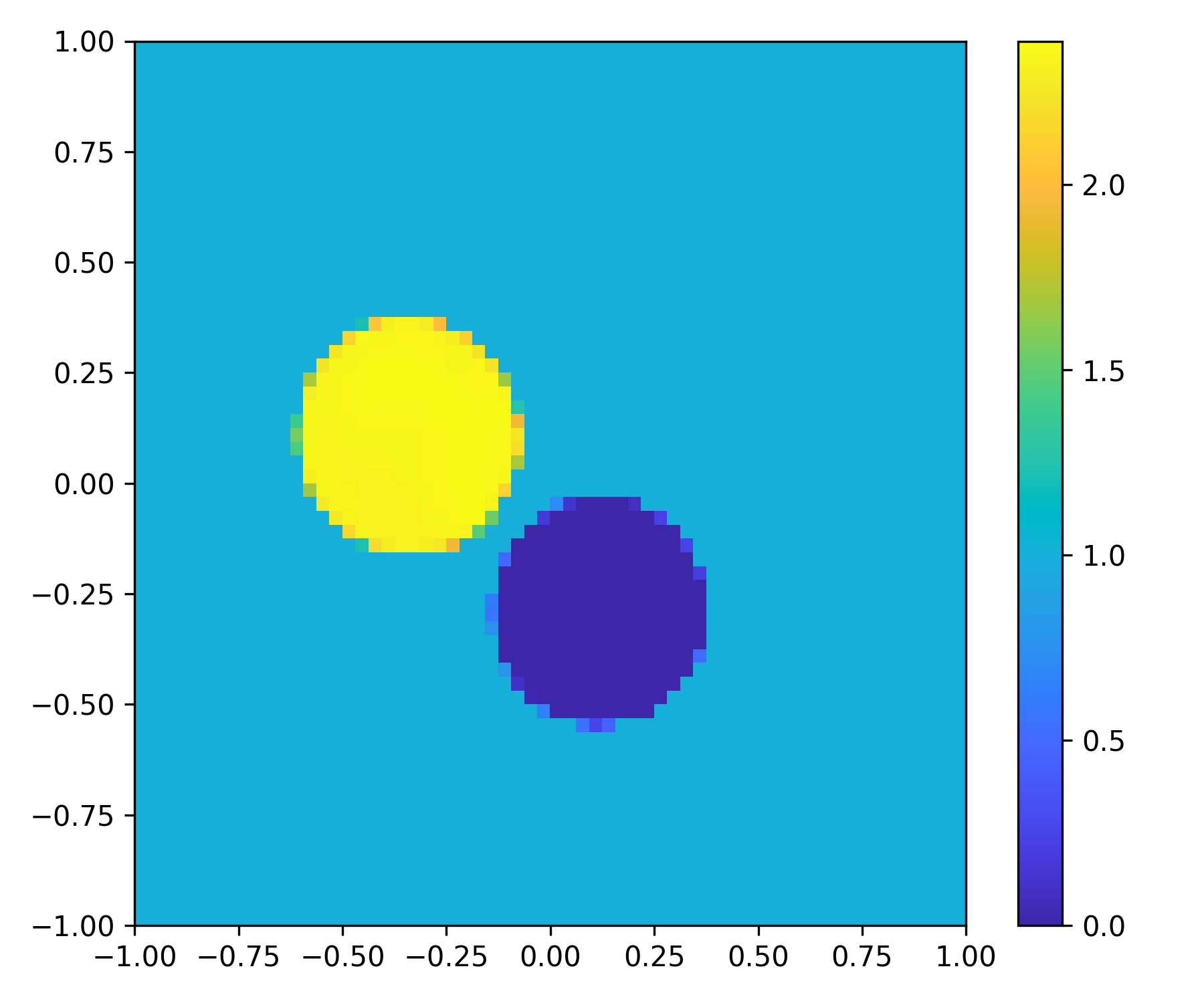}
	    	&\includegraphics[width=0.18\textwidth]{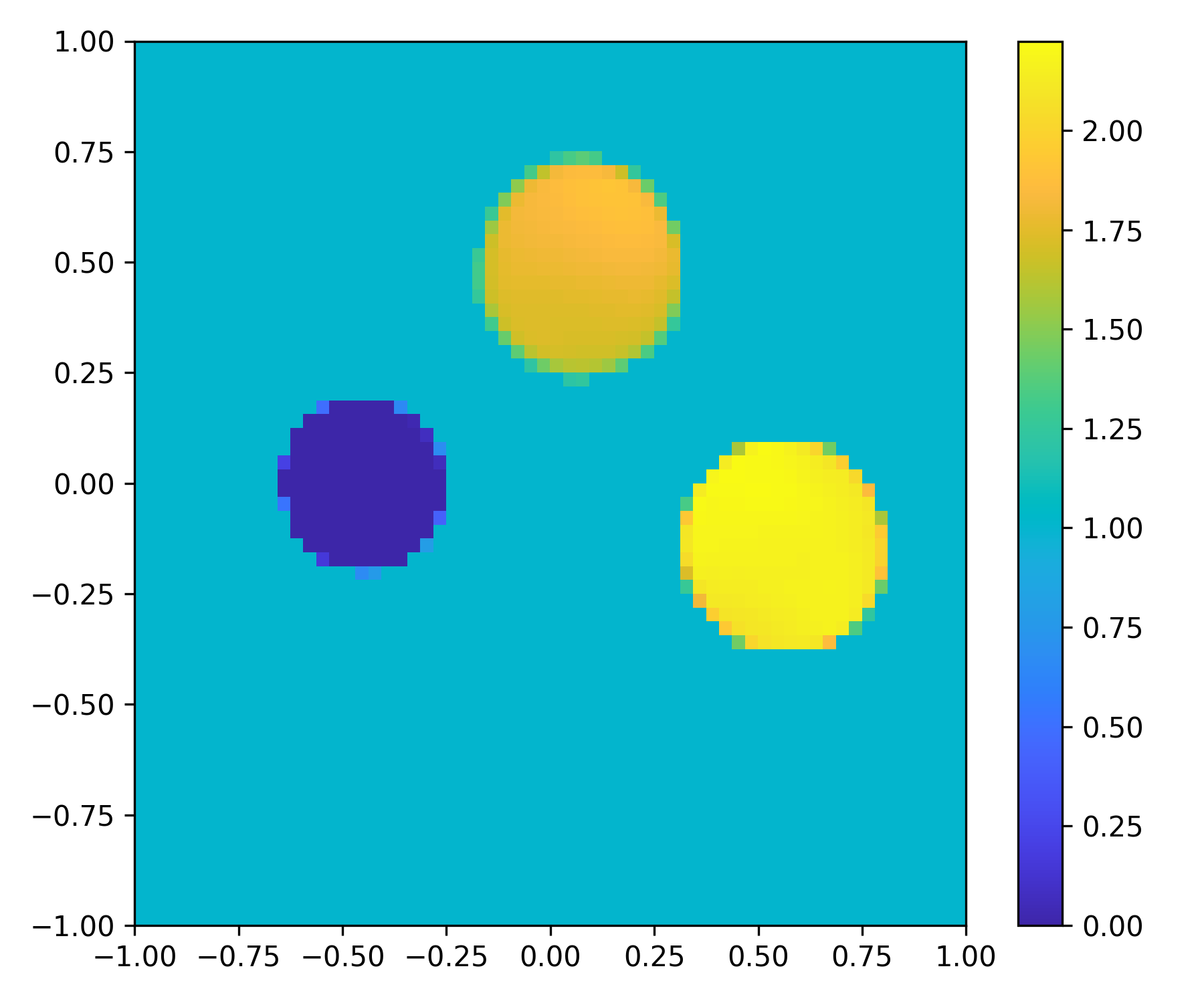} &\includegraphics[width=0.18\textwidth]{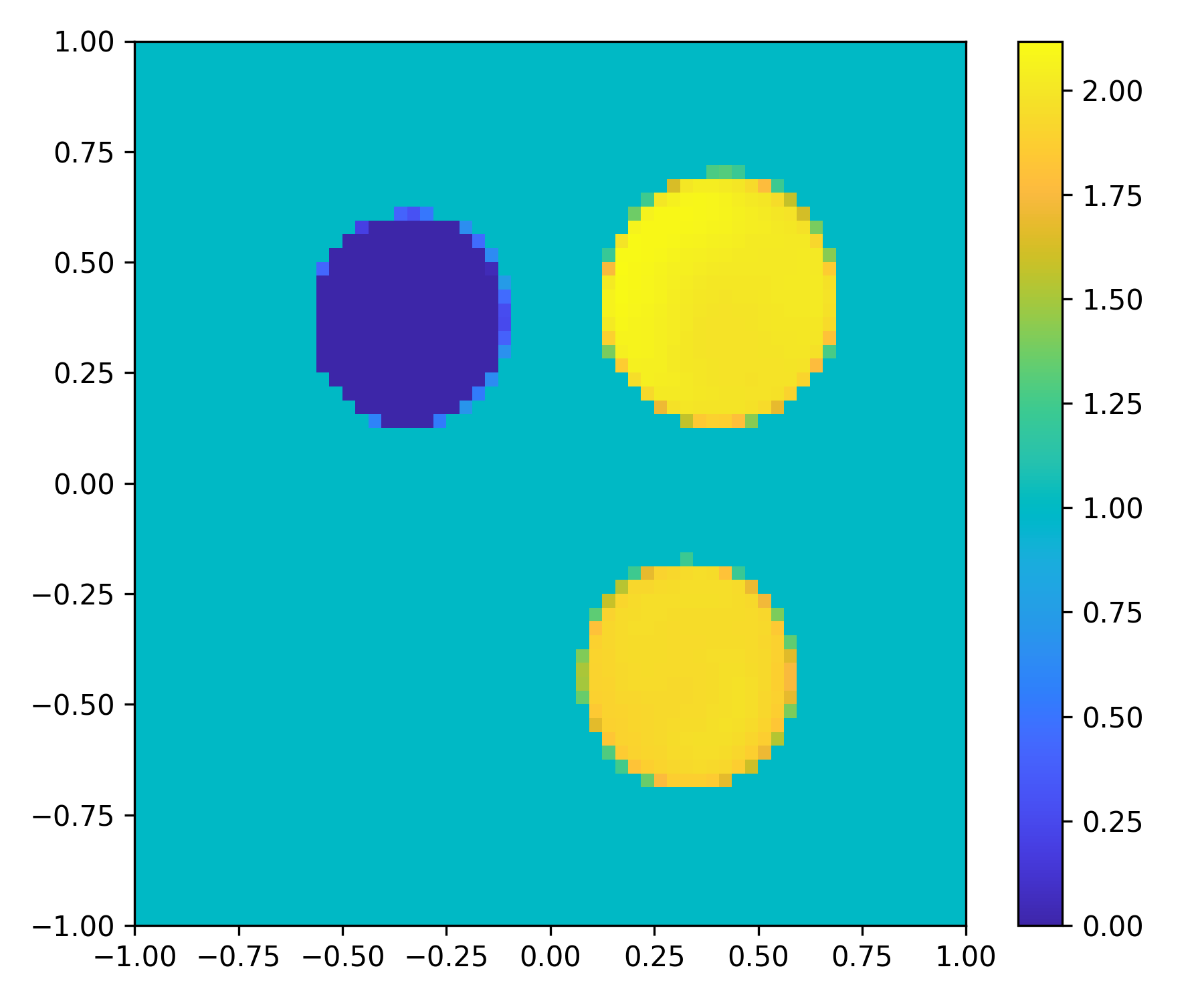}
	     	&\includegraphics[width=0.18\textwidth]{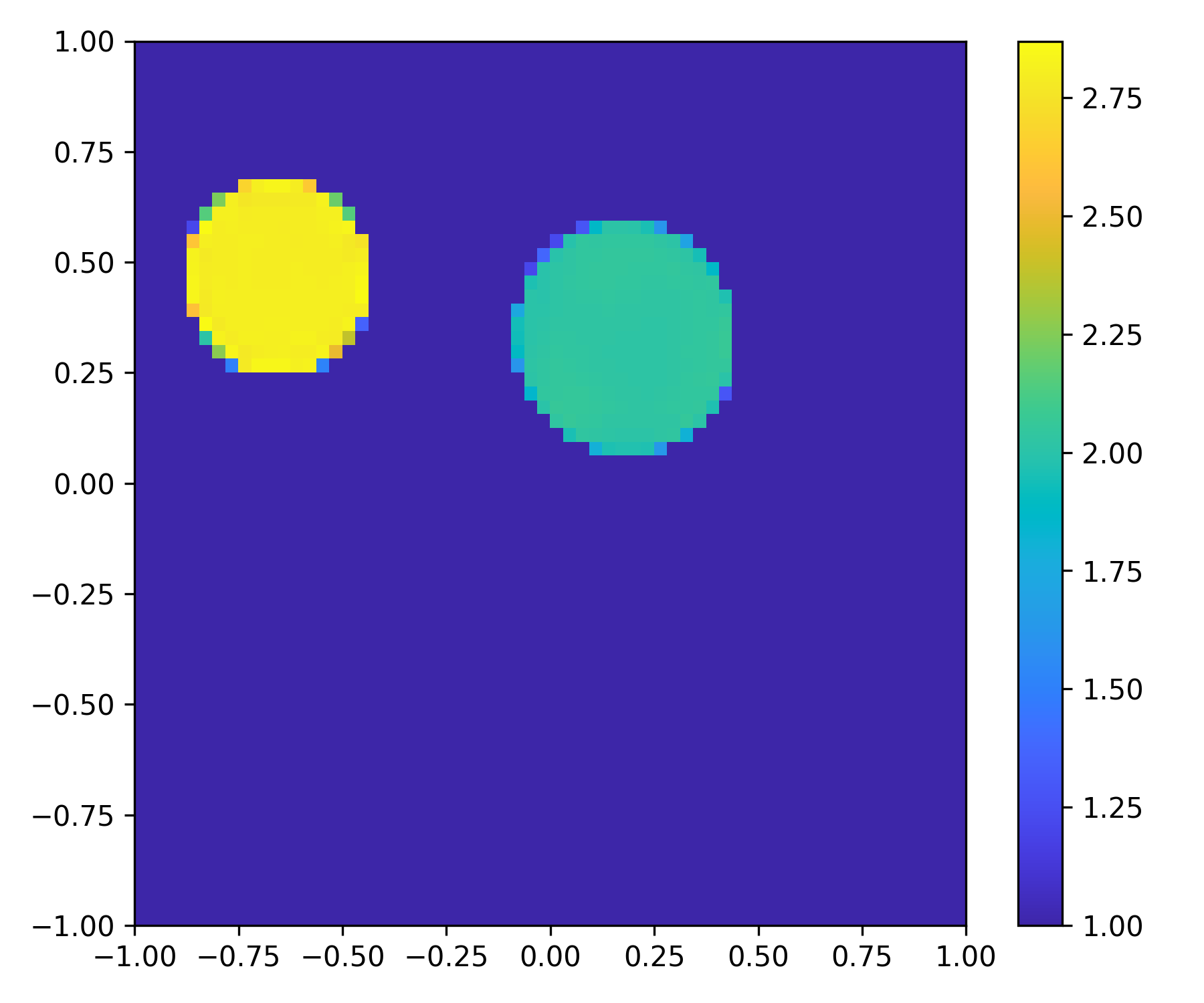}
	    	&\includegraphics[width=0.18\textwidth]{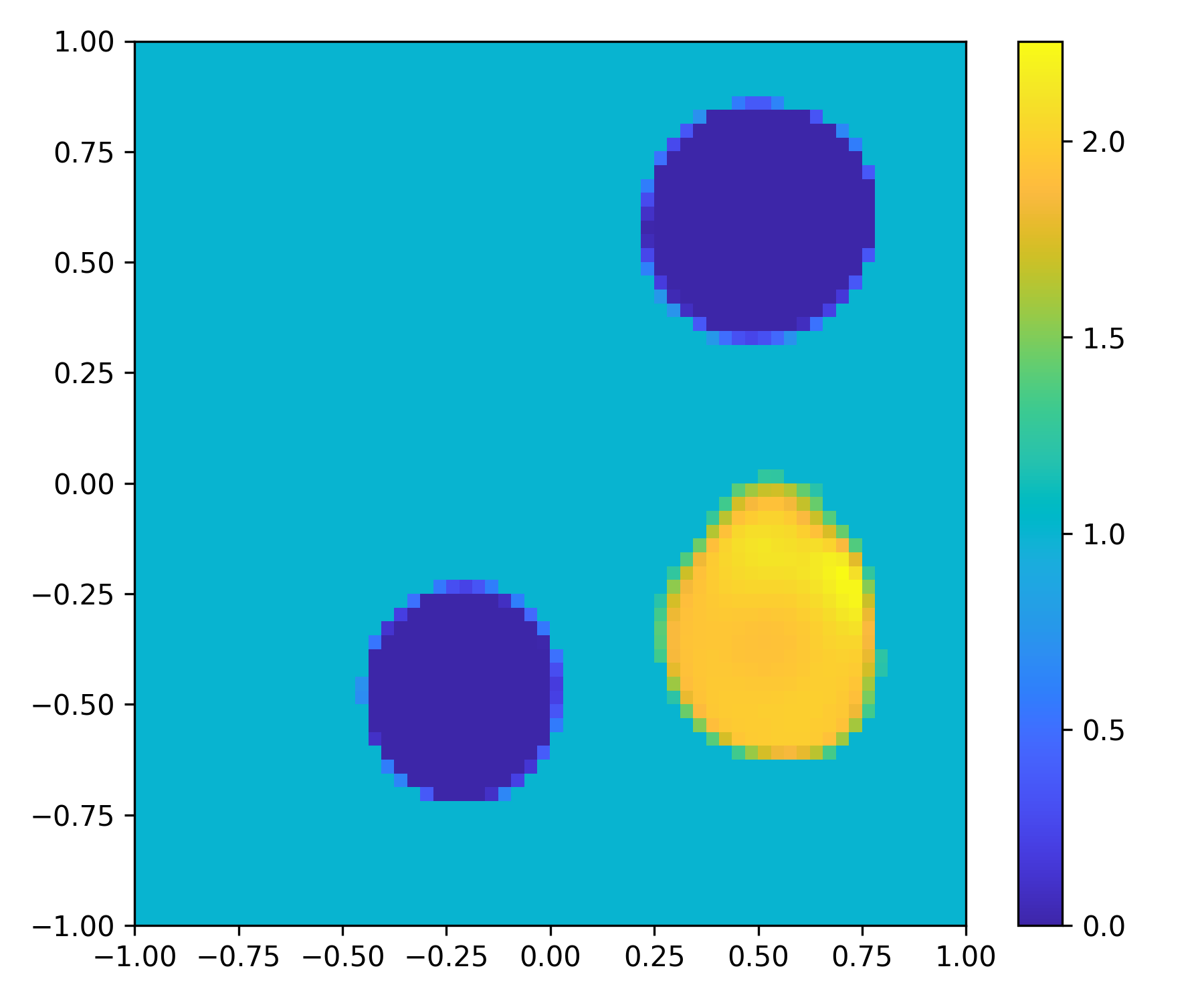}\\
			{$N_i=10$\\ $\delta=10\%$}&
			\includegraphics[width=0.18\textwidth]{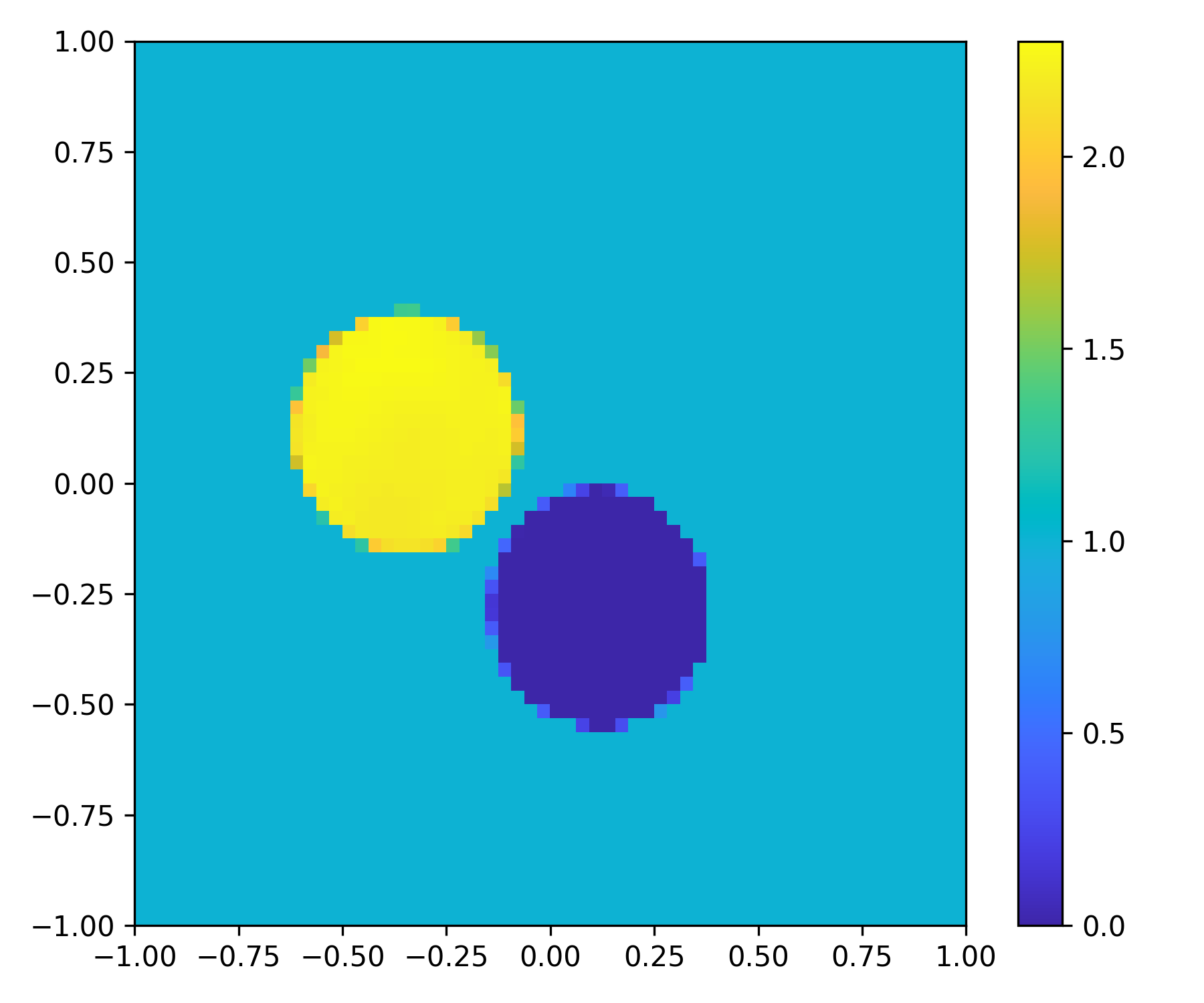}
	    	&\includegraphics[width=0.18\textwidth]{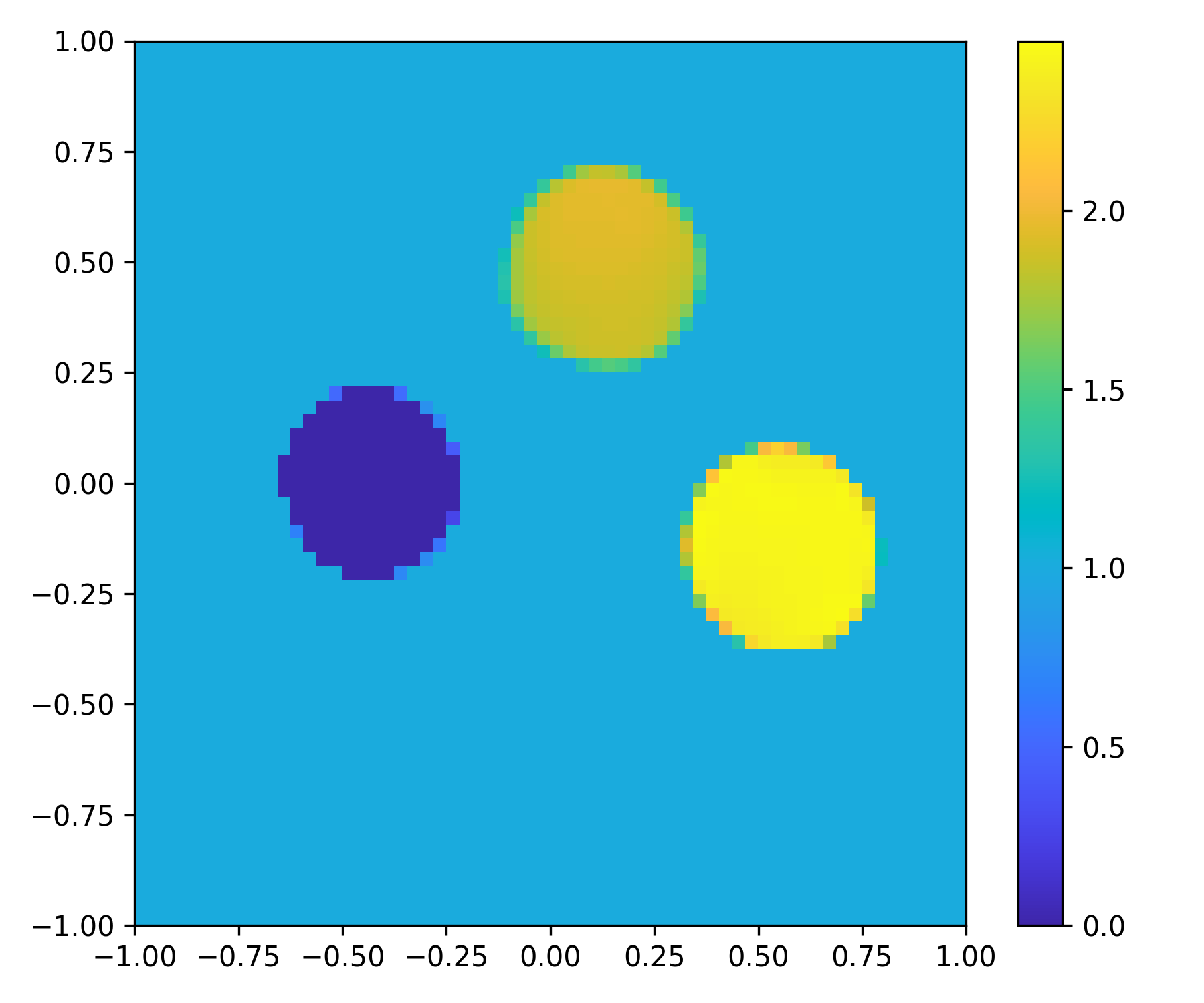} &\includegraphics[width=0.18\textwidth]{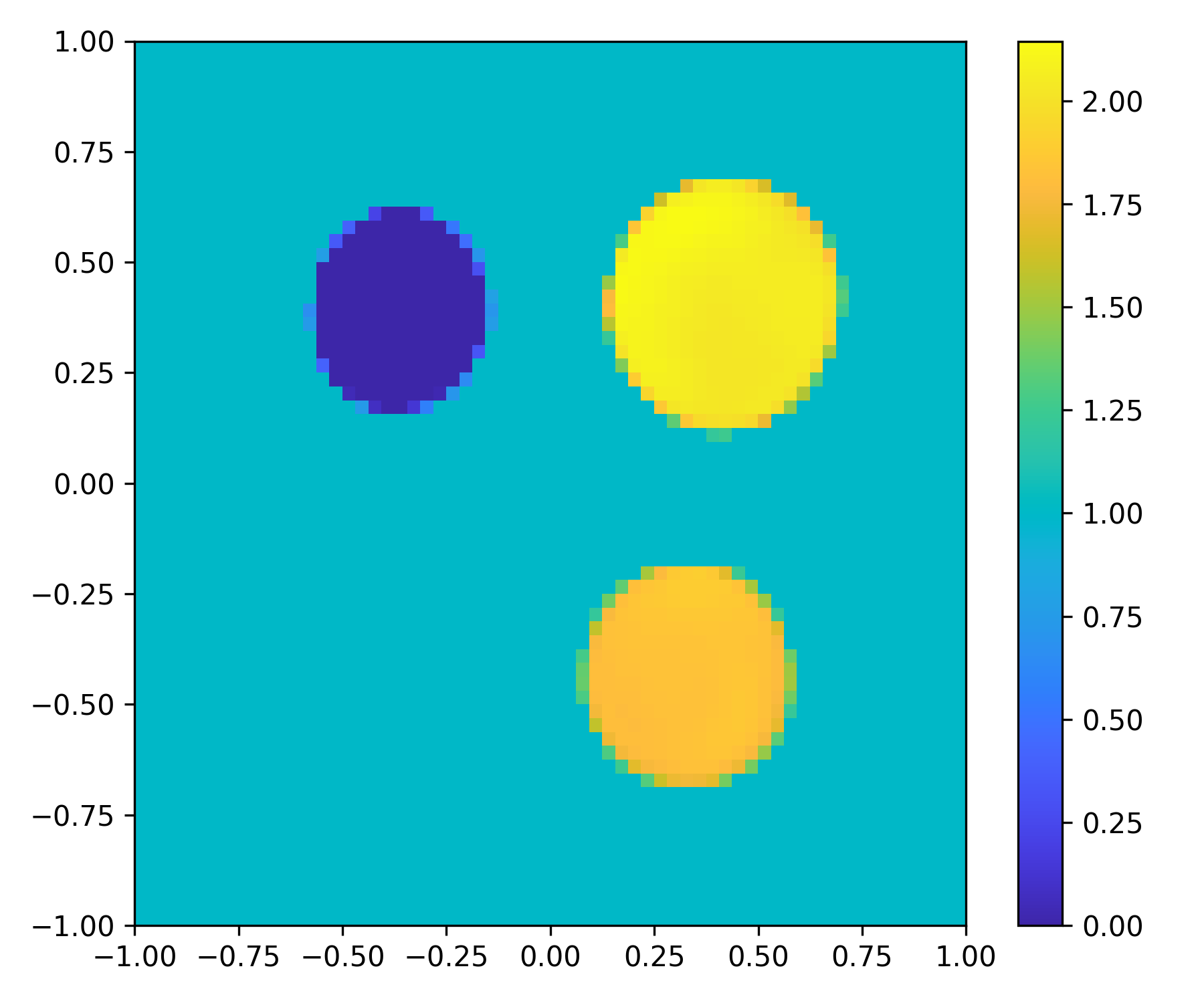}
	    	&\includegraphics[width=0.18\textwidth]{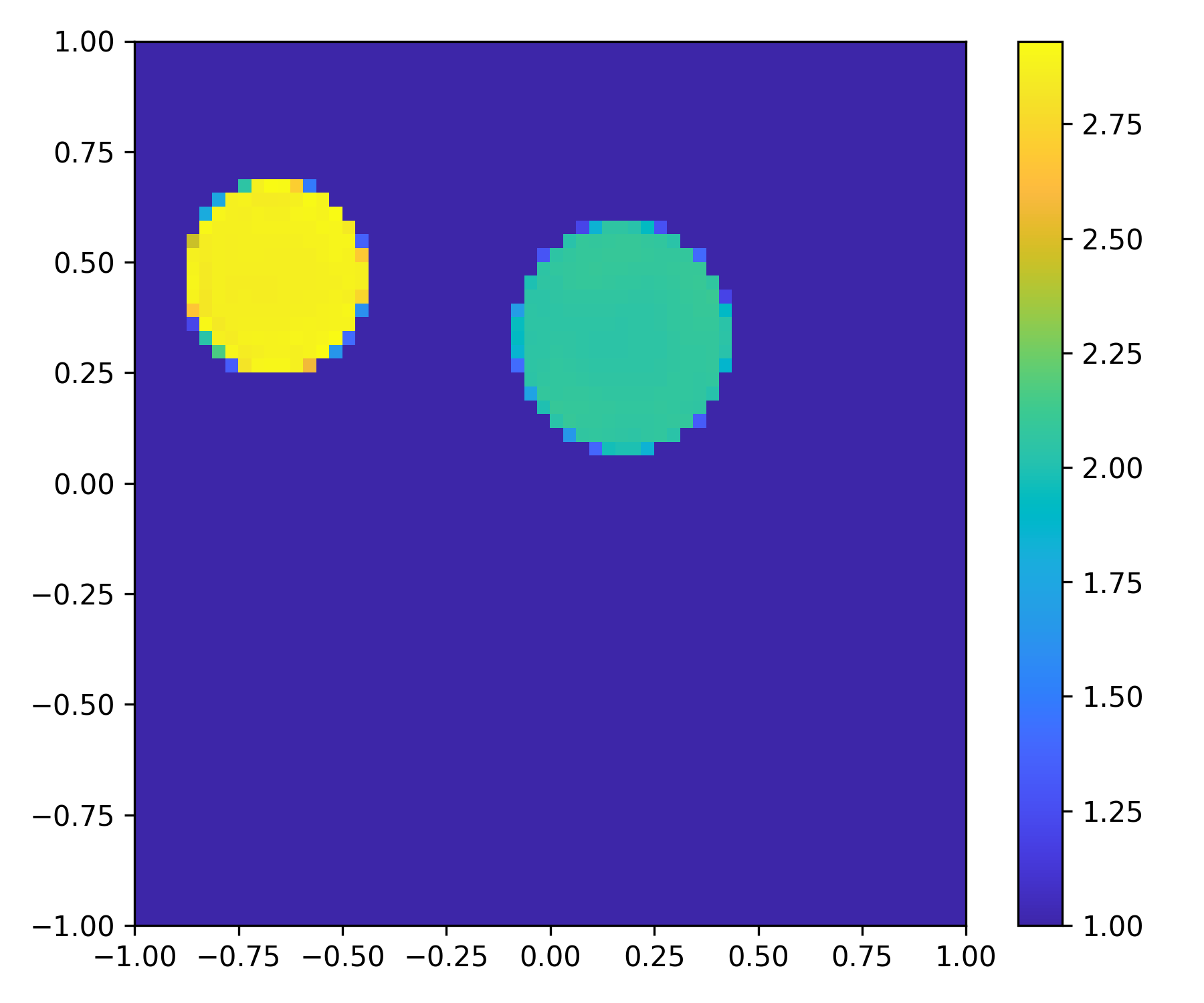}
	    	&\includegraphics[width=0.18\textwidth]{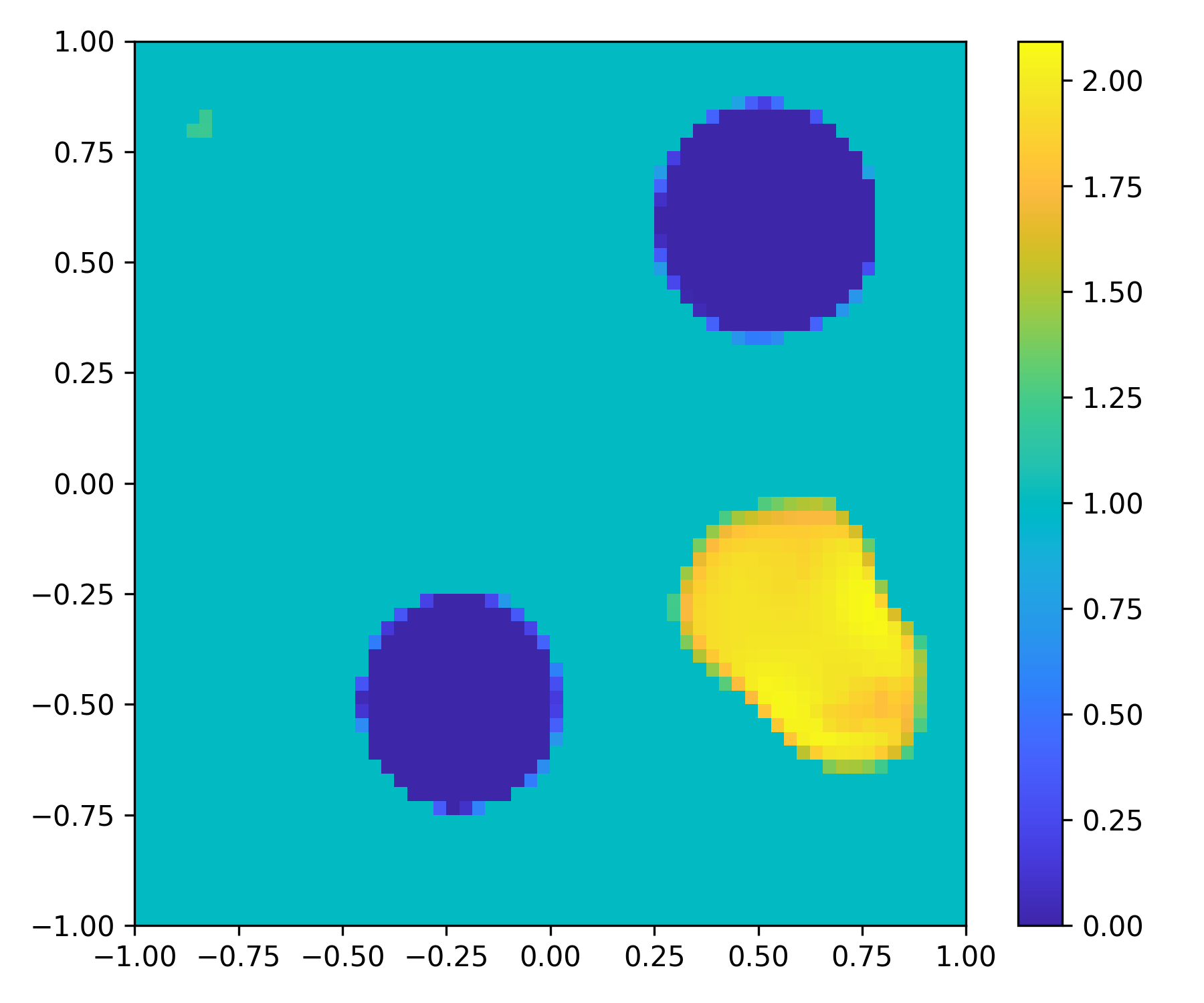}
			
		\end{tblr}
		
		\caption{Example  \ref{examp:mixed}: Reconstructed images from phaseless total fields with different Gaussian noises by using the networks trained by the Mixed circle dataset with phased data. The top row presents the true images, and the other rows are reconstructions with $N_i=10$ corresponding to different noise levels.}
		\label{fig:MixCir}
	\end{center}
\end{figure}

\section{Conclusions}
\label{sec:conclusion}
We have studied the inverse scattering problems under highly challenging conditions, where the measurement data are 
phaseless total fields and available only from one or a few incident waves. 
To address this demanding problem, we introduce a novel technique known as the Direct Sampling Method (DSM) for 
phaseless data. The DSM offers a robust and reasonable means of estimating the shapes and locations of unknown scatterers. Remarkably, our proposed DSM is applicable to both inverse medium and obstacle problems and does not necessitate prior knowledge of the physical properties of the scatterers.

To enhance the precision of our reconstructions, we present the DSM-DL, a fusion of the DSM with deep learning methodologies, offering high-quality reconstructions. The DSM-DL not only recovers the geometrical attributes and positions of the scatterers but also identifies their physical properties and the coefficient values of the medium scatterers. Moreover, it is capable of addressing mixed problems involving both medium and obstacle scatterers. It is worth noting that, supported by both theoretical analysis and numerical results, the DSM-DL network trained with phased data can also effectively tackle problems with phaseless data which implies a broad application of the algorithm.

In conjunction with our present work, there are several promising directions for future exploration. Firstly, a more rigorous theoretical treatment for DSM on phaseless near-field data would be an interesting direction, where we have observed robust results.
Secondly, in the context of DSM, it would be noteworthy to investigate the feasibility of employing a neural network to compute optimal probing functions under specific noise levels.
Lastly, our proposed DSM-DL has the potential to address a wide range of other challenging and practically critical inverse problems, including phaseless reconstruction without information about the incident field, limited aperture scattering problems, and electromagnetic inverse problems.
\\
\\
\textbf{Acknowledgment.}
We would like to thank two anonymous reviewers for many constructive comments and suggestions that have helped us improve our paper significantly.

\bibliography{bibfile}

\end{document}